\documentclass[reqno,a4paper,twoside,english]{amsart}

\usepackage[cmyk,usenames,dvipsnames,svgnames]{xcolor}

\usepackage{amssymb,dsfont,mathrsfs,verbatim,bm,mathtools}
\usepackage{babel}
\usepackage{microtype}

\usepackage{float}
\usepackage[section]{placeins}

\usepackage{tikz, tikz-cd}
\usetikzlibrary{matrix}
\usetikzlibrary{shapes}

\usepackage{lmodern}

\usepackage[raggedright]{titlesec}
\titleformat{\section}
{\normalfont\Large\bfseries\center}{\thesection}{1em}{}
\titleformat{\subsection}
{\normalfont\large\bfseries}{\thesubsection}{1em}{}
\titleformat{\subsubsection}
{\normalfont\normalsize\bfseries}{\thesubsubsection}{1em}{}
\titleformat{\paragraph}[runin]
{\normalfont\normalsize\bfseries}{\theparagraph}{1em}{}
\titleformat{\subparagraph}[runin]
{\normalfont\normalsize\bfseries}{\thesubparagraph}{1em}{}
\titlespacing*{\section} {0pt}{3.5ex plus 1ex minus .2ex}{2.3ex plus .2ex}
\titlespacing*{\subsection} {0pt}{3.25ex plus 1ex minus .2ex}{1.5ex plus .2ex}
\titlespacing*{\subsubsection}{0pt}{3.25ex plus 1ex minus .2ex}{1.5ex plus .2ex}
\titlespacing*{\paragraph} {0pt}{3.25ex plus 1ex minus .2ex}{1em}
\titlespacing*{\subparagraph} {\parindent}{3.25ex plus 1ex minus .2ex}{1em}

\usepackage{titletoc}
\contentsmargin{2.55em}
\dottedcontents{section}[1.5em]{}{2.8em}{1pc}
\dottedcontents{subsection}[3.7em]{}{3.4em}{1pc}
\dottedcontents{subsubsection}[7.6em]{}{3.2em}{1pc}
\dottedcontents{paragraph}[10.3em]{}{3.2em}{1pc}

\newcommand{\Br}{\mathfrak{Br}}
\newcommand{\Bc}{\mathfrak{C}}
\newcommand{\Bo}{\mathfrak{Op}}
\newcommand{\MP}{\mathfrak{MP}}

\newcommand{\ov}{\overline}
\newcommand{\N}{\mathds{N}}
\newcommand{\Z}{\mathds{Z}}
\newcommand{\F}{\mathds{F}}
\newcommand{\C}{\mathds{C}}
\newcommand{\B}{\mathds{B}}
\newcommand{\D}{\mathds{D}}
\newcommand{\ot}{\otimes}
\newcommand{\Sym}{\mathbb{S}}
\newcommand{\Alt}{\mathbb{A}}
\newcommand{\GAP}{\textsf{GAP}}
\newcommand{\MAGMA}{\textsf{Magma}}

\DeclareMathOperator{\comp}{comp}
\DeclareMathOperator{\ev}{ev}
\DeclareMathOperator{\id}{id}
\DeclareMathOperator{\Aff}{Aff}
\DeclareMathOperator{\Hom}{Hom}
\DeclareMathOperator{\End}{End}
\DeclareMathOperator{\PSL}{PSL}
\DeclareMathOperator{\PSU}{PSU}
\DeclareMathOperator{\SL}{SL}
\DeclareMathOperator{\Sz}{Sz}
\DeclareMathOperator{\Sp}{Sp}
\DeclareMathOperator{\Ree}{R}
\DeclareMathOperator{\GO}{O}
\DeclareMathOperator{\ide}{id}
\DeclareMathOperator{\Aut}{Aut}
\DeclareMathOperator{\Inn}{Inn}
\DeclareMathOperator{\ima}{Im}
\DeclareMathOperator{\Hol}{Hol}
\DeclareMathOperator{\rad}{rad}
\DeclareMathOperator{\Soc}{Soc}
\DeclareMathOperator{\op}{op}

\allowdisplaybreaks[3]

\numberwithin{table}{section}
\numberwithin{equation}{section}
\numberwithin{figure}{section}
\theoremstyle{plain}
\newtheorem{thm}{Theorem}[section]
\theoremstyle{plain}
\newtheorem{lem}[thm]{Lemma}
\newtheorem{defn}[thm]{Definition}
\newtheorem{cor}[thm]{Corollary}
\theoremstyle{plain}
\newtheorem{pro}[thm]{Proposition}

\newtheorem{notations}[thm]{Notations}

\newtheorem{example}[thm]{Example}
\theoremstyle{remark}
\newtheorem{rem}[thm]{Remark}

\theoremstyle{plain}
\newtheorem{exa}[thm]{Example}

\begin{document}

\title{Yang--Baxter operators in symmetric categories}

\begin{abstract}
We introduce non-degenerate solutions of the Yang--Baxter equation in the setting of symmetric monoidal categories. Our theory includes non-degenerate set-theoretical solutions as basic examples. However, infinite families of non-degenerate solutions (that are not of set-theoretical type) appear. As in the classical theory of Etingof, Schedler and Soloviev, non-degenerate solutions are classified in terms of invertible 1-cocycles. Braces and matched pairs of cocommutative Hopf algebras (or braiding operators) are also generalized to the context of symmetric monoidal categories and turn out to be equivalent to invertible
1-cocycles.
\end{abstract}

\author{Jorge A. Guccione}
\address{Departamento de Matem\'atica\\ Facultad de Ciencias Exactas y Naturales-UBA, Pabell\'on~1-Ciudad Universitaria\\ Intendente Guiraldes 2160 (C1428EGA) Buenos Aires, Argentina.}
\address{Instituto de Investigaciones Matem\'aticas ``Luis A. Santal\'o"\\ Pabell\'on~1-Ciudad Universitaria\\ Intendente Guiraldes 2160 (C1428EGA) Buenos Aires, Argentina.}
\email{vander@dm.uba.ar}

\author{Juan J. Guccione}
\address{Departamento de Matem\'atica\\ Facultad de Ciencias Exactas y Naturales-UBA, Pabell\'on~1-Ciudad Universitaria\\ Intendente Guiraldes 2160 (C1428EGA) Buenos Aires, Argentina.}
\address{Instituto Argentino de Matem\'atica-CONICET\\ Saavedra 15 3er piso\\ (C1083ACA) Buenos Aires, Argentina.}
\email{jjgucci@dm.uba.ar}

\author{Leandro Vendramin}
\address{Departamento de Matem\'atica\\ Facultad de Ciencias Exactas y Naturales-UBA, Pabell\'on~1-Ciudad Universitaria\\ Intendente Guiraldes 2160 (C1428EGA) Buenos Aires, Argentina.}
\address{Instituto de Investigaciones Matem\'aticas ``Luis A. Santal\'o"\\ Pabell\'on~1-Ciudad Universitaria\\ Intendente Guiraldes 2160 (C1428EGA) Buenos Aires, Argentina.}
\email{lvendramin@dm.uba.ar}

\thanks{Jorge A. Guccione and Juan J. Guccione were supported by UBACyT
20020150100153BA (UBA) and PIP 11220110100800CO (CONICET). Leandro Vendramin was
partially supported by CONICET, PICT-2014-1376, MATH-AmSud and ICTP}


\maketitle

\setcounter{tocdepth}{1}
\tableofcontents

\section*{Introduction}

The celebrated Yang--Baxter equation was first established by Yang~\cite{MR0261870} in 1967 and then by Baxter~\cite{MR0290733} in 1972. Since then, more solutions of various forms of the Yang--Baxter equation have been constructed by physicists and mathematicians. In the past three decades this equation has been widely studied from different perspectives and attracted the attention of a wide range of mathematicians because of the applications to knot theory, representations of braid groups, Hopf algebras and quantum groups, operator theory, etc.

Let $V$ be a vector space and let $R\colon V\otimes V\to V\otimes V$ be a linear isomorphism. We say that $R$ satisfies the {\em Yang-Baxter equation} on $V$ if
\begin{equation*}
R_{12}\circ R_{13}\circ R_{23}=R_{23}\circ R_{13}\circ R_{12}\quad\text{in $\End(V\ot V\ot V)$,}
\end{equation*}
where $R_{ij}$ means $R$ acting in the $i$-th and $j$-th components. It is easy to check that this occurs if and only if $r\coloneqq \tau\circ R$, where $\tau\colon V\ot V \to V\ot V$ denotes the flip, satisfies the {\em braid equation}
\begin{equation}
r_{12}\circ r_{23}\circ r_{12}=r_{23}\circ r_{12}\circ r_{23}.\label{eq:trenzas}
\end{equation}

The importance of this equation in mathematics and physics led Drinfeld~\cite{MR1183474} to ask for the simplest family of solutions, the so-called set-theoretical (or combinatorial) solutions, i.e. pairs $(X,r)$, where $X$ is a set and $r\colon X\times X\to X\times X$ is an invertible map satisfying (\ref{eq:trenzas}).  Each one of these solutions yields in an evident way a linear solution on the vector space with basis~$X$.

This approach was first considered by Etingof, Schedler and Soloviev~\cite{MR1722951} and Gateva-Ivanova and Van den Bergh~\cite{MR1637256} for involutive solutions and by Lu, Yan and Zhu~\cite{MR1769723}, and Soloviev~\cite{MR1809284} for non-involutive solutions. Now it is known that there are several connections between solutions and bijective $1$-cocycles, Biberbach groups and groups of I-type, Garside structures, biracks, cyclic sets, braces, Hopf algebras, left symmetric algebras, etc.; see for example~\cite{AGV,MR3177933,MR3374524,GI,MR3558231,MR2278047,MR3291816,MR1809284}.

In this paper we study the set-type solutions of the braid equation in the context of symmetric tensor categories. The underlying idea is simple: to use cocommutative coalgebras in symmetric tensor categories to generalize sets. This basic idea leads us to infinite families of new examples and several new theoretical developments. Despite we are mainly interested in solutions lying in the category of vector spaces, we work in the most general setting  of symmetric categories. This is free of cost and has many potential applications.

In order to carry out this task we must translate to categorical language the notion of non-degenerate map. Although we follow the ideas of the proofs given in~\cite{MR1722951},~\cite{MR1769723} and~\cite{MR1809284}, our arguments are not always the ones given in those papers. For instance one cannot adapt the arguments of Proposition~1.8 of \cite{MR1769723} to prove our Proposition~\ref{pro:YB}, which is its categorical version.

Our solutions turn out to be classified in terms of invertible $1$-cocycles. This construction is similar to that of Lu, Yan and Zhu and has many different equivalent formulations including braces and matched pairs of coalgebras.

\medskip

The paper is organized as follows. In Section~\ref{preliminaries} we review the basic notions needed to deal with the universal construction of a solution. In Section~\ref{YB} we define the concept of non-degenerate solution and study their main properties. Although we are working in an arbitrary symmetric monoidal category~$\mathscr{C}$, we call a pair $(X,r)$ consisting of a cocommutative coalgebra in~$\mathscr{C}$ and a non-degenerate solution, a non-degenerate braided set. Section~\ref{derived} is devoted to studying racks in arbitrary tensor categories; we prove that non-degenerate solutions produce racks and study the representations of the braid group induced by each one of these solutions and its rack. In Section~\ref{section: Braiding operators, invertible cocycles and braces} we introduce the categories of braces, braided operators and invertible cocycles, and in Theorem~\ref{thm:equivalencias} we prove that they are equivalent. In Section~\ref{LYZ},  Theorem~\ref{thm:main}, we prove our main result: the existence of the universal solutions. Finally, in Section~\ref{examples} we study solutions over $k\oplus V$, where $k$ is a field, $V$ is the vector space of primitive elements and $1\in k$ is the unique group-like element.

\section{Preliminaries} \label{preliminaries}

\subsubsection*{Notations}
Throughout this paper $\mathscr{C}=(\mathcal{C},\ot,\mathds{1},\alpha,\lambda,\varrho, c)$ is a symmetric monoidal category. Given an object $X\in \mathcal{C}$, we let $X^n$ denote the $n$-fold tensor product of $X$. We assume that the reader is familiar with the notions of algebra, coalgebra, bialgebra and Hopf algebra in $\mathscr{C}$, and we  use the well known graphic calculus for symmetric monoidal categories. As usual the morphisms are composed from top to bottom and the tensor product of two morphisms will be represented by horizontal concatenation. The identity morphism of an object will be represented by a vertical line and the symmetry isomorphism $c$, by the diagram
\begin{tikzpicture}
\def\flip(#1,#2)[#3]{\draw (
#1+1*#3,#2) .. controls (#1+1*#3,#2-0.05*#3) and (#1+0.96*#3,#2-0.15*#3).. (#1+0.9*#3,#2-0.2*#3)
(#1+0.1*#3,#2-0.8*#3)--(#1+0.9*#3,#2-0.2*#3)
(#1,#2-1*#3) .. controls (#1,#2-0.95*#3) and (#1+0.04*#3,#2-0.85*#3).. (#1+0.1*#3,#2-0.8*#3)
(#1,#2) .. controls (#1,#2-0.05*#3) and (#1+0.04*#3,#2-0.15*#3).. (#1+0.1*#3,#2-0.2*#3)
(#1+0.1*#3,#2-0.2*#3) -- (#1+0.9*#3,#2-0.8*#3)
(#1+1*#3,#2-1*#3) .. controls (#1+1*#3,#2-0.95*#3) and (#1+0.96*#3,#2-0.85*#3).. (#1+0.9*#3,#2-0.8*#3)
}
\begin{scope}[xshift=0cm, yshift=0cm]
\flip(0,0)[0.28];
\end{scope}
\end{tikzpicture}. For each algebra $A$ the diagrams
$$
\begin{tikzpicture}[scale=0.395]
\def\mult(#1,#2)[#3]{\draw (#1,#2) .. controls (#1,#2-0.555*#3/2) and (#1+0.445*#3/2,#2-#3/2) .. (#1+#3/2,#2-#3/2) .. controls (#1+1.555*#3/2,#2-#3/2) and (#1+2*#3/2,#2-0.555*#3/2) .. (#1+2*#3/2,#2) (#1+#3/2,#2-#3/2) -- (#1+#3/2,#2-2*#3/2)}
\def\unit(#1,#2){\draw (#1,#2) circle[radius=2pt] (#1,#2-0.07) -- (#1,#2-1)}
\def\laction(#1,#2)[#3,#4]{\draw (#1,#2) .. controls (#1,#2-0.555*#4/2) and (#1+0.445*#4/2,#2-1*#4/2) .. (#1+1*#4/2,#2-1*#4/2) -- (#1+2*#4/2+#3*#4/2,#2-1*#4/2) (#1+2*#4/2+#3*#4/2,#2)--(#1+2*#4/2+#3*#4/2,#2-2*#4/2)}
\def\raction(#1,#2)[#3,#4]{\draw (#1,#2) -- (#1,#2-2*#4/2)  (#1,#2-1*#4/2)--(#1+1*#4/2+#3*#4/2,#2-1*#4/2) .. controls (#1+1.555*#4/2+#3*#4/2,#2-1*#4/2) and (#1+2*#4/2+#3*#4/2,#2-0.555*#4/2) .. (#1+2*#4/2+#3*#4/2,#2)}
\begin{scope}[xshift=0cm,yshift=0cm]
\mult(0,0)[1];
\end{scope}
\begin{scope}[xshift=0cm,yshift=0cm]
\node at (3.5,-0.5){\text{and}};
\end{scope}
\begin{scope}[xshift=0cm,yshift=0cm]
\unit(6,0);
\end{scope}
\end{tikzpicture}
$$
stand for the multiplication and the unit, respectively, and for each coalgebra $X$, the diagrams
$$
\begin{tikzpicture}[scale=0.395]
\def\comult(#1,#2)[#3]{\draw (#1,#2)-- (#1,#2-1*#3/2) (#1,#2-1*#3/2) .. controls (#1+0.555*#3/2,#2-1*#3/2) and (#1+1*#3/2,#2-1.445*#3/2) .. (#1+1*#3/2,#2-2*#3/2) (#1,#2-1*#3/2) .. controls (#1-0.555*#3/2,#2-1*#3/2) and (#1-1*#3/2,#2-1.445*#3/2) .. (#1-1*#3/2,#2-2*#3/2)}
\def\counit(#1,#2){\draw (#1,#2) -- (#1,#2-0.93) (#1,#2-1) circle[radius=2pt]}
\def\rcoaction(#1,#2)[#3,#4]{\draw (#1,#2)-- (#1,#2-2*#4/2) (#1,#2-1*#4/2) -- (#1+1*#4/2+#3*#4/2,#2-1*#4/2).. controls (#1+1.555*#4/2+#3*#4/2,#2-1*#4/2) and (#1+2*#4/2+#3*#4/2,#2-1.445*#4/2) .. (#1+2*#4/2+#3*#4/2,#2-2*#4/2)}
\begin{scope}[xshift=0cm,yshift=0cm]
\comult(0,0)[1];
\end{scope}
\begin{scope}[xshift=0cm,yshift=0cm]
\node at (3,-0.5){\text{and}};
\end{scope}
\begin{scope}[xshift=0cm,yshift=0cm]
\counit(5.5,0);
\end{scope}
\end{tikzpicture}
$$
stand for the comultiplication and the counit. Furthermore, each coalgebra morphism $f\colon X\to Y$, each coalgebra morphism $R\colon X^2\to X^2$ and each candidate $r\colon X^2 \to X^2$ to be a solution of the braid equation will be represent by the diagrams
$$
\begin{tikzpicture}[scale=0.395]
\def\counit(#1,#2){\draw (#1,#2) -- (#1,#2-0.93) (#1,#2-1) circle[radius=2pt]}
\def\comult(#1,#2)[#3,#4]{\draw (#1,#2) -- (#1,#2-0.5*#4) arc (90:0:0.5*#3 and 0.5*#4) (#1,#2-0.5*#4) arc (90:180:0.5*#3 and 0.5*#4)}
\def\laction(#1,#2)[#3,#4]{\draw (#1,#2) .. controls (#1,#2-0.555*#4/2) and (#1+0.445*#4/2,#2-1*#4/2) .. (#1+1*#4/2,#2-1*#4/2) -- (#1+2*#4/2+#3*#4/2,#2-1*#4/2) (#1+2*#4/2+#3*#4/2,#2)--(#1+2*#4/2+#3*#4/2,#2-2*#4/2)}
\def\map(#1,#2)[#3]{\draw (#1,#2-0.5)  node[name=nodemap,inner sep=0pt,  minimum size=10pt, shape=circle,draw]{$#3$} (#1,#2)-- (nodemap)  (nodemap)-- (#1,#2-1)}
\def\solbraid(#1,#2)[#3]{\draw (#1,#2-0.5)  node[name=nodemap,inner sep=0pt,  minimum size=9pt, shape=circle,draw]{$#3$}
(#1-0.5,#2) .. controls (#1-0.5,#2-0.15) and (#1-0.4,#2-0.2) .. (#1-0.3,#2-0.3) (#1-0.3,#2-0.3) -- (nodemap)
(#1+0.5,#2) .. controls (#1+0.5,#2-0.15) and (#1+0.4,#2-0.2) .. (#1+0.3,#2-0.3) (#1+0.3,#2-0.3) -- (nodemap)
(#1+0.5,#2-1) .. controls (#1+0.5,#2-0.85) and (#1+0.4,#2-0.8) .. (#1+0.3,#2-0.7) (#1+0.3,#2-0.7) -- (nodemap)
(#1-0.5,#2-1) .. controls (#1-0.5,#2-0.85) and (#1-0.4,#2-0.8) .. (#1-0.3,#2-0.7) (#1-0.3,#2-0.7) -- (nodemap)
}
\def\flip(#1,#2)[#3]{\draw (
#1+1*#3,#2) .. controls (#1+1*#3,#2-0.05*#3) and (#1+0.96*#3,#2-0.15*#3).. (#1+0.9*#3,#2-0.2*#3)
(#1+0.1*#3,#2-0.8*#3)--(#1+0.9*#3,#2-0.2*#3)
(#1,#2-1*#3) .. controls (#1,#2-0.95*#3) and (#1+0.04*#3,#2-0.85*#3).. (#1+0.1*#3,#2-0.8*#3)
(#1,#2) .. controls (#1,#2-0.05*#3) and (#1+0.04*#3,#2-0.15*#3).. (#1+0.1*#3,#2-0.2*#3)
(#1+0.1*#3,#2-0.2*#3) -- (#1+0.9*#3,#2-0.8*#3)
(#1+1*#3,#2-1*#3) .. controls (#1+1*#3,#2-0.95*#3) and (#1+0.96*#3,#2-0.85*#3).. (#1+0.9*#3,#2-0.8*#3)
}
\def\raction(#1,#2)[#3,#4]{\draw (#1,#2) -- (#1,#2-2*#4/2)  (#1,#2-1*#4/2)--(#1+1*#4/2+#3*#4/2,#2-1*#4/2) .. controls (#1+1.555*#4/2+#3*#4/2,#2-1*#4/2) and (#1+2*#4/2+#3*#4/2,#2-0.555*#4/2) .. (#1+2*#4/2+#3*#4/2,#2)}
\def\doublemap(#1,#2)[#3]{\draw (#1+0.5,#2-0.5) node [name=doublemapnode,inner xsep=0pt, inner ysep=0pt, minimum height=9pt, minimum width=25pt,shape=rectangle,draw,rounded corners] {$#3$} (#1,#2) .. controls (#1,#2-0.075) .. (doublemapnode) (#1+1,#2) .. controls (#1+1,#2-0.075).. (doublemapnode) (doublemapnode) .. controls (#1,#2-0.925)..(#1,#2-1) (doublemapnode) .. controls (#1+1,#2-0.925).. (#1+1,#2-1)}
\begin{scope}[xshift=0cm, yshift=0cm]
\map(0,0)[\scriptstyle f];
\end{scope}
\begin{scope}[xshift=0.7cm, yshift=-0.25cm]
\node at (0,-0.5){,};
\end{scope}
\begin{scope}[xshift=0cm, yshift=0cm]
\doublemap(3,0)[\scriptstyle R];
\end{scope}
\begin{scope}[xshift=0cm,yshift=0cm]
\node at (6.9,-0.5){\text{and}};
\end{scope}
\begin{scope}[xshift=0cm, yshift=0cm]
\solbraid(9.85,0)[\scriptstyle r];
\end{scope}
\begin{scope}[xshift=10.75cm, yshift=-0.25cm]
\node at (0,-0.5){,};
\end{scope}
\end{tikzpicture}
$$
respectively (note that we use different diagrams to represent $R$ and $r$ in spite of both are endomorphisms of $X^2$). Moreover, the first and second coordinate maps $\sigma$ and $\tau$ of~$r$  (see Definition~\ref{def: funciones coordenadas}), and their inverses $\sigma^{-1}$ and $\tau^{-1}$ (when $r$ is non-degenerate), will be represented by
$$
\begin{tikzpicture}[scale=0.395]
\def\counit(#1,#2){\draw (#1,#2) -- (#1,#2-0.93) (#1,#2-1) circle[radius=2pt]}
\def\comult(#1,#2)[#3,#4]{\draw (#1,#2) -- (#1,#2-0.5*#4) arc (90:0:0.5*#3 and 0.5*#4) (#1,#2-0.5*#4) arc (90:180:0.5*#3 and 0.5*#4)}
\def\laction(#1,#2)[#3,#4]{\draw (#1,#2) .. controls (#1,#2-0.555*#4/2) and (#1+0.445*#4/2,#2-1*#4/2) .. (#1+1*#4/2,#2-1*#4/2) -- (#1+2*#4/2+#3*#4/2,#2-1*#4/2) (#1+2*#4/2+#3*#4/2,#2)--(#1+2*#4/2+#3*#4/2,#2-2*#4/2)}
\def\map(#1,#2)[#3]{\draw (#1,#2-0.5)  node[name=nodemap,inner sep=0pt,  minimum size=10pt, shape=circle,draw]{$#3$} (#1,#2)-- (nodemap)  (nodemap)-- (#1,#2-1)}
\def\lactiontr(#1,#2)[#3,#4,#5]{\draw (#1,#2) .. controls (#1,#2-0.555*#4/2) and (#1+0.445*#4/2,#2-1*#4/2) .. (#1+1*#4/2,#2-1*#4/2) -- (#1+2*#4/2+#3*#4/2,#2-1*#4/2)  node [inner sep=0pt, minimum size=3pt,shape=isosceles triangle,fill, shape border rotate=#5] {} (#1+2*#4/2+#3*#4/2,#2) --(#1+2*#4/2+#3*#4/2,#2-2*#4/2)}
\def\ractiontr(#1,#2)[#3,#4,#5]{\draw (#1,#2) -- (#1,#2-2*#4/2)  (#1,#2-1*#4/2) node [inner sep=0pt, minimum size=3pt,shape=isosceles triangle,fill, shape border rotate=#5] {}  --(#1+1*#4/2+#3*#4/2,#2-1*#4/2) .. controls (#1+1.555*#4/2+#3*#4/2,#2-1*#4/2) and (#1+2*#4/2+#3*#4/2,#2-0.555*#4/2) .. (#1+2*#4/2+#3*#4/2,#2)  }
\def\solbraid(#1,#2)[#3]{\draw (#1,#2-0.5)  node[name=nodemap,inner sep=0pt,  minimum size=9pt, shape=circle,draw]{$#3$}
(#1-0.5,#2) .. controls (#1-0.5,#2-0.15) and (#1-0.4,#2-0.2) .. (#1-0.3,#2-0.3) (#1-0.3,#2-0.3) -- (nodemap)
(#1+0.5,#2) .. controls (#1+0.5,#2-0.15) and (#1+0.4,#2-0.2) .. (#1+0.3,#2-0.3) (#1+0.3,#2-0.3) -- (nodemap)
(#1+0.5,#2-1) .. controls (#1+0.5,#2-0.85) and (#1+0.4,#2-0.8) .. (#1+0.3,#2-0.7) (#1+0.3,#2-0.7) -- (nodemap)
(#1-0.5,#2-1) .. controls (#1-0.5,#2-0.85) and (#1-0.4,#2-0.8) .. (#1-0.3,#2-0.7) (#1-0.3,#2-0.7) -- (nodemap)
}
\def\flip(#1,#2)[#3]{\draw (
#1+1*#3,#2) .. controls (#1+1*#3,#2-0.05*#3) and (#1+0.96*#3,#2-0.15*#3).. (#1+0.9*#3,#2-0.2*#3)
(#1+0.1*#3,#2-0.8*#3)--(#1+0.9*#3,#2-0.2*#3)
(#1,#2-1*#3) .. controls (#1,#2-0.95*#3) and (#1+0.04*#3,#2-0.85*#3).. (#1+0.1*#3,#2-0.8*#3)
(#1,#2) .. controls (#1,#2-0.05*#3) and (#1+0.04*#3,#2-0.15*#3).. (#1+0.1*#3,#2-0.2*#3)
(#1+0.1*#3,#2-0.2*#3) -- (#1+0.9*#3,#2-0.8*#3)
(#1+1*#3,#2-1*#3) .. controls (#1+1*#3,#2-0.95*#3) and (#1+0.96*#3,#2-0.85*#3).. (#1+0.9*#3,#2-0.8*#3)
}
\def\raction(#1,#2)[#3,#4]{\draw (#1,#2) -- (#1,#2-2*#4/2)  (#1,#2-1*#4/2)--(#1+1*#4/2+#3*#4/2,#2-1*#4/2) .. controls (#1+1.555*#4/2+#3*#4/2,#2-1*#4/2) and (#1+2*#4/2+#3*#4/2,#2-0.555*#4/2) .. (#1+2*#4/2+#3*#4/2,#2)}
\def\doublemap(#1,#2)[#3]{\draw (#1+0.5,#2-0.5) node [name=doublemapnode,inner xsep=0pt, inner ysep=0pt, minimum height=9pt, minimum width=25pt,shape=rectangle,draw,rounded corners] {$#3$} (#1,#2) .. controls (#1,#2-0.075) .. (doublemapnode) (#1+1,#2) .. controls (#1+1,#2-0.075).. (doublemapnode) (doublemapnode) .. controls (#1,#2-0.925)..(#1,#2-1) (doublemapnode) .. controls (#1+1,#2-0.925).. (#1+1,#2-1)}
\begin{scope}[xshift=0cm, yshift=0cm]
\laction(0,0)[0,1];
\end{scope}
\begin{scope}[xshift=0cm,yshift=-0.4cm]
\node at (1.4,-0.4){,};
\end{scope}
\begin{scope}[xshift=0cm, yshift=0cm]
\raction(3.2,0)[0,1];
\end{scope}
\begin{scope}[xshift=4.5cm, yshift=-0.4cm]
\node at (0,-0.4){,};
\end{scope}
\begin{scope}[xshift=0cm, yshift=0cm]
\lactiontr(6.35,0)[0,1,90];
\end{scope}
\begin{scope}[xshift=0cm,yshift=0cm]
\node at (9.9,-0.4){\text{and}};
\end{scope}
\begin{scope}[xshift=0cm, yshift=0cm]
\ractiontr(12.25,0)[0,1,90];
\end{scope}
\begin{scope}[xshift=13.6cm, yshift=-0.4cm]
\node at (0,-0.4){,};
\end{scope}
\end{tikzpicture}
$$
respectively. However, in Section~\ref{section: Braiding operators, invertible cocycles and braces}, when we deal with braces and braiding operators, the maps $\sigma$ and $\tau$ will be named $\lambda$ and $\rho$, and they will be represented by
$$
\begin{tikzpicture}[scale=0.395]
\def\unit(#1,#2){\draw (#1,#2) circle[radius=2pt] (#1,#2-0.07) -- (#1,#2-1)}
\def\mult(#1,#2)[#3,#4]{\draw (#1,#2) arc (180:360:0.5*#3 and 0.5*#4) (#1+0.5*#3, #2-0.5*#4) -- (#1+0.5*#3,#2-#4)}
\def\counit(#1,#2){\draw (#1,#2) -- (#1,#2-0.93) (#1,#2-1) circle[radius=2pt]}
\def\comult(#1,#2)[#3,#4]{\draw (#1,#2) -- (#1,#2-0.5*#4) arc (90:0:0.5*#3 and 0.5*#4) (#1,#2-0.5*#4) arc (90:180:0.5*#3 and 0.5*#4)}
\def\laction(#1,#2)[#3,#4]{\draw (#1,#2) .. controls (#1,#2-0.555*#4/2) and (#1+0.445*#4/2,#2-1*#4/2) .. (#1+1*#4/2,#2-1*#4/2) -- (#1+2*#4/2+#3*#4/2,#2-1*#4/2) (#1+2*#4/2+#3*#4/2,#2)--(#1+2*#4/2+#3*#4/2,#2-2*#4/2)}
\def\map(#1,#2)[#3]{\draw (#1,#2-0.5)  node[name=nodemap,inner sep=0pt,  minimum size=10pt, shape=circle, draw]{$#3$} (#1,#2)-- (nodemap)  (nodemap)-- (#1,#2-1)}
\def\solbraid(#1,#2)[#3]{\draw (#1,#2-0.5)  node[name=nodemap,inner sep=0pt,  minimum size=9pt, shape=circle,draw]{$#3$}
(#1-0.5,#2) .. controls (#1-0.5,#2-0.15) and (#1-0.4,#2-0.2) .. (#1-0.3,#2-0.3) (#1-0.3,#2-0.3) -- (nodemap)
(#1+0.5,#2) .. controls (#1+0.5,#2-0.15) and (#1+0.4,#2-0.2) .. (#1+0.3,#2-0.3) (#1+0.3,#2-0.3) -- (nodemap)
(#1+0.5,#2-1) .. controls (#1+0.5,#2-0.85) and (#1+0.4,#2-0.8) .. (#1+0.3,#2-0.7) (#1+0.3,#2-0.7) -- (nodemap)
(#1-0.5,#2-1) .. controls (#1-0.5,#2-0.85) and (#1-0.4,#2-0.8) .. (#1-0.3,#2-0.7) (#1-0.3,#2-0.7) -- (nodemap)
}
\def\flip(#1,#2)[#3]{\draw (
#1+1*#3,#2) .. controls (#1+1*#3,#2-0.05*#3) and (#1+0.96*#3,#2-0.15*#3).. (#1+0.9*#3,#2-0.2*#3)
(#1+0.1*#3,#2-0.8*#3)--(#1+0.9*#3,#2-0.2*#3)
(#1,#2-1*#3) .. controls (#1,#2-0.95*#3) and (#1+0.04*#3,#2-0.85*#3).. (#1+0.1*#3,#2-0.8*#3)
(#1,#2) .. controls (#1,#2-0.05*#3) and (#1+0.04*#3,#2-0.15*#3).. (#1+0.1*#3,#2-0.2*#3)
(#1+0.1*#3,#2-0.2*#3) -- (#1+0.9*#3,#2-0.8*#3)
(#1+1*#3,#2-1*#3) .. controls (#1+1*#3,#2-0.95*#3) and (#1+0.96*#3,#2-0.85*#3).. (#1+0.9*#3,#2-0.8*#3)
}
\def\raction(#1,#2)[#3,#4]{\draw (#1,#2) -- (#1,#2-2*#4/2)  (#1,#2-1*#4/2)--(#1+1*#4/2+#3*#4/2,#2-1*#4/2) .. controls (#1+1.555*#4/2+#3*#4/2,#2-1*#4/2) and (#1+2*#4/2+#3*#4/2,#2-0.555*#4/2) .. (#1+2*#4/2+#3*#4/2,#2)}
\def\doublemap(#1,#2)[#3]{\draw (#1+0.5,#2-0.5) node [name=doublemapnode,inner xsep=0pt, inner ysep=0pt, minimum height=11pt, minimum width=23pt,shape=rectangle,draw,rounded corners] {$#3$} (#1,#2) .. controls (#1,#2-0.075) .. (doublemapnode) (#1+1,#2) .. controls (#1+1,#2-0.075).. (doublemapnode) (doublemapnode) .. controls (#1,#2-0.925)..(#1,#2-1) (doublemapnode) .. controls (#1+1,#2-0.925).. (#1+1,#2-1)}
\def\doublesinglemap(#1,#2)[#3]{\draw (#1+0.5,#2-0.5) node [name=doublesinglemapnode,inner xsep=0pt, inner ysep=0pt, minimum height=11pt, minimum width=23pt,shape=rectangle,draw,rounded corners] {$#3$} (#1,#2) .. controls (#1,#2-0.075) .. (doublesinglemapnode) (#1+1,#2) .. controls (#1+1,#2-0.075).. (doublesinglemapnode) (doublesinglemapnode)-- (#1+0.5,#2-1)}
\def\ractiontr(#1,#2)[#3,#4,#5]{\draw (#1,#2) -- (#1,#2-2*#4/2)  (#1,#2-1*#4/2) node [inner sep=0pt, minimum size=3pt,shape=isosceles triangle,fill, shape border rotate=#5] {}  --(#1+1*#4/2+#3*#4/2,#2-1*#4/2) .. controls (#1+1.555*#4/2+#3*#4/2,#2-1*#4/2) and (#1+2*#4/2+#3*#4/2,#2-0.555*#4/2) .. (#1+2*#4/2+#3*#4/2,#2)  }
\def\rack(#1,#2)[#3]{\draw (#1,#2-0.5)  node[name=nodemap,inner sep=0pt,  minimum size=7.5pt, shape=circle,draw]{$#3$} (#1-1,#2) .. controls (#1-1,#2-0.5) and (#1-0.5,#2-0.5) .. (nodemap) (#1,#2)-- (nodemap)  (nodemap)-- (#1,#2-1)}
\def\rackmenoslarge(#1,#2)[#3]{\draw (#1,#2-0.5)  node[name=nodemap,inner sep=0pt,  minimum size=7.5pt, shape=circle,draw]{$#3$} (#1-1.5,#2+0.5) .. controls (#1-1.5,#2-0.5) and (#1-0.5,#2-0.5) .. (nodemap) (#1,#2)-- (nodemap)  (nodemap)-- (#1,#2-1)}
\def\racklarge(#1,#2)[#3]{\draw (#1,#2-0.5)  node[name=nodemap,inner sep=0pt,  minimum size=7.5pt, shape=circle,draw]{$#3$} (#1-2,#2+0.5) .. controls (#1-2,#2-0.5) and (#1-0.5,#2-0.5) .. (nodemap) (#1,#2)-- (nodemap)  (nodemap)-- (#1,#2-1)}
\def\rackmaslarge(#1,#2)[#3]{\draw (#1,#2-0.5)  node[name=nodemap,inner sep=0pt,  minimum size=7.5pt, shape=circle,draw]{$#3$} (#1-2.5,#2+0.5) .. controls (#1-2.5,#2-0.5) and (#1-0.5,#2-0.5) .. (nodemap) (#1,#2)-- (nodemap)  (nodemap)-- (#1,#2-1)}
\def\rackextralarge(#1,#2)[#3]{\draw (#1,#2-0.75)  node[name=nodemap,inner sep=0pt,  minimum size=7.5pt, shape=circle, draw]{$#3$} (#1-3,#2+1) .. controls (#1-3,#2-0.75) and (#1-0.5,#2-0.75) .. (nodemap) (#1,#2)-- (nodemap)  (nodemap)-- (#1,#2-1.5)}
\def\lactionnamed(#1,#2)[#3,#4][#5]{\draw (#1 + 0.5*#3 + 0.5 + 0.5*#4, #2- 0.5*#3) node[name=nodemap,inner sep=0pt,  minimum size=8pt, shape=circle,draw]{$#5$} (#1,#2)  arc (180:270:0.5*#3) (#1 + 0.5*#3,#2- 0.5*#3) --  (nodemap) (#1 + 0.5*#3 + 0.5 + 0.5*#4, #2) --  (nodemap) (nodemap) -- (#1 + 0.5*#3 + 0.5 + 0.5*#4, #2-#3)}
\def\ractionnamed(#1,#2)[#3,#4][#5]{\draw  (#1 - 0.5*#3- 0.5 - 0.5*#4, #2- 0.5*#3)  node[name=nodemap,inner sep=0pt,  minimum size=8pt, shape=circle,draw]{$#5$} (#1 - 0.5*#3, #2- 0.5*#3)  arc (270:360:0.5*#3) (#1 - 0.5*#3, #2- 0.5*#3) -- (nodemap)(#1 - 0.5*#3- 0.5 - 0.5*#4, #2)-- (nodemap) (nodemap) -- (#1 - 0.5*#3- 0.5 - 0.5*#4, #2-#3)}
\def\multsubzero(#1,#2)[#3,#4]{\draw (#1+0.5*#3, #2-0.5*#4) node [name=nodemap,inner sep=0pt, minimum size=3pt,shape=circle,fill=white, draw]{} (#1,#2) arc (180:360:0.5*#3 and 0.5*#4) (#1+0.5*#3, #2-0.5*#4) -- (#1+0.5*#3,#2-#4)}
\begin{scope}[xshift=0cm, yshift=0cm]
\lactionnamed(0,0)[1,0][\scriptstyle \lambda];
\end{scope}
\begin{scope}[xshift=0cm,yshift=0cm]
\node at (3.7,-0.4){\text{and}};
\end{scope}
\begin{scope}[xshift=7.3cm, yshift=0cm]
\ractionnamed(0,0)[1,0][\scriptstyle \rho];
\end{scope}
\begin{scope}[xshift=0cm,yshift=-0.3cm]
\node at (7.5,-0.4){.};
\end{scope}
\end{tikzpicture}
$$
Finally, the diagram
$$
\begin{tikzpicture}[scale=0.395]
\def\rack(#1,#2)[#3]{\draw (#1,#2-0.5)  node[name=nodemap,inner sep=0pt,  minimum size=7.5pt, shape=circle,draw]{$#3$} (#1-1,#2) .. controls (#1-1,#2-0.5) and (#1-0.5,#2-0.5) .. (nodemap) (#1,#2)-- (nodemap)  (nodemap)-- (#1,#2-1)}
\begin{scope}[xshift=0cm, yshift=0cm]
\rack(0,0)[\scriptstyle \triangleright];
\end{scope}
\begin{scope}[xshift=0cm,yshift=-0.3cm]
\node at (0.5,-0.4){.};
\end{scope}
\end{tikzpicture}
$$
stands for the action $\triangleright$ of a rack~$(X,\triangleright)$.

\subsubsection*{General constructions}


For each coalgebra $X$ and each $n\in \mathds{N}_0$, we let $\Delta_n\colon X \to X^{n+1}$ denote the map recursively defined by

\begin{itemize}

\smallskip

\item[-] $\Delta_0=\ide$,

\smallskip

\item[-] $\Delta_{i+1}\coloneqq (\Delta\ot X^i)\circ \Delta_i$.

\end{itemize}

\begin{defn}\label{def: funciones coordenadas} Let $Y, X_1,\dots, X_n$ be coalgebras in $\mathscr{C}$ and let $f\colon Y\longrightarrow X_1 \ot \cdots \ot X_n$ be a morphism. The {\em coordinate maps} of $f$ are the maps $f_i\colon Y\to X_i$, defined by
$$
f_i\coloneqq (\epsilon^{i-1}\ot X_i\ot \epsilon^{n-i-1}) \circ f\qquad\text{($1\le i \le n$)}.
$$
\end{defn}

\begin{pro}\label{coro: morfismos de coalgebra al producto tensorial de n coalgebras} Let $Y, X_1,\dots, X_n$ be cocommutative coalgebras in $\mathscr{C}$. If
$$
f\colon Y\longrightarrow X_1 \ot \cdots \ot X_n
$$
is a coalgebra morphism, then the coordinate maps $f_i\colon Y\to X_i$, of $f$, are coalgebra morphisms. Morevover, $f=(f_1\ot \cdots \ot f_n)\circ \Delta_{n-1}$. Conversely, given coalgebra morphisms $f_i\colon Y\to X_i$, the map $f \coloneqq (f_1\ot \cdots \ot f_n)\circ \Delta_{n-1}$is a coalgebra morphism. So, the categorical product of a finite number of cocommutative coalgebras is their tensor product.
\end{pro}

\begin{cor}\label{coro: igualdad de morfismos} Let $Y, X_1,\dots, X_n$ be cocommutative coalgebras in $\mathscr{C}$ and let
$$
f,g\colon Y\longrightarrow X_1 \ot \cdots \ot X_n
$$
be maps. If $f$ and $g$ are coalgebra morphisms, then $f=g$ if and only if $f_i=g_i$ for each $i$.
\end{cor}

Let $C$ be a coalgebra and let $A$ be an algebra. The set $\Hom_{\mathcal{C}}(C,A)$ is a monoid via the {\em convolution product}
$$
f*g \coloneqq m_A\circ (f\ot g)\circ \Delta_C.
$$
The unit of $\Hom_{\mathcal{C}}(C,A)$ is $\eta_A\circ \epsilon_C$. If $\psi\colon C'\to C$ and $\phi\colon A\to A'$ are morphisms of coalgebras and algebras respectively, then the function
$$
\theta_{\psi,\phi}\colon \Hom_{\mathcal{C}}(C,A)\longrightarrow \Hom_{\mathcal{C}}(C',A'),
$$
defined by $\theta_{\psi,\phi}(f):= \phi\circ f\circ \psi$ is a morphism of
monoids.

\smallskip

We will use that if $A$ is a Hopf algebra with antipode $S$ and $f\colon C\to A$ is a coalgebra morphism, then $f$ is convolution invertible with inverse $S\circ f$.

Next we recall without proof some well known results that we will need in Section~\ref{LYZ}. In the rest of this section we assume that $\mathscr{C}$ has countable colimits and that these colimits commute with the tensor product.

\begin{rem}\label{rem:producto_tensorial_colimite} If an arrow $\pi\colon X\to \widetilde{X}$ is the coequalizer of a family $(f_i\colon Z\to X)_{i\in I}$  of arrows, then $\pi\ot \pi\colon X^2\to \widetilde{X}^2$ is the colimit of the diagram made out by the maps $f_i\ot X$ and~$X\ot f_i$, with $i$ running on $I$.
\end{rem}

For each $Z$ in $\mathcal{C}$ the tensor algebra $T(Z)$ is by definition
$$
T(Z):= \coprod_{i\ge 0} Z^i,
$$
where $(Z^i)_{i\ge 0}$ is recursively defined by $Z^0:=\mathds{1}$ and $Z^{i+1}:=Z^i\ot Z$. The product on $T(Z)$ is induced by the canonical isomorphisms $Z^i\ot Z^j\simeq Z^{i+j}$. By the coherence theorem of Mac Lane, $T(Z)$ is an associative and unitary algebra. We let $\iota\colon Z\to T(Z)$ denote the canonical inclusion. The algebra $T(Z)$ satisfies the following universal property:
\begin{quotation}
For each associative and unitary algebra $A$ in $\mathscr{C}$ and each morphism $\varphi\colon Z\to A$ in $\mathscr{C}$, there is a unique algebra morphism $\ov{\varphi}\colon T(Z)\to A$ such that $\varphi = \ov{\varphi}\circ \iota$.
\end{quotation}

If $Z$ is a coalgebra in $\mathscr{C}$, then $T(Z)$ is a bialgebra in $\mathscr{C}$ via the comultiplication and counit induced by the morphisms $(\iota\ot \iota)\circ \Delta_Z$ and $\epsilon$.

Clearly if $Z$ is cocommutative, then so is $T(Z)$. The bialgebra $T(Z)$ satisfies the following universal property:
\begin{quotation}
For each bialgebra $A$ in $\mathscr{C}$ and each morphism $\varphi\colon Z\to A$ of coalgebras, there is a unique bialgebra morphism $\ov{\varphi}\colon T(Z)\to A$ such that $\varphi=\ov{\varphi}\circ\iota$.
\end{quotation}
\noindent Thus $T(Z)$ is the free bialgebra on the coalgebra $Z$. Our next purpose is to describe the free Hopf algebra on a cocommutative coalgebra $Z$ endowed with a coalgebra morphism $S\colon Z\to Z$. Given an arbitrary bialgebra $L$ in $\mathscr{C}$ we let $L^{\op}$ denotes the bialgebra with the same underlying coalgebra structure as $L$ and multiplication $m^{\op}:=m\circ c$.

\begin{pro}\label{coegalizador} Let $(L,m,\eta)$ be an associative and unitary algebra in $\mathscr{C}$ and let $(f_i\colon Z\to L)_{i\in I}$ be a finite family of morphisms in $\mathscr{C}$. For each $i$ let $\hat{f}_i\colon L\ot Z\ot L\to L$ be the morphism defined by $\hat{f}_i:= m\circ (m\ot L)\circ (L\ot f_i\ot L)$ and let $\pi\colon L\to \widetilde{L}$ be the coequalizer of $(\hat{f}_i)_{i\in I}$. There is a unique associative and unitary algebra structure on $\widetilde{L}$, such that $\pi$ is a morphism of algebras.
\end{pro}

The canonical map $\pi\colon L\to \widetilde{L}$ satisfies the following universal property:
\begin{quotation}
For each associative and unitary algebra $A$ in $\mathscr{C}$ and each algebra morphism $\varphi\colon L\to A$ such that $\varphi\circ f_i = \varphi\circ f_j$ for all $i,j\in I$, there exists a unique algebra morphism $\widetilde{\varphi}\colon \widetilde{L}\to A$ such that $\varphi = \widetilde{\varphi}\circ \pi$.
\end{quotation}

\begin{rem}\label{coegalizador'} Assume now that $(L,m,\eta,\Delta,\epsilon)$ is a bialgebra in $\mathscr{C}$ and that the $f_i$'s are coalgebra morphisms in $\mathscr{C}$. Then there is a unique bialgebra structure on $\widetilde{L}$ such that~$\pi$ is a morphism of bialgebras. Moreover $\pi$ has the following universal	property:
\begin{quotation}
For each bialgebra $H$ in $\mathscr{C}$ and each bialgebra morphism \hbox{$\varphi\colon L\to H$} such that $\varphi\circ f_i = \varphi\circ f_j$ for all $i,j\in I$, there exists a unique bialgebra	morphism $\widetilde{\varphi}\colon \widetilde{L}\to H$ such that $\varphi = \widetilde{\varphi}\circ \pi$.
\end{quotation}
\end{rem}

\begin{cor}\label{cociente de Hopf algebra} If $L$ is a Hopf algebra, then so is $\widetilde{L}$.
\end{cor}

Let $Z$ be a cocommutative coalgebra, let $S_Z\colon Z\to Z$ be a coalgebra morphism and let $L:=T(Z)$. Let $S\colon L\to L^{\op}$ be the bialgebra morphism induced by
$$
Z \xrightarrow[]{S_Z} Z \xrightarrow[]{\iota} L^{\op},
$$
where $\iota$ is the canonical inclusion. Let $f_1,f_2,f_3\colon Z\to L$ be the arrows defined by
$$
f_1:=\eta\circ \epsilon_Z, \quad f_2:= m\circ (S\ot L)\circ \Delta\circ \iota\quad\text{and}\quad f_3:= m \circ (L\ot S)\circ \Delta \circ \iota,
$$
and let $\pi\colon L\to \ov{L}$ be the coequalizer of $(\hat{f}_i)_{i\in \{1,2,3\}}$, where $\hat{f}_i$ is as in Proposition~\ref{coegalizador}. Since~$L$ is cocommutative, $f_1$, $f_2$ and $f_3$ are coalgebra morphisms. Consequently, by Remark~\ref{coegalizador'} there is a unique bialgebra structure on $\ov{L}$ such that $\pi$ is a morphism of bialgebras.

\begin{pro}\label{const de antipoda} There is a unique bialgebra morphism $\ov{S}\colon \ov{L}\to \ov{L}^{\op}$ such that $\ov{S}\circ \pi = \pi^{\op}\circ S$, where $\pi\colon L\to \ov{L}$ and $\pi^{\op}\colon L^{\op}\to \ov{L}^{\op}$ are the canonical morphisms. Moreover $\ov{L}$ is a cocommutative Hopf algebra with antipode $\ov{S}$.
\end{pro}

Let $j\colon Z\to \ov{L}$ be the morphism defined by $j:=\pi\circ \iota$, where $\iota\colon Z\to L$ is the canonical inclusion. The cocommutative Hopf algebra $\ov{L}$ has the following universal property:
\begin{quotation}
For each Hopf algebra $H$ in $\mathscr{C}$ and each coalgebra morphism \mbox{$\varphi\colon Z\to H$} in~$\mathscr{C}$ such that $S_H\circ \varphi = \varphi\circ S_Z$, there exists a unique morphism of Hopf algebras $\ov{\varphi}\colon \ov{L}\to H$ such that $\ov{\varphi} \circ j = \varphi$.
\end{quotation}

\begin{rem}\label{Z especial} Assume that $Z=X\coprod SX$, where $X$ and $SX$ are coalgebras and that there is a coalgebra isomorphism $S\colon X\to SX$ such that $S_Z$ is the map induced by~$S$ and $S^{-1}$ (thus, $S_Z^2=\ide$). Let $\iota_X\colon X\to Z$ be the canonical arrow. From the above universal property it follows that:
\begin{quotation}
For each Hopf algebra $H$ in $\mathscr{C}$ and each coalgebra morphism \mbox{$\varphi\colon X\to H$} in~$\mathscr{C}$, there exists a unique morphism of Hopf algebras $\ov{\varphi}\colon \ov{L}\to H$ such that $\ov{\varphi} \circ \ov{\jmath}=\varphi$ , where $\ov{\jmath}\coloneqq j\circ \iota_X = \varphi$.
\end{quotation}
So, $\ov{\jmath}\colon X\to \ov{L}$ is the free Hopf algebra over the coalgebra $X$.
\end{rem}

\section{Braided and symmetric coalgebras}\label{YB}

In this section we introduce the notion of non-degenerate braided set in $\mathscr{C}$ and we begin the study of its properties. In the sequel $X$ is a cocommutative coalgebra in~$\mathscr{C}$ and $r$ is a coalgebra automorphism of $X^2$.

Recall from Definition~\ref{def: funciones coordenadas} that the first and second coordinate maps of $r$ are the coalgebra maps $\sigma\coloneqq (X\ot \epsilon)\circ r$ and $\tau\coloneqq (\epsilon \ot X)\circ r$, respectively.

\begin{rem}\label{algunas formulas} From the fact that $r$ is compatible with the comultiplication map of $X^2$ it follows immediately that
\begin{enumerate}
\item $(r\otimes \tau)\circ \Delta_{X^2} = (X\otimes \Delta)\circ r$,

\smallskip

\item $(\tau \otimes r) \circ \Delta_{X^2} = (c\otimes X)\circ (X\otimes \Delta)\circ r$,

\smallskip

\item $(\sigma \otimes r) \circ \Delta_{X^2} = (\Delta\otimes X)\circ r$,

\smallskip

\item $(r\otimes \sigma)\circ \Delta_{X^2} = (X\otimes c)\circ (\Delta\otimes X)\circ r$.

\end{enumerate}
\end{rem}

\begin{defn}\label{braided, nodegenerado, involutivo} We will say that $(X,r)$ is a {\em braided set} if $r$ satisfies the braid equation
\begin{equation}\label{ec de trenzas}
r_{12}\circ r_{23} \circ r_{12} = r_{23}\circ r_{12} \circ r_{23},
\end{equation}
where $r_{12}\coloneqq r\ot X$ and $r_{23}\coloneqq X \ot r$; we will say that it is {\em involutive} if
\begin{equation}\label{involutiva}
r^2=\ide,
\end{equation}
that is, if $r$ is an involutive arrow; and we will say that it is {\em non-degenerate} if there exist maps $\sigma^{-1}\colon X^2 \to X$ and $\tau^{-1}\colon X^2 \to X$ such that
\begin{align}
&\sigma^{-1}\circ (X\ot \sigma)\circ (\Delta\ot X)= \sigma\circ (X\ot \sigma^{-1})\circ (\Delta\ot X)= \epsilon \ot X\label{no deg a izq}
\shortintertext{and}
&\tau^{-1}\circ (\tau \ot X)\circ (X \ot \Delta)= \tau \circ (\tau^{-1} \ot X)\circ (X \ot \Delta)= X \ot \epsilon.\label{no deg a der}
\end{align}
If $(X,r)$ is a non-degenerate pair, then we will say that $r$ is non-degenerate.
\end{defn}

\begin{pro}\label{no degenerado} The isomorphism $r$ is non-degenerate if and only if the maps $(X \ot \sigma)\circ (\Delta \ot X)$ and $(\tau \ot X)\circ (X\ot \Delta)$ are isomorphisms. Moreover, their compositional inverses are the maps $(X \ot \sigma^{-1})\circ (\Delta \ot X)$ and $(\tau^{-1} \ot X)\circ (X\ot \Delta)$, respectively.
\end{pro}

\begin{proof} If there exists $\sigma^{-1}$, then
$$
\begin{tikzpicture}[scale=0.395]
\def\counit(#1,#2){\draw (#1,#2) -- (#1,#2-0.93) (#1,#2-1) circle[radius=2pt]}
\def\comult(#1,#2)[#3,#4]{\draw (#1,#2) -- (#1,#2-0.5*#4) arc (90:0:0.5*#3 and 0.5*#4) (#1,#2-0.5*#4) arc (90:180:0.5*#3 and 0.5*#4)}
\def\laction(#1,#2)[#3,#4]{\draw (#1,#2) .. controls (#1,#2-0.555*#4/2) and (#1+0.445*#4/2,#2-1*#4/2) .. (#1+1*#4/2,#2-1*#4/2) -- (#1+2*#4/2+#3*#4/2,#2-1*#4/2) (#1+2*#4/2+#3*#4/2,#2)--(#1+2*#4/2+#3*#4/2,#2-2*#4/2)}
\def\lactiontr(#1,#2)[#3,#4,#5]{\draw (#1,#2) .. controls (#1,#2-0.555*#4/2) and (#1+0.445*#4/2,#2-1*#4/2) .. (#1+1*#4/2,#2-1*#4/2) -- (#1+2*#4/2+#3*#4/2,#2-1*#4/2)  node [inner sep=0pt, minimum size=3pt,shape=isosceles triangle,fill, shape border rotate=#5] {} (#1+2*#4/2+#3*#4/2,#2) --(#1+2*#4/2+#3*#4/2,#2-2*#4/2)}
\begin{scope}[xshift=0cm, yshift=-0.5cm]
\comult(0,0)[1,1]; \lactiontr(0,-2)[1,1,90]; \laction(0.5,-1)[0,1]; \draw (1.5,0) -- (1.5,-1); \comult(-0.5,-1)[1,1]; \draw (-1,-2) -- (-1,-3);
\end{scope}
\begin{scope}[xshift=2.1cm, yshift=-1.5cm]
\node at (0,-0.5){=};
\end{scope}
\begin{scope}[xshift=3.6cm, yshift=0cm]
\comult(-0.5,0)[1,1]; \comult(0,-1)[1,1]; \lactiontr(-0.5,-3)[2,1,90]; \laction(0.5,-2)[0,1]; \draw (1.5,0) -- (1.5,-2);  \draw (-1,-1) -- (-1,-4); \draw (-0.5,-2) -- (-0.5,-3);
\end{scope}
\begin{scope}[xshift=6.6cm, yshift=-1.5cm]
\node at (0,-0.5){$= \ide_{X^2}$};
\end{scope}
\begin{scope}[xshift=9.9cm, yshift=-1.5cm]
\node at (0,-0.5){and};
\end{scope}
\begin{scope}[xshift=13.1cm, yshift=-0.5cm]
\comult(0,0)[1,1]; \laction(0,-2)[1,1]; \lactiontr(0.5,-1)[0,1,90]; \draw (1.5,0) -- (1.5,-1); \comult(-0.5,-1)[1,1]; \draw (-1,-2) -- (-1,-3);
\end{scope}
\begin{scope}[xshift=15.2cm, yshift=-1.5cm]
\node at (0,-0.5){=};
\end{scope}
\begin{scope}[xshift=16.8cm, yshift=0cm]
\comult(-0.5,0)[1,1]; \comult(0,-1)[1,1]; \lactiontr(-0.5,-3)[2,1,90]; \laction(0.5,-2)[0,1]; \draw (1.5,0) -- (1.5,-2);  \draw (-1,-1) -- (-1,-4); \draw (-0.5,-2) -- (-0.5,-3);
\end{scope}
\begin{scope}[xshift=19.95cm, yshift=-1.5cm]
\node at (0,-0.5){$= \ide_{X^2}$.};
\end{scope}
\end{tikzpicture}
$$
Conversely, if $(X\ot \sigma)\circ (\Delta\ot X)$ is invertible with inverse $\mathrm{Inv}$, then
$$
\begin{tikzpicture}[scale=0.395]
\def\counit(#1,#2){\draw (#1,#2) -- (#1,#2-0.93) (#1,#2-1) circle[radius=2pt]}
\def\comult(#1,#2)[#3,#4]{\draw (#1,#2) -- (#1,#2-0.5*#4) arc (90:0:0.5*#3 and 0.5*#4) (#1,#2-0.5*#4) arc (90:180:0.5*#3 and 0.5*#4)}
\def\laction(#1,#2)[#3,#4]{\draw (#1,#2) .. controls (#1,#2-0.555*#4/2) and (#1+0.445*#4/2,#2-1*#4/2) .. (#1+1*#4/2,#2-1*#4/2) -- (#1+2*#4/2+#3*#4/2,#2-1*#4/2) (#1+2*#4/2+#3*#4/2,#2)--(#1+2*#4/2+#3*#4/2,#2-2*#4/2)}
\def\lactiontr(#1,#2)[#3,#4,#5]{\draw (#1,#2) .. controls (#1,#2-0.555*#4/2) and (#1+0.445*#4/2,#2-1*#4/2) .. (#1+1*#4/2,#2-1*#4/2) -- (#1+2*#4/2+#3*#4/2,#2-1*#4/2)  node [inner sep=0pt, minimum size=3pt,shape=isosceles triangle,fill, shape border rotate=#5] {} (#1+2*#4/2+#3*#4/2,#2) --(#1+2*#4/2+#3*#4/2,#2-2*#4/2)}
\def\doublemap(#1,#2)[#3]{\draw (#1+0.5,#2-0.5) node [name=doublemapnode,inner xsep=0pt, inner ysep=0pt, minimum height=10pt, minimum width=23pt,shape=rectangle,draw,rounded corners] {$#3$} (#1,#2) .. controls (#1,#2-0.075) .. (doublemapnode) (#1+1,#2) .. controls (#1+1,#2-0.075).. (doublemapnode) (doublemapnode) .. controls (#1,#2-0.925)..(#1,#2-1) (doublemapnode) .. controls (#1+1,#2-0.925).. (#1+1,#2-1)}
\begin{scope}[xshift=0cm, yshift=0cm]
\laction(0.5,-1)[0,1]; \comult(0,0)[1,1]; \draw (1.5,0) -- (1.5,-2); \draw (1.5,-2) -- (1.5,-2.5); \doublemap(0.5,-2.5)[\scriptstyle \mathrm{Inv}];  \counit(0.5,-3.5); \draw (1.5,-3.5) -- (1.5,-4.5); \draw (-0.5,-1) .. controls (-0.5,-1.5) and (0.5,-2) .. (0.5,-2.5);
\end{scope}
\begin{scope}[xshift=3.6cm, yshift=-1.7cm]
\node at (0,-0.5){$= \epsilon \ot X$};
\end{scope}
\begin{scope}[xshift=6.95cm, yshift=-1.7cm]
\node at (0,-0.5){and};
\end{scope}
\begin{scope}[xshift=9.7cm, yshift=0cm]
\laction(-0.5,-3.5)[2,1]; \comult(0,0)[1,1]; \draw (1.5,0) -- (1.5,-1); \doublemap(0.5,-1)[\scriptstyle \mathrm{Inv}]; \counit(0.5,-2); \draw (-0.5,-1) -- (-0.5,-3.5); \draw (1.5,-2) -- (1.5,-3.5);
\end{scope}
\begin{scope}[xshift=12.3cm, yshift=-1.7cm]
\node at (0,-0.5){=};
\end{scope}
\begin{scope}[xshift=12.85cm, yshift=-0.25cm]
\comult(0.5,-1)[1,1]; \draw (1.5,-1) -- (1.5,-2); \draw (0,-2) -- (0,-3); \draw (1,-2) -- (1,-3); \doublemap(0.5,0)[\scriptstyle \mathrm{Inv}]; \counit(0,-3); \draw (1.5,-2) .. controls (1.5,-2.5) and (2,-2.5) .. (2,-3); \laction(1,-3)[0,1];
\end{scope}
\begin{scope}[xshift=16.65cm, yshift=-1.7cm]
\node at (0,-0.5){$= \epsilon \ot X$,};
\end{scope}
\end{tikzpicture}
$$
since $\mathrm{Inv}$ is left colinear. For $\tau$ the proof is similar.
\end{proof}

\begin{cor}\label{sigma^-1 y tau^-1 son morfismos de coalgebras} If $r$ is non-degenerate, then $\sigma^{-1}$ and $\tau^{-1}$ are the unique maps satisfying conditions~\eqref{no deg a izq} and~\eqref{no deg a der} and they are morphisms of coalgebras.
\end{cor}

\begin{proof} It follows since $\sigma^{-1}$ and $\tau^{-1}$ are coordinate maps of the inverse maps of the coalgebra isomorphisms $(X \ot \sigma)\circ (\Delta \ot X)$ and $(\tau \ot X)\circ (X\ot \Delta)$, respectively.
\end{proof}

\begin{rem}\label{prop de involutiva y de braided} If $(X,r)$ is involutive, then
$$
\tau\circ r=\epsilon\ot X\quad\text{and}\quad \sigma \circ r =X\ot \epsilon.
$$
On the other hand, from Corollary~\ref{coro: igualdad de morfismos} one obtains immediately that $(X,r)$ is braided if and only if
\begin{align*}
&(\tau \otimes X)\circ (X\otimes r) \circ (r\otimes X) = r\circ (\tau \otimes X) \circ (X\otimes r),\\
&(\sigma \otimes X)\circ (X\otimes r) \circ (r\otimes X) = (X \otimes \tau)\circ (r\otimes X) \circ (X\otimes r)
\shortintertext{and}
&(X\otimes \sigma)\circ (r\otimes X) \circ (X\otimes r) = r\circ (X\otimes \sigma) \circ (r\otimes X),
\end{align*}
and that this occurs if and only if
\begin{align*}
&\tau\circ (\tau\ot X)= \tau\circ (\tau\ot X)\circ (X\ot r),\\
& \sigma \circ (X\ot\sigma)= \sigma\circ (X\ot \sigma)\circ (r\ot X)
\shortintertext{and}
& \tau\circ (X\ot \sigma) \circ (r\ot X) = \sigma\circ (\tau\ot X)\circ (X\ot r) && \text{(linking relation).}
\end{align*}
The last three conditions are geometrically represented by the diagrams
$$
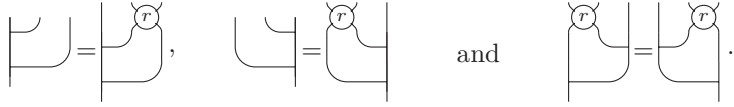
\begin{tikzpicture}[scale=0.395]
\def\mult(#1,#2)[#3,#4]{\draw (#1,#2) arc (180:360:0.5*#3 and 0.5*#4) (#1+0.5*#3, #2-0.5*#4) -- (#1+0.5*#3,#2-#4)}
\def\counit(#1,#2){\draw (#1,#2) -- (#1,#2-0.93) (#1,#2-1) circle[radius=2pt]}
\def\comult(#1,#2)[#3,#4]{\draw (#1,#2) -- (#1,#2-0.5*#4) arc (90:0:0.5*#3 and 0.5*#4) (#1,#2-0.5*#4) arc (90:180:0.5*#3 and 0.5*#4)}
\def\laction(#1,#2)[#3,#4]{\draw (#1,#2) .. controls (#1,#2-0.555*#4/2) and (#1+0.445*#4/2,#2-1*#4/2) .. (#1+1*#4/2,#2-1*#4/2) -- (#1+2*#4/2+#3*#4/2,#2-1*#4/2) (#1+2*#4/2+#3*#4/2,#2)--(#1+2*#4/2+#3*#4/2,#2-2*#4/2)}
\def\map(#1,#2)[#3]{\draw (#1,#2-0.5)  node[name=nodemap,inner sep=0pt,  minimum size=10pt, shape=circle, draw]{$#3$} (#1,#2)-- (nodemap)  (nodemap)-- (#1,#2-1)}
\def\solbraid(#1,#2)[#3]{\draw (#1,#2-0.5)  node[name=nodemap,inner sep=0pt,  minimum size=9pt, shape=circle,draw]{$#3$}
(#1-0.5,#2) .. controls (#1-0.5,#2-0.15) and (#1-0.4,#2-0.2) .. (#1-0.3,#2-0.3) (#1-0.3,#2-0.3) -- (nodemap)
(#1+0.5,#2) .. controls (#1+0.5,#2-0.15) and (#1+0.4,#2-0.2) .. (#1+0.3,#2-0.3) (#1+0.3,#2-0.3) -- (nodemap)
(#1+0.5,#2-1) .. controls (#1+0.5,#2-0.85) and (#1+0.4,#2-0.8) .. (#1+0.3,#2-0.7) (#1+0.3,#2-0.7) -- (nodemap)
(#1-0.5,#2-1) .. controls (#1-0.5,#2-0.85) and (#1-0.4,#2-0.8) .. (#1-0.3,#2-0.7) (#1-0.3,#2-0.7) -- (nodemap)
}
\def\flip(#1,#2)[#3]{\draw (
#1+1*#3,#2) .. controls (#1+1*#3,#2-0.05*#3) and (#1+0.96*#3,#2-0.15*#3).. (#1+0.9*#3,#2-0.2*#3)
(#1+0.1*#3,#2-0.8*#3)--(#1+0.9*#3,#2-0.2*#3)
(#1,#2-1*#3) .. controls (#1,#2-0.95*#3) and (#1+0.04*#3,#2-0.85*#3).. (#1+0.1*#3,#2-0.8*#3)
(#1,#2) .. controls (#1,#2-0.05*#3) and (#1+0.04*#3,#2-0.15*#3).. (#1+0.1*#3,#2-0.2*#3)
(#1+0.1*#3,#2-0.2*#3) -- (#1+0.9*#3,#2-0.8*#3)
(#1+1*#3,#2-1*#3) .. controls (#1+1*#3,#2-0.95*#3) and (#1+0.96*#3,#2-0.85*#3).. (#1+0.9*#3,#2-0.8*#3)
}
\def\raction(#1,#2)[#3,#4]{\draw (#1,#2) -- (#1,#2-2*#4/2)  (#1,#2-1*#4/2)--(#1+1*#4/2+#3*#4/2,#2-1*#4/2) .. controls (#1+1.555*#4/2+#3*#4/2,#2-1*#4/2) and (#1+2*#4/2+#3*#4/2,#2-0.555*#4/2) .. (#1+2*#4/2+#3*#4/2,#2)}
\def\doublemap(#1,#2)[#3]{\draw (#1+0.5,#2-0.5) node [name=doublemapnode,inner xsep=0pt, inner ysep=0pt, minimum height=11pt, minimum width=23pt,shape=rectangle,draw,rounded corners] {$#3$} (#1,#2) .. controls (#1,#2-0.075) .. (doublemapnode) (#1+1,#2) .. controls (#1+1,#2-0.075).. (doublemapnode) (doublemapnode) .. controls (#1,#2-0.925)..(#1,#2-1) (doublemapnode) .. controls (#1+1,#2-0.925).. (#1+1,#2-1)}
\def\doublesinglemap(#1,#2)[#3]{\draw (#1+0.5,#2-0.5) node [name=doublesinglemapnode,inner xsep=0pt, inner ysep=0pt, minimum height=11pt, minimum width=23pt,shape=rectangle,draw,rounded corners] {$#3$} (#1,#2) .. controls (#1,#2-0.075) .. (doublesinglemapnode) (#1+1,#2) .. controls (#1+1,#2-0.075).. (doublesinglemapnode) (doublesinglemapnode)-- (#1+0.5,#2-1)}
\def\ractiontr(#1,#2)[#3,#4,#5]{\draw (#1,#2) -- (#1,#2-2*#4/2)  (#1,#2-1*#4/2) node [inner sep=0pt, minimum size=3pt,shape=isosceles triangle,fill, shape border rotate=#5] {}  --(#1+1*#4/2+#3*#4/2,#2-1*#4/2) .. controls (#1+1.555*#4/2+#3*#4/2,#2-1*#4/2) and (#1+2*#4/2+#3*#4/2,#2-0.555*#4/2) .. (#1+2*#4/2+#3*#4/2,#2)  }
\def\rack(#1,#2)[#3]{\draw (#1,#2-0.5)  node[name=nodemap,inner sep=0pt,  minimum size=7.5pt, shape=circle,draw]{$#3$} (#1-1,#2) .. controls (#1-1,#2-0.5) and (#1-0.5,#2-0.5) .. (nodemap) (#1,#2)-- (nodemap)  (nodemap)-- (#1,#2-1)}
\def\rackmenoslarge(#1,#2)[#3]{\draw (#1,#2-0.5)  node[name=nodemap,inner sep=0pt,  minimum size=7.5pt, shape=circle,draw]{$#3$} (#1-1.5,#2+0.5) .. controls (#1-1.5,#2-0.5) and (#1-0.5,#2-0.5) .. (nodemap) (#1,#2)-- (nodemap)  (nodemap)-- (#1,#2-1)}
\def\racklarge(#1,#2)[#3]{\draw (#1,#2-0.5)  node[name=nodemap,inner sep=0pt,  minimum size=7.5pt, shape=circle,draw]{$#3$} (#1-2,#2+0.5) .. controls (#1-2,#2-0.5) and (#1-0.5,#2-0.5) .. (nodemap) (#1,#2)-- (nodemap)  (nodemap)-- (#1,#2-1)}
\def\rackmaslarge(#1,#2)[#3]{\draw (#1,#2-0.5)  node[name=nodemap,inner sep=0pt,  minimum size=7.5pt, shape=circle,draw]{$#3$} (#1-2.5,#2+0.5) .. controls (#1-2.5,#2-0.5) and (#1-0.5,#2-0.5) .. (nodemap) (#1,#2)-- (nodemap)  (nodemap)-- (#1,#2-1)}
\def\rackextralarge(#1,#2)[#3]{\draw (#1,#2-0.75)  node[name=nodemap,inner sep=0pt,  minimum size=7.5pt, shape=circle, draw]{$#3$} (#1-3,#2+1) .. controls (#1-3,#2-0.75) and (#1-0.5,#2-0.75) .. (nodemap) (#1,#2)-- (nodemap)  (nodemap)-- (#1,#2-1.5)}
\def\lactionnamed(#1,#2)[#3,#4][#5]{\draw (#1 + 0.5*#3 + 0.5 + 0.5*#4, #2- 0.5*#3) node[name=nodemap,inner sep=0pt,  minimum size=8pt, shape=circle,draw]{$#5$} (#1,#2)  arc (180:270:0.5*#3) (#1 + 0.5*#3,#2- 0.5*#3) --  (nodemap) (#1 + 0.5*#3 + 0.5 + 0.5*#4, #2) --  (nodemap) (nodemap) -- (#1 + 0.5*#3 + 0.5 + 0.5*#4, #2-#3)}
\def\ractionnamed(#1,#2)[#3,#4][#5]{\draw  (#1 - 0.5*#3- 0.5 - 0.5*#4, #2- 0.5*#3)  node[name=nodemap,inner sep=0pt,  minimum size=8pt, shape=circle,draw]{$#5$} (#1 - 0.5*#3, #2- 0.5*#3)  arc (270:360:0.5*#3) (#1 - 0.5*#3, #2- 0.5*#3) -- (nodemap)(#1 - 0.5*#3- 0.5 - 0.5*#4, #2)-- (nodemap) (nodemap) -- (#1 - 0.5*#3- 0.5 - 0.5*#4, #2-#3)}
\def\multsubzero(#1,#2)[#3,#4]{\draw (#1+0.5*#3, #2-0.5*#4) node [name=nodemap,inner sep=0pt, minimum size=3pt,shape=circle,fill=white, draw]{} (#1,#2) arc (180:360:0.5*#3 and 0.5*#4) (#1+0.5*#3, #2-0.5*#4) -- (#1+0.5*#3,#2-#4)}
\begin{scope}[xshift=0cm, yshift=-0.5cm]
\draw(0,0) -- (0,-2); \raction(0,0)[0,1]; \draw(2,0) -- (2,-1); \raction(0,-1)[1,1.335];
\end{scope}
\begin{scope}[xshift=2.55cm, yshift=-1.2cm]
\node at (0,-0.5){=};
\end{scope}
\begin{scope}[xshift=3.05cm, yshift=-0cm]
\draw(0,0) -- (0,-3); \raction(0,-1)[0,1]; \draw(2,-1) -- (2,-2); \raction(0,-2)[1,1.335];  \solbraid(1.5,0)[\scriptstyle r];
\end{scope}
\begin{scope}[xshift=5.4cm, yshift=-1.2cm]
\node at (0,-0.5){,};
\end{scope}
\begin{scope}[xshift=7.5cm, yshift=-0.5cm]
\draw(2,0) -- (2,-2); \laction(1,0)[0,1]; \draw(0,0) -- (0,-1); \laction(0,-1)[1,1.335];
\end{scope}
\begin{scope}[xshift=10.05cm, yshift=-1.2cm]
\node at (0,-0.5){=};
\end{scope}
\begin{scope}[xshift=10.55cm, yshift=-0cm]
\draw(2,0) -- (2,-3); \laction(1,-1)[0,1]; \draw(0,-1) -- (0,-2); \laction(0,-2)[1,1.335];  \solbraid(0.5,0)[\scriptstyle r];
\end{scope}
\begin{scope}[xshift=15.6cm, yshift=-1.2cm]
\node at (0,-0.5){and};
\end{scope}
\begin{scope}[xshift=18.6cm, yshift=-0cm]
 \solbraid(0.5,0)[\scriptstyle r];  \draw(0,-1) -- (0,-2); \draw(2,0) -- (2,-1); \laction(1,-1)[0,1]; \raction(0,-2)[1,1.335];
\end{scope}
\begin{scope}[xshift=21.1cm, yshift=-1.2cm]
\node at (0,-0.5){=};
\end{scope}
\begin{scope}[xshift=21.6cm, yshift=-0cm]
 \solbraid(1.5,0)[\scriptstyle r];  \draw(2,-1) -- (2,-2); \draw(0,0) -- (0,-1); \raction(0,-1)[0,1]; \laction(0,-2)[1,1.335];
\end{scope}
\begin{scope}[xshift=24cm, yshift=-1.2cm]
\node at (0,-0.5){.};
\end{scope}
\end{tikzpicture}
$$

\end{rem}

\begin{defn}\label{morf de function-like} Let $r'$ be a coalgebra automorphism of ${X'}^2$. A coalgebra map $\phi\colon X \to X'$ is a {\em homomorphism} from $r$ to $r'$ if $r'\circ (\phi\ot \phi)= (\phi\ot \phi) \circ r$.
\end{defn}

Let $n>1$. Recall that the braid group $B_n$ is the group generated by elements $b_1,\dots b_{n-1}$ subject to the relations
$$
b_ib_j=b_jb_i\quad \text{if $|i-j|>1$}\qquad\text{and}\qquad b_ib_{i+1}b_i=b_{i+1}b_ib_{i+1}\quad\text{for $i<n-1$}.
$$
Recall also that the symmetric group $S_n$ is the quotient of $B_n$ by the relations~$b_i^2=1$.

\begin{rem} The pair $(X,r)$ is a braided set if and only if the assignment $b_i \mapsto r_{i,i+1}$ extends to an action of $B_n$ on $X^n$, and it is an involutive braided set if and only if the assignment $b_i \mapsto r_{i,i+1}$ extends to an action of $S_n$ on $X^n$.
\end{rem}

\begin{defn}\label{Rmatriz} The $\Ree$-matrix of a map $r$ is the map $R\coloneqq c\circ r$.
\end{defn}

\begin{pro}\label{quantum YB} The pair $(X,r)$ is a braided set if and only if $R$ satisfies the quantum Yang-Baxter equation
\begin{equation}\label{QYB}
R_{12}\circ R_{13}\circ R_{23} = R_{23}\circ R_{13}\circ R_{12},
\end{equation}
and $(X,r)$ is an involutive braided set if, in addition to \eqref{QYB}, $R$ satisfies the {\em unitary condition} $R_{21}\circ R = \ide_{X^2}$, where $R_{21}\coloneqq c\circ R \circ c = r\circ c$.
\end{pro}

\begin{proof} This is clear.
\end{proof}

\begin{rem}\label{formula para R} An easy calculation shows that $R= (\tau\ot \sigma) \circ \Delta_{X^2}$.
\end{rem}

\begin{pro} If $(X,r)$ is a non-degenerate pair, then
$$
(\sigma^{-1}\ot \tau^{-1}) \circ (X\ot r \ot X) \circ(\Delta\ot \Delta)=c.
$$
\end{pro}

\begin{proof} We have
$$
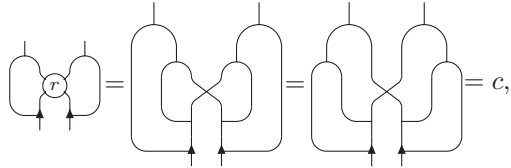
\begin{tikzpicture}[scale=0.395]
\def\counit(#1,#2){\draw (#1,#2) -- (#1,#2-0.93) (#1,#2-1) circle[radius=2pt]}
\def\comult(#1,#2)[#3,#4]{\draw (#1,#2) -- (#1,#2-0.5*#4) arc (90:0:0.5*#3 and 0.5*#4) (#1,#2-0.5*#4) arc (90:180:0.5*#3 and 0.5*#4)}
\def\laction(#1,#2)[#3,#4]{\draw (#1,#2) .. controls (#1,#2-0.555*#4/2) and (#1+0.445*#4/2,#2-1*#4/2) .. (#1+1*#4/2,#2-1*#4/2) -- (#1+2*#4/2+#3*#4/2,#2-1*#4/2) (#1+2*#4/2+#3*#4/2,#2)--(#1+2*#4/2+#3*#4/2,#2-2*#4/2)}
\def\lactiontr(#1,#2)[#3,#4,#5]{\draw (#1,#2) .. controls (#1,#2-0.555*#4/2) and (#1+0.445*#4/2,#2-1*#4/2) .. (#1+1*#4/2,#2-1*#4/2) -- (#1+2*#4/2+#3*#4/2,#2-1*#4/2)  node [inner sep=0pt, minimum size=3pt,shape=isosceles triangle,fill, shape border rotate=#5] {} (#1+2*#4/2+#3*#4/2,#2) --(#1+2*#4/2+#3*#4/2,#2-2*#4/2)}
\def\doublemap(#1,#2)[#3]{\draw (#1+0.5,#2-0.5) node [name=doublemapnode,inner xsep=0pt, inner ysep=0pt, minimum height=11pt, minimum width=23pt,shape=rectangle,draw,rounded corners] {$#3$} (#1,#2) .. controls (#1,#2-0.075) .. (doublemapnode) (#1+1,#2) .. controls (#1+1,#2-0.075).. (doublemapnode) (doublemapnode) .. controls (#1,#2-0.925)..(#1,#2-1) (doublemapnode) .. controls (#1+1,#2-0.925).. (#1+1,#2-1)}
\def\solbraid(#1,#2)[#3]{\draw (#1,#2-0.5)  node[name=nodemap,inner sep=0pt,  minimum size=9pt, shape=circle,draw]{$#3$}
(#1-0.5,#2) .. controls (#1-0.5,#2-0.15) and (#1-0.4,#2-0.2) .. (#1-0.3,#2-0.3) (#1-0.3,#2-0.3) -- (nodemap)
(#1+0.5,#2) .. controls (#1+0.5,#2-0.15) and (#1+0.4,#2-0.2) .. (#1+0.3,#2-0.3) (#1+0.3,#2-0.3) -- (nodemap)
(#1+0.5,#2-1) .. controls (#1+0.5,#2-0.85) and (#1+0.4,#2-0.8) .. (#1+0.3,#2-0.7) (#1+0.3,#2-0.7) -- (nodemap)
(#1-0.5,#2-1) .. controls (#1-0.5,#2-0.85) and (#1-0.4,#2-0.8) .. (#1-0.3,#2-0.7) (#1-0.3,#2-0.7) -- (nodemap)
}
\def\ractiontr(#1,#2)[#3,#4,#5]{\draw (#1,#2) -- (#1,#2-2*#4/2)  (#1,#2-1*#4/2) node [inner sep=0pt, minimum size=3pt,shape=isosceles triangle,fill, shape border rotate=#5] {}  --(#1+1*#4/2+#3*#4/2,#2-1*#4/2) .. controls (#1+1.555*#4/2+#3*#4/2,#2-1*#4/2) and (#1+2*#4/2+#3*#4/2,#2-0.555*#4/2) .. (#1+2*#4/2+#3*#4/2,#2)  }
\def\flip(#1,#2)[#3]{\draw (
#1+1*#3,#2) .. controls (#1+1*#3,#2-0.05*#3) and (#1+0.96*#3,#2-0.15*#3).. (#1+0.9*#3,#2-0.2*#3)
(#1+0.1*#3,#2-0.8*#3)--(#1+0.9*#3,#2-0.2*#3)
(#1,#2-1*#3) .. controls (#1,#2-0.95*#3) and (#1+0.04*#3,#2-0.85*#3).. (#1+0.1*#3,#2-0.8*#3)
(#1,#2) .. controls (#1,#2-0.05*#3) and (#1+0.04*#3,#2-0.15*#3).. (#1+0.1*#3,#2-0.2*#3)
(#1+0.1*#3,#2-0.2*#3) -- (#1+0.9*#3,#2-0.8*#3)
(#1+1*#3,#2-1*#3) .. controls (#1+1*#3,#2-0.95*#3) and (#1+0.96*#3,#2-0.85*#3).. (#1+0.9*#3,#2-0.8*#3)
}
\def\raction(#1,#2)[#3,#4]{\draw (#1,#2) -- (#1,#2-2*#4/2)  (#1,#2-1*#4/2)--(#1+1*#4/2+#3*#4/2,#2-1*#4/2) .. controls (#1+1.555*#4/2+#3*#4/2,#2-1*#4/2) and (#1+2*#4/2+#3*#4/2,#2-0.555*#4/2) .. (#1+2*#4/2+#3*#4/2,#2)}
\begin{scope}[xshift=0cm, yshift=-1.3cm]
\comult(0,0)[1,1]; \comult(2,0)[1,1]; \draw (-0.5,-1) -- (-0.5,-2); \solbraid(1,-1)[\scriptstyle r]; \draw (2.5,-1) -- (2.5,-2); \lactiontr(-0.5,-2)[0,1,90]; \ractiontr(1.5,-2)[0,1,90];
\end{scope}
\begin{scope}[xshift=3cm, yshift=-2.3cm]
\node at (0,-0.5){=};
\end{scope}
\begin{scope}[xshift=4.3cm, yshift=0cm]
\comult(0,0)[1.5,1.5]; \comult(0.75,-1.5)[1,1]; \comult(3.5,0)[1.5,1.5]; \comult(2.75,-1.5)[1,1]; \flip(1.25,-2.5)[1]; \draw (-0.75,-1.5) -- (-0.75,-4.5); \draw (0.25,-2.5) -- (0.25,-3.5); \draw (3.25,-2.5) -- (3.25,-3.5); \draw (4.25,-1.5) -- (4.25,-4.5); \laction(0.25,-3.5)[0,1]; \raction(2.25,-3.5)[0,1];  \lactiontr(-0.75,-4.5)[2,1,90]; \ractiontr(2.25,-4.5)[2,1,90];
\end{scope}
\begin{scope}[xshift=9cm, yshift=-2.3cm]
\node at (0,-0.5){=};
\end{scope}
\begin{scope}[xshift=10.8cm, yshift=0cm]
\comult(0,0)[1.5,1.5]; \comult(-0.75,-1.5)[1,1]; \comult(2.5,0)[1.5,1.5]; \comult(3.25,-1.5)[1,1]; \flip(0.75,-2.5)[1]; \draw (-1.25,-2.5) -- (-1.25,-4.5); \draw (-0.25,-2.5) -- (-0.25,-3.5); \draw (2.75,-2.5) -- (2.75,-3.5); \draw (3.75,-2.5) -- (3.75,-4.5); \laction(-0.25,-3.5)[0,1]; \raction(1.75,-3.5)[0,1];  \lactiontr(-1.25,-4.5)[2,1,90]; \ractiontr(1.75,-4.5)[2,1,90]; \draw (1.75,-1.5) -- (1.75,-2.5); \draw (0.75,-1.5) -- (0.75,-2.5);
\end{scope}
\begin{scope}[xshift=15.4cm, yshift=-2.3cm]
\node at (0,-0.5){$= c$,};
\end{scope}
\end{tikzpicture}
$$
as desired.
\end{proof}

\begin{defn}\label{transposicones en variables} Assume that $(X,r)$ is a non-degenerate pair. The {\em transpositions of $R$ and $R_{21}$ in the first and second variables} are the maps~$R^{t_1}$, $R^{t_2}$, $R_{21}^{t_1}$ and $R_{21}^{t_2}$,   defined by requiring that
\begin{align*}
& R^{t_1}\circ (\tau \ot X) \circ (X\ot\Delta) = (X\ot \sigma)\circ (\Delta \ot X), \\
& R^{t_2}\circ (X\ot \sigma) \circ (\Delta \ot X) = (\tau \ot X)\circ (X\ot \Delta), \\
& R^{t_1}_{21}\circ (\sigma \ot X)\circ (X\ot c)\circ (\Delta\ot X)= (X\ot\tau) \circ (c\ot X)\circ (X\ot \Delta)
\shortintertext{and}
& R^{t_2}_{21}\circ (X\ot\tau) \circ (c\ot X)\circ (X\ot \Delta) = (\sigma \ot X)\circ (X\ot c)\circ (\Delta\ot X).
\end{align*}
\end{defn}

\begin{exa} Let $\mathscr{C}$ be the category of sets and functions. Let $(X,r)$ be a non-degenerate set-theoretical solution. Assume that $r(x,y)=(\sigma_x(y),\tau_y(x))$. Then $R^{t_1}(\tau_y(x),y)=(x,\sigma_x(y))$ and $R^{t_2}(x,\sigma_x(y))=(\tau_y(x),y)$. These maps were considered in~\cite[Proposition 1.3]{MR1722951} and in~\cite[Lemma 7]{MR1769723}.
\end{exa}

\begin{rem}\label{Rt2=Rt1a la-1} Note that
$$
R^{t_1}\circ R^{t_2} = R^{t_2}\circ R^{t_1} = R^{t_1}_{21}\circ R^{t_2}_{21} = R^{t_2}_{21}\circ R^{t_1}_{21} = \ide_{X^2}.
$$
Thus, $R^{t_1}$, $R^{t_2}$, $R^{t_1}_{21}$ and $R^{t_2}_{21}$ are isomorphisms.
\end{rem}

\begin{pro} Let $r$ be as in Definition~\ref{transposicones en variables}. If $(X,r)$ is an involutive pair, then $R^{t_1}\circ R^{t_1}_{21}= R^{t_2}\circ R^{t_2}_{21}=\ide_{X^2}$.
\end{pro}

\begin{proof} By Remark~\ref{Rt2=Rt1a la-1} it suffices to prove that $R^{t_1}_{21}=R^{t_2}$ and $R^{t_2}_{21}=R^{t_1}$. We only treat with the first equality, because the second one is similar. Since
$$
\ide = r\circ r = (\sigma\ot\tau)\circ \Delta_{X^2}\circ r = (\sigma\ot\tau)\circ (r\ot r) \circ \Delta_{X^2},
$$
we have
$$
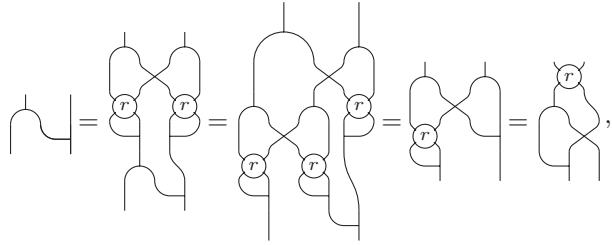
\begin{tikzpicture}[scale=0.395]
\def\counit(#1,#2){\draw (#1,#2) -- (#1,#2-0.93) (#1,#2-1) circle[radius=2pt]}
\def\comult(#1,#2)[#3,#4]{\draw (#1,#2) -- (#1,#2-0.5*#4) arc (90:0:0.5*#3 and 0.5*#4) (#1,#2-0.5*#4) arc (90:180:0.5*#3 and 0.5*#4)}
\def\laction(#1,#2)[#3,#4]{\draw (#1,#2) .. controls (#1,#2-0.555*#4/2) and (#1+0.445*#4/2,#2-1*#4/2) .. (#1+1*#4/2,#2-1*#4/2) -- (#1+2*#4/2+#3*#4/2,#2-1*#4/2) (#1+2*#4/2+#3*#4/2,#2)--(#1+2*#4/2+#3*#4/2,#2-2*#4/2)}
\def\map(#1,#2)[#3]{\draw (#1,#2-0.5)  node[name=nodemap,inner sep=0pt,  minimum size=10pt, shape=circle,draw]{$#3$} (#1,#2)-- (nodemap)  (nodemap)-- (#1,#2-1)}
\def\solbraid(#1,#2)[#3]{\draw (#1,#2-0.5)  node[name=nodemap,inner sep=0pt,  minimum size=9pt, shape=circle,draw]{$#3$}
(#1-0.5,#2) .. controls (#1-0.5,#2-0.15) and (#1-0.4,#2-0.2) .. (#1-0.3,#2-0.3) (#1-0.3,#2-0.3) -- (nodemap)
(#1+0.5,#2) .. controls (#1+0.5,#2-0.15) and (#1+0.4,#2-0.2) .. (#1+0.3,#2-0.3) (#1+0.3,#2-0.3) -- (nodemap)
(#1+0.5,#2-1) .. controls (#1+0.5,#2-0.85) and (#1+0.4,#2-0.8) .. (#1+0.3,#2-0.7) (#1+0.3,#2-0.7) -- (nodemap)
(#1-0.5,#2-1) .. controls (#1-0.5,#2-0.85) and (#1-0.4,#2-0.8) .. (#1-0.3,#2-0.7) (#1-0.3,#2-0.7) -- (nodemap)
}
\def\flip(#1,#2)[#3]{\draw (
#1+1*#3,#2) .. controls (#1+1*#3,#2-0.05*#3) and (#1+0.96*#3,#2-0.15*#3).. (#1+0.9*#3,#2-0.2*#3)
(#1+0.1*#3,#2-0.8*#3)--(#1+0.9*#3,#2-0.2*#3)
(#1,#2-1*#3) .. controls (#1,#2-0.95*#3) and (#1+0.04*#3,#2-0.85*#3).. (#1+0.1*#3,#2-0.8*#3)
(#1,#2) .. controls (#1,#2-0.05*#3) and (#1+0.04*#3,#2-0.15*#3).. (#1+0.1*#3,#2-0.2*#3)
(#1+0.1*#3,#2-0.2*#3) -- (#1+0.9*#3,#2-0.8*#3)
(#1+1*#3,#2-1*#3) .. controls (#1+1*#3,#2-0.95*#3) and (#1+0.96*#3,#2-0.85*#3).. (#1+0.9*#3,#2-0.8*#3)
}
\def\raction(#1,#2)[#3,#4]{\draw (#1,#2) -- (#1,#2-2*#4/2)  (#1,#2-1*#4/2)--(#1+1*#4/2+#3*#4/2,#2-1*#4/2) .. controls (#1+1.555*#4/2+#3*#4/2,#2-1*#4/2) and (#1+2*#4/2+#3*#4/2,#2-0.555*#4/2) .. (#1+2*#4/2+#3*#4/2,#2)}
\begin{scope}[xshift=0cm, yshift=-3.1cm]
\comult(0,0)[1,1]; \laction(0.5,-1)[0,1]; \laction(0.5,-1)[0,1]; \draw (1.5,0) -- (1.5,-1); \draw (-0.5,-1) -- (-0.5,-2);
\end{scope}
\begin{scope}[xshift=2.1cm, yshift=-3.6cm]
\node at (0,-0.5){=};
\end{scope}
\begin{scope}[xshift=2.8cm, yshift=-1cm]
\comult(0.5,0)[1,1]; \comult(2.5,0)[1,1]; \draw (0,-1) -- (0,-2);  \draw (3,-1) -- (3,-2);\flip(1,-1)[1]; \solbraid(0.5,-2)[\scriptstyle r]; \solbraid(2.5,-2)[\scriptstyle r]; \laction(0,-3)[0,1]; \raction(2,-3)[0,1]; \comult(1,-4)[1,1]; \laction(1.5,-5)[0,1]; \draw (2,-4) .. controls (2,-4.5) and (2.5,-4.5) .. (2.5,-5); \draw (0.5,-5) -- (0.5,-6);
\end{scope}
\begin{scope}[xshift=6.4cm, yshift=-3.6cm]
\node at (0,-0.5){=};
\end{scope}
\begin{scope}[xshift=7.1cm, yshift=0cm]
\comult(1.5,0)[2,2]; \draw (0.5,-2) -- (0.5,-3); \comult(0.5,-3)[1,1]; \comult(2.5,-3)[1,1]; \flip(2.5,-2)[1]; \draw (4,0) -- (4,-1); \comult(4,-1)[1,1]; \flip(1,-4)[1]; \draw (0,-4) -- (0,-5); \solbraid(0.5,-5)[\scriptstyle r]; \draw (3,-4) -- (3,-5); \solbraid(2.5,-5)[\scriptstyle r]; \draw (3.5,-5) .. controls (3.5,-6) and (4,-6) .. (4,-7); \solbraid(4,-3)[\scriptstyle r]; \draw (4.5,-2) -- (4.5,-3); \laction(0,-6)[0,1]; \draw (1,-7) -- (1,-8); \laction(2,-6)[0,1]; \raction(3.5,-4)[0,1]; \laction(3,-7)[0,1];
\end{scope}
\begin{scope}[xshift=12.2cm, yshift=-3.6cm]
\node at (0,-0.5){=};
\end{scope}
\begin{scope}[xshift=12.8cm, yshift=-2cm]
\comult(0.5,0)[1,1]; \comult(2.5,0)[1,1]; \draw (0,-1) -- (0,-2);  \draw (3,-1) -- (3,-2);\flip(1,-1)[1]; \solbraid(0.5,-2)[\scriptstyle r]; \laction(0,-3)[0,1]; \laction(2,-2)[0,1]; \draw (3,-3) -- (3,-4);
\end{scope}
\begin{scope}[xshift=16.4cm, yshift=-3.6cm]
\node at (0,-0.5){=};
\end{scope}
\begin{scope}[xshift=17.1cm, yshift=-2cm]
\solbraid(1,0)[\scriptstyle r]; \comult(0.5,-1)[1,1]; \flip(1,-2)[1]; \draw (1.5,-1) .. controls (1.5,-1.5) and (2,-1.5) .. (2,-2); \draw (0,-2) -- (0,-3);  \laction(0,-3)[0,1]; \draw (2,-3) -- (2,-4);
\end{scope}
\begin{scope}[xshift=19.4cm, yshift=-3.6cm]
\node at (0,-0.5){,};
\end{scope}
\end{tikzpicture}
$$
where we had used the fact that $\sigma\circ r$ is a coalgebra homomorphism and Remark~\ref{algunas formulas}(4). On the other hand, from the definition of~$R^{t_1}_{21}$ it follows immediately that
$$
R^{t_1}_{21} \circ (\sigma \ot X)\circ (X\ot c)\circ (\Delta\ot X)= (X\ot\tau) \circ (c\ot X)\circ (X\ot \Delta).
$$
Combining this with the previous equality and Remark~\ref{algunas formulas}(2), we obtain
$$
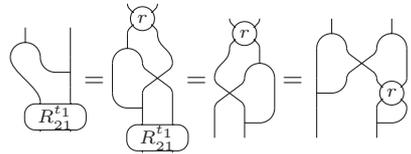
\begin{tikzpicture}[scale=0.395]
\def\counit(#1,#2){\draw (#1,#2) -- (#1,#2-0.93) (#1,#2-1) circle[radius=2pt]}
\def\comult(#1,#2)[#3,#4]{\draw (#1,#2) -- (#1,#2-0.5*#4) arc (90:0:0.5*#3 and 0.5*#4) (#1,#2-0.5*#4) arc (90:180:0.5*#3 and 0.5*#4)}
\def\laction(#1,#2)[#3,#4]{\draw (#1,#2) .. controls (#1,#2-0.555*#4/2) and (#1+0.445*#4/2,#2-1*#4/2) .. (#1+1*#4/2,#2-1*#4/2) -- (#1+2*#4/2+#3*#4/2,#2-1*#4/2) (#1+2*#4/2+#3*#4/2,#2)--(#1+2*#4/2+#3*#4/2,#2-2*#4/2)}
\def\map(#1,#2)[#3]{\draw (#1,#2-0.5)  node[name=nodemap,inner sep=0pt,  minimum size=10pt, shape=circle,draw]{$#3$} (#1,#2)-- (nodemap)  (nodemap)-- (#1,#2-1)}
\def\solbraid(#1,#2)[#3]{\draw (#1,#2-0.5)  node[name=nodemap,inner sep=0pt,  minimum size=9pt, shape=circle,draw]{$#3$}
(#1-0.5,#2) .. controls (#1-0.5,#2-0.15) and (#1-0.4,#2-0.2) .. (#1-0.3,#2-0.3) (#1-0.3,#2-0.3) -- (nodemap)
(#1+0.5,#2) .. controls (#1+0.5,#2-0.15) and (#1+0.4,#2-0.2) .. (#1+0.3,#2-0.3) (#1+0.3,#2-0.3) -- (nodemap)
(#1+0.5,#2-1) .. controls (#1+0.5,#2-0.85) and (#1+0.4,#2-0.8) .. (#1+0.3,#2-0.7) (#1+0.3,#2-0.7) -- (nodemap)
(#1-0.5,#2-1) .. controls (#1-0.5,#2-0.85) and (#1-0.4,#2-0.8) .. (#1-0.3,#2-0.7) (#1-0.3,#2-0.7) -- (nodemap)
}
\def\flip(#1,#2)[#3]{\draw (
#1+1*#3,#2) .. controls (#1+1*#3,#2-0.05*#3) and (#1+0.96*#3,#2-0.15*#3).. (#1+0.9*#3,#2-0.2*#3)
(#1+0.1*#3,#2-0.8*#3)--(#1+0.9*#3,#2-0.2*#3)
(#1,#2-1*#3) .. controls (#1,#2-0.95*#3) and (#1+0.04*#3,#2-0.85*#3).. (#1+0.1*#3,#2-0.8*#3)
(#1,#2) .. controls (#1,#2-0.05*#3) and (#1+0.04*#3,#2-0.15*#3).. (#1+0.1*#3,#2-0.2*#3)
(#1+0.1*#3,#2-0.2*#3) -- (#1+0.9*#3,#2-0.8*#3)
(#1+1*#3,#2-1*#3) .. controls (#1+1*#3,#2-0.95*#3) and (#1+0.96*#3,#2-0.85*#3).. (#1+0.9*#3,#2-0.8*#3)
}
\def\raction(#1,#2)[#3,#4]{\draw (#1,#2) -- (#1,#2-2*#4/2)  (#1,#2-1*#4/2)--(#1+1*#4/2+#3*#4/2,#2-1*#4/2) .. controls (#1+1.555*#4/2+#3*#4/2,#2-1*#4/2) and (#1+2*#4/2+#3*#4/2,#2-0.555*#4/2) .. (#1+2*#4/2+#3*#4/2,#2)}
\def\doublemap(#1,#2)[#3]{\draw (#1+0.5,#2-0.5) node [name=doublemapnode,inner xsep=0pt, inner ysep=0pt, minimum height=9.85pt, minimum width=23pt,shape=rectangle,draw,rounded corners] {$#3$} (#1,#2) .. controls (#1,#2-0.075) .. (doublemapnode) (#1+1,#2) .. controls (#1+1,#2-0.075).. (doublemapnode) (doublemapnode) .. controls (#1,#2-0.925)..(#1,#2-1) (doublemapnode) .. controls (#1+1,#2-0.925).. (#1+1,#2-1)}
\begin{scope}[xshift=0cm, yshift=-0.8cm]
\comult(0.5,0)[1,1]; \laction(1,-1)[0,1]; \draw (2,0) -- (2,-1);
\draw (0,-1) .. controls (0,-1.5) and (1,-2) .. (1,-2.5); \draw (2,-2) -- (2,-2.5); \doublemap(1,-2.5)[\scriptstyle R_{21}^{t_1}];
\end{scope}
\begin{scope}[xshift=2.8cm, yshift=-2.1cm]
\node at (0,-0.5){=};
\end{scope}
\begin{scope}[xshift=3.4cm, yshift=0cm]
\solbraid(1,0)[\scriptstyle r]; \comult(0.5,-1)[1,1]; \flip(1,-2)[1]; \draw (1.5,-1) .. controls (1.5,-1.5) and (2,-1.5) .. (2,-2); \draw (0,-2) -- (0,-3);  \laction(0,-3)[0,1]; \draw (2,-3) -- (2,-4); \doublemap(1,-4)[\scriptstyle R_{21}^{t_1}];
\end{scope}
\begin{scope}[xshift=6.2cm, yshift=-2.1cm]
\node at (0,-0.5){=};
\end{scope}
\begin{scope}[xshift=6.8cm, yshift=-0.5cm]
\solbraid(1,0)[\scriptstyle r]; \comult(1.5,-1)[1,1]; \flip(0,-2)[1]; \draw (0.5,-1) .. controls (0.5,-1.5) and (0,-1.5) .. (0,-2); \draw (2,-2) -- (2,-3);  \raction(1,-3)[0,1]; \draw (0,-3) -- (0,-4);
\end{scope}
\begin{scope}[xshift=9.4cm, yshift=-2.1cm]
\node at (0,-0.5){=};
\end{scope}
\begin{scope}[xshift=10.2cm, yshift=-0.5cm]
\comult(0.5,0)[1,1]; \comult(2.5,0)[1,1]; \flip(1,-1)[1]; \draw (0,-1) -- (0,-2); \draw (3,-1) -- (3,-2);\raction(0,-2)[0,1]; \solbraid(2.5,-2)[\scriptstyle r]; \raction(2,-3)[0,1]; \draw (0,-3) -- (0,-4);
\end{scope}
\begin{scope}[xshift=13.45cm, yshift=-2.1cm]
\node at (0,-0.5){.};
\end{scope}
\end{tikzpicture}
$$
Therefore, we are reduced to prove that
$$
(X \ot \tau) \circ (\tau \ot r) \circ \Delta_{X^2}=(\tau \ot X)\circ (X \ot \Delta).
$$
But this follows immediately applying  $X\ot \epsilon\ot X$ to  the equality
$$
(\tau \ot r^2)\circ \Delta_{X^2} = (\tau \ot X^2)\circ \Delta_{X^2},
$$
since $(\epsilon\ot X)\circ r= \tau$.
\end{proof}

\section{Derived solutions}
In this section we study the relation between non-degenerate braided sets and left racks. It is evident that each definition and result has a symmetric right hand one. Since we will do not mention these ones explicitly, we will write rack and derived map instead use the most correct terminology left rack and left derived map. In the case of set-theoretical solutions derived solutions were first considered in~\cite{MR1809284}. Racks in categories were studied in~\cite{MR3105304}.

\label{derived}
\begin{defn}\label{rack} A {\em left rack in $\mathscr{C}$} is a cocommutative coalgebra $X$ in $\mathscr{C}$ endowed with a coalgebra map $\triangleright\colon X^2 \to X$ that satisfies:

\begin{enumerate}

\smallskip

\item $\triangleright\circ (X \ot \triangleright) = \triangleright\circ (\triangleright\ot \triangleright)\circ (X\ot c\ot X)\circ (\Delta \ot X \ot X)$,

\smallskip

\item there exists a map $\bar{\triangleright}\colon X^2\to X$ such that
\begin{equation}\label{un rak esno deg a izq}
\qquad\quad \bar{\triangleright}\circ (X\ot \triangleright)\circ (\Delta\ot X)= \triangleright\circ (X\ot \bar{\triangleright})\circ (\Delta\ot X)= \epsilon \ot X.
\end{equation}
\end{enumerate}
\end{defn}

\begin{example}\label{rack trivial}
Each coalgebra $X$ in $\mathscr{C}$ is a rack, named {\em rack trivial on $X$}, via the {\em trivial rack} $\triangleright\coloneqq \epsilon \ot X$.
\end{example}

\begin{rem}\label{propiedades de bartriangleleft} Arguing as in Proposition~\ref{no degenerado} and Corollary~\ref{sigma^-1 y tau^-1 son morfismos de coalgebras} it is easy to see that $\bar{\triangleright}$ is the unique map satisfying condition~\eqref{un rak esno deg a izq}, that $(X\ot \triangleright)\circ (\Delta\ot X)$ is invertible with inverse $(X\ot \bar{\triangleright})\circ (\Delta\ot X)$ and that~$\bar{\triangleright}$ is a coalgebra homomorphism.
\end{rem}

Let $X$ be a cocommutative coalgebra in $\mathscr{C}$. Given a map $\triangleright\colon X^2\to X$, we let $r_{\triangleright}$ denote the map defined by
$$
r_{\triangleright}\coloneqq (\triangleright \ot X)\circ (X\ot c)\circ (\Delta\ot X).
$$
It is clear that $r_{\triangleright}$ is a coalgebra homomorphism if and only if $\triangleright$ is.

\begin{pro}\label{solucion tipo rack} The pair $(X, r_{\triangleright})$ is a non-degenerate braided set if and only if $(X,\triangleright)$ is a rack.
\end{pro}

\begin{proof} Assume that $(X, r_{\triangleright})$ is a non-degenerate braided set. Then $\triangleright$ is a coalgebra homomorphism and there exists $\bar{\triangleright}$ satisfying~\eqref{un rak esno deg a izq}, since $\triangleright=\sigma_{r_{\triangleright}}$. So, in order to prove that $(X,\triangleright)$ is a rack we only must check that $\triangleright$ fulfills condition~(1) of Definition~\ref{rack}. But this follows immediately applying $X\ot \epsilon \ot \epsilon$ to the equality
$$
{r_{\triangleright}}_{_{12}}\circ {r_{\triangleright}}_{_{23}}\circ {r_{\triangleright}}_{_{12}}={r_{\triangleright}}_{_{23}}\circ {r_{\triangleright}}_{_{12}}\circ {r_{\triangleright}}_{_{23}}.
$$
Conversely, if $(X,\triangleright)$ is a rack, then $r_{\triangleright}$ is a coalgebra homomorphism and the equalities
$$
\begin{tikzpicture}[scale=0.395]
\def\counit(#1,#2){\draw (#1,#2) -- (#1,#2-0.93) (#1,#2-1) circle[radius=2pt]}
\def\comult(#1,#2)[#3,#4]{\draw (#1,#2) -- (#1,#2-0.5*#4) arc (90:0:0.5*#3 and 0.5*#4) (#1,#2-0.5*#4) arc (90:180:0.5*#3 and 0.5*#4)}
\def\laction(#1,#2)[#3,#4]{\draw (#1,#2) .. controls (#1,#2-0.555*#4/2) and (#1+0.445*#4/2,#2-1*#4/2) .. (#1+1*#4/2,#2-1*#4/2) -- (#1+2*#4/2+#3*#4/2,#2-1*#4/2) (#1+2*#4/2+#3*#4/2,#2)--(#1+2*#4/2+#3*#4/2,#2-2*#4/2)}
\def\lactiontr(#1,#2)[#3,#4,#5]{\draw (#1,#2) .. controls (#1,#2-0.555*#4/2) and (#1+0.445*#4/2,#2-1*#4/2) .. (#1+1*#4/2,#2-1*#4/2) -- (#1+2*#4/2+#3*#4/2,#2-1*#4/2)  node [inner sep=0pt, minimum size=3pt,shape=isosceles triangle,fill, shape border rotate=#5] {} (#1+2*#4/2+#3*#4/2,#2) --(#1+2*#4/2+#3*#4/2,#2-2*#4/2)}
\def\doublemap(#1,#2)[#3]{\draw (#1+0.5,#2-0.5) node [name=doublemapnode,inner xsep=0pt, inner ysep=0pt, minimum height=11pt, minimum width=23pt,shape=rectangle,draw,rounded corners] {$#3$} (#1,#2) .. controls (#1,#2-0.075) .. (doublemapnode) (#1+1,#2) .. controls (#1+1,#2-0.075).. (doublemapnode) (doublemapnode) .. controls (#1,#2-0.925)..(#1,#2-1) (doublemapnode) .. controls (#1+1,#2-0.925).. (#1+1,#2-1)}
\def\solbraid(#1,#2)[#3]{\draw (#1,#2-0.5)  node[name=nodemap,inner sep=0pt,  minimum size=9pt, shape=circle,draw]{$#3$}
(#1-0.5,#2) .. controls (#1-0.5,#2-0.15) and (#1-0.4,#2-0.2) .. (#1-0.3,#2-0.3) (#1-0.3,#2-0.3) -- (nodemap)
(#1+0.5,#2) .. controls (#1+0.5,#2-0.15) and (#1+0.4,#2-0.2) .. (#1+0.3,#2-0.3) (#1+0.3,#2-0.3) -- (nodemap)
(#1+0.5,#2-1) .. controls (#1+0.5,#2-0.85) and (#1+0.4,#2-0.8) .. (#1+0.3,#2-0.7) (#1+0.3,#2-0.7) -- (nodemap)
(#1-0.5,#2-1) .. controls (#1-0.5,#2-0.85) and (#1-0.4,#2-0.8) .. (#1-0.3,#2-0.7) (#1-0.3,#2-0.7) -- (nodemap)
}
\def\ractiontr(#1,#2)[#3,#4,#5]{\draw (#1,#2) -- (#1,#2-2*#4/2)  (#1,#2-1*#4/2) node [inner sep=0pt, minimum size=3pt,shape=isosceles triangle,fill, shape border rotate=#5] {}  --(#1+1*#4/2+#3*#4/2,#2-1*#4/2) .. controls (#1+1.555*#4/2+#3*#4/2,#2-1*#4/2) and (#1+2*#4/2+#3*#4/2,#2-0.555*#4/2) .. (#1+2*#4/2+#3*#4/2,#2)  }
\def\flip(#1,#2)[#3]{\draw (
#1+1*#3,#2) .. controls (#1+1*#3,#2-0.05*#3) and (#1+0.96*#3,#2-0.15*#3).. (#1+0.9*#3,#2-0.2*#3)
(#1+0.1*#3,#2-0.8*#3)--(#1+0.9*#3,#2-0.2*#3)
(#1,#2-1*#3) .. controls (#1,#2-0.95*#3) and (#1+0.04*#3,#2-0.85*#3).. (#1+0.1*#3,#2-0.8*#3)
(#1,#2) .. controls (#1,#2-0.05*#3) and (#1+0.04*#3,#2-0.15*#3).. (#1+0.1*#3,#2-0.2*#3)
(#1+0.1*#3,#2-0.2*#3) -- (#1+0.9*#3,#2-0.8*#3)
(#1+1*#3,#2-1*#3) .. controls (#1+1*#3,#2-0.95*#3) and (#1+0.96*#3,#2-0.85*#3).. (#1+0.9*#3,#2-0.8*#3)
}
\def\raction(#1,#2)[#3,#4]{\draw (#1,#2) -- (#1,#2-2*#4/2)  (#1,#2-1*#4/2)--(#1+1*#4/2+#3*#4/2,#2-1*#4/2) .. controls (#1+1.555*#4/2+#3*#4/2,#2-1*#4/2) and (#1+2*#4/2+#3*#4/2,#2-0.555*#4/2) .. (#1+2*#4/2+#3*#4/2,#2)}
\def\rack(#1,#2)[#3]{\draw (#1,#2-0.5)  node[name=nodemap,inner sep=0pt,  minimum size=7.5pt, shape=circle,draw]{$#3$} (#1-1,#2) .. controls (#1-1,#2-0.5) and (#1-0.5,#2-0.5) .. (nodemap) (#1,#2)-- (nodemap)  (nodemap)-- (#1,#2-1)}
\def\racklarge(#1,#2)[#3]{\draw (#1,#2-0.5)  node[name=nodemap,inner sep=0pt,  minimum size=7.5pt, shape=circle,draw]{$#3$} (#1-2,#2+0.5) .. controls (#1-2,#2-0.5) and (#1-0.5,#2-0.5) .. (nodemap) (#1,#2)-- (nodemap)  (nodemap)-- (#1,#2-1)}
\begin{scope}[xshift=0cm, yshift=-0.5cm]
\comult(0,0)[1,1];  \draw (1.5,0) -- (1.5,-1); \draw (-0.5,-1) -- (-0.5,-2); \flip(0.5,-1)[1]; \rack(0.5,-2)[\scriptstyle \triangleright];  \draw (1.5,-2) -- (1.5,-3); \flip(0.5,-3)[1]; \comult(0.5,-4)[1,1]; \draw (1.5,-4) .. controls (1.5,-4.5) and (2,-4.5) .. (2,-5); \rack(2,-5)[\scriptstyle \bar{\triangleright}]; \draw (0,-5) -- (0,-6);
\end{scope}
\begin{scope}[xshift=2.6cm, yshift=-3.05cm]
\node at (0,-0.5){=};
\end{scope}
\begin{scope}[xshift=3.7cm, yshift=0cm]
\comult(0,0)[1,1];  \draw (1.5,0) -- (1.5,-1); \draw (-0.5,-1) -- (-0.5,-3); \flip(0.5,-1)[1]; \flip(0.5,-2)[1]; \flip(-0.5,-3)[1]; \draw (1.5,-3) -- (1.5,-4); \rack(1.5,-4)[\scriptstyle \triangleright]; \draw (-0.5,-4) .. controls (-0.5,-4.5) and (0,-4.5) .. (0,-5); \comult(0,-5)[1,1]; \rack(1.5,-6)[\scriptstyle \bar{\triangleright}];  \draw (1.5,-5) -- (1.5,-6); \draw (-0.5,-6) -- (-0.5,-7);
\end{scope}
\begin{scope}[xshift=6cm, yshift=-3.05cm]
\node at (0,-0.5){=};
\end{scope}
\begin{scope}[xshift=7cm, yshift=-1.5cm]
\comult(0,0)[1,1];  \draw (1.5,0) -- (1.5,-1);  \rack(1.5,-1)[\scriptstyle \triangleright]; \comult(0,-2)[1,1]; \rack(1.5,-3)[\scriptstyle \bar{\triangleright}]; \draw (-0.5,-1) .. controls (-0.5,-1.5) and (0,-1.5) .. (0,-2); \draw (1.5,-2) -- (1.5,-3); \draw (-0.5,-3) -- (-0.5,-3); \draw (-0.5,-3) -- (-0.5,-4);
\end{scope}
\begin{scope}[xshift=9.3cm, yshift=-3.05cm]
\node at (0,-0.5){=};
\end{scope}
\begin{scope}[xshift=10.4cm, yshift=-1.5cm]
\comult(0,0)[1,1];  \draw (2,0) -- (2,-2);  \rack(2,-2)[\scriptstyle \triangleright]; \comult(0.5,-1)[1,1]; \racklarge(2,-3)[\scriptstyle \bar{\triangleright}]; \draw (0,-2) -- (0,-2.5); \draw (-0.5,-1) -- (-0.5,-4);
\end{scope}
\begin{scope}[xshift=14.3cm, yshift=-3.05cm]
\node at (0,-0.5){$=\ide_{X^2}$};
\end{scope}
\end{tikzpicture}
$$
and
$$
\begin{tikzpicture}[scale=0.395]
\def\counit(#1,#2){\draw (#1,#2) -- (#1,#2-0.93) (#1,#2-1) circle[radius=2pt]}
\def\comult(#1,#2)[#3,#4]{\draw (#1,#2) -- (#1,#2-0.5*#4) arc (90:0:0.5*#3 and 0.5*#4) (#1,#2-0.5*#4) arc (90:180:0.5*#3 and 0.5*#4)}
\def\laction(#1,#2)[#3,#4]{\draw (#1,#2) .. controls (#1,#2-0.555*#4/2) and (#1+0.445*#4/2,#2-1*#4/2) .. (#1+1*#4/2,#2-1*#4/2) -- (#1+2*#4/2+#3*#4/2,#2-1*#4/2) (#1+2*#4/2+#3*#4/2,#2)--(#1+2*#4/2+#3*#4/2,#2-2*#4/2)}
\def\lactiontr(#1,#2)[#3,#4,#5]{\draw (#1,#2) .. controls (#1,#2-0.555*#4/2) and (#1+0.445*#4/2,#2-1*#4/2) .. (#1+1*#4/2,#2-1*#4/2) -- (#1+2*#4/2+#3*#4/2,#2-1*#4/2)  node [inner sep=0pt, minimum size=3pt,shape=isosceles triangle,fill, shape border rotate=#5] {} (#1+2*#4/2+#3*#4/2,#2) --(#1+2*#4/2+#3*#4/2,#2-2*#4/2)}
\def\doublemap(#1,#2)[#3]{\draw (#1+0.5,#2-0.5) node [name=doublemapnode,inner xsep=0pt, inner ysep=0pt, minimum height=11pt, minimum width=23pt,shape=rectangle,draw,rounded corners] {$#3$} (#1,#2) .. controls (#1,#2-0.075) .. (doublemapnode) (#1+1,#2) .. controls (#1+1,#2-0.075).. (doublemapnode) (doublemapnode) .. controls (#1,#2-0.925)..(#1,#2-1) (doublemapnode) .. controls (#1+1,#2-0.925).. (#1+1,#2-1)}
\def\solbraid(#1,#2)[#3]{\draw (#1,#2-0.5)  node[name=nodemap,inner sep=0pt,  minimum size=9pt, shape=circle,draw]{$#3$}
(#1-0.5,#2) .. controls (#1-0.5,#2-0.15) and (#1-0.4,#2-0.2) .. (#1-0.3,#2-0.3) (#1-0.3,#2-0.3) -- (nodemap)
(#1+0.5,#2) .. controls (#1+0.5,#2-0.15) and (#1+0.4,#2-0.2) .. (#1+0.3,#2-0.3) (#1+0.3,#2-0.3) -- (nodemap)
(#1+0.5,#2-1) .. controls (#1+0.5,#2-0.85) and (#1+0.4,#2-0.8) .. (#1+0.3,#2-0.7) (#1+0.3,#2-0.7) -- (nodemap)
(#1-0.5,#2-1) .. controls (#1-0.5,#2-0.85) and (#1-0.4,#2-0.8) .. (#1-0.3,#2-0.7) (#1-0.3,#2-0.7) -- (nodemap)
}
\def\ractiontr(#1,#2)[#3,#4,#5]{\draw (#1,#2) -- (#1,#2-2*#4/2)  (#1,#2-1*#4/2) node [inner sep=0pt, minimum size=3pt,shape=isosceles triangle,fill, shape border rotate=#5] {}  --(#1+1*#4/2+#3*#4/2,#2-1*#4/2) .. controls (#1+1.555*#4/2+#3*#4/2,#2-1*#4/2) and (#1+2*#4/2+#3*#4/2,#2-0.555*#4/2) .. (#1+2*#4/2+#3*#4/2,#2)  }
\def\flip(#1,#2)[#3]{\draw (
#1+1*#3,#2) .. controls (#1+1*#3,#2-0.05*#3) and (#1+0.96*#3,#2-0.15*#3).. (#1+0.9*#3,#2-0.2*#3)
(#1+0.1*#3,#2-0.8*#3)--(#1+0.9*#3,#2-0.2*#3)
(#1,#2-1*#3) .. controls (#1,#2-0.95*#3) and (#1+0.04*#3,#2-0.85*#3).. (#1+0.1*#3,#2-0.8*#3)
(#1,#2) .. controls (#1,#2-0.05*#3) and (#1+0.04*#3,#2-0.15*#3).. (#1+0.1*#3,#2-0.2*#3)
(#1+0.1*#3,#2-0.2*#3) -- (#1+0.9*#3,#2-0.8*#3)
(#1+1*#3,#2-1*#3) .. controls (#1+1*#3,#2-0.95*#3) and (#1+0.96*#3,#2-0.85*#3).. (#1+0.9*#3,#2-0.8*#3)
}
\def\raction(#1,#2)[#3,#4]{\draw (#1,#2) -- (#1,#2-2*#4/2)  (#1,#2-1*#4/2)--(#1+1*#4/2+#3*#4/2,#2-1*#4/2) .. controls (#1+1.555*#4/2+#3*#4/2,#2-1*#4/2) and (#1+2*#4/2+#3*#4/2,#2-0.555*#4/2) .. (#1+2*#4/2+#3*#4/2,#2)}
\def\rack(#1,#2)[#3]{\draw (#1,#2-0.5)  node[name=nodemap,inner sep=0pt,  minimum size=7.5pt, shape=circle,draw]{$#3$} (#1-1,#2) .. controls (#1-1,#2-0.5) and (#1-0.5,#2-0.5) .. (nodemap) (#1,#2)-- (nodemap)  (nodemap)-- (#1,#2-1)}
\def\racklarge(#1,#2)[#3]{\draw (#1,#2-0.5)  node[name=nodemap,inner sep=0pt,  minimum size=7.5pt, shape=circle,draw]{$#3$} (#1-2,#2+0.5) .. controls (#1-2,#2-0.5) and (#1-0.5,#2-0.5) .. (nodemap) (#1,#2)-- (nodemap)  (nodemap)-- (#1,#2-1)}
\begin{scope}[xshift=0.2cm, yshift=-0.5cm]
\flip(0.5,0)[1]; \comult(0.5,-1)[1,1]; \draw (1.5,-1) .. controls (1.5,-1.5) and (2,-1.5) .. (2,-2); \draw (0,-2) .. controls (0,-2.5) and (0.5,-2.5) .. (0.5,-3); \rack(2,-2)[\scriptstyle \bar{\triangleright}]; \comult(0.5,-3)[1,1]; \flip(1,-4)[1]; \rack(1,-5)[\scriptstyle \triangleright]; \draw (2,-3) -- (2,-4); \draw (0,-4) -- (0,-5); \draw (2,-5) -- (2,-6);
\end{scope}
\begin{scope}[xshift=2.9cm, yshift=-3.05cm]
\node at (0,-0.5){=};
\end{scope}
\begin{scope}[xshift=3.5cm, yshift=0cm]
\flip(0.5,0)[1]; \comult(0.5,-1)[1,1]; \draw (1.5,-1) .. controls (1.5,-1.5) and (2,-1.5) .. (2,-2); \draw (0,-2) .. controls (0,-2.5) and (0.5,-2.5) .. (0.5,-3); \rack(2,-2)[\scriptstyle \bar{\triangleright}]; \comult(0.5,-3)[1,1]; \flip(0,-4)[1]; \flip(1,-5)[1]; \rack(1,-6)[\scriptstyle \triangleright]; \draw (2,-3) -- (2,-5); \draw (0,-5) -- (0,-6); \draw (2,-6) -- (2,-7);
\end{scope}
\begin{scope}[xshift=6.3cm, yshift=-3.05cm]
\node at (0,-0.5){=};
\end{scope}
\begin{scope}[xshift=7.4cm, yshift=0cm]
\flip(0,0)[1]; \comult(0,-1)[1,1]; \draw (1,-1) .. controls (1,-1.5) and (1.5,-1.5) .. (1.5,-2);  \rack(1.5,-2)[\scriptstyle \bar{\triangleright}]; \comult(0,-3)[1,1]; \rack(1.5,-4)[\scriptstyle \triangleright]; \draw (-0.5,-2) .. controls (-0.5,-2.5) and (0,-2.5) .. (0,-3); \draw (1.5,-3) -- (1.5,-4); \draw (-0.5,-4) .. controls (-0.5,-4.5) and (0,-4.5) .. (0,-5); \draw (0,-5) -- (0,-6); \draw (1.5,-5) .. controls (1.5,-5.5) and (1,-5.5) .. (1,-6); \flip(0,-6)[1];
\end{scope}
\begin{scope}[xshift=9.7cm, yshift=-3.05cm]
\node at (0,-0.5){=};
\end{scope}
\begin{scope}[xshift=10.8cm, yshift=0cm]
\flip(0,0)[1]; \comult(0,-1)[1,1]; \draw (1,-1) .. controls (1,-2) and (2,-2) .. (2,-3);  \rack(2,-3)[\scriptstyle \triangleright]; \comult(0.5,-2)[1,1]; \racklarge(2,-4)[\scriptstyle \bar{\triangleright}]; \draw (0,-3) -- (0,-3.5); \draw (-0.5,-2) -- (-0.5,-4); \draw (-0.5,-4) .. controls (-0.5,-5) and (0.5,-5) .. (0.5,-6);  \draw (2,-5) .. controls (2,-5.5) and (1.5,-5.5) .. (1.5,-6); \flip(0.5,-6)[1];
\end{scope}
\begin{scope}[xshift=14.6cm, yshift=-3.05cm]
\node at (0,-0.5){$=\ide_{X^2}$};
\end{scope}
\end{tikzpicture}
$$
show that $r_{\triangleright}$ is invertible. It remains to check that $r_{\triangleright}$ satisfies the braid equation. But this follows easily from Corollary~\ref{coro: igualdad de morfismos}.
\end{proof}

\begin{defn}\label{solucion asociada} Let $r$ be a coalgebra automorphism of $X^2$. Assume that $(X,r)$ is non-degenerate.  The {\em derived map of $r$} is the map $s\colon X^2\to X^2$ defined by
$$
s\coloneqq (\tau \ot X) \circ (X\ot \Delta) \circ c \circ (X \ot \sigma) \circ (X \ot \tau^{-1} \ot X) \circ (\Delta\ot \Delta).
$$
If necessary we will write $s_r$ instead of $s$.
\end{defn}

\begin{lem}\label{le previo} The arrow
\begin{equation}\label{lactionractioninvcomultotcomult}
(X\ot \sigma)\circ (X\ot \tau^{-1}\ot X) \circ (\Delta\ot \Delta)
\end{equation}
is an isomorphism of coalgebras.
\end{lem}

\begin{proof} The equalities
$$
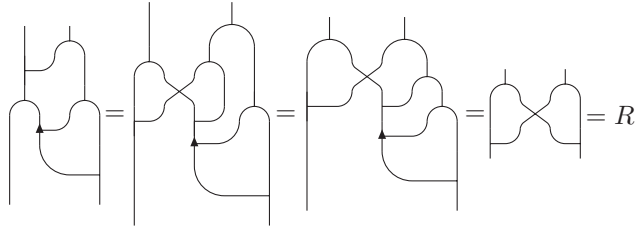
\begin{tikzpicture}[scale=0.395]
\def\counit(#1,#2){\draw (#1,#2) -- (#1,#2-0.93) (#1,#2-1) circle[radius=2pt]}
\def\comult(#1,#2)[#3,#4]{\draw (#1,#2) -- (#1,#2-0.5*#4) arc (90:0:0.5*#3 and 0.5*#4) (#1,#2-0.5*#4) arc (90:180:0.5*#3 and 0.5*#4)}
\def\laction(#1,#2)[#3,#4]{\draw (#1,#2) .. controls (#1,#2-0.555*#4/2) and (#1+0.445*#4/2,#2-1*#4/2) .. (#1+1*#4/2,#2-1*#4/2) -- (#1+2*#4/2+#3*#4/2,#2-1*#4/2) (#1+2*#4/2+#3*#4/2,#2)--(#1+2*#4/2+#3*#4/2,#2-2*#4/2)}
\def\lactiontr(#1,#2)[#3,#4,#5]{\draw (#1,#2) .. controls (#1,#2-0.555*#4/2) and (#1+0.445*#4/2,#2-1*#4/2) .. (#1+1*#4/2,#2-1*#4/2) -- (#1+2*#4/2+#3*#4/2,#2-1*#4/2)  node [inner sep=0pt, minimum size=3pt,shape=isosceles triangle,fill, shape border rotate=#5] {} (#1+2*#4/2+#3*#4/2,#2) --(#1+2*#4/2+#3*#4/2,#2-2*#4/2)}
\def\doublemap(#1,#2)[#3]{\draw (#1+0.5,#2-0.5) node [name=doublemapnode,inner xsep=0pt, inner ysep=0pt, minimum height=11pt, minimum width=25pt,shape=rectangle,draw,rounded corners] {$#3$} (#1,#2) .. controls (#1,#2-0.075) .. (doublemapnode) (#1+1,#2) .. controls (#1+1,#2-0.075).. (doublemapnode) (doublemapnode) .. controls (#1,#2-0.925)..(#1,#2-1) (doublemapnode) .. controls (#1+1,#2-0.925).. (#1+1,#2-1)}
\def\solbraid(#1,#2)[#3]{\draw (#1,#2-0.5)  node[name=nodemap,inner sep=0pt,  minimum size=9pt, shape=circle,draw]{$#3$}
(#1-0.5,#2) .. controls (#1-0.5,#2-0.15) and (#1-0.4,#2-0.2) .. (#1-0.3,#2-0.3) (#1-0.3,#2-0.3) -- (nodemap)
(#1+0.5,#2) .. controls (#1+0.5,#2-0.15) and (#1+0.4,#2-0.2) .. (#1+0.3,#2-0.3) (#1+0.3,#2-0.3) -- (nodemap)
(#1+0.5,#2-1) .. controls (#1+0.5,#2-0.85) and (#1+0.4,#2-0.8) .. (#1+0.3,#2-0.7) (#1+0.3,#2-0.7) -- (nodemap)
(#1-0.5,#2-1) .. controls (#1-0.5,#2-0.85) and (#1-0.4,#2-0.8) .. (#1-0.3,#2-0.7) (#1-0.3,#2-0.7) -- (nodemap)
}
\def\ractiontr(#1,#2)[#3,#4,#5]{\draw (#1,#2) -- (#1,#2-2*#4/2)  (#1,#2-1*#4/2) node [inner sep=0pt, minimum size=3pt,shape=isosceles triangle,fill, shape border rotate=#5] {}  --(#1+1*#4/2+#3*#4/2,#2-1*#4/2) .. controls (#1+1.555*#4/2+#3*#4/2,#2-1*#4/2) and (#1+2*#4/2+#3*#4/2,#2-0.555*#4/2) .. (#1+2*#4/2+#3*#4/2,#2)  }
\def\flip(#1,#2)[#3]{\draw (
#1+1*#3,#2) .. controls (#1+1*#3,#2-0.05*#3) and (#1+0.96*#3,#2-0.15*#3).. (#1+0.9*#3,#2-0.2*#3)
(#1+0.1*#3,#2-0.8*#3)--(#1+0.9*#3,#2-0.2*#3)
(#1,#2-1*#3) .. controls (#1,#2-0.95*#3) and (#1+0.04*#3,#2-0.85*#3).. (#1+0.1*#3,#2-0.8*#3)
(#1,#2) .. controls (#1,#2-0.05*#3) and (#1+0.04*#3,#2-0.15*#3).. (#1+0.1*#3,#2-0.2*#3)
(#1+0.1*#3,#2-0.2*#3) -- (#1+0.9*#3,#2-0.8*#3)
(#1+1*#3,#2-1*#3) .. controls (#1+1*#3,#2-0.95*#3) and (#1+0.96*#3,#2-0.85*#3).. (#1+0.9*#3,#2-0.8*#3)
}
\def\raction(#1,#2)[#3,#4]{\draw (#1,#2) -- (#1,#2-2*#4/2)  (#1,#2-1*#4/2)--(#1+1*#4/2+#3*#4/2,#2-1*#4/2) .. controls (#1+1.555*#4/2+#3*#4/2,#2-1*#4/2) and (#1+2*#4/2+#3*#4/2,#2-0.555*#4/2) .. (#1+2*#4/2+#3*#4/2,#2)}
\begin{scope}[xshift=0cm, yshift=-0.8cm]
\draw (0,0) -- (0,-2); \comult(1.5,0)[1,1]; \comult(0,-2)[1,1]; \comult(2,-2)[1,1]; \raction(0,-1)[0,1]; \draw (2,-1) -- (2,-2); \ractiontr(0.5,-3)[0,1,90]; \draw (-0.5,-3) -- (-0.5,-6); \draw (2.5,-3) -- (2.5,-4); \laction(0.5,-4)[0,2];
\end{scope}
\begin{scope}[xshift=3cm, yshift=-3.3cm]
\node at (0,-0.5){=};
\end{scope}
\begin{scope}[xshift=4.15cm, yshift=0cm]
\draw (0,0) -- (0,-1.5); \comult(2.75,0)[1.5,1.5]; \comult(0,-1.5)[1,1]; \comult(2,-1.5)[1,1]; \flip(0.5,-2.5)[1]; \draw (-0.5,-2.5) -- (-0.5,-3.5); \draw (2.5,-2.5) -- (2.5,-3.5); \raction(-0.5,-3.5)[0,1]; \raction(1.5,-3.5)[0,1]; \draw (3.5,-1.5) -- (3.5,-3.5); \laction(1.5,-5.5)[0.5,2];
\comult(3.5,-3)[1,1]; \ractiontr(1.5,-4)[0,1.5,90]; \draw (4,-4) -- (4,-5.5); \draw (-0.5,-4) -- (-0.5,-7.5);
\end{scope}
\begin{scope}[xshift=8.7cm, yshift=-3.3cm]
\node at (0,-0.5){=};
\end{scope}
\begin{scope}[xshift=10.15cm, yshift=-0.5cm]
\comult(0,0)[1.5,1.5]; \comult(2.5,0)[1.5,1.5]; \flip(0.75,-1.5)[1]; \comult(3.25,-1.5)[1,1]; \raction(1.75,-2.5)[0,1]; \comult(3.75,-2.5)[1,1];
\ractiontr(1.75,-3.5)[1,1,90]; \laction(1.75,-4.5)[0.5,2]; \draw (4.25,-3.5) -- (4.25,-4.5); \draw (-0.75,-1.5) -- (-0.75,-2.5); \raction(-0.75,-2.5)[1,1]; \draw (-0.75,-2.5) -- (-0.75,-6.5);
\end{scope}
\begin{scope}[xshift=15cm, yshift=-3.3cm]
\node at (0,-0.5){=};
\end{scope}
\begin{scope}[xshift=16cm, yshift=-2.25cm]
\comult(0,0)[1,1]; \comult(2,0)[1,1]; \flip(0.5,-1)[1]; \draw (-0.5,-1) -- (-0.5,-2); \draw (2.5,-1) -- (2.5,-2); \raction(-0.5,-2)[0,1]; \laction(1.5,-2)[0,1];
\end{scope}
\begin{scope}[xshift=19.5cm, yshift=-3.3cm]
\node at (0,-0.5){$= R$};
\end{scope}
\end{tikzpicture}
$$
show that the arrow~\eqref{lactionractioninvcomultotcomult} is invertible with inverse
$$
(\tau\ot X) \circ (X \ot \Delta)\circ R^{-1}.
$$
Therefore it is also a morphisms of coalgebras, since $R$, $\tau$ and $\Delta$ are.
\end{proof}

\begin{pro}\label{derivado es iso de coalgebras} The derived map of $r$ is an isomorphism of coalgebras.
\end{pro}

\begin{proof} By Proposition~\ref{no degenerado} we know that $(\tau\ot X)\circ (X \ot \Delta) \circ c$ is an isomorphism of coalgebras. So, the proposition follows immediately from Lemma~\ref{le previo}.
\end{proof}

\begin{lem}\label{cs es de comodulos} The map $c\circ s$ is a morphisms of left comodules (where $X^2$ is considered as a left comodule via $\Delta\ot X$).
\end{lem}

\begin{proof} Using that $\sigma$, $\tau$ and $\tau^{-1}$ are coalgebra maps it follows easily that
$$
(X\ot \epsilon)\circ c\circ s = (\epsilon\ot X)\circ s = X\ot \epsilon.
$$
Combining this with the fact that $\hat{s}:= c\circ s$ is a coalgebra map, we obtain that
$$
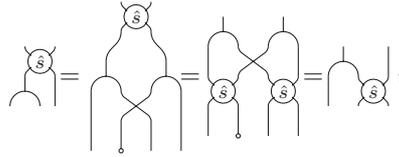
\begin{tikzpicture}[scale=0.395, baseline=(current  bounding  box.center)]
\def\counit(#1,#2){\draw (#1,#2) -- (#1,#2-0.93) (#1,#2-1) circle[radius=2pt]}
\def\comult(#1,#2)[#3,#4]{\draw (#1,#2) -- (#1,#2-0.5*#4) arc (90:0:0.5*#3 and 0.5*#4) (#1,#2-0.5*#4) arc (90:180:0.5*#3 and 0.5*#4)}
\def\laction(#1,#2)[#3,#4]{\draw (#1,#2) .. controls (#1,#2-0.555*#4/2) and (#1+0.445*#4/2,#2-1*#4/2) .. (#1+1*#4/2,#2-1*#4/2) -- (#1+2*#4/2+#3*#4/2,#2-1*#4/2) (#1+2*#4/2+#3*#4/2,#2)--(#1+2*#4/2+#3*#4/2,#2-2*#4/2)}
\def\lactiontr(#1,#2)[#3,#4,#5]{\draw (#1,#2) .. controls (#1,#2-0.555*#4/2) and (#1+0.445*#4/2,#2-1*#4/2) .. (#1+1*#4/2,#2-1*#4/2) -- (#1+2*#4/2+#3*#4/2,#2-1*#4/2)  node [inner sep=0pt, minimum size=3pt,shape=isosceles triangle,fill, shape border rotate=#5] {} (#1+2*#4/2+#3*#4/2,#2) --(#1+2*#4/2+#3*#4/2,#2-2*#4/2)}
\def\doublemap(#1,#2)[#3]{\draw (#1+0.5,#2-0.5) node [name=doublemapnode,inner xsep=0pt, inner ysep=0pt, minimum height=11pt, minimum width=23pt,shape=rectangle,draw,rounded corners] {$#3$} (#1,#2) .. controls (#1,#2-0.075) .. (doublemapnode) (#1+1,#2) .. controls (#1+1,#2-0.075).. (doublemapnode) (doublemapnode) .. controls (#1,#2-0.925)..(#1,#2-1) (doublemapnode) .. controls (#1+1,#2-0.925).. (#1+1,#2-1)}
\def\solbraid(#1,#2)[#3]{\draw (#1,#2-0.5)  node[name=nodemap,inner sep=0pt,  minimum size=9pt, shape=circle,draw]{$#3$}
(#1-0.5,#2) .. controls (#1-0.5,#2-0.15) and (#1-0.4,#2-0.2) .. (#1-0.3,#2-0.3) (#1-0.3,#2-0.3) -- (nodemap)
(#1+0.5,#2) .. controls (#1+0.5,#2-0.15) and (#1+0.4,#2-0.2) .. (#1+0.3,#2-0.3) (#1+0.3,#2-0.3) -- (nodemap)
(#1+0.5,#2-1) .. controls (#1+0.5,#2-0.85) and (#1+0.4,#2-0.8) .. (#1+0.3,#2-0.7) (#1+0.3,#2-0.7) -- (nodemap)
(#1-0.5,#2-1) .. controls (#1-0.5,#2-0.85) and (#1-0.4,#2-0.8) .. (#1-0.3,#2-0.7) (#1-0.3,#2-0.7) -- (nodemap)
}
\def\ractiontr(#1,#2)[#3,#4,#5]{\draw (#1,#2) -- (#1,#2-2*#4/2)  (#1,#2-1*#4/2) node [inner sep=0pt, minimum size=3pt,shape=isosceles triangle,fill, shape border rotate=#5] {}  --(#1+1*#4/2+#3*#4/2,#2-1*#4/2) .. controls (#1+1.555*#4/2+#3*#4/2,#2-1*#4/2) and (#1+2*#4/2+#3*#4/2,#2-0.555*#4/2) .. (#1+2*#4/2+#3*#4/2,#2)  }
\def\flip(#1,#2)[#3]{\draw (
#1+1*#3,#2) .. controls (#1+1*#3,#2-0.05*#3) and (#1+0.96*#3,#2-0.15*#3).. (#1+0.9*#3,#2-0.2*#3)
(#1+0.1*#3,#2-0.8*#3)--(#1+0.9*#3,#2-0.2*#3)
(#1,#2-1*#3) .. controls (#1,#2-0.95*#3) and (#1+0.04*#3,#2-0.85*#3).. (#1+0.1*#3,#2-0.8*#3)
(#1,#2) .. controls (#1,#2-0.05*#3) and (#1+0.04*#3,#2-0.15*#3).. (#1+0.1*#3,#2-0.2*#3)
(#1+0.1*#3,#2-0.2*#3) -- (#1+0.9*#3,#2-0.8*#3)
(#1+1*#3,#2-1*#3) .. controls (#1+1*#3,#2-0.95*#3) and (#1+0.96*#3,#2-0.85*#3).. (#1+0.9*#3,#2-0.8*#3)}
\def\raction(#1,#2)[#3,#4]{\draw (#1,#2) -- (#1,#2-2*#4/2)  (#1,#2-1*#4/2)--(#1+1*#4/2+#3*#4/2,#2-1*#4/2) .. controls (#1+1.555*#4/2+#3*#4/2,#2-1*#4/2) and (#1+2*#4/2+#3*#4/2,#2-0.555*#4/2) .. (#1+2*#4/2+#3*#4/2,#2)}
\begin{scope}[xshift=2.5cm, yshift=-1.5cm]
\solbraid(0.5,0)[\scriptstyle \hat{s}]; \comult(0,-1)[1,1]; \draw (1,-1)-- (1,-2);
\end{scope}
\begin{scope}[xshift=4cm, yshift=-2.05cm]
\node at (0,-0.5){=};
\end{scope}
\begin{scope}[xshift=5.7cm, yshift=0cm]
\solbraid(0.5,0)[\scriptstyle \hat{s}]; \draw (1,-1) .. controls (1,-1.5) and (1.5,-1.5) .. (1.5,-2);  \draw (0,-1) .. controls (0,-1.5) and (-0.5,-1.5) .. (-0.5,-2); \comult(-0.5,-2)[1,1]; \comult(1.5,-2)[1,1]; \flip(0,-3)[1]; \draw (-1,-3)-- (-1,-5); \draw (1,-4)-- (1,-5); \draw (2,-3)-- (2,-5); \counit(0,-4);
\end{scope}
\begin{scope}[xshift=8cm, yshift=-2.05cm]
\node at (0,-0.5){=};
\end{scope}
\begin{scope}[xshift=9.6cm, yshift=-0.5cm]
\comult(-0.5,0)[1,1]; \comult(1.5,0)[1,1]; \flip(0,-1)[1]; \draw (-1,-1)-- (-1,-2); \draw (2,-1)-- (2,-2); \solbraid(-0.5,-2)[\scriptstyle \hat{s}];
\solbraid(1.5,-2)[\scriptstyle \hat{s}]; \counit(0,-3); \draw (-1,-3)-- (-1,-4); \draw (1,-3)-- (1,-4);  \draw (2,-3)-- (2,-4);
\end{scope}
\begin{scope}[xshift=12.1cm, yshift=-2.05cm]
\node at (0,-0.5){=};
\end{scope}
\begin{scope}[xshift=13.1cm, yshift=-1.5cm]
\solbraid(1,-1)[\scriptstyle \hat{s}]; \comult(0,0)[1,1]; \draw (-0.5,-1)-- (-0.5,-2); \draw (1.5,0)-- (1.5,-1);
\end{scope}
\begin{scope}[xshift=15cm, yshift=-2.05cm]
\node at (0,-0.5){,};
\end{scope}
\end{tikzpicture}
$$
as desired.
\end{proof}

\begin{pro}\label{formula para s} The derived map $s$ of $r$ satisfies $s=(\triangleright \ot X)\circ (X\ot c)\circ (\Delta\ot X)$, where $\triangleright :=(X \ot \epsilon) \circ s$.
\end{pro}

\begin{proof} Since by Lemma~\ref{cs es de comodulos}
$$
\begin{tikzpicture}[scale=0.395, baseline=(current  bounding  box.center)]
\def\counit(#1,#2){\draw (#1,#2) -- (#1,#2-0.93) (#1,#2-1) circle[radius=2pt]}
\def\comult(#1,#2)[#3,#4]{\draw (#1,#2) -- (#1,#2-0.5*#4) arc (90:0:0.5*#3 and 0.5*#4) (#1,#2-0.5*#4) arc (90:180:0.5*#3 and 0.5*#4)}
\def\laction(#1,#2)[#3,#4]{\draw (#1,#2) .. controls (#1,#2-0.555*#4/2) and (#1+0.445*#4/2,#2-1*#4/2) .. (#1+1*#4/2,#2-1*#4/2) -- (#1+2*#4/2+#3*#4/2,#2-1*#4/2) (#1+2*#4/2+#3*#4/2,#2)--(#1+2*#4/2+#3*#4/2,#2-2*#4/2)}
\def\lactiontr(#1,#2)[#3,#4,#5]{\draw (#1,#2) .. controls (#1,#2-0.555*#4/2) and (#1+0.445*#4/2,#2-1*#4/2) .. (#1+1*#4/2,#2-1*#4/2) -- (#1+2*#4/2+#3*#4/2,#2-1*#4/2)  node [inner sep=0pt, minimum size=3pt,shape=isosceles triangle,fill, shape border rotate=#5] {} (#1+2*#4/2+#3*#4/2,#2) --(#1+2*#4/2+#3*#4/2,#2-2*#4/2)}
\def\doublemap(#1,#2)[#3]{\draw (#1+0.5,#2-0.5) node [name=doublemapnode,inner xsep=0pt, inner ysep=0pt, minimum height=11pt, minimum width=23pt,shape=rectangle,draw,rounded corners] {$#3$} (#1,#2) .. controls (#1,#2-0.075) .. (doublemapnode) (#1+1,#2) .. controls (#1+1,#2-0.075).. (doublemapnode) (doublemapnode) .. controls (#1,#2-0.925)..(#1,#2-1) (doublemapnode) .. controls (#1+1,#2-0.925).. (#1+1,#2-1)}
\def\solbraid(#1,#2)[#3]{\draw (#1,#2-0.5)  node[name=nodemap,inner sep=0pt,  minimum size=9pt, shape=circle,draw]{$#3$}
(#1-0.5,#2) .. controls (#1-0.5,#2-0.15) and (#1-0.4,#2-0.2) .. (#1-0.3,#2-0.3) (#1-0.3,#2-0.3) -- (nodemap)
(#1+0.5,#2) .. controls (#1+0.5,#2-0.15) and (#1+0.4,#2-0.2) .. (#1+0.3,#2-0.3) (#1+0.3,#2-0.3) -- (nodemap)
(#1+0.5,#2-1) .. controls (#1+0.5,#2-0.85) and (#1+0.4,#2-0.8) .. (#1+0.3,#2-0.7) (#1+0.3,#2-0.7) -- (nodemap)
(#1-0.5,#2-1) .. controls (#1-0.5,#2-0.85) and (#1-0.4,#2-0.8) .. (#1-0.3,#2-0.7) (#1-0.3,#2-0.7) -- (nodemap)
}
\def\ractiontr(#1,#2)[#3,#4,#5]{\draw (#1,#2) -- (#1,#2-2*#4/2)  (#1,#2-1*#4/2) node [inner sep=0pt, minimum size=3pt,shape=isosceles triangle,fill, shape border rotate=#5] {}  --(#1+1*#4/2+#3*#4/2,#2-1*#4/2) .. controls (#1+1.555*#4/2+#3*#4/2,#2-1*#4/2) and (#1+2*#4/2+#3*#4/2,#2-0.555*#4/2) .. (#1+2*#4/2+#3*#4/2,#2)  }
\def\flip(#1,#2)[#3]{\draw (
#1+1*#3,#2) .. controls (#1+1*#3,#2-0.05*#3) and (#1+0.96*#3,#2-0.15*#3).. (#1+0.9*#3,#2-0.2*#3)
(#1+0.1*#3,#2-0.8*#3)--(#1+0.9*#3,#2-0.2*#3)
(#1,#2-1*#3) .. controls (#1,#2-0.95*#3) and (#1+0.04*#3,#2-0.85*#3).. (#1+0.1*#3,#2-0.8*#3)
(#1,#2) .. controls (#1,#2-0.05*#3) and (#1+0.04*#3,#2-0.15*#3).. (#1+0.1*#3,#2-0.2*#3)
(#1+0.1*#3,#2-0.2*#3) -- (#1+0.9*#3,#2-0.8*#3)
(#1+1*#3,#2-1*#3) .. controls (#1+1*#3,#2-0.95*#3) and (#1+0.96*#3,#2-0.85*#3).. (#1+0.9*#3,#2-0.8*#3)}
\def\raction(#1,#2)[#3,#4]{\draw (#1,#2) -- (#1,#2-2*#4/2)  (#1,#2-1*#4/2)--(#1+1*#4/2+#3*#4/2,#2-1*#4/2) .. controls (#1+1.555*#4/2+#3*#4/2,#2-1*#4/2) and (#1+2*#4/2+#3*#4/2,#2-0.555*#4/2) .. (#1+2*#4/2+#3*#4/2,#2)}
\begin{scope}[xshift=0cm, yshift=-1cm]
\solbraid(0.5,0)[\scriptstyle s]; \flip(0,-1)[1];
\end{scope}
\begin{scope}[xshift=1.5cm, yshift=-1.55cm]
\node at (0,-0.5){=};
\end{scope}
\begin{scope}[xshift=2.5cm, yshift=0cm]
\solbraid(0.5,0)[\scriptstyle s]; \flip(0,-1)[1]; \comult(0,-2)[1,1]; \draw (1,-2)-- (1,-4); \draw (-0.5,-3)-- (-0.5,-4); \counit(0.5,-3);
\end{scope}
\begin{scope}[xshift=4cm, yshift=-1.55cm]
\node at (0,-0.5){=};
\end{scope}
\begin{scope}[xshift=5.05cm, yshift=0cm]
\solbraid(1,-1)[\scriptstyle s]; \flip(0.5,-2)[1]; \comult(0,0)[1,1]; \draw (-0.5,-1)-- (-0.5,-3); \draw (-0.5,-3)-- (-0.5,-4); \counit(0.5,-3); \draw (1.5,0)-- (1.5,-1); \draw (1.5,-3)-- (1.5,-4);
\end{scope}
\begin{scope}[xshift=7.1cm, yshift=-1.55cm]
\node at (0,-0.5){=};
\end{scope}
\begin{scope}[xshift=8.1cm, yshift=-0.5cm]
\solbraid(1,-1)[\scriptstyle s]; \comult(0,0)[1,1]; \draw (-0.5,-1)-- (-0.5,-3); \counit(1.5,-2); \draw (0.5,-2)-- (0.5,-3); \draw (1.5,0)-- (1.5,-1);
\end{scope}
\begin{scope}[xshift=10cm, yshift=-1.55cm]
\node at (0,-0.5){,};
\end{scope}
\end{tikzpicture}
$$
we have
$$
\begin{tikzpicture}[scale=0.395, baseline=(current  bounding  box.center)]
\def\counit(#1,#2){\draw (#1,#2) -- (#1,#2-0.93) (#1,#2-1) circle[radius=2pt]}
\def\comult(#1,#2)[#3,#4]{\draw (#1,#2) -- (#1,#2-0.5*#4) arc (90:0:0.5*#3 and 0.5*#4) (#1,#2-0.5*#4) arc (90:180:0.5*#3 and 0.5*#4)}
\def\laction(#1,#2)[#3,#4]{\draw (#1,#2) .. controls (#1,#2-0.555*#4/2) and (#1+0.445*#4/2,#2-1*#4/2) .. (#1+1*#4/2,#2-1*#4/2) -- (#1+2*#4/2+#3*#4/2,#2-1*#4/2) (#1+2*#4/2+#3*#4/2,#2)--(#1+2*#4/2+#3*#4/2,#2-2*#4/2)}
\def\lactiontr(#1,#2)[#3,#4,#5]{\draw (#1,#2) .. controls (#1,#2-0.555*#4/2) and (#1+0.445*#4/2,#2-1*#4/2) .. (#1+1*#4/2,#2-1*#4/2) -- (#1+2*#4/2+#3*#4/2,#2-1*#4/2)  node [inner sep=0pt, minimum size=3pt,shape=isosceles triangle,fill, shape border rotate=#5] {} (#1+2*#4/2+#3*#4/2,#2) --(#1+2*#4/2+#3*#4/2,#2-2*#4/2)}
\def\doublemap(#1,#2)[#3]{\draw (#1+0.5,#2-0.5) node [name=doublemapnode,inner xsep=0pt, inner ysep=0pt, minimum height=11pt, minimum width=23pt,shape=rectangle,draw,rounded corners] {$#3$} (#1,#2) .. controls (#1,#2-0.075) .. (doublemapnode) (#1+1,#2) .. controls (#1+1,#2-0.075).. (doublemapnode) (doublemapnode) .. controls (#1,#2-0.925)..(#1,#2-1) (doublemapnode) .. controls (#1+1,#2-0.925).. (#1+1,#2-1)}
\def\solbraid(#1,#2)[#3]{\draw (#1,#2-0.5)  node[name=nodemap,inner sep=0pt,  minimum size=9pt, shape=circle,draw]{$#3$}
(#1-0.5,#2) .. controls (#1-0.5,#2-0.15) and (#1-0.4,#2-0.2) .. (#1-0.3,#2-0.3) (#1-0.3,#2-0.3) -- (nodemap)
(#1+0.5,#2) .. controls (#1+0.5,#2-0.15) and (#1+0.4,#2-0.2) .. (#1+0.3,#2-0.3) (#1+0.3,#2-0.3) -- (nodemap)
(#1+0.5,#2-1) .. controls (#1+0.5,#2-0.85) and (#1+0.4,#2-0.8) .. (#1+0.3,#2-0.7) (#1+0.3,#2-0.7) -- (nodemap)
(#1-0.5,#2-1) .. controls (#1-0.5,#2-0.85) and (#1-0.4,#2-0.8) .. (#1-0.3,#2-0.7) (#1-0.3,#2-0.7) -- (nodemap)
}
\def\ractiontr(#1,#2)[#3,#4,#5]{\draw (#1,#2) -- (#1,#2-2*#4/2)  (#1,#2-1*#4/2) node [inner sep=0pt, minimum size=3pt,shape=isosceles triangle,fill, shape border rotate=#5] {}  --(#1+1*#4/2+#3*#4/2,#2-1*#4/2) .. controls (#1+1.555*#4/2+#3*#4/2,#2-1*#4/2) and (#1+2*#4/2+#3*#4/2,#2-0.555*#4/2) .. (#1+2*#4/2+#3*#4/2,#2)  }
\def\rack(#1,#2)[#3]{\draw (#1,#2-0.5)  node[name=nodemap,inner sep=0pt,  minimum size=7.5pt, shape=circle,draw]{$#3$} (#1-1,#2) .. controls (#1-1,#2-0.5) and (#1-0.5,#2-0.5) .. (nodemap) (#1,#2)-- (nodemap)  (nodemap)-- (#1,#2-1)}
\def\flip(#1,#2)[#3]{\draw (
#1+1*#3,#2) .. controls (#1+1*#3,#2-0.05*#3) and (#1+0.96*#3,#2-0.15*#3).. (#1+0.9*#3,#2-0.2*#3)
(#1+0.1*#3,#2-0.8*#3)--(#1+0.9*#3,#2-0.2*#3)
(#1,#2-1*#3) .. controls (#1,#2-0.95*#3) and (#1+0.04*#3,#2-0.85*#3).. (#1+0.1*#3,#2-0.8*#3)
(#1,#2) .. controls (#1,#2-0.05*#3) and (#1+0.04*#3,#2-0.15*#3).. (#1+0.1*#3,#2-0.2*#3)
(#1+0.1*#3,#2-0.2*#3) -- (#1+0.9*#3,#2-0.8*#3)
(#1+1*#3,#2-1*#3) .. controls (#1+1*#3,#2-0.95*#3) and (#1+0.96*#3,#2-0.85*#3).. (#1+0.9*#3,#2-0.8*#3)}
\def\raction(#1,#2)[#3,#4]{\draw (#1,#2) -- (#1,#2-2*#4/2)  (#1,#2-1*#4/2)--(#1+1*#4/2+#3*#4/2,#2-1*#4/2) .. controls (#1+1.555*#4/2+#3*#4/2,#2-1*#4/2) and (#1+2*#4/2+#3*#4/2,#2-0.555*#4/2) .. (#1+2*#4/2+#3*#4/2,#2)}
\def\doublesinglemap(#1,#2)[#3]{\draw (#1+0.5,#2-0.5) node [name=doublesinglemapnode,inner xsep=0pt, inner ysep=0pt, minimum height=11pt, minimum width=23pt,shape=rectangle,draw,rounded corners] {$#3$} (#1,#2) .. controls (#1,#2-0.075) .. (doublesinglemapnode) (#1+1,#2) .. controls (#1+1,#2-0.075).. (doublesinglemapnode) (doublesinglemapnode)-- (#1+0.5,#2-1)}
\begin{scope}[xshift=1.5cm, yshift=-2.55cm]
\node at (0,-0.5){$s=$};
\end{scope}
\begin{scope}[xshift=2.8cm, yshift=-1.5cm]
\solbraid(1,-1)[\scriptstyle s]; \comult(0,0)[1,1];  \counit(1.5,-2); \draw (1.5,0)-- (1.5,-1); \flip(-0.5,-2)[1]; \draw (-0.5,-1)-- (-0.5,-2);
\end{scope}
\begin{scope}[xshift=4.8cm, yshift=-2.55cm]
\node at (0,-0.5){=};
\end{scope}
\begin{scope}[xshift=5.8cm, yshift=-0.5cm]
\solbraid(0,-3)[\scriptstyle s]; \comult(0,0)[1,1];  \flip(-0.5,-1)[1]; \flip(0.5,-2)[1]; \counit(0.5,-4); \draw (-0.5,-2)-- (-0.5,-3); \draw (-0.5,-4)-- (-0.5,-5); \draw (1.5,0)-- (1.5,-2); \draw (1.5,-3)-- (1.5,-5);
\end{scope}
\begin{scope}[xshift=7.95cm, yshift=-2.55cm]
\node at (0,-0.5){=};
\end{scope}
\begin{scope}[xshift=9.1cm, yshift=-1cm]
\solbraid(0,-2)[\scriptstyle s]; \comult(0,0)[1,1]; \flip(0.5,-1)[1]; \counit(0.5,-3); \draw (-0.5,-1)-- (-0.5,-2); \draw (-0.5,-3)-- (-0.5,-4);
\draw (1.5,0)-- (1.5,-1); \draw (1.5,-2)-- (1.5,-4);
\end{scope}
\begin{scope}[xshift=11.2cm, yshift=-2.55cm]
\node at (0,-0.5){=};
\end{scope}
\begin{scope}[xshift=12.5cm, yshift=-1.5cm]
\comult(0,0)[1,1]; \flip(0.5,-1)[1]; \rack(0.5,-2)[\scriptstyle \triangleright]; \draw (-0.5,-1)-- (-0.5,-2); \draw (1.5,0)-- (1.5,-1); \draw (1.5,-2)-- (1.5,-3);
\end{scope}
\begin{scope}[xshift=14.3cm, yshift=-2.55cm]
\node at (0,-0.5){,};
\end{scope}
\end{tikzpicture}
$$
as desired.
\end{proof}

\begin{rem}\label{el rack asociado} From Proposition~\ref{formula para s} it follows that $s=r_{\triangleright}$, where $\triangleright \coloneqq (X\ot \epsilon)\circ s$. In the following subsection we will see that $(X,s)$ is a non-degenerate braided set. By Proposition~\ref{solucion tipo rack} this implies that $(X,\triangleright)$ is a rack.
\end{rem}

\subsection{The guitar map}
We now define a remarkable map that allows us to study braid group representations related to braidings. In the case of set-theoretical solutions, this map was considered in~\cite{MR1722951,MR3558231,MR1769723,MR2906433,MR1809284}.

In this subsection, for each $i,j\in \mathds{N}$ with $(i,j)\ne (1,1)$ we set $c_{ij}\coloneqq c_{X^i,X^j}$ (we keep the notation $c$, which we have been using so far, for $c_{X,X}$).

\begin{notations}\label{notacion alfan Qn} Let $r$ be a coalgebra automorphism of $X^2$. In the sequel, for each $n\ge 2$, we let $\alpha_n,Q_n \colon X^n\to X^n$ denote the maps recursively defined by

\begin{itemize}

\smallskip

\item[-] $\alpha_2= Q_2\coloneqq (\tau\ot X)\circ (X\ot \Delta)$,

\smallskip

\item[-] $\alpha_{n+1} \coloneqq (\tau \ot X^n) \circ (X \ot c_{n-1,1} \ot X) \circ (\alpha_n \ot \Delta)$,

\smallskip

\item[-] $Q_{n+1} \coloneqq (\tau \ot Q_n) \circ (X \ot c_{n-1,1} \ot X) \circ (X^n \ot \Delta)$.

\end{itemize}
\end{notations}

\begin{defn}\label{map Jn} For each $n\in\mathds{N}$,  we define the map $J_n\colon X^n\to X^n$ recursively~by:

\begin{itemize}

\smallskip

\item[-] $J_1 \coloneqq \ide_X$,

\smallskip

\item[-] $J_{n+1}\coloneqq (X\ot J_n)\circ \alpha_{n+1}$.

\end{itemize}
\end{defn}

\begin{rem}\label{Jn es inversible} Let $n\ge 2$. A direct computation
proves that
$$
(c_{n-1,1} \ot X) \circ (X^{n-1} \ot \Delta) = (X\ot c_{1,n-1}) \circ (\Delta\ot X^{n-1}) \circ c_{n-1,1}.
$$
Using this fact it is easy to see that
\begin{itemize}
\item[-] $\alpha_{n+1} = (X\ot c_{1,n-1})\circ (\alpha_2 \ot X^{n-1}) \circ (X \ot c_{n-1,1}) \circ (\alpha_n \ot X)$,

\item[-] $Q_{n+1} = (X \ot Q_n) \circ (X \ot c_{1,n-1}) \circ (Q_2\ot X^{n-1})\circ (X \ot c_{n-1,1})$.
\end{itemize}
Thus if $r$ is non-degenerate, then $\alpha_n$ and $Q_n$ are invertible for all~$n$. Consequently, under this condition~$J_n$ is also.
\end{rem}

\begin{pro}\label{Jn con Qn} The equality $J_n= Q_n\circ (J_{n-1}\ot X)$ holds for each $n\ge 2$.
\end{pro}

\begin{proof} For $n=2$ this is true by definition. Suppose that it is true for~$n$. Then,
\begin{align*}
J_{n+1} & = (X\ot J_n) \circ \alpha_{n+1}\\[0.7pt]
& = (X\ot Q_n)\circ (X\ot J_{n-1}\ot X) \circ \alpha_{n+1}\\[0.7pt]
& = (X\ot Q_n)\circ (X\ot J_{n-1}\ot X) \circ (\tau \ot X^n) \circ (X \ot c_{n-1,1} \ot X)\circ (\alpha_n \ot \Delta)\\[0.7pt]
& = (\tau \ot Q_n)\circ (X \ot c_{n-1,1} \ot X) \circ (X^n \ot \Delta)\circ (X\ot J_{n-1}\ot X) \circ (\alpha_n\ot X)\\[0.7pt]
& = (\tau \ot Q_n)\circ (X \ot c_{n-1,1} \ot X) \circ (X^n \ot \Delta)\circ (J_n\ot X)\\[0.7pt]
& = Q_{n+1}\circ (J_n\ot X),
\end{align*}
as desired.
\end{proof}

\begin{pro}\label{prop auxiliar} Let $(X,r)$ be a non-degenerate braided set and let $\triangleright$ be as in Remark~\ref{el rack asociado}. The following equality holds:
$$
\tau\circ r= \triangleright \circ (\tau \ot X) \circ (X \ot \Delta).
$$
\end{pro}

\begin{proof} Let $R$ be the $\Ree$-matrix of $r$.  Since $\triangleright = (X\ot \epsilon)\circ s$ (where $s$ is the derived map of $r$ introduced in Definition~\ref{solucion asociada}), $\tau$ is a morphisms of coalgebras, $\Delta$ is coassociative and $(X,r)$ is non-degenerate, we have
$$
\begin{tikzpicture}[scale=0.395]
\def\counit(#1,#2){\draw (#1,#2) -- (#1,#2-0.93) (#1,#2-1) circle[radius=2pt]}
\def\comult(#1,#2)[#3,#4]{\draw (#1,#2) -- (#1,#2-0.5*#4) arc (90:0:0.5*#3 and 0.5*#4) (#1,#2-0.5*#4) arc (90:180:0.5*#3 and 0.5*#4)}
\def\laction(#1,#2)[#3,#4]{\draw (#1,#2) .. controls (#1,#2-0.555*#4/2) and (#1+0.445*#4/2,#2-1*#4/2) .. (#1+1*#4/2,#2-1*#4/2) -- (#1+2*#4/2+#3*#4/2,#2-1*#4/2) (#1+2*#4/2+#3*#4/2,#2)--(#1+2*#4/2+#3*#4/2,#2-2*#4/2)}
\def\lactiontr(#1,#2)[#3,#4,#5]{\draw (#1,#2) .. controls (#1,#2-0.555*#4/2) and (#1+0.445*#4/2,#2-1*#4/2) .. (#1+1*#4/2,#2-1*#4/2) -- (#1+2*#4/2+#3*#4/2,#2-1*#4/2)  node [inner sep=0pt, minimum size=3pt,shape=isosceles triangle,fill, shape border rotate=#5] {} (#1+2*#4/2+#3*#4/2,#2) --(#1+2*#4/2+#3*#4/2,#2-2*#4/2)}
\def\doublemap(#1,#2)[#3]{\draw (#1+0.5,#2-0.5) node [name=doublemapnode,inner xsep=0pt, inner ysep=0pt, minimum height=9pt, minimum width=23pt,shape=rectangle,draw,rounded corners] {$#3$} (#1,#2) .. controls (#1,#2-0.075) .. (doublemapnode) (#1+1,#2) .. controls (#1+1,#2-0.075).. (doublemapnode) (doublemapnode) .. controls (#1,#2-0.925)..(#1,#2-1) (doublemapnode) .. controls (#1+1,#2-0.925).. (#1+1,#2-1)}
\def\solbraid(#1,#2)[#3]{\draw (#1,#2-0.5)  node[name=nodemap,inner sep=0pt,  minimum size=9pt, shape=circle,draw]{$#3$}
(#1-0.5,#2) .. controls (#1-0.5,#2-0.15) and (#1-0.4,#2-0.2) .. (#1-0.3,#2-0.3) (#1-0.3,#2-0.3) -- (nodemap)
(#1+0.5,#2) .. controls (#1+0.5,#2-0.15) and (#1+0.4,#2-0.2) .. (#1+0.3,#2-0.3) (#1+0.3,#2-0.3) -- (nodemap)
(#1+0.5,#2-1) .. controls (#1+0.5,#2-0.85) and (#1+0.4,#2-0.8) .. (#1+0.3,#2-0.7) (#1+0.3,#2-0.7) -- (nodemap)
(#1-0.5,#2-1) .. controls (#1-0.5,#2-0.85) and (#1-0.4,#2-0.8) .. (#1-0.3,#2-0.7) (#1-0.3,#2-0.7) -- (nodemap)
}
\def\ractiontr(#1,#2)[#3,#4,#5]{\draw (#1,#2) -- (#1,#2-2*#4/2)  (#1,#2-1*#4/2) node [inner sep=0pt, minimum size=3pt,shape=isosceles triangle,fill, shape border rotate=#5] {}  --(#1+1*#4/2+#3*#4/2,#2-1*#4/2) .. controls (#1+1.555*#4/2+#3*#4/2,#2-1*#4/2) and (#1+2*#4/2+#3*#4/2,#2-0.555*#4/2) .. (#1+2*#4/2+#3*#4/2,#2)  }
\def\flip(#1,#2)[#3]{\draw (
#1+1*#3,#2) .. controls (#1+1*#3,#2-0.05*#3) and (#1+0.96*#3,#2-0.15*#3).. (#1+0.9*#3,#2-0.2*#3)
(#1+0.1*#3,#2-0.8*#3)--(#1+0.9*#3,#2-0.2*#3)
(#1,#2-1*#3) .. controls (#1,#2-0.95*#3) and (#1+0.04*#3,#2-0.85*#3).. (#1+0.1*#3,#2-0.8*#3)
(#1,#2) .. controls (#1,#2-0.05*#3) and (#1+0.04*#3,#2-0.15*#3).. (#1+0.1*#3,#2-0.2*#3)
(#1+0.1*#3,#2-0.2*#3) -- (#1+0.9*#3,#2-0.8*#3)
(#1+1*#3,#2-1*#3) .. controls (#1+1*#3,#2-0.95*#3) and (#1+0.96*#3,#2-0.85*#3).. (#1+0.9*#3,#2-0.8*#3)
}
\def\raction(#1,#2)[#3,#4]{\draw (#1,#2) -- (#1,#2-2*#4/2)  (#1,#2-1*#4/2)--(#1+1*#4/2+#3*#4/2,#2-1*#4/2) .. controls (#1+1.555*#4/2+#3*#4/2,#2-1*#4/2) and (#1+2*#4/2+#3*#4/2,#2-0.555*#4/2) .. (#1+2*#4/2+#3*#4/2,#2)}
\def\rack(#1,#2)[#3]{\draw (#1,#2-0.5)  node[name=nodemap,inner sep=0pt,  minimum size=7.5pt, shape=circle,draw]{$#3$} (#1-1,#2) .. controls (#1-1,#2-0.5) and (#1-0.5,#2-0.5) .. (nodemap) (#1,#2)-- (nodemap)  (nodemap)-- (#1,#2-1)}
\def\racklarge(#1,#2)[#3]{\draw (#1,#2-0.5)  node[name=nodemap,inner sep=0pt,  minimum size=7.5pt, shape=circle,draw]{$#3$} (#1-2,#2+0.5) .. controls (#1-2,#2-0.5) and (#1-0.5,#2-0.5) .. (nodemap) (#1,#2)-- (nodemap)  (nodemap)-- (#1,#2-1)}
\begin{scope}[xshift=0cm, yshift=-3.45cm]
\comult(1.5,0)[1,1];  \raction(0,-1)[0,1]; \draw (2,-1) -- (2,-2.5); \draw (0,0) -- (0,-1); \racklarge(2,-2.5)[\scriptstyle \triangleright];
\end{scope}
\begin{scope}[xshift=2.7cm, yshift=-4.75cm]
\node at (0,-0.5){=};
\end{scope}
\begin{scope}[xshift=3.8cm, yshift=-1.25cm]
\comult(1.5,0)[1,1];  \raction(0,-1)[0,1]; \draw (2,-1) -- (2,-2); \draw (0,0) -- (0,-1); \comult(0,-2)[1,1]; \comult(2,-2)[1,1]; \ractiontr(0.5,-3)[0,1,90]; \laction(0.5,-4)[0,2]; \draw (2.5,-3) -- (2.5,-4); \flip(1.5,-6)[1]; \draw (-0.5,-3) -- (-0.5,-4);   \draw (-0.5,-4) .. controls (-0.5,-5) and (1.5,-5.5) .. (1.5,-6); \raction(1.5,-7)[0,1];
\end{scope}
\begin{scope}[xshift=6.8cm, yshift=-4.75cm]
\node at (0,-0.5){=};
\end{scope}
\begin{scope}[xshift=7.4cm, yshift=-0.75cm]
\comult(0.5,-1)[1,1]; \comult(2.5,-1)[1,1]; \comult(3,0)[1,1]; \draw (0.5,0) -- (0.5,-1); \flip(1,-2)[1]; \raction(0,-3)[0,1]; \raction(2,-3)[0,1]; \draw (0,-2) -- (0,-3); \draw (3,-2) -- (3,-3); \draw (3.5,-1) -- (3.5,-3.5); \comult(3.5,-3)[1,1]; \ractiontr(2,-4)[0,1,90]; \laction(2,-5)[0,2]; \draw (4,-4) -- (4,-5);\flip(3,-7)[1]; \draw (0,-4) .. controls (0,-6) and (3,-6.5) .. (3,-7); \raction(3,-8)[0,1];
\end{scope}
\begin{scope}[xshift=11.95cm, yshift=-4.75cm]
\node at (0,-0.5){=};
\end{scope}
\begin{scope}[xshift=12.5cm, yshift=0cm]
\comult(0.5,-1)[1,1];  \comult(3,0)[2,2];\comult(4,-2)[1,1]; \flip(1,-2)[1]; \comult(3.5,-3)[1,1]; \raction(2,-4)[0,1]; \ractiontr(2,-4.5)[0,2,90]; \laction(2,-6.5)[0.5,2]; \draw (4,-4) -- (4,-4.5); \draw (4.5,-3) -- (4.5,-6.5); \draw (2,-3) -- (2,-4); \raction(0,-3)[0,1]; \draw (0,-2) -- (0,-3); \draw (0,-4) .. controls (0,-7.5) and (3.5,-8) .. (3.5,-8.5); \flip(3.5,-8.5)[1]; \raction(3.5,-9.5)[0,1];
\end{scope}
\begin{scope}[xshift=17.55cm, yshift=-4.75cm]
\node at (0,-0.5){=};
\end{scope}
\begin{scope}[xshift=18.40cm, yshift=-3.7cm]
\doublemap(0,0)[\scriptstyle R]; \flip(0,-1)[1]; \raction(0,-2)[0,1];
\end{scope}
\begin{scope}[xshift=20.35cm, yshift=-4.75cm]
\node at (0,-0.5){=};
\end{scope}
\begin{scope}[xshift=21cm, yshift=-4.2cm]
\solbraid(0.5,0)[\scriptstyle r]; \raction(0,-1)[0,1];
\end{scope}
\begin{scope}[xshift=22.40cm, yshift=-4.75cm]
\node at (0,-0.5){,};
\end{scope}
\end{tikzpicture}
$$
as desired.
\end{proof}

\begin{pro}\label{tau es morfismo de comodulos} Let $(X,r)$ and $\triangleright$ be as in Proposition~\ref{prop auxiliar}. Then
\begin{equation*}\label{eq: tau es morfismo de comodulos}
\tau\circ (\triangleright \ot X) = \triangleright \circ (\tau \ot \tau) \circ (X \ot c \ot X) \circ (X^2 \ot \Delta).
\end{equation*}
\end{pro}

\begin{proof} By Remarks~\ref{prop de involutiva y de braided} and~\ref{algunas formulas}(1), the fact that $(X,r)$ is a non-degenerate braided set and Proposition~\ref{prop auxiliar}, we have
$$
\begin{tikzpicture}[scale=0.395]
\def\counit(#1,#2){\draw (#1,#2) -- (#1,#2-0.93) (#1,#2-1) circle[radius=2pt]}
\def\comult(#1,#2)[#3,#4]{\draw (#1,#2) -- (#1,#2-0.5*#4) arc (90:0:0.5*#3 and 0.5*#4) (#1,#2-0.5*#4) arc (90:180:0.5*#3 and 0.5*#4)}
\def\laction(#1,#2)[#3,#4]{\draw (#1,#2) .. controls (#1,#2-0.555*#4/2) and (#1+0.445*#4/2,#2-1*#4/2) .. (#1+1*#4/2,#2-1*#4/2) -- (#1+2*#4/2+#3*#4/2,#2-1*#4/2) (#1+2*#4/2+#3*#4/2,#2)--(#1+2*#4/2+#3*#4/2,#2-2*#4/2)}
\def\lactiontr(#1,#2)[#3,#4,#5]{\draw (#1,#2) .. controls (#1,#2-0.555*#4/2) and (#1+0.445*#4/2,#2-1*#4/2) .. (#1+1*#4/2,#2-1*#4/2) -- (#1+2*#4/2+#3*#4/2,#2-1*#4/2)  node [inner sep=0pt, minimum size=3pt,shape=isosceles triangle,fill, shape border rotate=#5] {} (#1+2*#4/2+#3*#4/2,#2) --(#1+2*#4/2+#3*#4/2,#2-2*#4/2)}
\def\doublemap(#1,#2)[#3]{\draw (#1+0.5,#2-0.5) node [name=doublemapnode,inner xsep=0pt, inner ysep=0pt, minimum height=11pt, minimum width=23pt,shape=rectangle,draw,rounded corners] {$#3$} (#1,#2) .. controls (#1,#2-0.075) .. (doublemapnode) (#1+1,#2) .. controls (#1+1,#2-0.075).. (doublemapnode) (doublemapnode) .. controls (#1,#2-0.925)..(#1,#2-1) (doublemapnode) .. controls (#1+1,#2-0.925).. (#1+1,#2-1)}
\def\solbraid(#1,#2)[#3]{\draw (#1,#2-0.5)  node[name=nodemap,inner sep=0pt,  minimum size=9pt, shape=circle,draw]{$#3$}
(#1-0.5,#2) .. controls (#1-0.5,#2-0.15) and (#1-0.4,#2-0.2) .. (#1-0.3,#2-0.3) (#1-0.3,#2-0.3) -- (nodemap)
(#1+0.5,#2) .. controls (#1+0.5,#2-0.15) and (#1+0.4,#2-0.2) .. (#1+0.3,#2-0.3) (#1+0.3,#2-0.3) -- (nodemap)
(#1+0.5,#2-1) .. controls (#1+0.5,#2-0.85) and (#1+0.4,#2-0.8) .. (#1+0.3,#2-0.7) (#1+0.3,#2-0.7) -- (nodemap)
(#1-0.5,#2-1) .. controls (#1-0.5,#2-0.85) and (#1-0.4,#2-0.8) .. (#1-0.3,#2-0.7) (#1-0.3,#2-0.7) -- (nodemap)
}
\def\ractiontr(#1,#2)[#3,#4,#5]{\draw (#1,#2) -- (#1,#2-2*#4/2)  (#1,#2-1*#4/2) node [inner sep=0pt, minimum size=3pt,shape=isosceles triangle,fill, shape border rotate=#5] {}  --(#1+1*#4/2+#3*#4/2,#2-1*#4/2) .. controls (#1+1.555*#4/2+#3*#4/2,#2-1*#4/2) and (#1+2*#4/2+#3*#4/2,#2-0.555*#4/2) .. (#1+2*#4/2+#3*#4/2,#2)  }
\def\flip(#1,#2)[#3]{\draw (
#1+1*#3,#2) .. controls (#1+1*#3,#2-0.05*#3) and (#1+0.96*#3,#2-0.15*#3).. (#1+0.9*#3,#2-0.2*#3)
(#1+0.1*#3,#2-0.8*#3)--(#1+0.9*#3,#2-0.2*#3)
(#1,#2-1*#3) .. controls (#1,#2-0.95*#3) and (#1+0.04*#3,#2-0.85*#3).. (#1+0.1*#3,#2-0.8*#3)
(#1,#2) .. controls (#1,#2-0.05*#3) and (#1+0.04*#3,#2-0.15*#3).. (#1+0.1*#3,#2-0.2*#3)
(#1+0.1*#3,#2-0.2*#3) -- (#1+0.9*#3,#2-0.8*#3)
(#1+1*#3,#2-1*#3) .. controls (#1+1*#3,#2-0.95*#3) and (#1+0.96*#3,#2-0.85*#3).. (#1+0.9*#3,#2-0.8*#3)
}
\def\raction(#1,#2)[#3,#4]{\draw (#1,#2) -- (#1,#2-2*#4/2)  (#1,#2-1*#4/2)--(#1+1*#4/2+#3*#4/2,#2-1*#4/2) .. controls (#1+1.555*#4/2+#3*#4/2,#2-1*#4/2) and (#1+2*#4/2+#3*#4/2,#2-0.555*#4/2) .. (#1+2*#4/2+#3*#4/2,#2)}
\def\rack(#1,#2)[#3]{\draw (#1,#2-0.5)  node[name=nodemap,inner sep=0pt,  minimum size=7.5pt, shape=circle,draw]{$#3$} (#1-1,#2) .. controls (#1-1,#2-0.5) and (#1-0.5,#2-0.5) .. (nodemap) (#1,#2)-- (nodemap)  (nodemap)-- (#1,#2-1)}
\def\racklarge(#1,#2)[#3]{\draw (#1,#2-0.5)  node[name=nodemap,inner sep=0pt,  minimum size=7.5pt, shape=circle,draw]{$#3$} (#1-2,#2+0.5) .. controls (#1-2,#2-0.5) and (#1-0.5,#2-0.5) .. (nodemap) (#1,#2)-- (nodemap)  (nodemap)-- (#1,#2-1)}
\def\rackextralarge(#1,#2)[#3]{\draw (#1,#2-0.5)  node[name=nodemap,inner sep=0pt,  minimum size=7.5pt, shape=circle,draw]{$#3$} (#1-3,#2+1) .. controls (#1-3,#2-0.5) and (#1-0.5,#2-0.5) .. (nodemap) (#1,#2)-- (nodemap)  (nodemap)-- (#1,#2-1)}
\begin{scope}[xshift=0cm, yshift=-0.75cm]
\draw (0,0) -- (0,-1); \comult(1.5,0)[1,1];  \comult(3.5,0)[1,1]; \raction(0,-1)[0,1]; \raction(0,-2)[0.665,1.5]; \flip(2,-1)[1]; \draw (4,-1) -- (4,-2); \raction(3,-2)[0,1]; \rackextralarge(3,-5)[\scriptstyle \triangleright]; \draw (0,-3) -- (0,-4); \draw (3,-3) -- (3,-5);
\end{scope}
\begin{scope}[xshift=4.45cm, yshift=-3.25cm]
\node at (0,-0.5){=};
\end{scope}
\begin{scope}[xshift=5cm, yshift=0cm]
\draw (0,0) -- (0,-3); \comult(1.5,0)[1,1];  \comult(3.5,0)[1,1];  \flip(2,-1)[1]; \draw (4,-1) -- (4,-2); \draw (1,-1) -- (1,-2); \raction(3,-2)[0,1]; \solbraid(1.5,-2)[\scriptstyle r]; \raction(0,-3)[0,1]; \raction(0,-4)[0.665,1.5]; \draw (2,-3) -- (2,-4); \rackextralarge(3,-6.5)[\scriptstyle \triangleright]; \draw (3,-3) -- (3,-6.5);
\end{scope}
\begin{scope}[xshift=9.5cm, yshift=-3.25cm]
\node at (0,-0.5){=};
\end{scope}
\begin{scope}[xshift=10.1cm, yshift=-1.2cm]
\draw (0,0) -- (0,-1); \solbraid(1.5,0)[\scriptstyle r]; \raction(0,-1)[0,1]; \comult(2,-1)[1,1]; \raction(0,-2)[1,1];
\racklarge(2,-4)[\scriptstyle \triangleright]; \draw (2.5,-2) .. controls (2.5,-2.5) and (2,-3.5) .. (2,-4); \draw (0,-3) -- (0,-3.5);
\end{scope}
\begin{scope}[xshift=13.15cm, yshift=-3.25cm]
\node at (0,-0.5){=};
\end{scope}
\begin{scope}[xshift=13.75cm, yshift=-1.5cm]
\draw (0,0) -- (0,-1); \solbraid(1.5,0)[\scriptstyle r]; \raction(0,-1)[0,1]; \draw (0,-2) -- (0,-2.5); \solbraid(0.5,-2.5)[\scriptstyle r];  \draw (2,-1) .. controls (2,-1.5) and (1,-2) .. (1,-2.5); \raction(0,-3.5)[0,1];
\end{scope}
\begin{scope}[xshift=16.3cm, yshift=-3.25cm]
\node at (0,-0.5){=};
\end{scope}
\begin{scope}[xshift=16.8cm, yshift=-1.5cm]
\draw (2,0) -- (2,-1); \solbraid(0.5,0)[\scriptstyle r]; \solbraid(1.5,-1)[\scriptstyle r]; \draw (0,-1) -- (0,-2);  \raction(0,-2)[0,1]; \draw (2,-2) -- (2,-2.5); \raction(0,-2.5)[0,2];
\end{scope}
\begin{scope}[xshift=19.3cm, yshift=-3.25cm]
\node at (0,-0.5){=};
\end{scope}
\begin{scope}[xshift=19.8cm, yshift=-2.25cm]
\draw (1.5,0) -- (1.5,-1.5); \solbraid(0.5,0)[\scriptstyle r]; \raction(0,-1)[0,1]; \raction(0,-1.5)[0,1.5];
\end{scope}
\begin{scope}[xshift=21.7cm, yshift=-3.25cm]
\node at (0,-0.5){=};
\end{scope}
\begin{scope}[xshift=22.3cm, yshift=-1.45cm]
\comult(1.5,0)[1,1];  \raction(0,-1)[0,1]; \draw (2,-1) -- (2,-2.5); \draw (0,0) -- (0,-1); \racklarge(2,-2.5)[\scriptstyle \triangleright]; \draw (3,0) -- (3,-3.5); \raction(2,-3.5)[0,1];
\end{scope}
\begin{scope}[xshift=25.6cm, yshift=-3.25cm]
\node at (0,-0.5){.};
\end{scope}
\end{tikzpicture}
$$
Since $(\tau \ot X) \circ (X \ot \Delta)$ is invertible, this finishes the proof.
\end{proof}

\begin{pro}\label{compatibilidad de s y tau} Let $(X,r)$ and $\triangleright$ be as in Proposition~\ref{prop auxiliar}. Then
$$
\tilde{\tau}_2 \circ (s \ot X) = s \circ \tilde{\tau}_2.
$$
where $\tilde{\tau}_2:=(\tau \ot \tau) \circ (X \ot c \ot X) \circ (X^2 \ot \Delta)$ and $s$ is the derived map of $r$.
\end{pro}

\begin{proof} On one hand, by Proposition~\ref{tau es morfismo de comodulos} and the fact that $X$ is a cocommutative coalgebra, we have
\begin{equation}\label{ecu1}
\begin{tikzpicture}[scale=0.395, baseline=(current  bounding  box.center)]
\def\counit(#1,#2){\draw (#1,#2) -- (#1,#2-0.93) (#1,#2-1) circle[radius=2pt]}
\def\comult(#1,#2)[#3,#4]{\draw (#1,#2) -- (#1,#2-0.5*#4) arc (90:0:0.5*#3 and 0.5*#4) (#1,#2-0.5*#4) arc (90:180:0.5*#3 and 0.5*#4)}
\def\laction(#1,#2)[#3,#4]{\draw (#1,#2) .. controls (#1,#2-0.555*#4/2) and (#1+0.445*#4/2,#2-1*#4/2) .. (#1+1*#4/2,#2-1*#4/2) -- (#1+2*#4/2+#3*#4/2,#2-1*#4/2) (#1+2*#4/2+#3*#4/2,#2)--(#1+2*#4/2+#3*#4/2,#2-2*#4/2)}
\def\lactiontr(#1,#2)[#3,#4,#5]{\draw (#1,#2) .. controls (#1,#2-0.555*#4/2) and (#1+0.445*#4/2,#2-1*#4/2) .. (#1+1*#4/2,#2-1*#4/2) -- (#1+2*#4/2+#3*#4/2,#2-1*#4/2)  node [inner sep=0pt, minimum size=3pt,shape=isosceles triangle,fill, shape border rotate=#5] {} (#1+2*#4/2+#3*#4/2,#2) --(#1+2*#4/2+#3*#4/2,#2-2*#4/2)}
\def\doublemap(#1,#2)[#3]{\draw (#1+0.5,#2-0.5) node [name=doublemapnode,inner xsep=0pt, inner ysep=0pt, minimum height=11pt, minimum width=23pt,shape=rectangle,draw,rounded corners] {$#3$} (#1,#2) .. controls (#1,#2-0.075) .. (doublemapnode) (#1+1,#2) .. controls (#1+1,#2-0.075).. (doublemapnode) (doublemapnode) .. controls (#1,#2-0.925)..(#1,#2-1) (doublemapnode) .. controls (#1+1,#2-0.925).. (#1+1,#2-1)}
\def\solbraid(#1,#2)[#3]{\draw (#1,#2-0.5)  node[name=nodemap,inner sep=0pt,  minimum size=9pt, shape=circle,draw]{$#3$}
(#1-0.5,#2) .. controls (#1-0.5,#2-0.15) and (#1-0.4,#2-0.2) .. (#1-0.3,#2-0.3) (#1-0.3,#2-0.3) -- (nodemap)
(#1+0.5,#2) .. controls (#1+0.5,#2-0.15) and (#1+0.4,#2-0.2) .. (#1+0.3,#2-0.3) (#1+0.3,#2-0.3) -- (nodemap)
(#1+0.5,#2-1) .. controls (#1+0.5,#2-0.85) and (#1+0.4,#2-0.8) .. (#1+0.3,#2-0.7) (#1+0.3,#2-0.7) -- (nodemap)
(#1-0.5,#2-1) .. controls (#1-0.5,#2-0.85) and (#1-0.4,#2-0.8) .. (#1-0.3,#2-0.7) (#1-0.3,#2-0.7) -- (nodemap)
}
\def\ractiontr(#1,#2)[#3,#4,#5]{\draw (#1,#2) -- (#1,#2-2*#4/2)  (#1,#2-1*#4/2) node [inner sep=0pt, minimum size=3pt,shape=isosceles triangle,fill, shape border rotate=#5] {}  --(#1+1*#4/2+#3*#4/2,#2-1*#4/2) .. controls (#1+1.555*#4/2+#3*#4/2,#2-1*#4/2) and (#1+2*#4/2+#3*#4/2,#2-0.555*#4/2) .. (#1+2*#4/2+#3*#4/2,#2)  }
\def\flip(#1,#2)[#3]{\draw (
#1+1*#3,#2) .. controls (#1+1*#3,#2-0.05*#3) and (#1+0.96*#3,#2-0.15*#3).. (#1+0.9*#3,#2-0.2*#3)
(#1+0.1*#3,#2-0.8*#3)--(#1+0.9*#3,#2-0.2*#3)
(#1,#2-1*#3) .. controls (#1,#2-0.95*#3) and (#1+0.04*#3,#2-0.85*#3).. (#1+0.1*#3,#2-0.8*#3)
(#1,#2) .. controls (#1,#2-0.05*#3) and (#1+0.04*#3,#2-0.15*#3).. (#1+0.1*#3,#2-0.2*#3)
(#1+0.1*#3,#2-0.2*#3) -- (#1+0.9*#3,#2-0.8*#3)
(#1+1*#3,#2-1*#3) .. controls (#1+1*#3,#2-0.95*#3) and (#1+0.96*#3,#2-0.85*#3).. (#1+0.9*#3,#2-0.8*#3)
}
\def\raction(#1,#2)[#3,#4]{\draw (#1,#2) -- (#1,#2-2*#4/2)  (#1,#2-1*#4/2)--(#1+1*#4/2+#3*#4/2,#2-1*#4/2) .. controls (#1+1.555*#4/2+#3*#4/2,#2-1*#4/2) and (#1+2*#4/2+#3*#4/2,#2-0.555*#4/2) .. (#1+2*#4/2+#3*#4/2,#2)}
\def\rack(#1,#2)[#3]{\draw (#1,#2-0.5)  node[name=nodemap,inner sep=0pt,  minimum size=7.5pt, shape=circle,draw]{$#3$} (#1-1,#2) .. controls (#1-1,#2-0.5) and (#1-0.5,#2-0.5) .. (nodemap) (#1,#2)-- (nodemap)  (nodemap)-- (#1,#2-1)}
\def\racklarge(#1,#2)[#3]{\draw (#1,#2-0.5)  node[name=nodemap,inner sep=0pt,  minimum size=7.5pt, shape=circle,draw]{$#3$} (#1-2,#2+0.5) .. controls (#1-2,#2-0.5) and (#1-0.5,#2-0.5) .. (nodemap) (#1,#2)-- (nodemap)  (nodemap)-- (#1,#2-1)}
\def\rackextralarge(#1,#2)[#3]{\draw (#1,#2-0.5)  node[name=nodemap,inner sep=0pt,  minimum size=7.5pt, shape=circle,draw]{$#3$} (#1-3,#2+1) .. controls (#1-3,#2-0.5) and (#1-0.5,#2-0.5) .. (nodemap) (#1,#2)-- (nodemap)  (nodemap)-- (#1,#2-1)}
\begin{scope}[xshift=0cm, yshift=-2.25cm]
\comult(0.5,0)[1,1]; \draw (2,0) -- (2,-1); \draw (0,-1) -- (0,-2); \flip(1,-1)[1]; \rack(1,-2)[\scriptstyle \triangleright]; \flip(2,-2)[1]; \comult(3.5,-1)[1,1]; \draw (4,-2) -- (4,-3); \raction(3,-3)[0,1]; \raction(1,-3)[0,1]; \draw (3.5,0) -- (3.5,-1);
\end{scope}
\begin{scope}[xshift=4.55cm, yshift=-3.8cm]
\node at (0,-0.5){=};
\end{scope}
\begin{scope}[xshift=5.1cm, yshift=-0.25cm]
\comult(0.5,0)[1,1]; \draw (0,-1) -- (0,-5.5); \raction(0,-5.5)[0,1]; \draw (1,-2) -- (1,-4.5); \flip(1,-4.5)[1]; \flip(1,-1)[1]; \flip(2.5,-2.5)[1]; \comult(4,-1.5)[1,1]; \draw (4,0) -- (4,-1.5);  \draw (4.5,-2.5) -- (4.5,-3.5); \draw (2,-2) .. controls (2,-2.25) and (2.5, -2.25).. (2.5,-2.5); \raction(3.5,-3.5)[0,1]; \comult(2.5,-3.5)[1,1]; \draw (3,-4.5) -- (3,-5.5); \raction(2,-5.5)[0,1]; \racklarge(2,-7)[\scriptstyle \triangleright]; \draw (2,-6.5) -- (2,-7); \draw (2,0) -- (2,-1); \draw (3.5,-4.5) -- (3.5,-8);
\end{scope}
\begin{scope}[xshift=10.1cm, yshift=-3.8cm]
\node at (0,-0.5){=};
\end{scope}
\begin{scope}[xshift=10.7cm, yshift=-0.8cm]
\comult(0.5,0)[1,1]; \draw (0,-1) -- (0,-4.5); \raction(0,-4.5)[0,1]; \draw (1,-2) -- (1,-3.5); \flip(1,-3.5)[1]; \flip(1,-1)[1]; \comult(3.5,-1.5)[1,1]; \comult(4.25,0)[1.5,1.5]; \draw (5,-1.5) -- (5,-2.5); \flip(4,-2.5)[1]; \draw (2,-2) -- (2,-2.5); \flip(2,-2.5)[1]; \flip(3,-3.5)[1]; \draw (5,-3.5) -- (5,-4.5); \raction(4,-4.5)[0,1]; \raction(2,-4.5)[0,1]; \racklarge(2,-6)[\scriptstyle \triangleright];  \draw (2,-5.5) -- (2,-6);  \draw (4,-5) -- (4,-7); \draw (2,0) -- (2,-1);
\end{scope}
\begin{scope}[xshift=16.25cm, yshift=-3.8cm]
\node at (0,-0.5){=};
\end{scope}
\begin{scope}[xshift=16.8cm, yshift=0cm]
\flip(0,0)[1]; \draw (1,-1) .. controls (1,-1.25) and (1.5,-1.25) .. (1.5,-1.5); \comult(1.5,-1.5)[1,1]; \flip(2,-2.5)[1]; \comult(3.5,-1.5)[1,1];  \comult(4,-0.5)[1,1]; \draw (4,0) -- (4,-0.5); \draw (1,-2.5) -- (1,-3.5); \draw (4,-2.5) -- (4,-3.5); \raction(1,-3.5)[0,1]; \raction(3,-3.5)[0,1]; \flip(0,-4.5)[1]; \draw (0,-1) -- (0,-4.5); \draw (4.5,-1.5) -- (4.5,-4); \draw (4.5,-4) .. controls (4.5,-4.25) and (4,-4.25) .. (4,-4.5); \flip(3,-4.5)[1]; \raction(1,-5.5)[0,2]; \draw (0,-5.5) -- (0,-7.5); \rack(1,-7.5)[\scriptstyle \triangleright]; \draw (4,-5.5) -- (4,-8.5);
\end{scope}
\begin{scope}[xshift=21.6cm, yshift=-3.8cm]
\node at (0,-0.5){.};
\end{scope}
\end{tikzpicture}
\end{equation}
On the other hand,
\begin{equation}\label{ecu2}
\begin{tikzpicture}[scale=0.395, baseline=(current  bounding  box.center)]
\def\counit(#1,#2){\draw (#1,#2) -- (#1,#2-0.93) (#1,#2-1) circle[radius=2pt]}
\def\comult(#1,#2)[#3,#4]{\draw (#1,#2) -- (#1,#2-0.5*#4) arc (90:0:0.5*#3 and 0.5*#4) (#1,#2-0.5*#4) arc (90:180:0.5*#3 and 0.5*#4)}
\def\laction(#1,#2)[#3,#4]{\draw (#1,#2) .. controls (#1,#2-0.555*#4/2) and (#1+0.445*#4/2,#2-1*#4/2) .. (#1+1*#4/2,#2-1*#4/2) -- (#1+2*#4/2+#3*#4/2,#2-1*#4/2) (#1+2*#4/2+#3*#4/2,#2)--(#1+2*#4/2+#3*#4/2,#2-2*#4/2)}
\def\lactiontr(#1,#2)[#3,#4,#5]{\draw (#1,#2) .. controls (#1,#2-0.555*#4/2) and (#1+0.445*#4/2,#2-1*#4/2) .. (#1+1*#4/2,#2-1*#4/2) -- (#1+2*#4/2+#3*#4/2,#2-1*#4/2)  node [inner sep=0pt, minimum size=3pt,shape=isosceles triangle,fill, shape border rotate=#5] {} (#1+2*#4/2+#3*#4/2,#2) --(#1+2*#4/2+#3*#4/2,#2-2*#4/2)}
\def\doublemap(#1,#2)[#3]{\draw (#1+0.5,#2-0.5) node [name=doublemapnode,inner xsep=0pt, inner ysep=0pt, minimum height=11pt, minimum width=23pt,shape=rectangle,draw,rounded corners] {$#3$} (#1,#2) .. controls (#1,#2-0.075) .. (doublemapnode) (#1+1,#2) .. controls (#1+1,#2-0.075).. (doublemapnode) (doublemapnode) .. controls (#1,#2-0.925)..(#1,#2-1) (doublemapnode) .. controls (#1+1,#2-0.925).. (#1+1,#2-1)}
\def\solbraid(#1,#2)[#3]{\draw (#1,#2-0.5)  node[name=nodemap,inner sep=0pt,  minimum size=9pt, shape=circle,draw]{$#3$}
(#1-0.5,#2) .. controls (#1-0.5,#2-0.15) and (#1-0.4,#2-0.2) .. (#1-0.3,#2-0.3) (#1-0.3,#2-0.3) -- (nodemap)
(#1+0.5,#2) .. controls (#1+0.5,#2-0.15) and (#1+0.4,#2-0.2) .. (#1+0.3,#2-0.3) (#1+0.3,#2-0.3) -- (nodemap)
(#1+0.5,#2-1) .. controls (#1+0.5,#2-0.85) and (#1+0.4,#2-0.8) .. (#1+0.3,#2-0.7) (#1+0.3,#2-0.7) -- (nodemap)
(#1-0.5,#2-1) .. controls (#1-0.5,#2-0.85) and (#1-0.4,#2-0.8) .. (#1-0.3,#2-0.7) (#1-0.3,#2-0.7) -- (nodemap)
}
\def\ractiontr(#1,#2)[#3,#4,#5]{\draw (#1,#2) -- (#1,#2-2*#4/2)  (#1,#2-1*#4/2) node [inner sep=0pt, minimum size=3pt,shape=isosceles triangle,fill, shape border rotate=#5] {}  --(#1+1*#4/2+#3*#4/2,#2-1*#4/2) .. controls (#1+1.555*#4/2+#3*#4/2,#2-1*#4/2) and (#1+2*#4/2+#3*#4/2,#2-0.555*#4/2) .. (#1+2*#4/2+#3*#4/2,#2)  }
\def\flip(#1,#2)[#3]{\draw (
#1+1*#3,#2) .. controls (#1+1*#3,#2-0.05*#3) and (#1+0.96*#3,#2-0.15*#3).. (#1+0.9*#3,#2-0.2*#3)
(#1+0.1*#3,#2-0.8*#3)--(#1+0.9*#3,#2-0.2*#3)
(#1,#2-1*#3) .. controls (#1,#2-0.95*#3) and (#1+0.04*#3,#2-0.85*#3).. (#1+0.1*#3,#2-0.8*#3)
(#1,#2) .. controls (#1,#2-0.05*#3) and (#1+0.04*#3,#2-0.15*#3).. (#1+0.1*#3,#2-0.2*#3)
(#1+0.1*#3,#2-0.2*#3) -- (#1+0.9*#3,#2-0.8*#3)
(#1+1*#3,#2-1*#3) .. controls (#1+1*#3,#2-0.95*#3) and (#1+0.96*#3,#2-0.85*#3).. (#1+0.9*#3,#2-0.8*#3)
}
\def\raction(#1,#2)[#3,#4]{\draw (#1,#2) -- (#1,#2-2*#4/2)  (#1,#2-1*#4/2)--(#1+1*#4/2+#3*#4/2,#2-1*#4/2) .. controls (#1+1.555*#4/2+#3*#4/2,#2-1*#4/2) and (#1+2*#4/2+#3*#4/2,#2-0.555*#4/2) .. (#1+2*#4/2+#3*#4/2,#2)}
\def\rack(#1,#2)[#3]{\draw (#1,#2-0.5)  node[name=nodemap,inner sep=0pt,  minimum size=7.5pt, shape=circle,draw]{$#3$} (#1-1,#2) .. controls (#1-1,#2-0.5) and (#1-0.5,#2-0.5) .. (nodemap) (#1,#2)-- (nodemap)  (nodemap)-- (#1,#2-1)}
\def\racklarge(#1,#2)[#3]{\draw (#1,#2-0.5)  node[name=nodemap,inner sep=0pt,  minimum size=7.5pt, shape=circle,draw]{$#3$} (#1-2,#2+0.5) .. controls (#1-2,#2-0.5) and (#1-0.5,#2-0.5) .. (nodemap) (#1,#2)-- (nodemap)  (nodemap)-- (#1,#2-1)}
\def\rackextralarge(#1,#2)[#3]{\draw (#1,#2-0.5)  node[name=nodemap,inner sep=0pt,  minimum size=7.5pt, shape=circle,draw]{$#3$} (#1-3,#2+1) .. controls (#1-3,#2-0.5) and (#1-0.5,#2-0.5) .. (nodemap) (#1,#2)-- (nodemap)  (nodemap)-- (#1,#2-1)}
\begin{scope}[xshift=0cm, yshift=-0.8cm]
\flip(0,0)[1];  \draw (1,-1) .. controls (1,-1.25) and (1.5,-1.25) .. (1.5,-1.5); \raction(1.5,-1.5)[0,1]; \draw (0,-1) -- (0,-3.5); \comult(1.5,-2.5)[1,1]; \comult(3,-0.5)[1,1]; \draw (3,0) -- (3,-0.5); \draw (3.5,-1.5) .. controls (3.5,-2) and (3,-3) .. (3,-3.5); \flip(0,-3.5)[1]; \flip(2,-3.5)[1]; \raction(1,-4.5)[0,1];  \draw (3,-4.5) -- (3,-6.5); \draw (0,-4.5) -- (0,-5.5); \rack(1,-5.5)[\scriptstyle \triangleright];
\end{scope}
\begin{scope}[xshift=4.05cm, yshift=-3.55cm]
\node at (0,-0.5){=};
\end{scope}
\begin{scope}[xshift=5.1cm, yshift=0cm]
\flip(0,0)[1]; \raction(1,-1)[0,1]; \draw (0,-1) -- (0,-2);  \comult(2.5,0)[1,1]; \flip(0,-2)[1]; \draw (1,-3) .. controls (1,-3.25) and (1.5,-3.25) .. (1.5,-3.5);  \comult(0,-3)[1,1]; \draw (1.5,-3.5) -- (1.5,-4); \flip(0.5,-4)[1]; \flip(1.5,-5)[1]; \draw (3,-1) .. controls (3,-1.5) and (2.5,-4.5) .. (2.5,-5); \draw (-0.5,-4) -- (-0.5,-7);  \draw (0.5,-5) -- (0.5,-6); \raction(0.5,-6)[0,1]; \rack(0.5,-7)[\scriptstyle \triangleright]; \draw (2.5,-6) -- (2.5,-8);
\end{scope}
\begin{scope}[xshift=8.75cm, yshift=-3.55cm]
\node at (0,-0.5){=};
\end{scope}
\begin{scope}[xshift=10.35cm, yshift=-0.5cm]
\flip(0,0)[1]; \raction(1,-1)[0,1]; \draw (0,-1) -- (0,-2);  \comult(2.5,0)[1,1]; \flip(0,-2)[1]; \raction(1,-3)[0,2]; \draw (0,-3) .. controls (0,-3.5) and (-0.5,-3.5) .. (-0.5,-4); \comult(-0.5,-4)[1,1]; \flip(0,-5)[1]; \draw (-1,-5) -- (-1,-6); \rack(0,-6)[\scriptstyle \triangleright]; \draw (1,-6) -- (1,-7); \draw (3,-1) -- (3,-3);
\end{scope}
\begin{scope}[xshift=13.9cm, yshift=-3.55cm]
\node at (0,-0.5){=};
\end{scope}
\begin{scope}[xshift=15cm, yshift=-1cm]
\draw (0,0) -- (0,-2); \draw (1,0) -- (1,-1); \comult(2.5,0)[1,1]; \flip(1,-1)[1]; \draw (3,-1) -- (3,-2); \raction(0,-2)[0,1]; \raction(2,-2)[0,1]; \comult(0,-3)[1,1]; \draw (2,-3) .. controls (2,-3.5) and (1.5,-3.5) .. (1.5,-4); \draw (-0.5,-4) -- (-0.5,-5); \flip(0.5,-4)[1]; \rack(0.5,-5)[\scriptstyle \triangleright]; \draw (1.5,-5) -- (1.5,-6);
\end{scope}
\begin{scope}[xshift=18.3cm, yshift=-3.55cm]
\node at (0,-0.5){.};
\end{scope}
\end{tikzpicture}
\end{equation}
This finishes the proof, because the expression at the right hand of the last equality in~\eqref{ecu1} equals to the expression at the left hand of the first equality in~\eqref{ecu2}, since $\tau$ is a coalgebra homomorphism.
\end{proof}

\begin{lem}\label{alpha n con ri i>1} If $(X,r)$ is a braided set, then
$$
\alpha_n \circ r_{i,i+1} = r_{i,i+1}\circ \alpha_n \quad\text{for all $n>2$ and all $1<i<n$.}
$$
\end{lem}

\begin{proof} We first consider the case $i=n-1$. Let $\tau_2\colon X^3\to X$ be the map defined by
$$
\tau_2:= \tau\circ (\tau\ot X).
$$
It is easy to see that
$$
\alpha_n = (\tau_2\ot X^{n-1})\circ (X\ot c_{n-3,2}\ot X^2)\circ (\alpha_{n-2}\ot \Delta_{X^2}).
$$
By Remark~\ref{prop de involutiva y de braided} we know that $\tau_2\circ r = \tau_2$. Since, moreover $r$ is a coalgebra homo\-morphism,
\begin{align*}
\alpha_n \circ r_{n-1,n} & = (\tau_2\ot X^{n-1})\circ (X\ot c_{n-3,2}\ot X^2)\circ (\alpha_{n-2}\ot \Delta_{X^2})\circ (X^{n-2}\ot r)\\[0.7pt]
& = (\tau_2\ot X^{n-1})\circ (X\ot c_{n-3,2}\ot X^2)\circ (\alpha_{n-2}\ot r\ot r)\circ (X^{n-2}\ot \Delta_{X^2})\\[0.7pt]
& = (\tau_2\ot X^{n-1})\circ (X\ot c_{n-3,2}\ot X^2)\circ (X^{n-2}\ot r\ot r)\circ (\alpha_{n-2}\ot \Delta_{X^2})\\[0.7pt]
& = ((\tau_2\circ r)\ot X^{n-1})\circ (X\ot c_{n-3,2}\ot X^2)\circ (X^n\ot r)\circ (\alpha_{n-2}\ot \Delta_{X^2})\\[0.7pt]
& = (X^{n-2}\ot r)\circ (\tau_2\ot X^{n-1})\circ (X\ot c_{n-3,2}\ot X^2)\circ (\alpha_{n-2}\ot \Delta_{X^2})\\[0.7pt]
&=r_{n-1,n}\circ \alpha_n.
\end{align*}
We now proceed by induction on $n$. If $n=3$, then necessarily $i = 2 = n-1$, and so in this case the lemma is true as we already have proven. Suppose now that $n>3$ and the lemma is true for $n-1$. For every $m>i$ we will write $r^m_{i,i+1}$ in order to indicate that the domain of $r_{i,i+1}$ is $X^m$. We can assume that $1<i<n-1$. By inductive hypothesis, we have
\begin{align*}
\alpha_n \circ r_{i,i+1}^n & = (\tau \ot X^{n-1}) \circ (X \ot c_{n-2,1} \ot X) \circ (\alpha_{n-1}\ot \Delta)\circ r_{i,i+1}^n\\[0.7pt]
& = (\tau \ot X^{n-1}) \circ (X \ot c_{n-2,1} \ot X) \circ (\alpha_{n-1}\ot X^2)\circ (r_{i,i+1}^{n-1}\ot\Delta)\\[0.7pt]
& = (\tau \ot X^{n-1}) \circ (X \ot c_{n-2,1} \ot X) \circ (r_{i,i+1}^{n-1}\ot X^2)\circ (\alpha_{n-1}\ot\Delta)\\[0.7pt]
& = (\tau \ot X^{n-1}) \circ (r_{i+1,i+2}^n\ot X)\circ (X \ot c_{n-2,1}\ot X)\circ (\alpha_{n-1}\ot\Delta)\\[0.7pt]
& = r_{i,i+1}^n \circ (\tau \ot X^{n-1}) \circ (X \ot c_{n-2,1} \ot X) \circ (\alpha_{n-1} \ot \Delta) \\[0.7pt]
& = r_{i,i+1}^n \circ \alpha_n,
\end{align*}
as desired.
\end{proof}

\begin{thm}\label{relacion entre Jn r y s} If $(X,r)$ is a non-degenerate braided set, then
$$
J_n \circ r_{i,i+1} = s_{i,i+1} \circ J_n\quad\text{for all $n \ge 2$ and $i<n$,}
$$
where $s$ is the derived map of $r$.
\end{thm}

\begin{proof} We proceed by induction on $n$. Assume first that $n=2$ and $i=1$. By Remark~\ref{algunas formulas}(1), the fact that $\Delta$ is cocommutative and $\tau^{-1}$ is a coalgebra homomorphism, and Propositions~\ref{no degenerado} and~\ref{prop auxiliar}, we have:
$$
\begin{tikzpicture}[scale=0.395]
\def\counit(#1,#2){\draw (#1,#2) -- (#1,#2-0.93) (#1,#2-1) circle[radius=2pt]}
\def\comult(#1,#2)[#3,#4]{\draw (#1,#2) -- (#1,#2-0.5*#4) arc (90:0:0.5*#3 and 0.5*#4) (#1,#2-0.5*#4) arc (90:180:0.5*#3 and 0.5*#4)}
\def\laction(#1,#2)[#3,#4]{\draw (#1,#2) .. controls (#1,#2-0.555*#4/2) and (#1+0.445*#4/2,#2-1*#4/2) .. (#1+1*#4/2,#2-1*#4/2) -- (#1+2*#4/2+#3*#4/2,#2-1*#4/2) (#1+2*#4/2+#3*#4/2,#2)--(#1+2*#4/2+#3*#4/2,#2-2*#4/2)}
\def\lactiontr(#1,#2)[#3,#4,#5]{\draw (#1,#2) .. controls (#1,#2-0.555*#4/2) and (#1+0.445*#4/2,#2-1*#4/2) .. (#1+1*#4/2,#2-1*#4/2) -- (#1+2*#4/2+#3*#4/2,#2-1*#4/2)  node [inner sep=0pt, minimum size=3pt,shape=isosceles triangle,fill, shape border rotate=#5] {} (#1+2*#4/2+#3*#4/2,#2) --(#1+2*#4/2+#3*#4/2,#2-2*#4/2)}
\def\doublemap(#1,#2)[#3]{\draw (#1+0.5,#2-0.5) node [name=doublemapnode,inner xsep=0pt, inner ysep=0pt, minimum height=11pt, minimum width=23pt,shape=rectangle,draw,rounded corners] {$#3$} (#1,#2) .. controls (#1,#2-0.075) .. (doublemapnode) (#1+1,#2) .. controls (#1+1,#2-0.075).. (doublemapnode) (doublemapnode) .. controls (#1,#2-0.925)..(#1,#2-1) (doublemapnode) .. controls (#1+1,#2-0.925).. (#1+1,#2-1)}
\def\solbraid(#1,#2)[#3]{\draw (#1,#2-0.5)  node[name=nodemap,inner sep=0pt,  minimum size=9pt, shape=circle,draw]{$#3$}
(#1-0.5,#2) .. controls (#1-0.5,#2-0.15) and (#1-0.4,#2-0.2) .. (#1-0.3,#2-0.3) (#1-0.3,#2-0.3) -- (nodemap)
(#1+0.5,#2) .. controls (#1+0.5,#2-0.15) and (#1+0.4,#2-0.2) .. (#1+0.3,#2-0.3) (#1+0.3,#2-0.3) -- (nodemap)
(#1+0.5,#2-1) .. controls (#1+0.5,#2-0.85) and (#1+0.4,#2-0.8) .. (#1+0.3,#2-0.7) (#1+0.3,#2-0.7) -- (nodemap)
(#1-0.5,#2-1) .. controls (#1-0.5,#2-0.85) and (#1-0.4,#2-0.8) .. (#1-0.3,#2-0.7) (#1-0.3,#2-0.7) -- (nodemap)
}
\def\ractiontr(#1,#2)[#3,#4,#5]{\draw (#1,#2) -- (#1,#2-2*#4/2)  (#1,#2-1*#4/2) node [inner sep=0pt, minimum size=3pt,shape=isosceles triangle,fill, shape border rotate=#5] {}  --(#1+1*#4/2+#3*#4/2,#2-1*#4/2) .. controls (#1+1.555*#4/2+#3*#4/2,#2-1*#4/2) and (#1+2*#4/2+#3*#4/2,#2-0.555*#4/2) .. (#1+2*#4/2+#3*#4/2,#2)  }
\def\flip(#1,#2)[#3]{\draw (
#1+1*#3,#2) .. controls (#1+1*#3,#2-0.05*#3) and (#1+0.96*#3,#2-0.15*#3).. (#1+0.9*#3,#2-0.2*#3)
(#1+0.1*#3,#2-0.8*#3)--(#1+0.9*#3,#2-0.2*#3)
(#1,#2-1*#3) .. controls (#1,#2-0.95*#3) and (#1+0.04*#3,#2-0.85*#3).. (#1+0.1*#3,#2-0.8*#3)
(#1,#2) .. controls (#1,#2-0.05*#3) and (#1+0.04*#3,#2-0.15*#3).. (#1+0.1*#3,#2-0.2*#3)
(#1+0.1*#3,#2-0.2*#3) -- (#1+0.9*#3,#2-0.8*#3)
(#1+1*#3,#2-1*#3) .. controls (#1+1*#3,#2-0.95*#3) and (#1+0.96*#3,#2-0.85*#3).. (#1+0.9*#3,#2-0.8*#3)
}
\def\raction(#1,#2)[#3,#4]{\draw (#1,#2) -- (#1,#2-2*#4/2)  (#1,#2-1*#4/2)--(#1+1*#4/2+#3*#4/2,#2-1*#4/2) .. controls (#1+1.555*#4/2+#3*#4/2,#2-1*#4/2) and (#1+2*#4/2+#3*#4/2,#2-0.555*#4/2) .. (#1+2*#4/2+#3*#4/2,#2)}
\def\rack(#1,#2)[#3]{\draw (#1,#2-0.5)  node[name=nodemap,inner sep=0pt,  minimum size=7.5pt, shape=circle,draw]{$#3$} (#1-1,#2) .. controls (#1-1,#2-0.5) and (#1-0.5,#2-0.5) .. (nodemap) (#1,#2)-- (nodemap)  (nodemap)-- (#1,#2-1)}
\def\racklarge(#1,#2)[#3]{\draw (#1,#2-0.5)  node[name=nodemap,inner sep=0pt,  minimum size=7.5pt, shape=circle,draw]{$#3$} (#1-2,#2+0.5) .. controls (#1-2,#2-0.5) and (#1-0.5,#2-0.5) .. (nodemap) (#1,#2)-- (nodemap)  (nodemap)-- (#1,#2-1)}
\begin{scope}[xshift=0cm, yshift=-3.03cm]
\solbraid(0.5,0)[\scriptstyle r]; \draw (1,-1) .. controls (1,-1.5) and (1.5,-1.5) .. (1.5,-2); \comult(1.5,-2)[1,1]; \raction(0,-3)[0,1]; \draw (0,-1) -- (0,-3); \draw (2,-3) -- (2,-4);
\end{scope}
\begin{scope}[xshift=2.5cm, yshift=-4.6cm]
\node at (0,-0.5){=};
\end{scope}
\begin{scope}[xshift=3.1cm, yshift=-2.5cm]
\draw (0.5,0) -- (0.5,-1); \comult(0.5,-1)[1,1]; \comult(2.5,0)[1,1]; \draw (0,-2) -- (0,-3); \draw (3,-2) -- (3,-3); \flip(1,-2)[1]; \raction(0,-4)[0,1];  \solbraid(0.5,-3)[\scriptstyle r];  \draw (2,-4) -- (2,-5); \raction(2,-3)[0,1]; \flip(2,-1)[1];
\end{scope}
\begin{scope}[xshift=6.6cm, yshift=-4.6cm]
\node at (0,-0.5){=};
\end{scope}
\begin{scope}[xshift=7.7cm, yshift=-0.8cm]
\draw (0,0) -- (0,-1); \comult(1.5,0)[1,1]; \raction(0,-1)[0,1]; \comult(2,-1)[1,1]; \ractiontr(0,-2)[1,1,90]; \comult(0,-3)[1,1]; \comult(2.5,-2)[1,1]; \draw (2,-3) -- (2,-4); \draw (3,-3) -- (3,-3.5);  \raction(0.5,-4)[1,1]; \flip(0.5,-5.5)[1]; \draw (3,-3.5) .. controls (3,-4.5) and (1.5,-5) .. (1.5,-5.5); \draw (0.5,-5) -- (0.5,-5.5); \draw (-0.5,-4) -- (-0.5,-6.5); \solbraid(0,-6.5)[\scriptstyle r]; \raction(-0.5,-7.5)[0,1];  \draw (1.5,-6.5) -- (1.5,-8.5);
\end{scope}
\begin{scope}[xshift=11.2cm, yshift=-4.6cm]
\node at (0,-0.5){=};
\end{scope}
\begin{scope}[xshift=11.7cm, yshift=0cm]
\comult(0.5,-4)[1,1]; \comult(3.25,-2.5)[1.5,1.5];\comult(2.5,-4)[1,1];\draw (3,-5) -- (3,-6); \comult(4.25,-0.5)[2,2]; \draw (5.25,-2.5) -- (5.25,-6.5); \comult(3.25,1.5)[2,2];\draw (0.5,1.5) -- (0.5,-4); \raction(0.5,-0.5)[0.325,1.5]; \flip(1,-5)[1]; \draw (0,-5) -- (0,-6); \ractiontr(0,-6)[0,1,90]; \ractiontr(2,-6)[0,1,90]; \raction(2,-6.5)[0,2]; \draw (4,-4) -- (4,-6.5); \draw (5.25,-6.5) .. controls (5.25,-7) and (3,-8) .. (3,-8.5); \flip(2,-8.5)[1]; \draw (0,-7) .. controls (0,-8) and (1,-8.5) .. (1,-9.5); \solbraid(1.5,-9.5)[\scriptstyle r];\raction(1,-10.5)[0,1]; \draw (3,-9.5) -- (3,-11.5);
\end{scope}
\begin{scope}[xshift=17.4cm, yshift=-4.6cm]
\node at (0,-0.5){=};
\end{scope}
\begin{scope}[xshift=17.9cm, yshift=0.5cm]
\comult(0.5,-2)[1,1]; \comult(2,0)[1,1]; \raction(0.5,-1)[0,1]; \draw (2.5,-1) -- (2.5,-4); \draw (0.5,0) -- (0.5,-2); \flip(1.5,-4)[1]; \draw (1,-3) .. controls (1,-3.5) and (1.5,-3.5) .. (1.5,-4); \draw (2.5,-5) .. controls (2.5,-5.5) and (3,-6) .. (3,-6.5); \comult(1.5,-5)[1,1]; \draw (0,-3) -- (0,-8.5); \ractiontr(0,-6)[0,1,90]; \draw (2,-6) .. controls (2,-7) and (1.5,-7) .. (1.5,-7.5); \draw (0,-7) -- (0,-8);  \draw (3,-6.5) -- (3,-11); \comult(1.5,-7.5)[1,1];  \raction(0,-8.5)[0,1]; \draw (2,-8.5) -- (2,-10); \draw (0,-8.5) -- (0,-9.5); \racklarge(2,-10)[\scriptstyle \triangleright];
\end{scope}
\begin{scope}[xshift=21.4cm, yshift=-4.6cm]
\node at (0,-0.5){=};
\end{scope}
\begin{scope}[xshift=22.65cm, yshift=-2.05cm]
\comult(0.5,-2)[1,1]; \comult(2,0)[1,1]; \raction(0.5,-1)[0,1]; \draw (2.5,-1) -- (2.5,-4); \draw (0.5,0) -- (0.5,-2); \flip(1.5,-4)[1]; \draw (2.5,-5) -- (2.5,-6); \draw (1,-3) .. controls (1,-3.5) and (1.5,-3.5) .. (1.5,-4); \draw (0,-3) .. controls (0,-3.5) and (-0.5,-4) .. (-0.5,-4.5);
\racklarge(1.5,-5)[\scriptstyle \triangleright];
\end{scope}
\begin{scope}[xshift=25.35cm, yshift=-4.6cm]
\node at (0,-0.5){.};
\end{scope}
\end{tikzpicture}
$$
By Proposition~\ref{formula para s}, this proves that $J_2\circ r = s\circ J_2$. Assume now that, by inductive hypothesis, $J_n \circ r_{12}= s_{12}\circ J_n$. Let $\tilde{\tau}_2$ be as in Proposition~\ref{compatibilidad de s y tau}. It is easy to see that
$$
Q_{n+1} = (\tilde{\tau}_2\ot Q_{n-1})\circ (X^2\ot c_{n-2,1}\ot X)\circ (X^n\ot \Delta).
$$
Combining this with Propositions~\ref{Jn con Qn} and~\ref{compatibilidad de s y tau}, we obtain
\begin{align*}
J_{n+1}\circ r_{12} & = Q_{n+1}\circ (J_n \ot X)\circ (r\ot X^{n-1})\\[0.7pt]
& = (\tilde{\tau}_2\ot Q_{n-1})\circ (X^2\ot c_{n-2,1}\ot X)\circ (X^n\ot \Delta)\circ (J_n \ot X)\circ (r\ot X^{n-1}) \\[0.7pt]
& = (\tilde{\tau}_2\ot Q_{n-1})\circ (X^2\ot c_{n-2,1}\ot X)\circ (X^n\ot \Delta)\circ (s\ot X^{n-1})\circ (J_n \ot X)\\[0.7pt]
& = (\tilde{\tau}_2\ot Q_{n-1})\circ (s\ot c_{n-2,1}\ot X)\circ (X^n\ot \Delta)\circ (J_n \ot X)\\[0.7pt]
& = (s\ot X^{n-1})\circ (\tilde{\tau}_2\ot Q_{n-1})\circ (X^2\ot c_{n-2,1}\ot X)\circ (X^n\ot \Delta)\circ (J_n \ot X)\\[0.7pt]
& = (s\ot X^{n-1})\circ Q_{n+1}\circ (J_n \ot X)\\
& = s_{12}\circ J_{n+1},
\end{align*}
Finally, assume that the result is true for $n$ and that $1<i\le n$. As in the proof of Lemma~\ref{alpha n con ri i>1}, for every $m>j$ we will write $r^m_{j,j+1}$ in order to indicate that the domain of $r_{j,j+1}$ is $X^m$. By inductive hypothesis and Lemma~\ref{alpha n con ri i>1},
\begin{align*}
J_{n+1}\circ r_{i,i+1}^{n+1} & = (X\ot J_n) \circ \alpha_{n+1} \circ r_{i,i+1}^{n+1}\\
& = (X\ot J_n) \circ r_{i,i+1}^{n+1} \circ \alpha_{n+1}\\
& = (X\ot J_n) \circ (X\ot r_{i-1,i}^n) \circ \alpha_{n+1}\\
& = (X\ot r_{i-1,i}^n) \circ (X\ot J_n) \circ \alpha_{n+1}\\
& = r_{i,i+1}^{n+1} \circ J_{n+1},
\end{align*}
as we want.
\end{proof}

\begin{thm}\label{s es sol de yb} If $(X,r)$ is a non-degenerate braided set, then so is $(X,s)$.
\end{thm}

\begin{proof} By Theorem~\ref{relacion entre Jn r y s}, we have
$$
s_{12} \circ s_{23} \circ s_{12}= J_3\circ r_{12} \circ r_{23} \circ r_{12} \circ J_3^{-1} = J_3\circ r_{23} \circ r_{12} \circ r_{23} \circ J_3^{-1} = s_{23} \circ s_{12} \circ s_{23}.
$$
Thus $s$ satisfies the braid equation. Moreover, by Proposition~\ref{derivado es iso de coalgebras} and Lemma~\ref{cs es de comodulos}, we know that $s$ is invertible and that
$$
(X\ot s^{-1}) \circ (X\ot c) \circ (\Delta\ot X) = (\Delta \ot X) \circ s^{-1} \circ c.
$$
Using again Lemma~\ref{cs es de comodulos} and this equality we obtain that
$$
\begin{tikzpicture}[scale=0.395]
\def\counit(#1,#2){\draw (#1,#2) -- (#1,#2-0.93) (#1,#2-1) circle[radius=2pt]}
\def\comult(#1,#2)[#3,#4]{\draw (#1,#2) -- (#1,#2-0.5*#4) arc (90:0:0.5*#3 and 0.5*#4) (#1,#2-0.5*#4) arc (90:180:0.5*#3 and 0.5*#4)}
\def\laction(#1,#2)[#3,#4]{\draw (#1,#2) .. controls (#1,#2-0.555*#4/2) and (#1+0.445*#4/2,#2-1*#4/2) .. (#1+1*#4/2,#2-1*#4/2) -- (#1+2*#4/2+#3*#4/2,#2-1*#4/2) (#1+2*#4/2+#3*#4/2,#2)--(#1+2*#4/2+#3*#4/2,#2-2*#4/2)}
\def\lactiontr(#1,#2)[#3,#4,#5]{\draw (#1,#2) .. controls (#1,#2-0.555*#4/2) and (#1+0.445*#4/2,#2-1*#4/2) .. (#1+1*#4/2,#2-1*#4/2) -- (#1+2*#4/2+#3*#4/2,#2-1*#4/2)  node [inner sep=0pt, minimum size=3pt,shape=isosceles triangle,fill, shape border rotate=#5] {} (#1+2*#4/2+#3*#4/2,#2) --(#1+2*#4/2+#3*#4/2,#2-2*#4/2)}
\def\doublemap(#1,#2)[#3]{\draw (#1+0.5,#2-0.5) node [name=doublemapnode,inner xsep=0pt, inner ysep=0pt, minimum height=11pt, minimum width=23pt,shape=rectangle,draw,rounded corners] {$#3$} (#1,#2) .. controls (#1,#2-0.075) .. (doublemapnode) (#1+1,#2) .. controls (#1+1,#2-0.075).. (doublemapnode) (doublemapnode) .. controls (#1,#2-0.925)..(#1,#2-1) (doublemapnode) .. controls (#1+1,#2-0.925).. (#1+1,#2-1)}
\def\solbraid(#1,#2)[#3]{\draw (#1,#2-0.5)  node[name=nodemap,inner sep=0pt,  minimum size=9pt, shape=circle,draw]{$#3$}
(#1-0.5,#2) .. controls (#1-0.5,#2-0.15) and (#1-0.4,#2-0.2) .. (#1-0.3,#2-0.3) (#1-0.3,#2-0.3) -- (nodemap)
(#1+0.5,#2) .. controls (#1+0.5,#2-0.15) and (#1+0.4,#2-0.2) .. (#1+0.3,#2-0.3) (#1+0.3,#2-0.3) -- (nodemap)
(#1+0.5,#2-1) .. controls (#1+0.5,#2-0.85) and (#1+0.4,#2-0.8) .. (#1+0.3,#2-0.7) (#1+0.3,#2-0.7) -- (nodemap)
(#1-0.5,#2-1) .. controls (#1-0.5,#2-0.85) and (#1-0.4,#2-0.8) .. (#1-0.3,#2-0.7) (#1-0.3,#2-0.7) -- (nodemap)
}
\def\ractiontr(#1,#2)[#3,#4,#5]{\draw (#1,#2) -- (#1,#2-2*#4/2)  (#1,#2-1*#4/2) node [inner sep=0pt, minimum size=3pt,shape=isosceles triangle,fill, shape border rotate=#5] {}  --(#1+1*#4/2+#3*#4/2,#2-1*#4/2) .. controls (#1+1.555*#4/2+#3*#4/2,#2-1*#4/2) and (#1+2*#4/2+#3*#4/2,#2-0.555*#4/2) .. (#1+2*#4/2+#3*#4/2,#2)  }
\def\flip(#1,#2)[#3]{\draw (
#1+1*#3,#2) .. controls (#1+1*#3,#2-0.05*#3) and (#1+0.96*#3,#2-0.15*#3).. (#1+0.9*#3,#2-0.2*#3)
(#1+0.1*#3,#2-0.8*#3)--(#1+0.9*#3,#2-0.2*#3)
(#1,#2-1*#3) .. controls (#1,#2-0.95*#3) and (#1+0.04*#3,#2-0.85*#3).. (#1+0.1*#3,#2-0.8*#3)
(#1,#2) .. controls (#1,#2-0.05*#3) and (#1+0.04*#3,#2-0.15*#3).. (#1+0.1*#3,#2-0.2*#3)
(#1+0.1*#3,#2-0.2*#3) -- (#1+0.9*#3,#2-0.8*#3)
(#1+1*#3,#2-1*#3) .. controls (#1+1*#3,#2-0.95*#3) and (#1+0.96*#3,#2-0.85*#3).. (#1+0.9*#3,#2-0.8*#3)
}
\def\raction(#1,#2)[#3,#4]{\draw (#1,#2) -- (#1,#2-2*#4/2)  (#1,#2-1*#4/2)--(#1+1*#4/2+#3*#4/2,#2-1*#4/2) .. controls (#1+1.555*#4/2+#3*#4/2,#2-1*#4/2) and (#1+2*#4/2+#3*#4/2,#2-0.555*#4/2) .. (#1+2*#4/2+#3*#4/2,#2)}
\def\rack(#1,#2)[#3]{\draw (#1,#2-0.5)  node[name=nodemap,inner sep=0pt,  minimum size=7.5pt, shape=circle,draw]{$#3$} (#1-1,#2) .. controls (#1-1,#2-0.5) and (#1-0.5,#2-0.5) .. (nodemap) (#1,#2)-- (nodemap)  (nodemap)-- (#1,#2-1)}
\def\racklarge(#1,#2)[#3]{\draw (#1,#2-0.5)  node[name=nodemap,inner sep=0pt,  minimum size=7.5pt, shape=circle,draw]{$#3$} (#1-2,#2+0.5) .. controls (#1-2,#2-0.5) and (#1-0.5,#2-0.5) .. (nodemap) (#1,#2)-- (nodemap)  (nodemap)-- (#1,#2-1)}
\begin{scope}[xshift=0cm, yshift=-0.9cm]
\comult(0.5,0)[1,1]; \draw (0,-1) -- (0,-2);  \solbraid(1.5,-1)[\scriptstyle s]; \flip(0,-2)[1]; \solbraid(0.5,-3)[\scriptstyle \bar{s}]; \draw (2,-2) -- (2,-4); \counit(2,-4); \counit(0,-4); \draw (1,-4) -- (1,-5); \draw (2,0) -- (2,-1);
\end{scope}
\begin{scope}[xshift=2.55cm, yshift=-2.75cm]
\node at (0,-0.5){=};
\end{scope}
\begin{scope}[xshift=3.15cm, yshift=0cm]
\solbraid(0.5,0)[\scriptstyle s]; \draw (1,-1) .. controls (1,-1.25) and (1.5,-1.25) .. (1.5,-1.5);  \comult(1.5,-1.5)[1,1]; \flip(0,-2.5)[1]; \flip(0,-3.5)[1]; \counit(2,-5.5); \solbraid(0.5,-4.5)[\scriptstyle \bar{s}]; \counit(0,-5.5); \draw (2,-2.5) -- (2,-5.5); \draw (1,-5.5) -- (1,-6.5); \draw (0,-1) -- (0,-2.5);
\end{scope}
\begin{scope}[xshift=6.8cm, yshift=-2.75cm]
\node at (0,-0.5){$= \epsilon \ot X$};
\end{scope}
\begin{scope}[xshift=10.2cm, yshift=-2.75cm]
\node at (0,-0.5){and};
\end{scope}
\begin{scope}[xshift=12.4cm, yshift=-0.25cm]
\comult(0.5,0)[1,1]; \flip(1,-1)[1]; \draw (0,-1) -- (0,-3); \solbraid(1.5,-2)[\scriptstyle \bar{s}]; \flip(0,-3)[1]; \solbraid(1.5,-4)[\scriptstyle s]; \draw (0,-4) -- (0,-5); \draw (1,-5) -- (1,-6); \counit(0,-5); \counit(2,-5); \draw (2,0) -- (2,-1); \draw (2,-3) -- (2,-4);
\end{scope}
\begin{scope}[xshift=14.95cm, yshift=-2.75cm]
\node at (0,-0.5){=};
\end{scope}
\begin{scope}[xshift=15.5cm, yshift=0cm]
\flip(0.5,0)[1]; \solbraid(1,-1)[\scriptstyle \bar{s}]; \comult(0.5,-2)[1,1]; \flip(0,-3)[1]; \draw (1.5,-2) .. controls (1.5,-2.25) and (2,-2.25) .. (2,-2.5); \draw (2,-2.5) -- (2,-4); \draw (0,-4) -- (0,-5); \counit(0,-5); \solbraid(1.5,-4)[\scriptstyle s]; \draw (1,-5) -- (1,-6); \counit(2,-5);
\end{scope}
\begin{scope}[xshift=19.3cm, yshift=-2.75cm]
\node at (0,-0.5){$= \epsilon \ot X$,};
\end{scope}
\end{tikzpicture}
$$
where $\bar{s}$ denotes $s^{-1}$. Since $(\epsilon\ot X)\circ s\! = \!X\ot \epsilon$ this proves that $(X,s)$ is a non-degenerate pair.
\end{proof}

\begin{rem}
By Proposition~\ref{prop auxiliar}, if $(X,r)$ is a non-degenerate involutive braided set, then $\triangleright= \epsilon\ot X$ or, equivalently, $s=c$. Conversely, if $s=c$, then from Proposition~\ref{relacion entre Jn r y s} we obtain $J_2\circ r^2\circ J_2^{-1}= c^2=\ide$, which implies that $r$ is involutive.
\end{rem}

\section{Braces, invertible cocycles and braiding operators}\label{section: Braiding operators, invertible cocycles and braces}
In this section we adapt the notions of brace, braiding operator and invertible cocycle to the setting of symmetric monoidal category, and we prove that the obtained categories are equivalent. In the diagrammatic proofs the map $m_{\circ}$ will be  represented by the symbol
\begin{tikzpicture} \def\multsubzero(#1,#2)[#3,#4]{\draw (#1+0.5*#3, #2-0.5*#4) node
[name=nodemap,inner sep=0pt, minimum size=3pt,shape=circle,fill=white,
draw]{} (#1,#2) arc (180:360:0.5*#3 and 0.5*#4) (#1+0.5*#3, #2-0.5*#4) --
(#1+0.5*#3,#2-#4)} \begin{scope}[xshift=0cm, yshift=0cm]
\multsubzero(0,0)[0.29,0.29];
\end{scope}
\end{tikzpicture}.

\subsection{Braces}

Braces were introduced by Rump in~\cite{MR2278047} to study involutive set-theoretical solutions. Skew braces are generalizations useful for studying non-involutive set-theoretical solutions~\cite{GV}.

\begin{defn}\label{braces} A \emph{brace} in $\mathscr{C}$ is a pair $(A,m,\eta,\Delta,\epsilon,S)$ and $(A,m_{\circ},\eta_{\circ},\Delta,\epsilon,T)$, of cocommutative Hopf algebras in $\mathscr{C}$ with the same comultiplication, such that
\begin{equation}\label{ecuabraces}
m_{\circ}\circ(A\otimes m)=m\circ(m_{\circ}\otimes \lambda)\circ (A\otimes c\otimes A)\circ (\Delta\otimes A^2),
\end{equation}
where $\lambda\coloneqq m\circ (S\otimes m_{\circ})\circ (\Delta\otimes A)$. This brace will be denoted by $(A,m,m_{\circ})$. A \emph{morphism of braces} is a map in $\mathscr{C}$ that is a Hopf algebra morphism for both Hopf algebra structures.
\end{defn}

\begin{rem} It is easy to see that in any brace in $\mathscr{C}$ one has $\eta=\eta_\circ$.
\end{rem}

Fix a cocommutative Hopf algebra $A=(A,m,\eta,\Delta,\epsilon,S)$. We let $\Br(A)$ denote the full subcategory of the category of braces in $\mathscr{C}$ with objects $(A,m,m_{\circ})$. Note that here $m$ is fixed, but $m_{\circ}$ is not.

\smallskip

Let $H$ be a Hopf algebra in $\mathscr{C}$ and let $A$ be an object in $\mathcal{C}$. Recall that $A$ is a left $H$-module in $\mathscr{C}$ via a map $\lambda\colon H\otimes A\longrightarrow A$, named the {\em left action of $H$ on $A$}, if
$$
\lambda\circ(\eta\otimes A)=\id\quad \text{and} \quad \lambda\circ(m\otimes A)=\lambda\circ(H\otimes\lambda).
$$
Recall also that an algebra $A$ in $\mathscr{C}$ is a left $H$-module algebra if it is a left $H$-module such that
$$
\lambda\circ (H\otimes \eta)=\epsilon\otimes\eta \quad \text{and} \quad \lambda\circ (H\otimes m)=m\circ (\lambda\otimes\lambda)\circ (H\otimes c\otimes A)\circ (\Delta\otimes A^2),
$$
where $\lambda$ is the action, and that a coalgebra $A$ in $\mathscr{C}$ is a left $H$-module coalgebra if $A$ is a left $H$-module such that
$$
\Delta\circ\lambda=(\lambda\otimes\lambda)\circ\Delta_{H\otimes C}\quad \text{and} \quad  \epsilon\circ \lambda=\epsilon\otimes\epsilon,
$$
where $\lambda$ is the action. Recall finally that $A$ is a right $H$-module in $\mathscr{C}$ via a map $\rho\colon A\otimes H\longrightarrow A$, named the {\em right action of $H$ on $A$}, if it is a left $H$-module in $\mathscr{C}$ via $\rho\circ c$. The notions of right $H$-module-algebra and right $H$-module coalgebra are defined in the same way.

\begin{pro}\label{pro:modulo} If $(A,m,m_{\circ})$ is a brace in $\mathscr{C}$, then the algebra $(A,m)$ is a left $(A,m_{\circ},\Delta)$-module algebra and module coalgebra via $\lambda\coloneqq m\circ (S\otimes m_{\circ})\circ (\Delta\otimes A)$.
\end{pro}

\begin{proof} Since $\lambda$ is a composition of coalgebra morphisms, it is a coalgebra morphism. Moreover, it is easy to check that $\lambda$ is unitary and that $\lambda\circ (A\ot \eta) = \epsilon\ot \eta$. Consider $\Hom(A^3,A)$ endowed with the convolution product associated with the comultiplication $\Delta_{A^3}$ and the multiplication $m$. Multiplying by $S\ot \epsilon^2$ on the left the equality that appears in Definition~\ref{braces} we obtain that $\lambda$ is compatible with $m$. Next we prove the associativity of the action. We claim that
\begin{equation}\label{auxiliar}
m\circ (S\circ m_\circ\otimes A)\circ (A\otimes c)\circ (\Delta\otimes A) = \lambda \circ (A\otimes S).
\end{equation}
Consider $\Hom(A^2,A)$ endowed with the convolution product associated with the comultiplication $\Delta_{A^2}$ and the multiplication $m$. Let
$$
f\coloneqq m\circ (S\circ m_{\circ}\ot A)\circ (A\ot c)\circ (\Delta\ot A)\quad\text{and}\quad  g\coloneqq \lambda \circ (A\otimes S).
$$
We must show that $f=g$. By equality~\eqref{ecuabraces} we have $m_0\star g = A\ot \epsilon$. So
$$
f = (S\circ m_{\circ})\star (A\ot \epsilon) = (S\circ m_{\circ})\star m_{\circ}\star g = g,
$$
where the first equality is trivial and the third one follows from the fact that $S\circ m_0$ is the convolution inverse of $m_{\circ}$ (because  $m_{\circ}$ is a coalgebra morphism). This finishes the proof of the claim. The associativity of the action follows now from the fact that
$$
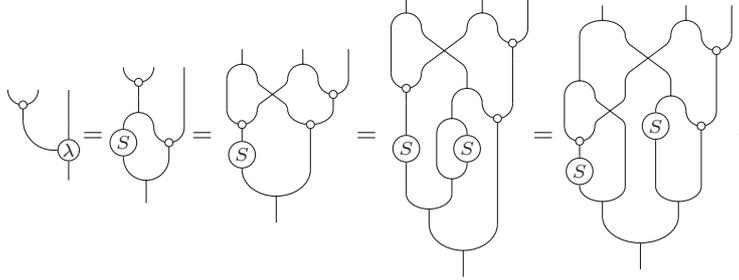
\begin{tikzpicture}[scale=0.40]
\def\mult(#1,#2)[#3,#4]{\draw (#1,#2) arc (180:360:0.5*#3 and 0.5*#4) (#1+0.5*#3, #2-0.5*#4) -- (#1+0.5*#3,#2-#4)}
\def\counit(#1,#2){\draw (#1,#2) -- (#1,#2-0.93) (#1,#2-1) circle[radius=2pt]}
\def\comult(#1,#2)[#3,#4]{\draw (#1,#2) -- (#1,#2-0.5*#4) arc (90:0:0.5*#3 and 0.5*#4) (#1,#2-0.5*#4) arc (90:180:0.5*#3 and 0.5*#4)}
\def\laction(#1,#2)[#3,#4]{\draw (#1,#2) .. controls (#1,#2-0.555*#4/2) and (#1+0.445*#4/2,#2-1*#4/2) .. (#1+1*#4/2,#2-1*#4/2) -- (#1+2*#4/2+#3*#4/2,#2-1*#4/2) (#1+2*#4/2+#3*#4/2,#2)--(#1+2*#4/2+#3*#4/2,#2-2*#4/2)}
\def\map(#1,#2)[#3]{\draw (#1,#2-0.5)  node[name=nodemap,inner sep=0pt,  minimum size=10pt, shape=circle, draw]{$#3$} (#1,#2)-- (nodemap)  (nodemap)-- (#1,#2-1)}
\def\solbraid(#1,#2)[#3]{\draw (#1,#2-0.5)  node[name=nodemap,inner sep=0pt,  minimum size=9pt, shape=circle,draw]{$#3$}
(#1-0.5,#2) .. controls (#1-0.5,#2-0.15) and (#1-0.4,#2-0.2) .. (#1-0.3,#2-0.3) (#1-0.3,#2-0.3) -- (nodemap)
(#1+0.5,#2) .. controls (#1+0.5,#2-0.15) and (#1+0.4,#2-0.2) .. (#1+0.3,#2-0.3) (#1+0.3,#2-0.3) -- (nodemap)
(#1+0.5,#2-1) .. controls (#1+0.5,#2-0.85) and (#1+0.4,#2-0.8) .. (#1+0.3,#2-0.7) (#1+0.3,#2-0.7) -- (nodemap)
(#1-0.5,#2-1) .. controls (#1-0.5,#2-0.85) and (#1-0.4,#2-0.8) .. (#1-0.3,#2-0.7) (#1-0.3,#2-0.7) -- (nodemap)
}
\def\flip(#1,#2)[#3]{\draw (
#1+1*#3,#2) .. controls (#1+1*#3,#2-0.05*#3) and (#1+0.96*#3,#2-0.15*#3).. (#1+0.9*#3,#2-0.2*#3)
(#1+0.1*#3,#2-0.8*#3)--(#1+0.9*#3,#2-0.2*#3)
(#1,#2-1*#3) .. controls (#1,#2-0.95*#3) and (#1+0.04*#3,#2-0.85*#3).. (#1+0.1*#3,#2-0.8*#3)
(#1,#2) .. controls (#1,#2-0.05*#3) and (#1+0.04*#3,#2-0.15*#3).. (#1+0.1*#3,#2-0.2*#3)
(#1+0.1*#3,#2-0.2*#3) -- (#1+0.9*#3,#2-0.8*#3)
(#1+1*#3,#2-1*#3) .. controls (#1+1*#3,#2-0.95*#3) and (#1+0.96*#3,#2-0.85*#3).. (#1+0.9*#3,#2-0.8*#3)
}
\def\raction(#1,#2)[#3,#4]{\draw (#1,#2) -- (#1,#2-2*#4/2)  (#1,#2-1*#4/2)--(#1+1*#4/2+#3*#4/2,#2-1*#4/2) .. controls (#1+1.555*#4/2+#3*#4/2,#2-1*#4/2) and (#1+2*#4/2+#3*#4/2,#2-0.555*#4/2) .. (#1+2*#4/2+#3*#4/2,#2)}
\def\doublemap(#1,#2)[#3]{\draw (#1+0.5,#2-0.5) node [name=doublemapnode,inner xsep=0pt, inner ysep=0pt, minimum height=11pt, minimum width=23pt,shape=rectangle,draw,rounded corners] {$#3$} (#1,#2) .. controls (#1,#2-0.075) .. (doublemapnode) (#1+1,#2) .. controls (#1+1,#2-0.075).. (doublemapnode) (doublemapnode) .. controls (#1,#2-0.925)..(#1,#2-1) (doublemapnode) .. controls (#1+1,#2-0.925).. (#1+1,#2-1)}
\def\doublesinglemap(#1,#2)[#3]{\draw (#1+0.5,#2-0.5) node [name=doublesinglemapnode,inner xsep=0pt, inner ysep=0pt, minimum height=11pt, minimum width=23pt,shape=rectangle,draw,rounded corners] {$#3$} (#1,#2) .. controls (#1,#2-0.075) .. (doublesinglemapnode) (#1+1,#2) .. controls (#1+1,#2-0.075).. (doublesinglemapnode) (doublesinglemapnode)-- (#1+0.5,#2-1)}
\def\ractiontr(#1,#2)[#3,#4,#5]{\draw (#1,#2) -- (#1,#2-2*#4/2)  (#1,#2-1*#4/2) node [inner sep=0pt, minimum size=3pt,shape=isosceles triangle,fill, shape border rotate=#5] {}  --(#1+1*#4/2+#3*#4/2,#2-1*#4/2) .. controls (#1+1.555*#4/2+#3*#4/2,#2-1*#4/2) and (#1+2*#4/2+#3*#4/2,#2-0.555*#4/2) .. (#1+2*#4/2+#3*#4/2,#2)  }
\def\rack(#1,#2)[#3]{\draw (#1,#2-0.5)  node[name=nodemap,inner sep=0pt,  minimum size=7.5pt, shape=circle,draw]{$#3$} (#1-1,#2) .. controls (#1-1,#2-0.5) and (#1-0.5,#2-0.5) .. (nodemap) (#1,#2)-- (nodemap)  (nodemap)-- (#1,#2-1)}
\def\rackmenoslarge(#1,#2)[#3]{\draw (#1,#2-0.5)  node[name=nodemap,inner sep=0pt,  minimum size=7.5pt, shape=circle,draw]{$#3$} (#1-1.5,#2+0.5) .. controls (#1-1.5,#2-0.5) and (#1-0.5,#2-0.5) .. (nodemap) (#1,#2)-- (nodemap)  (nodemap)-- (#1,#2-1)}
\def\racklarge(#1,#2)[#3]{\draw (#1,#2-0.5)  node[name=nodemap,inner sep=0pt,  minimum size=7.5pt, shape=circle,draw]{$#3$} (#1-2,#2+0.5) .. controls (#1-2,#2-0.5) and (#1-0.5,#2-0.5) .. (nodemap) (#1,#2)-- (nodemap)  (nodemap)-- (#1,#2-1)}
\def\rackmaslarge(#1,#2)[#3]{\draw (#1,#2-0.5)  node[name=nodemap,inner sep=0pt,  minimum size=7.5pt, shape=circle,draw]{$#3$} (#1-2.5,#2+0.5) .. controls (#1-2.5,#2-0.5) and (#1-0.5,#2-0.5) .. (nodemap) (#1,#2)-- (nodemap)  (nodemap)-- (#1,#2-1)}
\def\rackextralarge(#1,#2)[#3]{\draw (#1,#2-0.75)  node[name=nodemap,inner sep=0pt,  minimum size=7.5pt, shape=circle, draw]{$#3$} (#1-3,#2+1) .. controls (#1-3,#2-0.75) and (#1-0.5,#2-0.75) .. (nodemap) (#1,#2)-- (nodemap)  (nodemap)-- (#1,#2-1.5)}
\def\lactionnamed(#1,#2)[#3,#4][#5]{\draw (#1 + 0.5*#3 + 0.5 + 0.5*#4, #2- 0.5*#3) node[name=nodemap,inner sep=0pt,  minimum size=8pt, shape=circle,draw]{$#5$} (#1,#2)  arc (180:270:0.5*#3) (#1 + 0.5*#3,#2- 0.5*#3) --  (nodemap) (#1 + 0.5*#3 + 0.5 + 0.5*#4, #2) --  (nodemap) (nodemap) -- (#1 + 0.5*#3 + 0.5 + 0.5*#4, #2-#3)}
\def\ractionnamed(#1,#2)[#3,#4][#5]{\draw  (#1 - 0.5*#3- 0.5 - 0.5*#4, #2- 0.5*#3)  node[name=nodemap,inner sep=0pt,  minimum size=8pt, shape=circle,draw]{$#5$} (#1 - 0.5*#3, #2- 0.5*#3)  arc (270:360:0.5*#3) (#1 - 0.5*#3, #2- 0.5*#3) -- (nodemap)(#1 - 0.5*#3- 0.5 - 0.5*#4, #2)-- (nodemap) (nodemap) -- (#1 - 0.5*#3- 0.5 - 0.5*#4, #2-#3)}
\def\multsubzero(#1,#2)[#3,#4]{\draw (#1+0.5*#3, #2-0.5*#4) node [name=nodemap,inner sep=0pt, minimum size=3pt,shape=circle,fill=white, draw]{} (#1,#2) arc (180:360:0.5*#3 and 0.5*#4) (#1+0.5*#3, #2-0.5*#4) -- (#1+0.5*#3,#2-#4)}
\begin{scope}[xshift=-3.3cm, yshift=-3.05cm]
\multsubzero(-0.5,0)[1,1]; \draw (1.5,0) -- (1.5,-1);
\lactionnamed(0,-1)[2,0][\scriptstyle \lambda];
\end{scope}
\begin{scope}[xshift=-1cm, yshift=-4.1cm]
\node at (0,-0.5){=};
\end{scope}
\begin{scope}[xshift=0cm, yshift=-2.3cm]
\multsubzero(0,0)[1,1]; \draw (2,0) -- (2,-2);
\comult(0.5,-1)[1,1];
\map(0,-2)[\scriptstyle S]; \multsubzero(1,-2)[1,1];
\mult(0,-3)[1.5,1.5];
\end{scope}
\begin{scope}[xshift=2.6cm, yshift=-4.1cm]
\node at (0,-0.5){=};
\end{scope}
\begin{scope}[xshift=2.9cm, yshift=-1.7cm]
\comult(1,0)[1,1]; \comult(3,0)[1,1]; \draw (4.5,0) -- (4.5,-1);
\draw (0.5,-1) -- (0.5,-2); \flip(1.5,-1)[1]; \multsubzero(3.5,-1)[1,1];
\multsubzero(0.5,-2)[1,1]; \multsubzero(2.5,-2)[1.5,1];
\map(1,-3)[\scriptstyle S];\draw (3.25,-3) -- (3.25,-4);
\mult(1,-4)[2.25,1.75];
\end{scope}
\begin{scope}[xshift=8cm, yshift=-4.1cm]
\node at (0,-0.5){=};
\end{scope}
\begin{scope}[xshift=8.3cm, yshift=0cm]
\comult(1,0)[1,1]; \comult(3.5,0)[1,1]; \draw (5,0) -- (5,-1);
\draw (0.5,-1) -- (0.5,-2.5); \flip(1.5,-1)[1.5]; \multsubzero(4,-1)[1,1];
\draw (4.5,-2) -- (4.5,-3.5);
\multsubzero(0.5,-2.5)[1,1]; \comult(3,-2.5)[1,1];
\draw (1,-3.5) -- (1,-4.5);
\map(1,-4.5)[\scriptstyle S];\comult(2.5,-3.5)[1,1];
\multsubzero(3.5,-3.5)[1,1];
\draw (1,-5.5) -- (1,-6.5);  \draw (2,-4.5) -- (2,-5.5); \map(3,-4.5)[\scriptstyle S]; \draw (4,-4.5) -- (4,-7.5);
\mult(2,-5.5)[1,1];
\mult(1,-6.5)[1.5,1];
\mult(1.75,-7.5)[2.25,1.75];
\end{scope}
\begin{scope}[xshift=13.8cm, yshift=-4.1cm]
\node at (0,-0.5){=};
\end{scope}
\begin{scope}[xshift=14.5cm, yshift=-0.25cm]
\comult(1.25,0)[1.5,1]; \comult(4,0)[1,1]; \draw (5.5,0) -- (5.5,-1);
\draw (0.5,-1) -- (0.5,-2); \flip(2,-1)[1.5]; \multsubzero(4.5,-1)[1,1];
\comult(0.5,-2)[1,1]; \draw (2,-2.5) -- (2,-3); \draw (5,-2) -- (5,-3.5);
\draw (0,-3) -- (0,-4); \flip(1,-3)[1]; \comult(3.5,-2.5)[1,1]; \map(3,-3.5)[\scriptstyle S];  \multsubzero(4,-3.5)[1,1];
\multsubzero(0,-4)[1,1]; \draw (2,-4) -- (2,-6); \map(0.5,-5)[\scriptstyle S];
\draw (3,-4.5) -- (3,-6); \draw (4.5,-4.5) -- (4.5,-6);
\mult(0.5,-6)[1.5,1]; \mult(3,-6)[1.5,1];
\mult(1.25,-7)[2.5,1.75];
\end{scope}
\begin{scope}[xshift=20.3cm, yshift=-4.1cm]
\node at (0,-0.5){,};
\end{scope}
\end{tikzpicture}
$$
where the first equality holds by definition; the second one, since $m_{\circ}$ is a coalgebra morphism; the third one, since $S$ is the antipode of $(A,m,\Delta)$; and the last one, since $c$ in natural, $\Delta$ is coassociative and $m$ is associative; and
$$
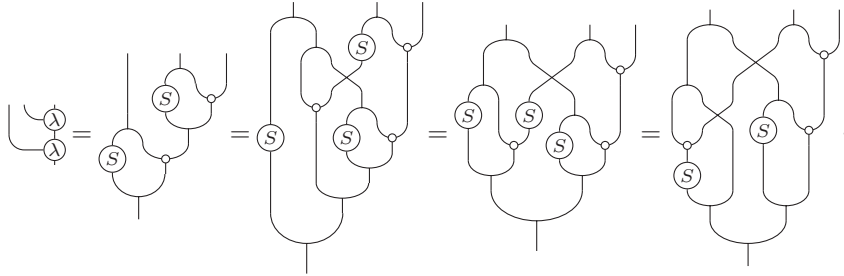
\begin{tikzpicture}[scale=0.40]
\def\mult(#1,#2)[#3,#4]{\draw (#1,#2) arc (180:360:0.5*#3 and 0.5*#4) (#1+0.5*#3, #2-0.5*#4) -- (#1+0.5*#3,#2-#4)}
\def\counit(#1,#2){\draw (#1,#2) -- (#1,#2-0.93) (#1,#2-1) circle[radius=2pt]}
\def\comult(#1,#2)[#3,#4]{\draw (#1,#2) -- (#1,#2-0.5*#4) arc (90:0:0.5*#3 and 0.5*#4) (#1,#2-0.5*#4) arc (90:180:0.5*#3 and 0.5*#4)}
\def\laction(#1,#2)[#3,#4]{\draw (#1,#2) .. controls (#1,#2-0.555*#4/2) and (#1+0.445*#4/2,#2-1*#4/2) .. (#1+1*#4/2,#2-1*#4/2) -- (#1+2*#4/2+#3*#4/2,#2-1*#4/2) (#1+2*#4/2+#3*#4/2,#2)--(#1+2*#4/2+#3*#4/2,#2-2*#4/2)}
\def\map(#1,#2)[#3]{\draw (#1,#2-0.5)  node[name=nodemap,inner sep=0pt,  minimum size=10pt, shape=circle, draw]{$#3$} (#1,#2)-- (nodemap)  (nodemap)-- (#1,#2-1)}
\def\solbraid(#1,#2)[#3]{\draw (#1,#2-0.5)  node[name=nodemap,inner sep=0pt,  minimum size=9pt, shape=circle,draw]{$#3$}
(#1-0.5,#2) .. controls (#1-0.5,#2-0.15) and (#1-0.4,#2-0.2) .. (#1-0.3,#2-0.3) (#1-0.3,#2-0.3) -- (nodemap)
(#1+0.5,#2) .. controls (#1+0.5,#2-0.15) and (#1+0.4,#2-0.2) .. (#1+0.3,#2-0.3) (#1+0.3,#2-0.3) -- (nodemap)
(#1+0.5,#2-1) .. controls (#1+0.5,#2-0.85) and (#1+0.4,#2-0.8) .. (#1+0.3,#2-0.7) (#1+0.3,#2-0.7) -- (nodemap)
(#1-0.5,#2-1) .. controls (#1-0.5,#2-0.85) and (#1-0.4,#2-0.8) .. (#1-0.3,#2-0.7) (#1-0.3,#2-0.7) -- (nodemap)
}
\def\flip(#1,#2)[#3]{\draw (
#1+1*#3,#2) .. controls (#1+1*#3,#2-0.05*#3) and (#1+0.96*#3,#2-0.15*#3).. (#1+0.9*#3,#2-0.2*#3)
(#1+0.1*#3,#2-0.8*#3)--(#1+0.9*#3,#2-0.2*#3)
(#1,#2-1*#3) .. controls (#1,#2-0.95*#3) and (#1+0.04*#3,#2-0.85*#3).. (#1+0.1*#3,#2-0.8*#3)
(#1,#2) .. controls (#1,#2-0.05*#3) and (#1+0.04*#3,#2-0.15*#3).. (#1+0.1*#3,#2-0.2*#3)
(#1+0.1*#3,#2-0.2*#3) -- (#1+0.9*#3,#2-0.8*#3)
(#1+1*#3,#2-1*#3) .. controls (#1+1*#3,#2-0.95*#3) and (#1+0.96*#3,#2-0.85*#3).. (#1+0.9*#3,#2-0.8*#3)
}
\def\raction(#1,#2)[#3,#4]{\draw (#1,#2) -- (#1,#2-2*#4/2)  (#1,#2-1*#4/2)--(#1+1*#4/2+#3*#4/2,#2-1*#4/2) .. controls (#1+1.555*#4/2+#3*#4/2,#2-1*#4/2) and (#1+2*#4/2+#3*#4/2,#2-0.555*#4/2) .. (#1+2*#4/2+#3*#4/2,#2)}
\def\doublemap(#1,#2)[#3]{\draw (#1+0.5,#2-0.5) node [name=doublemapnode,inner xsep=0pt, inner ysep=0pt, minimum height=11pt, minimum width=23pt,shape=rectangle,draw,rounded corners] {$#3$} (#1,#2) .. controls (#1,#2-0.075) .. (doublemapnode) (#1+1,#2) .. controls (#1+1,#2-0.075).. (doublemapnode) (doublemapnode) .. controls (#1,#2-0.925)..(#1,#2-1) (doublemapnode) .. controls (#1+1,#2-0.925).. (#1+1,#2-1)}
\def\doublesinglemap(#1,#2)[#3]{\draw (#1+0.5,#2-0.5) node [name=doublesinglemapnode,inner xsep=0pt, inner ysep=0pt, minimum height=11pt, minimum width=23pt,shape=rectangle,draw,rounded corners] {$#3$} (#1,#2) .. controls (#1,#2-0.075) .. (doublesinglemapnode) (#1+1,#2) .. controls (#1+1,#2-0.075).. (doublesinglemapnode) (doublesinglemapnode)-- (#1+0.5,#2-1)}
\def\ractiontr(#1,#2)[#3,#4,#5]{\draw (#1,#2) -- (#1,#2-2*#4/2)  (#1,#2-1*#4/2) node [inner sep=0pt, minimum size=3pt,shape=isosceles triangle,fill, shape border rotate=#5] {}  --(#1+1*#4/2+#3*#4/2,#2-1*#4/2) .. controls (#1+1.555*#4/2+#3*#4/2,#2-1*#4/2) and (#1+2*#4/2+#3*#4/2,#2-0.555*#4/2) .. (#1+2*#4/2+#3*#4/2,#2)  }
\def\rack(#1,#2)[#3]{\draw (#1,#2-0.5)  node[name=nodemap,inner sep=0pt,  minimum size=7.5pt, shape=circle,draw]{$#3$} (#1-1,#2) .. controls (#1-1,#2-0.5) and (#1-0.5,#2-0.5) .. (nodemap) (#1,#2)-- (nodemap)  (nodemap)-- (#1,#2-1)}
\def\rackmenoslarge(#1,#2)[#3]{\draw (#1,#2-0.5)  node[name=nodemap,inner sep=0pt,  minimum size=7.5pt, shape=circle,draw]{$#3$} (#1-1.5,#2+0.5) .. controls (#1-1.5,#2-0.5) and (#1-0.5,#2-0.5) .. (nodemap) (#1,#2)-- (nodemap)  (nodemap)-- (#1,#2-1)}
\def\racklarge(#1,#2)[#3]{\draw (#1,#2-0.5)  node[name=nodemap,inner sep=0pt,  minimum size=7.5pt, shape=circle,draw]{$#3$} (#1-2,#2+0.5) .. controls (#1-2,#2-0.5) and (#1-0.5,#2-0.5) .. (nodemap) (#1,#2)-- (nodemap)  (nodemap)-- (#1,#2-1)}
\def\rackmaslarge(#1,#2)[#3]{\draw (#1,#2-0.5)  node[name=nodemap,inner sep=0pt,  minimum size=7.5pt, shape=circle,draw]{$#3$} (#1-2.5,#2+0.5) .. controls (#1-2.5,#2-0.5) and (#1-0.5,#2-0.5) .. (nodemap) (#1,#2)-- (nodemap)  (nodemap)-- (#1,#2-1)}
\def\rackextralarge(#1,#2)[#3]{\draw (#1,#2-0.75)  node[name=nodemap,inner sep=0pt,  minimum size=7.5pt, shape=circle, draw]{$#3$} (#1-3,#2+1) .. controls (#1-3,#2-0.75) and (#1-0.5,#2-0.75) .. (nodemap) (#1,#2)-- (nodemap)  (nodemap)-- (#1,#2-1.5)}
\def\lactionnamed(#1,#2)[#3,#4][#5]{\draw (#1 + 0.5*#3 + 0.5 + 0.5*#4, #2- 0.5*#3) node[name=nodemap,inner sep=0pt,  minimum size=8pt, shape=circle,draw]{$#5$} (#1,#2)  arc (180:270:0.5*#3) (#1 + 0.5*#3,#2- 0.5*#3) --  (nodemap) (#1 + 0.5*#3 + 0.5 + 0.5*#4, #2) --  (nodemap) (nodemap) -- (#1 + 0.5*#3 + 0.5 + 0.5*#4, #2-#3)}
\def\ractionnamed(#1,#2)[#3,#4][#5]{\draw  (#1 - 0.5*#3- 0.5 - 0.5*#4, #2- 0.5*#3)  node[name=nodemap,inner sep=0pt,  minimum size=8pt, shape=circle,draw]{$#5$} (#1 - 0.5*#3, #2- 0.5*#3)  arc (270:360:0.5*#3) (#1 - 0.5*#3, #2- 0.5*#3) -- (nodemap)(#1 - 0.5*#3- 0.5 - 0.5*#4, #2)-- (nodemap) (nodemap) -- (#1 - 0.5*#3- 0.5 - 0.5*#4, #2-#3)}
\def\multsubzero(#1,#2)[#3,#4]{\draw (#1+0.5*#3, #2-0.5*#4) node [name=nodemap,inner sep=0pt, minimum size=3pt,shape=circle,fill=white, draw]{} (#1,#2) arc (180:360:0.5*#3 and 0.5*#4) (#1+0.5*#3, #2-0.5*#4) -- (#1+0.5*#3,#2-#4)}
\begin{scope}[xshift=0cm, yshift=-3.4cm]
\lactionnamed(0.5,0)[1,0][\scriptstyle \lambda]; \draw (0,0) -- (0,-1);
\lactionnamed(0,-1)[1,1][\scriptstyle \lambda];
\end{scope}
\begin{scope}[xshift=2.4cm, yshift=-3.95cm]
\node at (0,-0.5){=};
\end{scope}
\begin{scope}[xshift=3.9cm, yshift=-1.7cm]
\draw (0,0) -- (0,-2); \comult(1.725,0)[1,1]; \draw (3.25,0) -- (3.25,-1);
\map(1.25,-1)[\scriptstyle S]; \multsubzero(2.25,-1)[1,1];
\comult(0,-2)[1,1]; \mult(1.25,-2)[1.5,1];
\map(-0.5,-3)[\scriptstyle S]; \multsubzero(0.5,-3)[1.5,1];
\mult(-0.5,-4)[1.725,1.5];
\end{scope}
\begin{scope}[xshift=7.6cm, yshift=-3.95cm]
\node at (0,-0.5){=};
\end{scope}
\begin{scope}[xshift=8.6cm, yshift=0cm]
\comult(0.75,0)[1.5,1]; \comult(3.5,0)[1,1]; \draw (5,0) -- (5,-1);
\draw (0,-1) -- (0,-4); \comult(1.5,-1)[1,1]; \map(3,-1)[\scriptstyle S]; \multsubzero(4,-1)[1,1];
\draw (1,-2) -- (1,-3); \flip(2,-2)[1]; \draw (4.5,-2) -- (4.5,-4);
\multsubzero(1,-3)[1,1]; \comult(3,-3)[1,1];
\map(0,-4)[\scriptstyle S]; \draw (1.5,-4) -- (1.5,-6); \map(2.5,-4)[\scriptstyle S]; \multsubzero(3.5,-4)[1,1];
\draw (0,-5) -- (0,-7); \mult(2.5,-5)[1.5,1];
\mult(1.5,-6)[1.725,1];
\mult(0,-7)[2.375,2];
\end{scope}
\begin{scope}[xshift=14.1cm, yshift=-3.95cm]
\node at (0,-0.5){=};
\end{scope}
\begin{scope}[xshift=15.1cm, yshift=-0.75cm]
\comult(1.25,0)[1.5,1]; \comult(4,0)[1,1]; \draw (5.5,0) -- (5.5,-1);
\draw (0.5,-1) -- (0.5,-1.5); \flip(2,-1)[1.5]; \multsubzero(4.5,-1)[1,1];
\comult(0.5,-1.5)[1,1]; \draw (5,-2) -- (5,-3.5);
\map(0,-2.5)[\scriptstyle S];\draw (1,-2.5) -- (1,-3.5);\map(2,-2.5)[\scriptstyle S]; \comult(3.5,-2.5)[1,1];
\draw (0,-3.5) -- (0,-4.5); \multsubzero(1,-3.5)[1,1]; \map(3,-3.5)[\scriptstyle S]; \multsubzero(4,-3.5)[1,1];
\mult(0,-4.5)[1.5,1]; \mult(3,-4.5)[1.5,1];
\mult(0.75,-5.5)[3,2];
\end{scope}
\begin{scope}[xshift=21.1cm, yshift=-3.95cm]
\node at (0,-0.5){=};
\end{scope}
\begin{scope}[xshift=21.8cm, yshift=-0.25cm]
\comult(1.25,0)[1.5,1]; \comult(4,0)[1,1]; \draw (5.5,0) -- (5.5,-1);
\draw (0.5,-1) -- (0.5,-2); \flip(2,-1)[1.5]; \multsubzero(4.5,-1)[1,1];
\comult(0.5,-2)[1,1]; \draw (2,-2.5) -- (2,-3); \draw (5,-2) -- (5,-3.5);
\draw (0,-3) -- (0,-4); \flip(1,-3)[1]; \comult(3.5,-2.5)[1,1]; \map(3,-3.5)[\scriptstyle S];  \multsubzero(4,-3.5)[1,1];
\multsubzero(0,-4)[1,1]; \draw (2,-4) -- (2,-6); \map(0.5,-5)[\scriptstyle S];
\draw (3,-4.5) -- (3,-6); \draw (4.5,-4.5) -- (4.5,-6);
\mult(0.5,-6)[1.5,1]; \mult(3,-6)[1.5,1];
\mult(1.25,-7)[2.5,1.5];
\end{scope}
\begin{scope}[xshift=27.5cm, yshift=-3.95cm]
\node at (0,-0.5){,};
\end{scope}
\end{tikzpicture}
$$
where the first equality holds by definition; the second one, by~\eqref{ecuabraces}; the third one, since $c$ in natural, $\Delta$ is coassociative and $m$ is associative; and the last one, by~\eqref{auxiliar}.
\end{proof}

\begin{rem}\label{rem:productos} It follows from the definition that in any brace $(A,m,m_{\circ})$
\begin{align}
\label{eq:aob} & m_\circ=m\circ (A\otimes\lambda)\circ (\Delta\otimes A).
\shortintertext{Using this formula and Proposition~\ref{pro:modulo} one obtains that}
\label{eq:ab} & m=m_\circ\circ (A\otimes\lambda)\circ (A\otimes T\otimes A)\circ (\Delta\otimes A).
\end{align}
\end{rem}

\begin{rem}\label{rem:brace} In Proposition~\ref{pro:modulo} we proved in particular that
\begin{equation}\label{eq:brace}
\lambda\circ (A\otimes m)=m\circ(\lambda\otimes\lambda)\circ(A\otimes c\otimes A)\circ (\Delta\otimes A^2).			
\end{equation}
Conversely, a pair $(A,\Delta,\epsilon,m,\eta,S)$ and $(A,\Delta,\epsilon,m_{\circ},\eta_{\circ}, T)$ of cocommutative Hopf algebras with the same comultiplication, such that the map $\lambda\colon A\otimes A\to A$, given by
\[
\lambda\coloneqq m\circ  (S\otimes m_{\circ})\circ (\Delta\otimes A),
\]
satisfies Equality~\eqref{eq:brace}, is a brace in $\mathscr{C}$ (consequently, $\eta=\eta_{\circ}$).
\end{rem}

\begin{pro}\label{pro:rho} If $(A,m,m_{\circ})$ is a brace in $\mathscr{C}$, then $A$ is a right $(A,m_{\circ})$-mod\-ule via $\rho\coloneqq m_{\circ}\circ(T\otimes A)\circ (\lambda\otimes m_\circ)\circ\Delta_{A^2}$. Furthermore $\rho$ is a coalgebra morphism such that $\rho\circ (\eta\ot A)=\eta\ot \epsilon$.
\end{pro}

\begin{proof} A direct computation using Proposition~\ref{pro:modulo} shows that
$$
\rho\circ (A\ot \eta)=\ide\quad\text{and} \quad\rho\circ (\eta\ot A)=\eta \ot \epsilon.
$$
Moreover, $\rho$ is a coalgebra morphism because it is a composition of coalgebra morphisms. Since
$$
m=m_{\circ} \circ (A\ot \lambda)\circ (A\ot T\ot A) \circ (\Delta\ot A),
$$
in order to prove that $\rho$ is an action it suffices to note that
$$
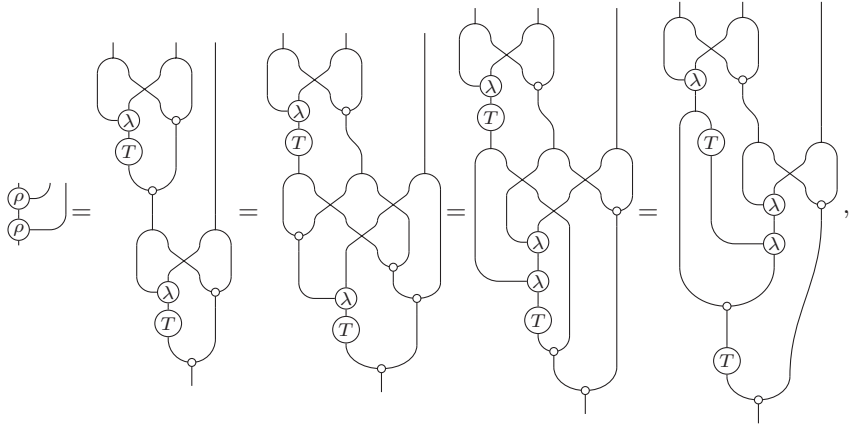
\begin{tikzpicture}[scale=0.414]
\def\mult(#1,#2)[#3,#4]{\draw (#1,#2) arc (180:360:0.5*#3 and 0.5*#4) (#1+0.5*#3,
#2-0.5*#4) -- (#1+0.5*#3,#2-#4)}
\def\counit(#1,#2){\draw (#1,#2) -- (#1,#2-0.93) (#1,#2-1) circle[radius=2pt]}
\def\comult(#1,#2)[#3,#4]{\draw (#1,#2) -- (#1,#2-0.5*#4) arc (90:0:0.5*#3 and
0.5*#4) (#1,#2-0.5*#4) arc (90:180:0.5*#3 and 0.5*#4)}
\def\laction(#1,#2)[#3,#4]{\draw (#1,#2) .. controls (#1,#2-0.555*#4/2) and
(#1+0.445*#4/2,#2-1*#4/2) .. (#1+1*#4/2,#2-1*#4/2) -- (#1+2*#4/2+#3*#4/2,#2-1*#4/2)
(#1+2*#4/2+#3*#4/2,#2)--(#1+2*#4/2+#3*#4/2,#2-2*#4/2)}
\def\map(#1,#2)[#3]{\draw (#1,#2-0.5)  node[name=nodemap,inner sep=0pt,  minimum
size=10pt, shape=circle, draw]{$#3$} (#1,#2)-- (nodemap)  (nodemap)-- (#1,#2-1)}
\def\solbraid(#1,#2)[#3]{\draw (#1,#2-0.5)  node[name=nodemap,inner sep=0pt,
minimum size=9pt, shape=circle,draw]{$#3$}
(#1-0.5,#2) .. controls (#1-0.5,#2-0.15) and (#1-0.4,#2-0.2) .. (#1-0.3,#2-0.3)
(#1-0.3,#2-0.3) -- (nodemap)
(#1+0.5,#2) .. controls (#1+0.5,#2-0.15) and (#1+0.4,#2-0.2) .. (#1+0.3,#2-0.3)
(#1+0.3,#2-0.3) -- (nodemap)
(#1+0.5,#2-1) .. controls (#1+0.5,#2-0.85) and (#1+0.4,#2-0.8) .. (#1+0.3,#2-0.7)
(#1+0.3,#2-0.7) -- (nodemap)
(#1-0.5,#2-1) .. controls (#1-0.5,#2-0.85) and (#1-0.4,#2-0.8) .. (#1-0.3,#2-0.7)
(#1-0.3,#2-0.7) -- (nodemap)
}
\def\flip(#1,#2)[#3]{\draw (
#1+1*#3,#2) .. controls (#1+1*#3,#2-0.05*#3) and (#1+0.96*#3,#2-0.15*#3)..
(#1+0.9*#3,#2-0.2*#3)
(#1+0.1*#3,#2-0.8*#3)--(#1+0.9*#3,#2-0.2*#3)
(#1,#2-1*#3) .. controls (#1,#2-0.95*#3) and (#1+0.04*#3,#2-0.85*#3)..
(#1+0.1*#3,#2-0.8*#3)
(#1,#2) .. controls (#1,#2-0.05*#3) and (#1+0.04*#3,#2-0.15*#3).. (#1+0.1*#3,#2-0.2*#3)
(#1+0.1*#3,#2-0.2*#3) -- (#1+0.9*#3,#2-0.8*#3)
(#1+1*#3,#2-1*#3) .. controls (#1+1*#3,#2-0.95*#3) and (#1+0.96*#3,#2-0.85*#3)..
(#1+0.9*#3,#2-0.8*#3)
}
\def\raction(#1,#2)[#3,#4]{\draw (#1,#2) -- (#1,#2-2*#4/2)
(#1,#2-1*#4/2)--(#1+1*#4/2+#3*#4/2,#2-1*#4/2) .. controls
(#1+1.555*#4/2+#3*#4/2,#2-1*#4/2) and (#1+2*#4/2+#3*#4/2,#2-0.555*#4/2) ..
(#1+2*#4/2+#3*#4/2,#2)}
\def\doublemap(#1,#2)[#3]{\draw (#1+0.5,#2-0.5) node [name=doublemapnode,inner
xsep=0pt, inner ysep=0pt, minimum height=11pt, minimum
width=23pt,shape=rectangle,draw,rounded corners] {$#3$} (#1,#2) .. controls
(#1,#2-0.075) .. (doublemapnode) (#1+1,#2) .. controls (#1+1,#2-0.075)..
(doublemapnode) (doublemapnode) .. controls (#1,#2-0.925)..(#1,#2-1) (doublemapnode)
.. controls (#1+1,#2-0.925).. (#1+1,#2-1)}
\def\doublesinglemap(#1,#2)[#3]{\draw (#1+0.5,#2-0.5) node
[name=doublesinglemapnode,inner xsep=0pt, inner ysep=0pt, minimum height=11pt,
minimum width=23pt,shape=rectangle,draw,rounded corners] {$#3$} (#1,#2) .. controls
(#1,#2-0.075) .. (doublesinglemapnode) (#1+1,#2) .. controls (#1+1,#2-0.075)..
(doublesinglemapnode) (doublesinglemapnode)-- (#1+0.5,#2-1)}
\def\ractiontr(#1,#2)[#3,#4,#5]{\draw (#1,#2) -- (#1,#2-2*#4/2)  (#1,#2-1*#4/2) node
[inner sep=0pt, minimum size=3pt,shape=isosceles triangle,fill, shape border
rotate=#5] {}  --(#1+1*#4/2+#3*#4/2,#2-1*#4/2) .. controls
(#1+1.555*#4/2+#3*#4/2,#2-1*#4/2) and (#1+2*#4/2+#3*#4/2,#2-0.555*#4/2) ..
(#1+2*#4/2+#3*#4/2,#2)  }
\def\rack(#1,#2)[#3]{\draw (#1,#2-0.5)  node[name=nodemap,inner sep=0pt,  minimum
size=7.5pt, shape=circle,draw]{$#3$} (#1-1,#2) .. controls (#1-1,#2-0.5) and
(#1-0.5,#2-0.5) .. (nodemap) (#1,#2)-- (nodemap)  (nodemap)-- (#1,#2-1)}
\def\rackmenoslarge(#1,#2)[#3]{\draw (#1,#2-0.5)  node[name=nodemap,inner sep=0pt,
minimum size=7.5pt, shape=circle,draw]{$#3$} (#1-1.5,#2+0.5) .. controls
(#1-1.5,#2-0.5) and (#1-0.5,#2-0.5) .. (nodemap) (#1,#2)-- (nodemap)  (nodemap)--
(#1,#2-1)}
\def\racklarge(#1,#2)[#3]{\draw (#1,#2-0.5)  node[name=nodemap,inner sep=0pt,
minimum size=7.5pt, shape=circle,draw]{$#3$} (#1-2,#2+0.5) .. controls (#1-2,#2-0.5)
and (#1-0.5,#2-0.5) .. (nodemap) (#1,#2)-- (nodemap)  (nodemap)-- (#1,#2-1)}
\def\rackmaslarge(#1,#2)[#3]{\draw (#1,#2-0.5)  node[name=nodemap,inner sep=0pt,
minimum size=7.5pt, shape=circle,draw]{$#3$} (#1-2.5,#2+0.5) .. controls
(#1-2.5,#2-0.5) and (#1-0.5,#2-0.5) .. (nodemap) (#1,#2)-- (nodemap)  (nodemap)--
(#1,#2-1)}
\def\rackextralarge(#1,#2)[#3]{\draw (#1,#2-0.75)  node[name=nodemap,inner sep=0pt,
minimum size=7.5pt, shape=circle, draw]{$#3$} (#1-3,#2+1) .. controls (#1-3,#2-0.75)
and (#1-0.5,#2-0.75) .. (nodemap) (#1,#2)-- (nodemap)  (nodemap)-- (#1,#2-1.5)}
\def\lactionnamed(#1,#2)[#3,#4][#5]{\draw (#1 + 0.5*#3 + 0.5 + 0.5*#4, #2- 0.5*#3)
node[name=nodemap,inner sep=0pt,  minimum size=8pt, shape=circle,draw]{$#5$} (#1,#2)
 arc (180:270:0.5*#3) (#1 + 0.5*#3,#2- 0.5*#3) --  (nodemap) (#1 + 0.5*#3 + 0.5 +
0.5*#4, #2) --  (nodemap) (nodemap) -- (#1 + 0.5*#3 + 0.5 + 0.5*#4, #2-#3)}
\def\ractionnamed(#1,#2)[#3,#4][#5]{\draw  (#1 - 0.5*#3- 0.5 - 0.5*#4, #2- 0.5*#3)
node[name=nodemap,inner sep=0pt,  minimum size=8pt, shape=circle,draw]{$#5$} (#1 -
0.5*#3, #2- 0.5*#3)  arc (270:360:0.5*#3) (#1 - 0.5*#3, #2- 0.5*#3) -- (nodemap)(#1
- 0.5*#3- 0.5 - 0.5*#4, #2)-- (nodemap) (nodemap) -- (#1 - 0.5*#3- 0.5 - 0.5*#4,
#2-#3)}
\def\multsubzero(#1,#2)[#3,#4]{\draw (#1+0.5*#3, #2-0.5*#4) node [name=nodemap,inner
sep=0pt, minimum size=3pt,shape=circle,fill=white, draw]{} (#1,#2) arc
(180:360:0.5*#3 and 0.5*#4) (#1+0.5*#3, #2-0.5*#4) -- (#1+0.5*#3,#2-#4)}
\begin{scope}[xshift=0cm, yshift=-5.8cm]
\ractionnamed(1,0)[1,0][\scriptstyle \rho]; \draw (1.5,0) -- (1.5,-1);
\ractionnamed(1.5,-1)[1,1][\scriptstyle \rho];
\end{scope}
\begin{scope}[xshift=2cm, yshift=-6.3cm]
\node at (0,-0.5){=};
\end{scope}
\begin{scope}[xshift=2.5cm, yshift=-1.3cm]
\comult(0.5,0)[1,1]; \comult(2.5,0)[1,1]; \draw (3.75,0) -- (3.75,-5.5); \draw
(0,-1) -- (0,-2); \flip(1,-1)[1]; \draw (3,-1) -- (3,-2);
\lactionnamed(0,-2)[1,0][\scriptstyle \lambda]; \multsubzero(2,-2)[1,1];
\map(1,-3)[\scriptstyle T]; \draw (2.5,-3) -- (2.5,-4); \multsubzero(1,-4)[1.5,1.5];
\comult(1.75,-5.5)[1,1]; \comult(3.75,-5.5)[1,1]; \draw (1.25,-6.5) -- (1.25,-7.5);
\flip(2.25,-6.5)[1]; \draw (4.25,-6.5) -- (4.25,-7.5);
\lactionnamed(1.25,-7.5)[1,0][\scriptstyle \lambda]; \multsubzero(3.25,-7.5)[1,1];
\map(2.25,-8.5)[\scriptstyle T]; \draw (3.75,-8.5) -- (3.75,-9.5);
\multsubzero(2.25,-9.5)[1.5,1.5];
\end{scope}
\begin{scope}[xshift=7.3cm, yshift=-6.3cm]
\node at (0,-0.5){=};
\end{scope}
\begin{scope}[xshift=7.9cm, yshift=-1cm]
\comult(0.5,0)[1,1]; \comult(2.5,0)[1,1]; \draw (5,0) -- (5,-4); \draw (0,-1) --
(0,-2); \flip(1,-1)[1]; \draw (3,-1) -- (3,-2);
\lactionnamed(0,-2)[1,0][\scriptstyle \lambda]; \multsubzero(2,-2)[1,1];
\map(1,-3)[\scriptstyle T]; \draw (2.5,-3) .. controls (2.5,-3.5) and (3,-3.5)..
(3,-4); \comult(1,-4)[1,1]; \comult(3,-4)[1,1]; \draw (0.5,-5) -- (0.5,-6);
\flip(1.5,-5)[1]; \draw (4.5,-6) -- (4.5,-7);
\flip(3.5,-5)[1]; \multsubzero(0.5,-6)[1,1]; \multsubzero(3.5,-7)[1,1];
\comult(5,-4)[1,1]; \flip(2.5,-6)[1]; \draw (1,-7) -- (1,-8);  \draw (2.5,-7) --
(2.5,-8); \lactionnamed(1,-8)[1,1][\scriptstyle \lambda]; \multsubzero(4,-8)[1.5,1];
\draw (5.5,-5) -- (5.5,-8); \map(2.5,-9)[\scriptstyle T];
\multsubzero(2.5,-10)[2.25,1.5]; \draw (4.75,-9) -- (4.75,-10);
\end{scope}
\begin{scope}[xshift=13.9cm, yshift=-6.3cm]
\node at (0,-0.5){=};
\end{scope}
\begin{scope}[xshift=14cm, yshift=-0.2cm]
\comult(0.5,0)[1,1]; \comult(2.5,0)[1,1]; \draw (5,0) -- (5,-4); \draw (0,-1) --
(0,-2); \flip(1,-1)[1]; \draw (3,-1) -- (3,-2);
\lactionnamed(0,-2)[1,0][\scriptstyle \lambda]; \multsubzero(2,-2)[1,1];
\comult(1,-4)[1,1]; \draw (2.5,-3) .. controls (2.5,-3.5) and (3,-3.5).. (3,-4);
\comult(3,-4)[1,1]; \comult(5,-4)[1,1]; \map(1,-3)[\scriptstyle T];
\flip(1.5,-5)[1]; \flip(3.5,-5)[1]; \flip(2.5,-6)[1]; \draw (1.5,-6) -- (1.5,-7);
\draw (0.5,-5) -- (0.5,-8);
\draw (5.5,-5) -- (5.5,-6); \multsubzero(4.5,-6)[1,1];
\lactionnamed(1.5,-7)[1,0][\scriptstyle \lambda];
\lactionnamed(0.5,-8)[1.5,1.5][\scriptstyle \lambda]; \draw (3.5,-7) -- (3.5,-10.5);
\map(2.5,-9.5)[\scriptstyle T];
\multsubzero(2.5,-10.5)[1,1];  \draw (5,-7) -- (5,-11.5); \multsubzero(3,-11.5)[2,1.5];
\end{scope}
\begin{scope}[xshift=20cm, yshift=-6.3cm]
\node at (0,-0.5){=};
\end{scope}
\begin{scope}[xshift=20.5cm, yshift=0cm]
\comult(0.5,0)[1,1]; \comult(2.5,0)[1,1]; \draw (5,0) -- (5,-4); \draw (0,-1) --
(0,-2); \flip(1,-1)[1]; \draw (3,-1) -- (3,-2);
\lactionnamed(0,-2)[1,0][\scriptstyle \lambda];
\multsubzero(2,-2)[1,1];\comult(1,-3)[1,1]; \draw (2.5,-3) .. controls (2.5,-3.5)
and (3,-3.5).. (3,-4); \comult(3,-4)[1,1];  \map(1.5,-4)[\scriptstyle T];
\comult(5,-4)[1,1];  \draw (1.5,-5) -- (1.5,-7); \draw (2.5,-5) -- (2.5,-6);
\flip(3.5,-5)[1]; \draw (5.5,-5) -- (5.5,-6);\lactionnamed(2.5,-6)[1,0][\scriptstyle
\lambda]; \multsubzero(4.5,-6)[1,1]; \draw (0.5,-4) -- (0.5,-8.5);
\lactionnamed(1.5,-7)[1.5,1.5][\scriptstyle \lambda]; \multsubzero(0.5,-8.5)[3,2.5];
\map(2,-11)[\scriptstyle T];  \draw (5,-7) .. controls (5,-9) and (4,-10).. (4,-12);
\multsubzero(2,-12)[2,1.5];
\end{scope}
\begin{scope}[xshift=26.3cm, yshift=-6.3cm]
\node at (0,-0.5){,};
\end{scope}
\end{tikzpicture}
$$
where the first equality holds by definition; the second one, since $m_{\circ}$ is a coalgebra morphism and $c$ is a natural isomorphism; the third one, since $m_{\circ}$ is associative and $\lambda$ is an action; and the fourth one, since $T$ is a coalgebra and an algebra antimorphism and $c$ is a natural isomorphism; and that
$$
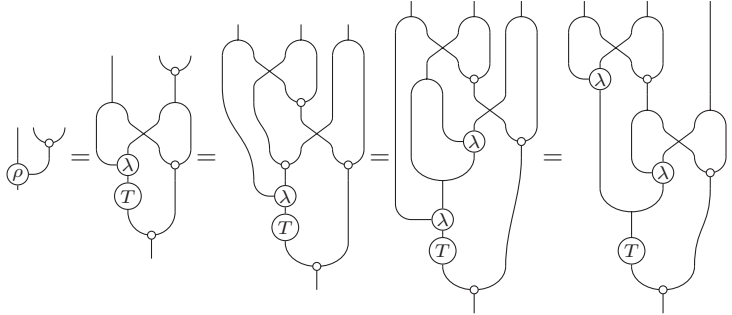
\begin{tikzpicture}[scale=0.414]
\def\mult(#1,#2)[#3,#4]{\draw (#1,#2) arc (180:360:0.5*#3 and 0.5*#4) (#1+0.5*#3, #2-0.5*#4) -- (#1+0.5*#3,#2-#4)}
\def\counit(#1,#2){\draw (#1,#2) -- (#1,#2-0.93) (#1,#2-1) circle[radius=2pt]}
\def\comult(#1,#2)[#3,#4]{\draw (#1,#2) -- (#1,#2-0.5*#4) arc (90:0:0.5*#3 and 0.5*#4) (#1,#2-0.5*#4) arc (90:180:0.5*#3 and 0.5*#4)}
\def\laction(#1,#2)[#3,#4]{\draw (#1,#2) .. controls (#1,#2-0.555*#4/2) and (#1+0.445*#4/2,#2-1*#4/2) .. (#1+1*#4/2,#2-1*#4/2) -- (#1+2*#4/2+#3*#4/2,#2-1*#4/2) (#1+2*#4/2+#3*#4/2,#2)--(#1+2*#4/2+#3*#4/2,#2-2*#4/2)}
\def\map(#1,#2)[#3]{\draw (#1,#2-0.5)  node[name=nodemap,inner sep=0pt,  minimum size=10pt, shape=circle, draw]{$#3$} (#1,#2)-- (nodemap)  (nodemap)-- (#1,#2-1)}
\def\solbraid(#1,#2)[#3]{\draw (#1,#2-0.5)  node[name=nodemap,inner sep=0pt,  minimum size=9pt, shape=circle,draw]{$#3$}
(#1-0.5,#2) .. controls (#1-0.5,#2-0.15) and (#1-0.4,#2-0.2) .. (#1-0.3,#2-0.3) (#1-0.3,#2-0.3) -- (nodemap)
(#1+0.5,#2) .. controls (#1+0.5,#2-0.15) and (#1+0.4,#2-0.2) .. (#1+0.3,#2-0.3) (#1+0.3,#2-0.3) -- (nodemap)
(#1+0.5,#2-1) .. controls (#1+0.5,#2-0.85) and (#1+0.4,#2-0.8) .. (#1+0.3,#2-0.7) (#1+0.3,#2-0.7) -- (nodemap)
(#1-0.5,#2-1) .. controls (#1-0.5,#2-0.85) and (#1-0.4,#2-0.8) .. (#1-0.3,#2-0.7) (#1-0.3,#2-0.7) -- (nodemap)
}
\def\flip(#1,#2)[#3]{\draw (
#1+1*#3,#2) .. controls (#1+1*#3,#2-0.05*#3) and (#1+0.96*#3,#2-0.15*#3).. (#1+0.9*#3,#2-0.2*#3)
(#1+0.1*#3,#2-0.8*#3)--(#1+0.9*#3,#2-0.2*#3)
(#1,#2-1*#3) .. controls (#1,#2-0.95*#3) and (#1+0.04*#3,#2-0.85*#3).. (#1+0.1*#3,#2-0.8*#3)
(#1,#2) .. controls (#1,#2-0.05*#3) and (#1+0.04*#3,#2-0.15*#3).. (#1+0.1*#3,#2-0.2*#3)
(#1+0.1*#3,#2-0.2*#3) -- (#1+0.9*#3,#2-0.8*#3)
(#1+1*#3,#2-1*#3) .. controls (#1+1*#3,#2-0.95*#3) and (#1+0.96*#3,#2-0.85*#3).. (#1+0.9*#3,#2-0.8*#3)
}
\def\raction(#1,#2)[#3,#4]{\draw (#1,#2) -- (#1,#2-2*#4/2)  (#1,#2-1*#4/2)--(#1+1*#4/2+#3*#4/2,#2-1*#4/2) .. controls (#1+1.555*#4/2+#3*#4/2,#2-1*#4/2) and (#1+2*#4/2+#3*#4/2,#2-0.555*#4/2) .. (#1+2*#4/2+#3*#4/2,#2)}
\def\doublemap(#1,#2)[#3]{\draw (#1+0.5,#2-0.5) node [name=doublemapnode,inner xsep=0pt, inner ysep=0pt, minimum height=11pt, minimum width=23pt,shape=rectangle,draw,rounded corners] {$#3$} (#1,#2) .. controls (#1,#2-0.075) .. (doublemapnode) (#1+1,#2) .. controls (#1+1,#2-0.075).. (doublemapnode) (doublemapnode) .. controls (#1,#2-0.925)..(#1,#2-1) (doublemapnode) .. controls (#1+1,#2-0.925).. (#1+1,#2-1)}
\def\doublesinglemap(#1,#2)[#3]{\draw (#1+0.5,#2-0.5) node [name=doublesinglemapnode,inner xsep=0pt, inner ysep=0pt, minimum height=11pt, minimum width=23pt,shape=rectangle,draw,rounded corners] {$#3$} (#1,#2) .. controls (#1,#2-0.075) .. (doublesinglemapnode) (#1+1,#2) .. controls (#1+1,#2-0.075).. (doublesinglemapnode) (doublesinglemapnode)-- (#1+0.5,#2-1)}
\def\ractiontr(#1,#2)[#3,#4,#5]{\draw (#1,#2) -- (#1,#2-2*#4/2)  (#1,#2-1*#4/2) node [inner sep=0pt, minimum size=3pt,shape=isosceles triangle,fill, shape border rotate=#5] {}  --(#1+1*#4/2+#3*#4/2,#2-1*#4/2) .. controls (#1+1.555*#4/2+#3*#4/2,#2-1*#4/2) and (#1+2*#4/2+#3*#4/2,#2-0.555*#4/2) .. (#1+2*#4/2+#3*#4/2,#2)  }
\def\rack(#1,#2)[#3]{\draw (#1,#2-0.5)  node[name=nodemap,inner sep=0pt,  minimum size=7.5pt, shape=circle,draw]{$#3$} (#1-1,#2) .. controls (#1-1,#2-0.5) and (#1-0.5,#2-0.5) .. (nodemap) (#1,#2)-- (nodemap)  (nodemap)-- (#1,#2-1)}
\def\rackmenoslarge(#1,#2)[#3]{\draw (#1,#2-0.5)  node[name=nodemap,inner sep=0pt,  minimum size=7.5pt, shape=circle,draw]{$#3$} (#1-1.5,#2+0.5) .. controls (#1-1.5,#2-0.5) and (#1-0.5,#2-0.5) .. (nodemap) (#1,#2)-- (nodemap)  (nodemap)-- (#1,#2-1)}
\def\racklarge(#1,#2)[#3]{\draw (#1,#2-0.5)  node[name=nodemap,inner sep=0pt,  minimum size=7.5pt, shape=circle,draw]{$#3$} (#1-2,#2+0.5) .. controls (#1-2,#2-0.5) and (#1-0.5,#2-0.5) .. (nodemap) (#1,#2)-- (nodemap)  (nodemap)-- (#1,#2-1)}
\def\rackmaslarge(#1,#2)[#3]{\draw (#1,#2-0.5)  node[name=nodemap,inner sep=0pt,  minimum size=7.5pt, shape=circle,draw]{$#3$} (#1-2.5,#2+0.5) .. controls (#1-2.5,#2-0.5) and (#1-0.5,#2-0.5) .. (nodemap) (#1,#2)-- (nodemap)  (nodemap)-- (#1,#2-1)}
\def\rackextralarge(#1,#2)[#3]{\draw (#1,#2-0.75)  node[name=nodemap,inner sep=0pt,  minimum size=7.5pt, shape=circle, draw]{$#3$} (#1-3,#2+1) .. controls (#1-3,#2-0.75) and (#1-0.5,#2-0.75) .. (nodemap) (#1,#2)-- (nodemap)  (nodemap)-- (#1,#2-1.5)}
\def\lactionnamed(#1,#2)[#3,#4][#5]{\draw (#1 + 0.5*#3 + 0.5 + 0.5*#4, #2- 0.5*#3) node[name=nodemap,inner sep=0pt,  minimum size=8pt, shape=circle,draw]{$#5$} (#1,#2)  arc (180:270:0.5*#3) (#1 + 0.5*#3,#2- 0.5*#3) --  (nodemap) (#1 + 0.5*#3 + 0.5 + 0.5*#4, #2) --  (nodemap) (nodemap) -- (#1 + 0.5*#3 + 0.5 + 0.5*#4, #2-#3)}
\def\ractionnamed(#1,#2)[#3,#4][#5]{\draw  (#1 - 0.5*#3- 0.5 - 0.5*#4, #2- 0.5*#3)  node[name=nodemap,inner sep=0pt,  minimum size=8pt, shape=circle,draw]{$#5$} (#1 - 0.5*#3, #2- 0.5*#3)  arc (270:360:0.5*#3) (#1 - 0.5*#3, #2- 0.5*#3) -- (nodemap)(#1 - 0.5*#3- 0.5 - 0.5*#4, #2)-- (nodemap) (nodemap) -- (#1 - 0.5*#3- 0.5 - 0.5*#4, #2-#3)}
\def\multsubzero(#1,#2)[#3,#4]{\draw (#1+0.5*#3, #2-0.5*#4) node [name=nodemap,inner sep=0pt, minimum size=3pt,shape=circle,fill=white, draw]{} (#1,#2) arc (180:360:0.5*#3 and 0.5*#4) (#1+0.5*#3, #2-0.5*#4) -- (#1+0.5*#3,#2-#4)}
\begin{scope}[xshift=0cm, yshift=-4.05cm]
\draw (0,0) -- (0,-1); \multsubzero(0.5,0)[1,1]; \ractionnamed(1,-1)[1,0][\scriptstyle \rho];
\end{scope}
\begin{scope}[xshift=2cm, yshift=-4.55cm]
\node at (0,-0.5){=};
\end{scope}
\begin{scope}[xshift=2.5cm, yshift=-1.75cm]
\draw (0.5,0) -- (0.5,-1); \multsubzero(2,0)[1,1]; \comult(0.5,-1)[1,1]; \comult(2.5,-1)[1,1]; \draw (0,-2) -- (0,-3); \flip(1,-2)[1]; \draw (3,-2) -- (3,-3); \lactionnamed(0,-3)[1,0][\scriptstyle \lambda]; \multsubzero(2,-3)[1,1]; \map(1,-4)[\scriptstyle T]; \draw (2.5,-4) -- (2.5,-5); \multsubzero(1,-5)[1.5,1.5];
\end{scope}
\begin{scope}[xshift=6cm, yshift=-4.55cm]
\node at (0,-0.5){=};
\end{scope}
\begin{scope}[xshift=6.5cm, yshift=-0.75cm]
\comult(0.5,0)[1,1]; \comult(2.5,0)[1,1]; \comult(4,0)[1,1]; \draw (0,-1) -- (0,-2); \draw (0,-2) .. controls (0,-3.5) and (1,-3.5) .. (1,-5); \flip(1,-1)[1]; \draw (3,-1) -- (3,-2); \draw (3.5,-1) -- (3.5,-3); \draw (4.5,-1) -- (4.5,-4); \draw (1,-2) -- (1,-2.5); \draw (1,-2.5) .. controls (1,-3.25) and (1.5,-3.25) .. (1.5,-4); \multsubzero(2,-2)[1,1]; \flip(2.5,-3)[1]; \multsubzero(1.5,-4)[1,1]; \lactionnamed(1,-5)[1,0][\scriptstyle \lambda]; \multsubzero(3.5,-4)[1,1]; \map(2,-6)[\scriptstyle T]; \multsubzero(2,-7)[2,1.5]; \draw (4,-5) -- (4,-7);
\end{scope}
\begin{scope}[xshift=11.5cm, yshift=-4.55cm]
\node at (0,-0.5){=};
\end{scope}
\begin{scope}[xshift=12cm, yshift=0cm]
\comult(0.5,0)[1,1]; \comult(2.5,0)[1,1]; \comult(4,0)[1,1]; \draw (0,-1) -- (0,-6.5);  \draw (3,-1) -- (3,-2); \draw (3.5,-1) -- (3.5,-3); \draw (4.5,-1) -- (4.5,-4);  \flip(1,-1)[1];  \multsubzero(2,-2)[1,1]; \flip(2.5,-3)[1]; \multsubzero(3.5,-4)[1,1]; \comult(1,-2)[1,1];  \draw (1.5,-3) -- (1.5,-4); \draw (0.5,-3) -- (0.5,-5); \lactionnamed(1.5,-4)[1,0][\scriptstyle \lambda]; \mult(0.5,-5)[2,1.5]; \lactionnamed(0,-6.5)[1,1][\scriptstyle \lambda]; \multsubzero(1.5,-8.5)[2,1.5]; \draw (4,-5) .. controls (4,-6.25) and (3.5,-7.25) ..  (3.5,-8.5); \map(1.5,-7.5)[\scriptstyle T];
\end{scope}
\begin{scope}[xshift=17cm, yshift=-4.55cm]
\node at (0,-0.5){=};
\end{scope}
\begin{scope}[xshift=17.5cm, yshift=0cm]
\comult(0.5,0)[1,1]; \comult(2.5,0)[1,1]; \draw (0,-1) -- (0,-2);  \draw (3,-1) -- (3,-2); \flip(1,-1)[1]; \draw (4.5,0) -- (4.5,-3); \lactionnamed(0,-2)[1,0][\scriptstyle \lambda]; \multsubzero(2,-2)[1,1]; \comult(2.5,-3)[1,1]; \comult(4.5,-3)[1,1]; \draw (2,-4) -- (2,-5);  \draw (5,-4) -- (5,-5); \flip(3,-4)[1]; \lactionnamed(2,-5)[1,0][\scriptstyle \lambda]; \multsubzero(4,-5)[1,1]; \draw (1,-3) -- (1,-6); \mult(1,-6)[2,1.5]; \map(2,-7.5)[\scriptstyle T]; \multsubzero(2,-8.5)[2,1.5]; \draw (4.5,-6) .. controls (4.5,-7.25) and (4,-7.25) ..  (4,-8.5);
\end{scope}
\begin{scope}[xshift=22.8cm, yshift=-4.55cm]
\node at (0,-0.5){,};
\end{scope}
\end{tikzpicture}
$$
where the first equality holds by definition; the second one, since $m_{\circ}$ is a coalgebra morphism, $c$ is a natural isomorphism and $m_{\circ}$ is associative; the third one, by~Remark~\ref{rem:productos}; and the last one, since $(A,m)$ is an $(A,m_{\circ},\Delta)$-module algebra, $\Delta_{A^2}$ is coassociative and $m_{\circ}$ is a coalgebra morphism.
\end{proof}

\begin{pro}\label{pro:rho'} For each brace $(A,m,m_{\circ})$ in $\mathscr{C}$ the following equalities hold:
$$
m_{\circ}\circ (\rho\otimes T)\circ (A\otimes \Delta) = m_{\circ}\circ (T\ot A)\circ (\lambda\ot A)\circ(A\ot c)\circ (\Delta\ot A)= T\circ m\circ (T\ot A).
$$
\end{pro}

\begin{proof} In fact, we have
$$
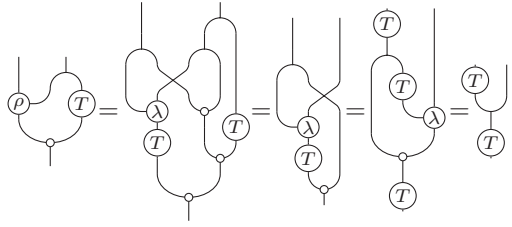
\begin{tikzpicture}[scale=0.414]
\def\unit(#1,#2){\draw (#1,#2) circle[radius=2pt] (#1,#2-0.07) -- (#1,#2-1)}
\def\mult(#1,#2)[#3,#4]{\draw (#1,#2) arc (180:360:0.5*#3 and 0.5*#4) (#1+0.5*#3, #2-0.5*#4) -- (#1+0.5*#3,#2-#4)}
\def\counit(#1,#2){\draw (#1,#2) -- (#1,#2-0.93) (#1,#2-1) circle[radius=2pt]}
\def\comult(#1,#2)[#3,#4]{\draw (#1,#2) -- (#1,#2-0.5*#4) arc (90:0:0.5*#3 and 0.5*#4) (#1,#2-0.5*#4) arc (90:180:0.5*#3 and 0.5*#4)}
\def\laction(#1,#2)[#3,#4]{\draw (#1,#2) .. controls (#1,#2-0.555*#4/2) and (#1+0.445*#4/2,#2-1*#4/2) .. (#1+1*#4/2,#2-1*#4/2) -- (#1+2*#4/2+#3*#4/2,#2-1*#4/2) (#1+2*#4/2+#3*#4/2,#2)--(#1+2*#4/2+#3*#4/2,#2-2*#4/2)}
\def\map(#1,#2)[#3]{\draw (#1,#2-0.5)  node[name=nodemap,inner sep=0pt,  minimum size=10pt, shape=circle, draw]{$#3$} (#1,#2)-- (nodemap)  (nodemap)-- (#1,#2-1)}
\def\solbraid(#1,#2)[#3]{\draw (#1,#2-0.5)  node[name=nodemap,inner sep=0pt,  minimum size=9pt, shape=circle,draw]{$#3$}
(#1-0.5,#2) .. controls (#1-0.5,#2-0.15) and (#1-0.4,#2-0.2) .. (#1-0.3,#2-0.3) (#1-0.3,#2-0.3) -- (nodemap)
(#1+0.5,#2) .. controls (#1+0.5,#2-0.15) and (#1+0.4,#2-0.2) .. (#1+0.3,#2-0.3) (#1+0.3,#2-0.3) -- (nodemap)
(#1+0.5,#2-1) .. controls (#1+0.5,#2-0.85) and (#1+0.4,#2-0.8) .. (#1+0.3,#2-0.7) (#1+0.3,#2-0.7) -- (nodemap)
(#1-0.5,#2-1) .. controls (#1-0.5,#2-0.85) and (#1-0.4,#2-0.8) .. (#1-0.3,#2-0.7) (#1-0.3,#2-0.7) -- (nodemap)
}
\def\flip(#1,#2)[#3]{\draw (
#1+1*#3,#2) .. controls (#1+1*#3,#2-0.05*#3) and (#1+0.96*#3,#2-0.15*#3).. (#1+0.9*#3,#2-0.2*#3)
(#1+0.1*#3,#2-0.8*#3)--(#1+0.9*#3,#2-0.2*#3)
(#1,#2-1*#3) .. controls (#1,#2-0.95*#3) and (#1+0.04*#3,#2-0.85*#3).. (#1+0.1*#3,#2-0.8*#3)
(#1,#2) .. controls (#1,#2-0.05*#3) and (#1+0.04*#3,#2-0.15*#3).. (#1+0.1*#3,#2-0.2*#3)
(#1+0.1*#3,#2-0.2*#3) -- (#1+0.9*#3,#2-0.8*#3)
(#1+1*#3,#2-1*#3) .. controls (#1+1*#3,#2-0.95*#3) and (#1+0.96*#3,#2-0.85*#3).. (#1+0.9*#3,#2-0.8*#3)
}
\def\raction(#1,#2)[#3,#4]{\draw (#1,#2) -- (#1,#2-2*#4/2)  (#1,#2-1*#4/2)--(#1+1*#4/2+#3*#4/2,#2-1*#4/2) .. controls (#1+1.555*#4/2+#3*#4/2,#2-1*#4/2) and (#1+2*#4/2+#3*#4/2,#2-0.555*#4/2) .. (#1+2*#4/2+#3*#4/2,#2)}
\def\doublemap(#1,#2)[#3]{\draw (#1+0.5,#2-0.5) node [name=doublemapnode,inner xsep=0pt, inner ysep=0pt, minimum height=11pt, minimum width=23pt,shape=rectangle,draw,rounded corners] {$#3$} (#1,#2) .. controls (#1,#2-0.075) .. (doublemapnode) (#1+1,#2) .. controls (#1+1,#2-0.075).. (doublemapnode) (doublemapnode) .. controls (#1,#2-0.925)..(#1,#2-1) (doublemapnode) .. controls (#1+1,#2-0.925).. (#1+1,#2-1)}
\def\doublesinglemap(#1,#2)[#3]{\draw (#1+0.5,#2-0.5) node [name=doublesinglemapnode,inner xsep=0pt, inner ysep=0pt, minimum height=11pt, minimum width=23pt,shape=rectangle,draw,rounded corners] {$#3$} (#1,#2) .. controls (#1,#2-0.075) .. (doublesinglemapnode) (#1+1,#2) .. controls (#1+1,#2-0.075).. (doublesinglemapnode) (doublesinglemapnode)-- (#1+0.5,#2-1)}
\def\ractiontr(#1,#2)[#3,#4,#5]{\draw (#1,#2) -- (#1,#2-2*#4/2)  (#1,#2-1*#4/2) node [inner sep=0pt, minimum size=3pt,shape=isosceles triangle,fill, shape border rotate=#5] {}  --(#1+1*#4/2+#3*#4/2,#2-1*#4/2) .. controls (#1+1.555*#4/2+#3*#4/2,#2-1*#4/2) and (#1+2*#4/2+#3*#4/2,#2-0.555*#4/2) .. (#1+2*#4/2+#3*#4/2,#2)  }
\def\rack(#1,#2)[#3]{\draw (#1,#2-0.5)  node[name=nodemap,inner sep=0pt,  minimum size=7.5pt, shape=circle,draw]{$#3$} (#1-1,#2) .. controls (#1-1,#2-0.5) and (#1-0.5,#2-0.5) .. (nodemap) (#1,#2)-- (nodemap)  (nodemap)-- (#1,#2-1)}
\def\rackmenoslarge(#1,#2)[#3]{\draw (#1,#2-0.5)  node[name=nodemap,inner sep=0pt,  minimum size=7.5pt, shape=circle,draw]{$#3$} (#1-1.5,#2+0.5) .. controls (#1-1.5,#2-0.5) and (#1-0.5,#2-0.5) .. (nodemap) (#1,#2)-- (nodemap)  (nodemap)-- (#1,#2-1)}
\def\racklarge(#1,#2)[#3]{\draw (#1,#2-0.5)  node[name=nodemap,inner sep=0pt,  minimum size=7.5pt, shape=circle,draw]{$#3$} (#1-2,#2+0.5) .. controls (#1-2,#2-0.5) and (#1-0.5,#2-0.5) .. (nodemap) (#1,#2)-- (nodemap)  (nodemap)-- (#1,#2-1)}
\def\rackmaslarge(#1,#2)[#3]{\draw (#1,#2-0.5)  node[name=nodemap,inner sep=0pt,  minimum size=7.5pt, shape=circle,draw]{$#3$} (#1-2.5,#2+0.5) .. controls (#1-2.5,#2-0.5) and (#1-0.5,#2-0.5) .. (nodemap) (#1,#2)-- (nodemap)  (nodemap)-- (#1,#2-1)}
\def\rackextralarge(#1,#2)[#3]{\draw (#1,#2-0.75)  node[name=nodemap,inner sep=0pt,  minimum size=7.5pt, shape=circle, draw]{$#3$} (#1-3,#2+1) .. controls (#1-3,#2-0.75) and (#1-0.5,#2-0.75) .. (nodemap) (#1,#2)-- (nodemap)  (nodemap)-- (#1,#2-1.5)}
\def\lactionnamed(#1,#2)[#3,#4][#5]{\draw (#1 + 0.5*#3 + 0.5 + 0.5*#4, #2- 0.5*#3) node[name=nodemap,inner sep=0pt,  minimum size=8pt, shape=circle,draw]{$#5$} (#1,#2)  arc (180:270:0.5*#3) (#1 + 0.5*#3,#2- 0.5*#3) --  (nodemap) (#1 + 0.5*#3 + 0.5 + 0.5*#4, #2) --  (nodemap) (nodemap) -- (#1 + 0.5*#3 + 0.5 + 0.5*#4, #2-#3)}
\def\ractionnamed(#1,#2)[#3,#4][#5]{\draw  (#1 - 0.5*#3- 0.5 - 0.5*#4, #2- 0.5*#3)  node[name=nodemap,inner sep=0pt,  minimum size=8pt, shape=circle,draw]{$#5$} (#1 - 0.5*#3, #2- 0.5*#3)  arc (270:360:0.5*#3) (#1 - 0.5*#3, #2- 0.5*#3) -- (nodemap)(#1 - 0.5*#3- 0.5 - 0.5*#4, #2)-- (nodemap) (nodemap) -- (#1 - 0.5*#3- 0.5 - 0.5*#4, #2-#3)}
\def\multsubzero(#1,#2)[#3,#4]{\draw (#1+0.5*#3, #2-0.5*#4) node [name=nodemap,inner sep=0pt, minimum size=3pt,shape=circle,fill=white, draw]{} (#1,#2) arc (180:360:0.5*#3 and 0.5*#4) (#1+0.5*#3, #2-0.5*#4) -- (#1+0.5*#3,#2-#4)}
\begin{scope}[xshift=0cm, yshift=-3.75cm]
\draw (0,0) -- (0,-1);\comult(1.5,0)[1,1]; \ractionnamed(1,-1)[1,0][\scriptstyle \rho]; \map(2,-1)[\scriptstyle T]; \multsubzero(0,-2)[2,1.5];
\end{scope}
\begin{scope}[xshift=2.85cm, yshift=-5.05cm]
\node at (0,-0.5){=};
\end{scope}
\begin{scope}[xshift=3.4cm, yshift=-2cm]
\draw (0.5,0) -- (0.5,-1);
\comult(3,0)[1,1]; \comult(0.5,-1)[1,1]; \comult(2.5,-1)[1,1]; \draw (0,-2) -- (0,-3); \flip(1,-2)[1]; \draw (3,-2) -- (3,-3); \lactionnamed(0,-3)[1,0][\scriptstyle \lambda]; \multsubzero(2,-3)[1,1];  \map(1,-4)[\scriptstyle T]; \draw (3.5,-1) -- (3.5,-3.5); \map(3.5,-3.5)[\scriptstyle T]; \multsubzero(2.5,-4.5)[1,1]; \draw (1,-5) -- (1,-5.5); \draw (2.5,-4) -- (2.5,-4.5);  \multsubzero(1,-5.5)[2,1.5];
\end{scope}
\begin{scope}[xshift=7.7cm, yshift=-5.05cm]
\node at (0,-0.5){=};
\end{scope}
\begin{scope}[xshift=8.2cm, yshift=-2.5cm]
\draw (0.5,0) -- (0.5,-1); \comult(0.5,-1)[1,1]; \draw (2,0) -- (2,-2); \flip(1,-2)[1]; \draw (0,-2) -- (0,-3); \lactionnamed(0,-3)[1,0][\scriptstyle \lambda]; \draw (2,-3) -- (2,-5); \map(1,-4)[\scriptstyle T]; \multsubzero(1,-5)[1,1];
\end{scope}
\begin{scope}[xshift=10.7cm, yshift=-5.05cm]
\node at (0,-0.5){=};
\end{scope}
\begin{scope}[xshift=11.2cm, yshift=-3.2cm]
\map(0.5,1)[\scriptstyle T];
\comult(0.5,0)[1,1]; \draw (2,1) -- (2,-2); \draw (0,-1) -- (0,-3); \lactionnamed(1,-2)[1,0][\scriptstyle \lambda];  \map(1,-1)[\scriptstyle T]; \multsubzero(0,-3)[2,1.5]; \map(1,-4.5)[\scriptstyle T];
\end{scope}
\begin{scope}[xshift=14cm, yshift=-5.05cm]
\node at (0,-0.5){=};
\end{scope}
\begin{scope}[xshift=14.5cm, yshift=-5cm]
\map(0,1)[\scriptstyle T]; \draw (1,1) -- (1,0);
\mult(0,-0)[1,1];  \map(0.5,-1)[\scriptstyle T];
\end{scope}
\begin{scope}[xshift=15.75cm, yshift=-5.05cm]
\node at (0,-0.5){,};
\end{scope}
\end{tikzpicture}
$$
where the first equality follows from the definition of $\rho$; the second one, since $\Delta$ is coassociative, $m_{\circ}$ is associative and $T$ is the antipode of $(A,m_{\circ},\Delta)$; the third one, since $c$ is a natural isomorphism and $T$ is an algebra and a coalgebra antihomomorphism; and the last one, from equality~\eqref{eq:ab}.
\end{proof}

\subsection{Invertible cocycles}

In~\cite{MR1722951,MR1809284} set-theoretical solutions of the Yang--Baxter equations were classified in terms of bijective $1$-cocycles. In
this subsection we adapt this concept to the setting of symmetric categories.

\begin{defn} Let $H$ and $A$ be cocommutative Hopf algebras in $\mathscr{C}$ an let $\lambda$ be a left action of $H$ on $A$. Assume that $A$ is a left $H$-module algebra and a left $H$-module coalgebra via $\lambda$.  An \emph{invertible $1$-cocycle} is a coalgebra isomorphism $\pi\colon H\to A$ such that
\begin{equation}\label{formula de brace}
\pi\circ m=m\circ(\pi\otimes\lambda)\circ(\Delta\otimes\pi).
\end{equation}
It is easy to see that this equality is equivalent to
\begin{equation}\label{formula equivalente a la de brace}
\lambda\circ (H\otimes \pi) = m\circ (S\otimes \pi)\circ (\pi\otimes m)\circ (\Delta\ot H).
\end{equation}

A \emph{morphism} from an invertible $1$-cocycle $\pi\colon H\to A$ to an invertible $1$-cocycle $\xi\colon K\to B$ is a pair $(f,g)$ of Hopf algebra morphism, $f\colon H\to K$ and $g\colon A\to B$, such that
\[
\xi\circ f=g\circ \pi\quad\text{and}\quad g\circ \lambda=\lambda\circ(f\otimes g),
\]
where $\lambda$ denotes the actions.
\end{defn}

Fix a cocommutative Hopf algebra $A$ in $\mathscr{C}$. We let $\Bc(A)$ denote the full sub\-cat\-e\-go\-ry of invertible $1$-cocycles in $\mathscr{C}$ whose objects are the invertible $1$-cocycles with codomain~$A$.

\subsection{Braiding operators}

In~\cite{MR1769723} Lu, Yan and Zhu introduced braiding operators to study non-degenerate set-theoretical solutions. Takeuchi noticed that braiding operators are equivalent to certain matched pairs of groups, see~\cite{MR2024436}. We now extend these ideas to our general
setting of non-degenerate solutions in symmetric categories.

\begin{defn} A \emph{braiding operator} in $\mathscr{C}$ is a pair $(A,r)$, where $A$ is a cocommutative Hopf algebra in $\mathscr{C}$ and $r\colon A\otimes A\to A\otimes A$ is a coalgebra isomorphism such that the following equalities hold:
\begin{align}
\label{eq:bo1}&m\circ r=m,\\
\label{eq:bo2}&r\circ (m\otimes A)=(A\otimes m)\circ (r\otimes A)\circ (A\otimes r),\\
\label{eq:bo3}&r\circ (A\otimes m)=(m\otimes A)\circ (A\otimes r)\circ (r\otimes A),\\
\label{eq:bo4}&r\circ (\eta\otimes A)=A\otimes\eta,\\
\label{eq:bo5}&r\circ (A\otimes\eta)=\eta\otimes A,
\end{align}
where $m$ is the multiplication map of $A$ and $\eta$ is the unit of $A$. A \emph{morphism} from a braiding operator $(A,r)$ to a braiding operator $(K,s)$ is Hopf algebra morphism $f\colon A\to K$ such that $(f\otimes f)\circ r=s\circ (f\otimes f)$.
\end{defn}

Given a braiding operator $(A,r)$ we will denote by $\lambda$ and $\rho$ the first and second coordinates of $r$ respectively.

\begin{pro}\label{lem:matched_pair} Let $A$ be a Hopf algebra, let $r\colon A^2\to A^2$ be a coalgebra isomorphism and let $\lambda$ and $\rho$ be the first and second coordinates of $r$. Then, conditions~\eqref{eq:bo2}--\eqref{eq:bo4} are satisfied if and only if $\lambda$ is a left action, $\rho$ is a right action and the following equalities hold:
\begin{enumerate}

\item $\lambda\circ (A\otimes m)=m\circ (A\otimes\lambda)\circ (r\otimes A)$.

\smallskip

\item $\rho\circ (m\otimes A)=m\circ (\rho\otimes A)\circ (A\otimes r)$.

\smallskip

\item $\lambda\circ (A\otimes\eta)=\eta\circ\epsilon$.

\smallskip

\item $\rho\circ(\eta\otimes A)=\eta\circ\epsilon$.
\end{enumerate}
\end{pro}

\begin{proof} By Corollary~\ref{coro: igualdad de morfismos}.
\end{proof}

Fix a cocommutative Hopf algebra $A$ in $\mathscr{C}$. We let $\Bo(A)$ denote the full subcategory of braiding operators $(A,r)$ with underlying Hopf algebra $A$.

\smallskip
The proof of the following theorem is based on~\cite[Theorem 1]{MR1769723}.

\begin{thm} If $(A,r)$ is a braiding operator in $\mathscr{C}$, then $r$ is a solution of the Yang--Baxter equation.	
\end{thm}

\begin{proof} Since $r_{12}\circ r_{23}\circ r_{12}$ and $r_{23}\circ r_{12}\circ r_{23}$ are coalgebra morphisms, by Proposition~\ref{coro: morfismos de coalgebra al producto tensorial de n coalgebras} we can write
\begin{align}
& r_{12}\circ r_{23}\circ r_{12} = (f_1\otimes f_2\otimes f_3)\circ
(\Delta_{A^3}\otimes A^3) \circ \Delta_{A^3}\label{ecua1}
\shortintertext{and}
& r_{23}\circ r_{12}\circ r_{23} = (g_1\otimes g_2\otimes g_3)\circ
(\Delta_{A^3}\otimes A^3) \circ \Delta_{A^3},\label{ecua2}
\end{align}	
where the $f_i$ and the $g_i$ are the coordinate maps of $r_{12}\circ r_{23}\circ r_{12}$ and $r_{23}\circ
r_{12}\circ r_{23}$, respectively. The following computation shows that $f_3=g_3$:
$$
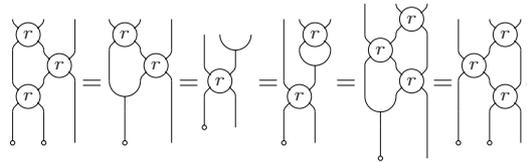
\begin{tikzpicture}[scale=0.41]
\def\mult(#1,#2)[#3,#4]{\draw (#1,#2) arc (180:360:0.5*#3 and 0.5*#4) (#1+0.5*#3, #2-0.5*#4) -- (#1+0.5*#3,#2-#4)}
\def\counit(#1,#2){\draw (#1,#2) -- (#1,#2-0.93) (#1,#2-1) circle[radius=2pt]}
\def\comult(#1,#2)[#3,#4]{\draw (#1,#2) -- (#1,#2-0.5*#4) arc (90:0:0.5*#3 and 0.5*#4) (#1,#2-0.5*#4) arc (90:180:0.5*#3 and 0.5*#4)}
\def\laction(#1,#2)[#3,#4]{\draw (#1,#2) .. controls (#1,#2-0.555*#4/2) and (#1+0.445*#4/2,#2-1*#4/2) .. (#1+1*#4/2,#2-1*#4/2) -- (#1+2*#4/2+#3*#4/2,#2-1*#4/2) (#1+2*#4/2+#3*#4/2,#2)--(#1+2*#4/2+#3*#4/2,#2-2*#4/2)}
\def\map(#1,#2)[#3]{\draw (#1,#2-0.5)  node[name=nodemap,inner sep=0pt,  minimum size=10pt, shape=circle, draw]{$#3$} (#1,#2)-- (nodemap)  (nodemap)-- (#1,#2-1)}
\def\solbraid(#1,#2)[#3]{\draw (#1,#2-0.5)  node[name=nodemap,inner sep=0pt,  minimum size=9pt, shape=circle,draw]{$#3$}
(#1-0.5,#2) .. controls (#1-0.5,#2-0.15) and (#1-0.4,#2-0.2) .. (#1-0.3,#2-0.3) (#1-0.3,#2-0.3) -- (nodemap)
(#1+0.5,#2) .. controls (#1+0.5,#2-0.15) and (#1+0.4,#2-0.2) .. (#1+0.3,#2-0.3) (#1+0.3,#2-0.3) -- (nodemap)
(#1+0.5,#2-1) .. controls (#1+0.5,#2-0.85) and (#1+0.4,#2-0.8) .. (#1+0.3,#2-0.7) (#1+0.3,#2-0.7) -- (nodemap)
(#1-0.5,#2-1) .. controls (#1-0.5,#2-0.85) and (#1-0.4,#2-0.8) .. (#1-0.3,#2-0.7) (#1-0.3,#2-0.7) -- (nodemap)
}
\def\flip(#1,#2)[#3]{\draw (
#1+1*#3,#2) .. controls (#1+1*#3,#2-0.05*#3) and (#1+0.96*#3,#2-0.15*#3).. (#1+0.9*#3,#2-0.2*#3)
(#1+0.1*#3,#2-0.8*#3)--(#1+0.9*#3,#2-0.2*#3)
(#1,#2-1*#3) .. controls (#1,#2-0.95*#3) and (#1+0.04*#3,#2-0.85*#3).. (#1+0.1*#3,#2-0.8*#3)
(#1,#2) .. controls (#1,#2-0.05*#3) and (#1+0.04*#3,#2-0.15*#3).. (#1+0.1*#3,#2-0.2*#3)
(#1+0.1*#3,#2-0.2*#3) -- (#1+0.9*#3,#2-0.8*#3)
(#1+1*#3,#2-1*#3) .. controls (#1+1*#3,#2-0.95*#3) and (#1+0.96*#3,#2-0.85*#3).. (#1+0.9*#3,#2-0.8*#3)
}
\def\raction(#1,#2)[#3,#4]{\draw (#1,#2) -- (#1,#2-2*#4/2)  (#1,#2-1*#4/2)--(#1+1*#4/2+#3*#4/2,#2-1*#4/2) .. controls (#1+1.555*#4/2+#3*#4/2,#2-1*#4/2) and (#1+2*#4/2+#3*#4/2,#2-0.555*#4/2) .. (#1+2*#4/2+#3*#4/2,#2)}
\def\doublemap(#1,#2)[#3]{\draw (#1+0.5,#2-0.5) node [name=doublemapnode,inner xsep=0pt, inner ysep=0pt, minimum height=11pt, minimum width=23pt,shape=rectangle,draw,rounded corners] {$#3$} (#1,#2) .. controls (#1,#2-0.075) .. (doublemapnode) (#1+1,#2) .. controls (#1+1,#2-0.075).. (doublemapnode) (doublemapnode) .. controls (#1,#2-0.925)..(#1,#2-1) (doublemapnode) .. controls (#1+1,#2-0.925).. (#1+1,#2-1)}
\def\doublesinglemap(#1,#2)[#3]{\draw (#1+0.5,#2-0.5) node [name=doublesinglemapnode,inner xsep=0pt, inner ysep=0pt, minimum height=11pt, minimum width=23pt,shape=rectangle,draw,rounded corners] {$#3$} (#1,#2) .. controls (#1,#2-0.075) .. (doublesinglemapnode) (#1+1,#2) .. controls (#1+1,#2-0.075).. (doublesinglemapnode) (doublesinglemapnode)-- (#1+0.5,#2-1)}
\def\ractiontr(#1,#2)[#3,#4,#5]{\draw (#1,#2) -- (#1,#2-2*#4/2)  (#1,#2-1*#4/2) node [inner sep=0pt, minimum size=3pt,shape=isosceles triangle,fill, shape border rotate=#5] {}  --(#1+1*#4/2+#3*#4/2,#2-1*#4/2) .. controls (#1+1.555*#4/2+#3*#4/2,#2-1*#4/2) and (#1+2*#4/2+#3*#4/2,#2-0.555*#4/2) .. (#1+2*#4/2+#3*#4/2,#2)  }
\def\rack(#1,#2)[#3]{\draw (#1,#2-0.5)  node[name=nodemap,inner sep=0pt,  minimum size=7.5pt, shape=circle,draw]{$#3$} (#1-1,#2) .. controls (#1-1,#2-0.5) and (#1-0.5,#2-0.5) .. (nodemap) (#1,#2)-- (nodemap)  (nodemap)-- (#1,#2-1)}
\def\rackmenoslarge(#1,#2)[#3]{\draw (#1,#2-0.5)  node[name=nodemap,inner sep=0pt,  minimum size=7.5pt, shape=circle,draw]{$#3$} (#1-1.5,#2+0.5) .. controls (#1-1.5,#2-0.5) and (#1-0.5,#2-0.5) .. (nodemap) (#1,#2)-- (nodemap)  (nodemap)-- (#1,#2-1)}
\def\racklarge(#1,#2)[#3]{\draw (#1,#2-0.5)  node[name=nodemap,inner sep=0pt,  minimum size=7.5pt, shape=circle,draw]{$#3$} (#1-2,#2+0.5) .. controls (#1-2,#2-0.5) and (#1-0.5,#2-0.5) .. (nodemap) (#1,#2)-- (nodemap)  (nodemap)-- (#1,#2-1)}
\def\rackmaslarge(#1,#2)[#3]{\draw (#1,#2-0.5)  node[name=nodemap,inner sep=0pt,  minimum size=7.5pt, shape=circle,draw]{$#3$} (#1-2.5,#2+0.5) .. controls (#1-2.5,#2-0.5) and (#1-0.5,#2-0.5) .. (nodemap) (#1,#2)-- (nodemap)  (nodemap)-- (#1,#2-1)}
\def\rackextralarge(#1,#2)[#3]{\draw (#1,#2-0.75)  node[name=nodemap,inner sep=0pt,  minimum size=7.5pt, shape=circle, draw]{$#3$} (#1-3,#2+1) .. controls (#1-3,#2-0.75) and (#1-0.5,#2-0.75) .. (nodemap) (#1,#2)-- (nodemap)  (nodemap)-- (#1,#2-1.5)}
\def\lactionnamed(#1,#2)[#3,#4][#5]{\draw (#1 + 0.5*#3 + 0.5 + 0.5*#4, #2- 0.5*#3) node[name=nodemap,inner sep=0pt,  minimum size=8pt, shape=circle,draw]{$#5$} (#1,#2)  arc (180:270:0.5*#3) (#1 + 0.5*#3,#2- 0.5*#3) --  (nodemap) (#1 + 0.5*#3 + 0.5 + 0.5*#4, #2) --  (nodemap) (nodemap) -- (#1 + 0.5*#3 + 0.5 + 0.5*#4, #2-#3)}
\def\ractionnamed(#1,#2)[#3,#4][#5]{\draw  (#1 - 0.5*#3- 0.5 - 0.5*#4, #2- 0.5*#3)  node[name=nodemap,inner sep=0pt,  minimum size=8pt, shape=circle,draw]{$#5$} (#1 - 0.5*#3, #2- 0.5*#3)  arc (270:360:0.5*#3) (#1 - 0.5*#3, #2- 0.5*#3) -- (nodemap)(#1 - 0.5*#3- 0.5 - 0.5*#4, #2)-- (nodemap) (nodemap) -- (#1 - 0.5*#3- 0.5 - 0.5*#4, #2-#3)}
\def\multsubzero(#1,#2)[#3,#4]{\draw (#1+0.5*#3, #2-0.5*#4) node [name=nodemap,inner sep=0pt, minimum size=3pt,shape=circle,fill=white, draw]{} (#1,#2) arc (180:360:0.5*#3 and 0.5*#4) (#1+0.5*#3, #2-0.5*#4) -- (#1+0.5*#3,#2-#4)}
\begin{scope}[xshift=0cm, yshift=-0.5cm]
\solbraid(0.5,0)[\scriptstyle r]; \draw (2,0) -- (2,-1); \solbraid(1.5,-1)[\scriptstyle r]; \draw (0,-1) -- (0,-2);  \solbraid(0.5,-2)[\scriptstyle r]; \counit(0,-3); \counit(1,-3); \draw (2,-2) -- (2,-4);
\end{scope}
\begin{scope}[xshift=2.55cm, yshift=-2.1cm]
\node at (0,-0.5){=};
\end{scope}
\begin{scope}[xshift=3.1cm, yshift=-0.5cm]
\solbraid(0.5,0)[\scriptstyle r]; \draw (2,0) -- (2,-1); \solbraid(1.5,-1)[\scriptstyle r]; \draw (0,-1) -- (0,-2); \mult(0,-2)[1,1]; \draw (2,-2) -- (2,-4); \counit(0.5,-3);
\end{scope}
\begin{scope}[xshift=5.65cm, yshift=-2.1cm]
\node at (0,-0.5){=};
\end{scope}
\begin{scope}[xshift=6.15cm, yshift=-1cm]
\draw (0,0) -- (0,-1); \mult(0.5,0)[1,1]; \solbraid(0.5,-1)[\scriptstyle r];  \counit(0,-2); \draw (1,-2) -- (1,-3);
\end{scope}
\begin{scope}[xshift=8.2cm, yshift=-2.1cm]
\node at (0,-0.5){=};
\end{scope}
\begin{scope}[xshift=8.7cm, yshift=-0.5cm]
\draw (0,0) -- (0,-2); \solbraid(1,0)[\scriptstyle r]; \mult(0.5,-1)[1,1]; \solbraid(0.5,-2)[\scriptstyle r]; \counit(0,-3); \draw (1,-3) -- (1,-4);
\end{scope}
\begin{scope}[xshift=10.7cm, yshift=-2.1cm]
\node at (0,-0.5){=};
\end{scope}
\begin{scope}[xshift=11.3cm, yshift=-0cm]
\draw (0,0) -- (0,-1); \solbraid(1.5,0)[\scriptstyle r]; \solbraid(0.5,-1)[\scriptstyle r]; \draw (2,-1) -- (2,-2); \solbraid(1.5,-2)[\scriptstyle r]; \draw (0,-2) -- (0,-3); \mult(0,-3)[1,1]; \draw (2,-3) -- (2,-5); \counit(0.5,-4);
\end{scope}
\begin{scope}[xshift=13.8cm, yshift=-2.1cm]
\node at (0,-0.5){=};
\end{scope}
\begin{scope}[xshift=14.3cm, yshift=-0.5cm]
\draw (0,0) -- (0,-1); \solbraid(1.5,0)[\scriptstyle r]; \solbraid(0.5,-1)[\scriptstyle r]; \draw (2,-1) -- (2,-2); \solbraid(1.5,-2)[\scriptstyle r]; \draw (0,-2) -- (0,-3);  \draw (2,-3) -- (2,-4); \counit(0,-3); \counit(1,-3);
\end{scope}
\begin{scope}[xshift=16.55cm, yshift=-2.1cm]
\node at (0,-0.5){.};
\end{scope}
\end{tikzpicture}
$$
By symmetry, $f_1=g_1$. Using Equalities~\eqref{ecua1} and~\eqref{ecua2}, the associativity of $m$ and that $m \circ r=m$, it is easy to check that
$$
f_1\star f_2\star f_3=m\circ (m\otimes A)=f_1\star g_2\star f_3,
$$
where $\star$ denotes the convolution product in~$\Hom(A^3,A)$. Since~$f_1$ and~$f_3$ are convolution invertible (because they are coalgebra morphisms), $f_2=g_2$.
\end{proof}

\begin{pro}\label{braiding operators y antipodas} For each braiding operator $(A,r)$ in $\mathscr{C}$, we have
\begin{align*}
&(A\otimes r)\circ (r\otimes A)\circ (A\otimes S\otimes \id)\circ(A\otimes\Delta) =(S\otimes A^2)\circ(\Delta\otimes A)\circ r\circ J_2^{-1}
\shortintertext{and}
&(r\otimes A)\circ (A\otimes r)\circ (A\otimes S\otimes A)\circ(\Delta\otimes A) =(A^2\otimes S)\circ(A\otimes\Delta)\circ r\circ K_2^{-1},
\end{align*}
where $J_2\coloneqq (\rho \ot A)\circ (A\ot \Delta)$  and $K_2\coloneqq (A\otimes\lambda)\circ (\Delta\otimes A)$.
\end{pro}

\begin{proof} By symmetry it suffices to prove the first equality. Let
\begin{align*}
&F\coloneqq (A\otimes r)\circ (r\otimes A)\circ (A\otimes S\otimes \id)\circ(A\otimes\Delta)\circ J_2
\shortintertext{and}
& G\coloneqq (S\otimes A^2)\circ(\Delta\otimes A)\circ r.
\end{align*}
We must prove that $F = G$. By Proposition~\ref{coro: morfismos de coalgebra al producto tensorial de n coalgebras} we know that
\begin{align}
& F=(f_1\ot f_2\ot f_3)\circ (\Delta_{A^2}\ot A)\circ
\Delta_{A^2}\label{ecua3}
\shortintertext{and}
&G=(g_1\ot g_2\ot g_3)\circ (\Delta_{A^2}\ot A)\circ
\Delta_{A^2},\label{ecua4}
\end{align}
where $f_1$, $f_2$ and $f_3$ are the coordinate maps of $F$ and $g_1$, $g_2$ and $g_3$ are the coordinate maps of $G$. A direct computation shows that $f_1 = g_1$ and $f_2 = g_2$. On the other hand, using equalities~\eqref{ecua1} and~\eqref{ecua2}, the associativity of $m$ and that $m\circ r=m$ we obtain that
$$
f_1\star f_2\star f_3 =(\epsilon\ot A)\circ r = g_1\star g_2\star g_3,
$$
where $\star$ denotes the convolution product in $\Hom(A^2,A)$. Since $f_1$ and $f_2$ are convolution invertible (because they are coalgebra morphisms),~$f_3=g_3$. Thus, $F=G$ by Corollary~\ref{coro: igualdad de morfismos}.
\end{proof}

\begin{cor}\label{braiding operators y antipodas 2} For each braiding operator $(A,r)$ in $\mathscr{C}$ it is true that
\begin{align*}
& r\circ (S\ot A)= (A\ot S)\circ c \circ R^{t_2},\\
& r\circ (A\ot S)= (S\ot A)\circ c \circ R^{t_1}
\shortintertext{and}
& r\circ (S\ot S)= (S\ot S)\circ r^{-1}.
\end{align*}
\end{cor}

\begin{proof} Applying $\epsilon \ot \epsilon\ot A$  to the second equality in Proposition~\ref{braiding operators y antipodas} and using Proposition~\ref{no degenerado} and the definition of $R^{t_2}$ we obtain that
$$
(\epsilon \ot A)\circ r\circ (S\ot A)= S\circ \rho \circ (A\ot \lambda^{-1})\circ (\Delta \ot A) = (\epsilon \ot A)\circ (A\ot S) \circ c\circ R^{t_2}.
$$
On the other hand
$$
(A\ot \epsilon)\circ r\circ (S\ot A)= \lambda\circ (S \ot A)= \lambda^{-1}= (A\ot \epsilon)\circ (A\ot S)\circ c\circ R^{t_2},
$$
where the first equality holds by the definition of $\lambda$; the second one, because $\lambda$ is an action (see Lemma~\ref{lem:matched_pair}); the third one, because $\epsilon \circ \rho = \epsilon\ot \epsilon$; and the last one, by the definition of $R^{t_2}$. So, by Corollary~\ref{coro: igualdad de morfismos} the first equality in the statement is true. The second one can be proved in a similar way. The last equality follows easily from the first two equalities and the fact that $R^{t_1}$ is the compositional inverse of $R^{t_2}$.
\end{proof}

\subsection{Braces, braiding operators and invertible cocycles}\label{Braces, etcetera}
Throughout this subsection for each braiding operator $(A,r)$ in $\mathscr{C}$ we will denote by~$m_{\circ}$ and $T$ the multiplication map and the antipode of $A$, respectively.

\begin{thm} \label{thm:equivalencias} For each cocommutative Hopf algebra $A$ in $\mathscr{C}$, the categories $\Br(A)$, $\Bc(A)$ and $\Bo(A)$ are equivalent.
\end{thm}

We shall need the following lemmas.

\begin{lem}\label{lem:bop->brace} Let $(A,r)$ be a braiding operator in $\mathscr{C}$.  Set
$$
m\coloneqq m_{\circ}\circ(A\otimes\lambda)\circ (A\otimes T\otimes A)\circ(\Delta\otimes A).
$$
The following facts hold:
\begin{enumerate}

\item $\lambda$ and $\rho$ are coalgebra morphisms.

\smallskip

\item $m_{\circ} = m\circ (A\otimes \lambda)\circ (\Delta\otimes A)$.

\smallskip

\item $\rho = m_{\circ}\circ (\lambda\otimes A)\circ (\Delta\otimes m_{\circ})\circ\Delta_{A^2}$.

\smallskip

\item $m_{\circ}\circ (\rho\otimes T)\circ (A\otimes \Delta) = T\circ m\circ (T\ot A)$.

\smallskip

\item Let $(A,r')$ be a braiding operator with underlying Hopf algebra $A$. If the first coordinate maps of $r$ and $r'$ coincide, then $r=r'$.

\end{enumerate}

\end{lem}

\begin{proof} In order to prove that $\lambda$ and $\rho$ are coalgebra morphisms it suffices to note that they are  composition of coalgebra morphisms.  The equalities
$$
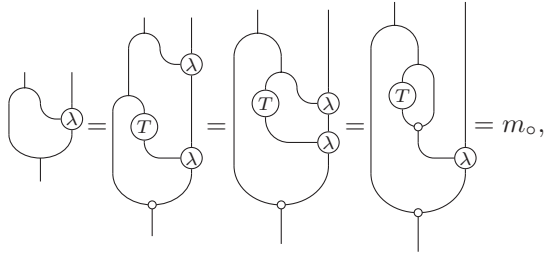
\begin{tikzpicture}[scale=0.414]
\def\mult(#1,#2)[#3,#4]{\draw (#1,#2) arc (180:360:0.5*#3 and 0.5*#4) (#1+0.5*#3, #2-0.5*#4) -- (#1+0.5*#3,#2-#4)}
\def\counit(#1,#2){\draw (#1,#2) -- (#1,#2-0.93) (#1,#2-1) circle[radius=2pt]}
\def\comult(#1,#2)[#3,#4]{\draw (#1,#2) -- (#1,#2-0.5*#4) arc (90:0:0.5*#3 and 0.5*#4) (#1,#2-0.5*#4) arc (90:180:0.5*#3 and 0.5*#4)}
\def\laction(#1,#2)[#3,#4]{\draw (#1,#2) .. controls (#1,#2-0.555*#4/2) and (#1+0.445*#4/2,#2-1*#4/2) .. (#1+1*#4/2,#2-1*#4/2) -- (#1+2*#4/2+#3*#4/2,#2-1*#4/2) (#1+2*#4/2+#3*#4/2,#2)--(#1+2*#4/2+#3*#4/2,#2-2*#4/2)}
\def\map(#1,#2)[#3]{\draw (#1,#2-0.5)  node[name=nodemap,inner sep=0pt,  minimum size=10pt, shape=circle, draw]{$#3$} (#1,#2)-- (nodemap)  (nodemap)-- (#1,#2-1)}
\def\solbraid(#1,#2)[#3]{\draw (#1,#2-0.5)  node[name=nodemap,inner sep=0pt,  minimum size=9pt, shape=circle,draw]{$#3$}
(#1-0.5,#2) .. controls (#1-0.5,#2-0.15) and (#1-0.4,#2-0.2) .. (#1-0.3,#2-0.3) (#1-0.3,#2-0.3) -- (nodemap)
(#1+0.5,#2) .. controls (#1+0.5,#2-0.15) and (#1+0.4,#2-0.2) .. (#1+0.3,#2-0.3) (#1+0.3,#2-0.3) -- (nodemap)
(#1+0.5,#2-1) .. controls (#1+0.5,#2-0.85) and (#1+0.4,#2-0.8) .. (#1+0.3,#2-0.7) (#1+0.3,#2-0.7) -- (nodemap)
(#1-0.5,#2-1) .. controls (#1-0.5,#2-0.85) and (#1-0.4,#2-0.8) .. (#1-0.3,#2-0.7) (#1-0.3,#2-0.7) -- (nodemap)
}
\def\flip(#1,#2)[#3]{\draw (
#1+1*#3,#2) .. controls (#1+1*#3,#2-0.05*#3) and (#1+0.96*#3,#2-0.15*#3).. (#1+0.9*#3,#2-0.2*#3)
(#1+0.1*#3,#2-0.8*#3)--(#1+0.9*#3,#2-0.2*#3)
(#1,#2-1*#3) .. controls (#1,#2-0.95*#3) and (#1+0.04*#3,#2-0.85*#3).. (#1+0.1*#3,#2-0.8*#3)
(#1,#2) .. controls (#1,#2-0.05*#3) and (#1+0.04*#3,#2-0.15*#3).. (#1+0.1*#3,#2-0.2*#3)
(#1+0.1*#3,#2-0.2*#3) -- (#1+0.9*#3,#2-0.8*#3)
(#1+1*#3,#2-1*#3) .. controls (#1+1*#3,#2-0.95*#3) and (#1+0.96*#3,#2-0.85*#3).. (#1+0.9*#3,#2-0.8*#3)
}
\def\raction(#1,#2)[#3,#4]{\draw (#1,#2) -- (#1,#2-2*#4/2)  (#1,#2-1*#4/2)--(#1+1*#4/2+#3*#4/2,#2-1*#4/2) .. controls (#1+1.555*#4/2+#3*#4/2,#2-1*#4/2) and (#1+2*#4/2+#3*#4/2,#2-0.555*#4/2) .. (#1+2*#4/2+#3*#4/2,#2)}
\def\doublemap(#1,#2)[#3]{\draw (#1+0.5,#2-0.5) node [name=doublemapnode,inner xsep=0pt, inner ysep=0pt, minimum height=11pt, minimum width=23pt,shape=rectangle,draw,rounded corners] {$#3$} (#1,#2) .. controls (#1,#2-0.075) .. (doublemapnode) (#1+1,#2) .. controls (#1+1,#2-0.075).. (doublemapnode) (doublemapnode) .. controls (#1,#2-0.925)..(#1,#2-1) (doublemapnode) .. controls (#1+1,#2-0.925).. (#1+1,#2-1)}
\def\doublesinglemap(#1,#2)[#3]{\draw (#1+0.5,#2-0.5) node [name=doublesinglemapnode,inner xsep=0pt, inner ysep=0pt, minimum height=11pt, minimum width=23pt,shape=rectangle,draw,rounded corners] {$#3$} (#1,#2) .. controls (#1,#2-0.075) .. (doublesinglemapnode) (#1+1,#2) .. controls (#1+1,#2-0.075).. (doublesinglemapnode) (doublesinglemapnode)-- (#1+0.5,#2-1)}
\def\ractiontr(#1,#2)[#3,#4,#5]{\draw (#1,#2) -- (#1,#2-2*#4/2)  (#1,#2-1*#4/2) node [inner sep=0pt, minimum size=3pt,shape=isosceles triangle,fill, shape border rotate=#5] {}  --(#1+1*#4/2+#3*#4/2,#2-1*#4/2) .. controls (#1+1.555*#4/2+#3*#4/2,#2-1*#4/2) and (#1+2*#4/2+#3*#4/2,#2-0.555*#4/2) .. (#1+2*#4/2+#3*#4/2,#2)  }
\def\rack(#1,#2)[#3]{\draw (#1,#2-0.5)  node[name=nodemap,inner sep=0pt,  minimum size=7.5pt, shape=circle,draw]{$#3$} (#1-1,#2) .. controls (#1-1,#2-0.5) and (#1-0.5,#2-0.5) .. (nodemap) (#1,#2)-- (nodemap)  (nodemap)-- (#1,#2-1)}
\def\rackmenoslarge(#1,#2)[#3]{\draw (#1,#2-0.5)  node[name=nodemap,inner sep=0pt,  minimum size=7.5pt, shape=circle,draw]{$#3$} (#1-1.5,#2+0.5) .. controls (#1-1.5,#2-0.5) and (#1-0.5,#2-0.5) .. (nodemap) (#1,#2)-- (nodemap)  (nodemap)-- (#1,#2-1)}
\def\racklarge(#1,#2)[#3]{\draw (#1,#2-0.5)  node[name=nodemap,inner sep=0pt,  minimum size=7.5pt, shape=circle,draw]{$#3$} (#1-2,#2+0.5) .. controls (#1-2,#2-0.5) and (#1-0.5,#2-0.5) .. (nodemap) (#1,#2)-- (nodemap)  (nodemap)-- (#1,#2-1)}
\def\rackmaslarge(#1,#2)[#3]{\draw (#1,#2-0.5)  node[name=nodemap,inner sep=0pt,  minimum size=7.5pt, shape=circle,draw]{$#3$} (#1-2.5,#2+0.5) .. controls (#1-2.5,#2-0.5) and (#1-0.5,#2-0.5) .. (nodemap) (#1,#2)-- (nodemap)  (nodemap)-- (#1,#2-1)}
\def\rackextralarge(#1,#2)[#3]{\draw (#1,#2-0.75)  node[name=nodemap,inner sep=0pt,  minimum size=7.5pt, shape=circle, draw]{$#3$} (#1-3,#2+1) .. controls (#1-3,#2-0.75) and (#1-0.5,#2-0.75) .. (nodemap) (#1,#2)-- (nodemap)  (nodemap)-- (#1,#2-1.5)}
\def\lactionnamed(#1,#2)[#3,#4][#5]{\draw (#1 + 0.5*#3 + 0.5 + 0.5*#4, #2- 0.5*#3) node[name=nodemap,inner sep=0pt,  minimum size=8pt, shape=circle,draw]{$#5$} (#1,#2)  arc (180:270:0.5*#3) (#1 + 0.5*#3,#2- 0.5*#3) --  (nodemap) (#1 + 0.5*#3 + 0.5 + 0.5*#4, #2) --  (nodemap) (nodemap) -- (#1 + 0.5*#3 + 0.5 + 0.5*#4, #2-#3)}
\def\ractionnamed(#1,#2)[#3,#4][#5]{\draw  (#1 - 0.5*#3- 0.5 - 0.5*#4, #2- 0.5*#3)  node[name=nodemap,inner sep=0pt,  minimum size=8pt, shape=circle,draw]{$#5$} (#1 - 0.5*#3, #2- 0.5*#3)  arc (270:360:0.5*#3) (#1 - 0.5*#3, #2- 0.5*#3) -- (nodemap)(#1 - 0.5*#3- 0.5 - 0.5*#4, #2)-- (nodemap) (nodemap) -- (#1 - 0.5*#3- 0.5 - 0.5*#4, #2-#3)}
\def\multsubzero(#1,#2)[#3,#4]{\draw (#1+0.5*#3, #2-0.5*#4) node [name=nodemap,inner sep=0pt, minimum size=3pt,shape=circle,fill=white, draw]{} (#1,#2) arc (180:360:0.5*#3 and 0.5*#4) (#1+0.5*#3, #2-0.5*#4) -- (#1+0.5*#3,#2-#4)}
\begin{scope}[xshift=0cm, yshift=-2.25cm]
\comult(0.5,0)[1,1]; \draw (0,-1) -- (0,-2); \draw (2,0) -- (2,-1);  \lactionnamed(1,-1)[1,0][\scriptstyle \lambda]; \mult(0,-2)[2,1.5];
\end{scope}
\begin{scope}[xshift=2.8cm, yshift=-3.55cm]
\node at (0,-0.5){=};
\end{scope}
\begin{scope}[xshift=3.3cm, yshift=-0.5cm]
\comult(1,0)[1,1]; \draw (0.5,-1) -- (0.5,-2); \draw (2.5,0) -- (2.5,-1);  \lactionnamed(1.5,-1)[1,0][\scriptstyle \lambda];\comult(0.5,-2)[1,1]; \draw (0,-3) -- (0,-5); \map(1,-3)[\scriptstyle T]; \draw (2.5,-2) -- (2.5,-4); \lactionnamed(1,-4)[1,1][\scriptstyle \lambda]; \multsubzero(0,-5)[2.5,2];
\end{scope}
\begin{scope}[xshift=6.6cm, yshift=-3.55cm]
\node at (0,-0.5){=};
\end{scope}
\begin{scope}[xshift=7.15cm, yshift=-0.25cm]
\comult(0.75,0)[1.5,1.5]; \draw (0,-1.5) -- (0,-5); \draw (3,0) -- (3,-2.5); \comult(1.5,-1.5)[1,1]; \lactionnamed(2,-2.5)[1,0][\scriptstyle \lambda]; \map(1,-2.5)[\scriptstyle T];  \lactionnamed(1,-3.5)[1.5,1.5][\scriptstyle \lambda]; \multsubzero(0,-5)[3,2.5];
\end{scope}
\begin{scope}[xshift=11cm, yshift=-3.55cm]
\node at (0,-0.5){=};
\end{scope}
\begin{scope}[xshift=11.5cm, yshift=0cm]
\comult(0.75,0)[1.5,1.5]; \draw (0,-1.5) -- (0,-5.5); \draw (3,0) -- (3,-4.5); \comult(1.5,-1.5)[1,1];  \map(1,-2.5)[\scriptstyle T]; \draw (2,-2.5) -- (2,-3.5); \multsubzero(1,-3.5)[1,1]; \lactionnamed(1.5,-4.5)[1,1][\scriptstyle \lambda]; \multsubzero(0,-5.5)[3,2.5];
\end{scope}
\begin{scope}[xshift=15.9cm, yshift=-3.55cm]
\node at (0,-0.5){$= m_{\circ}$,};
\end{scope}
\end{tikzpicture}
$$
prove that the equality in item~(2) holds. Item~(3) is true, since
$$
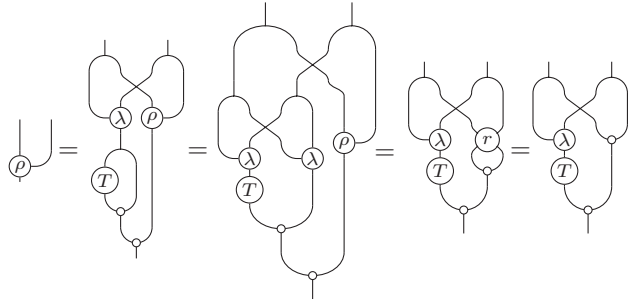
\begin{tikzpicture}[scale=0.414]
\def\unit(#1,#2){\draw (#1,#2) circle[radius=2pt] (#1,#2-0.07) -- (#1,#2-1)}
\def\mult(#1,#2)[#3,#4]{\draw (#1,#2) arc (180:360:0.5*#3 and 0.5*#4) (#1+0.5*#3, #2-0.5*#4) -- (#1+0.5*#3,#2-#4)}
\def\counit(#1,#2){\draw (#1,#2) -- (#1,#2-0.93) (#1,#2-1) circle[radius=2pt]}
\def\comult(#1,#2)[#3,#4]{\draw (#1,#2) -- (#1,#2-0.5*#4) arc (90:0:0.5*#3 and 0.5*#4) (#1,#2-0.5*#4) arc (90:180:0.5*#3 and 0.5*#4)}
\def\laction(#1,#2)[#3,#4]{\draw (#1,#2) .. controls (#1,#2-0.555*#4/2) and (#1+0.445*#4/2,#2-1*#4/2) .. (#1+1*#4/2,#2-1*#4/2) -- (#1+2*#4/2+#3*#4/2,#2-1*#4/2) (#1+2*#4/2+#3*#4/2,#2)--(#1+2*#4/2+#3*#4/2,#2-2*#4/2)}
\def\map(#1,#2)[#3]{\draw (#1,#2-0.5)  node[name=nodemap,inner sep=0pt,  minimum size=10pt, shape=circle, draw]{$#3$} (#1,#2)-- (nodemap)  (nodemap)-- (#1,#2-1)}
\def\solbraid(#1,#2)[#3]{\draw (#1,#2-0.5)  node[name=nodemap,inner sep=0pt,  minimum size=9pt, shape=circle,draw]{$#3$}
(#1-0.5,#2) .. controls (#1-0.5,#2-0.15) and (#1-0.4,#2-0.2) .. (#1-0.3,#2-0.3) (#1-0.3,#2-0.3) -- (nodemap)
(#1+0.5,#2) .. controls (#1+0.5,#2-0.15) and (#1+0.4,#2-0.2) .. (#1+0.3,#2-0.3) (#1+0.3,#2-0.3) -- (nodemap)
(#1+0.5,#2-1) .. controls (#1+0.5,#2-0.85) and (#1+0.4,#2-0.8) .. (#1+0.3,#2-0.7) (#1+0.3,#2-0.7) -- (nodemap)
(#1-0.5,#2-1) .. controls (#1-0.5,#2-0.85) and (#1-0.4,#2-0.8) .. (#1-0.3,#2-0.7) (#1-0.3,#2-0.7) -- (nodemap)
}
\def\flip(#1,#2)[#3]{\draw (
#1+1*#3,#2) .. controls (#1+1*#3,#2-0.05*#3) and (#1+0.96*#3,#2-0.15*#3).. (#1+0.9*#3,#2-0.2*#3)
(#1+0.1*#3,#2-0.8*#3)--(#1+0.9*#3,#2-0.2*#3)
(#1,#2-1*#3) .. controls (#1,#2-0.95*#3) and (#1+0.04*#3,#2-0.85*#3).. (#1+0.1*#3,#2-0.8*#3)
(#1,#2) .. controls (#1,#2-0.05*#3) and (#1+0.04*#3,#2-0.15*#3).. (#1+0.1*#3,#2-0.2*#3)
(#1+0.1*#3,#2-0.2*#3) -- (#1+0.9*#3,#2-0.8*#3)
(#1+1*#3,#2-1*#3) .. controls (#1+1*#3,#2-0.95*#3) and (#1+0.96*#3,#2-0.85*#3).. (#1+0.9*#3,#2-0.8*#3)
}
\def\raction(#1,#2)[#3,#4]{\draw (#1,#2) -- (#1,#2-2*#4/2)  (#1,#2-1*#4/2)--(#1+1*#4/2+#3*#4/2,#2-1*#4/2) .. controls (#1+1.555*#4/2+#3*#4/2,#2-1*#4/2) and (#1+2*#4/2+#3*#4/2,#2-0.555*#4/2) .. (#1+2*#4/2+#3*#4/2,#2)}
\def\doublemap(#1,#2)[#3]{\draw (#1+0.5,#2-0.5) node [name=doublemapnode,inner xsep=0pt, inner ysep=0pt, minimum height=11pt, minimum width=23pt,shape=rectangle,draw,rounded corners] {$#3$} (#1,#2) .. controls (#1,#2-0.075) .. (doublemapnode) (#1+1,#2) .. controls (#1+1,#2-0.075).. (doublemapnode) (doublemapnode) .. controls (#1,#2-0.925)..(#1,#2-1) (doublemapnode) .. controls (#1+1,#2-0.925).. (#1+1,#2-1)}
\def\doublesinglemap(#1,#2)[#3]{\draw (#1+0.5,#2-0.5) node [name=doublesinglemapnode,inner xsep=0pt, inner ysep=0pt, minimum height=11pt, minimum width=23pt,shape=rectangle,draw,rounded corners] {$#3$} (#1,#2) .. controls (#1,#2-0.075) .. (doublesinglemapnode) (#1+1,#2) .. controls (#1+1,#2-0.075).. (doublesinglemapnode) (doublesinglemapnode)-- (#1+0.5,#2-1)}
\def\ractiontr(#1,#2)[#3,#4,#5]{\draw (#1,#2) -- (#1,#2-2*#4/2)  (#1,#2-1*#4/2) node [inner sep=0pt, minimum size=3pt,shape=isosceles triangle,fill, shape border rotate=#5] {}  --(#1+1*#4/2+#3*#4/2,#2-1*#4/2) .. controls (#1+1.555*#4/2+#3*#4/2,#2-1*#4/2) and (#1+2*#4/2+#3*#4/2,#2-0.555*#4/2) .. (#1+2*#4/2+#3*#4/2,#2)  }
\def\rack(#1,#2)[#3]{\draw (#1,#2-0.5)  node[name=nodemap,inner sep=0pt,  minimum size=7.5pt, shape=circle,draw]{$#3$} (#1-1,#2) .. controls (#1-1,#2-0.5) and (#1-0.5,#2-0.5) .. (nodemap) (#1,#2)-- (nodemap)  (nodemap)-- (#1,#2-1)}
\def\rackmenoslarge(#1,#2)[#3]{\draw (#1,#2-0.5)  node[name=nodemap,inner sep=0pt,  minimum size=7.5pt, shape=circle,draw]{$#3$} (#1-1.5,#2+0.5) .. controls (#1-1.5,#2-0.5) and (#1-0.5,#2-0.5) .. (nodemap) (#1,#2)-- (nodemap)  (nodemap)-- (#1,#2-1)}
\def\racklarge(#1,#2)[#3]{\draw (#1,#2-0.5)  node[name=nodemap,inner sep=0pt,  minimum size=7.5pt, shape=circle,draw]{$#3$} (#1-2,#2+0.5) .. controls (#1-2,#2-0.5) and (#1-0.5,#2-0.5) .. (nodemap) (#1,#2)-- (nodemap)  (nodemap)-- (#1,#2-1)}
\def\rackmaslarge(#1,#2)[#3]{\draw (#1,#2-0.5)  node[name=nodemap,inner sep=0pt,  minimum size=7.5pt, shape=circle,draw]{$#3$} (#1-2.5,#2+0.5) .. controls (#1-2.5,#2-0.5) and (#1-0.5,#2-0.5) .. (nodemap) (#1,#2)-- (nodemap)  (nodemap)-- (#1,#2-1)}
\def\rackextralarge(#1,#2)[#3]{\draw (#1,#2-0.75)  node[name=nodemap,inner sep=0pt,  minimum size=7.5pt, shape=circle, draw]{$#3$} (#1-3,#2+1) .. controls (#1-3,#2-0.75) and (#1-0.5,#2-0.75) .. (nodemap) (#1,#2)-- (nodemap)  (nodemap)-- (#1,#2-1.5)}
\def\lactionnamed(#1,#2)[#3,#4][#5]{\draw (#1 + 0.5*#3 + 0.5 + 0.5*#4, #2- 0.5*#3) node[name=nodemap,inner sep=0pt,  minimum size=8pt, shape=circle,draw]{$#5$} (#1,#2)  arc (180:270:0.5*#3) (#1 + 0.5*#3,#2- 0.5*#3) --  (nodemap) (#1 + 0.5*#3 + 0.5 + 0.5*#4, #2) --  (nodemap) (nodemap) -- (#1 + 0.5*#3 + 0.5 + 0.5*#4, #2-#3)}
\def\ractionnamed(#1,#2)[#3,#4][#5]{\draw  (#1 - 0.5*#3- 0.5 - 0.5*#4, #2- 0.5*#3)  node[name=nodemap,inner sep=0pt,  minimum size=8pt, shape=circle,draw]{$#5$} (#1 - 0.5*#3, #2- 0.5*#3)  arc (270:360:0.5*#3) (#1 - 0.5*#3, #2- 0.5*#3) -- (nodemap)(#1 - 0.5*#3- 0.5 - 0.5*#4, #2)-- (nodemap) (nodemap) -- (#1 - 0.5*#3- 0.5 - 0.5*#4, #2-#3)}
\def\multsubzero(#1,#2)[#3,#4]{\draw (#1+0.5*#3, #2-0.5*#4) node [name=nodemap,inner sep=0pt, minimum size=3pt,shape=circle,fill=white, draw]{} (#1,#2) arc (180:360:0.5*#3 and 0.5*#4) (#1+0.5*#3, #2-0.5*#4) -- (#1+0.5*#3,#2-#4)}
\begin{scope}[xshift=0cm, yshift=-5.25cm]
\draw (0,0) -- (0,-1); \draw (1,0) -- (1,-1);
\ractionnamed(1,-1)[1,0][\scriptstyle \rho];
\end{scope}
\begin{scope}[xshift=1.55cm, yshift=-5.8cm]
\node at (0,-0.5){=};
\end{scope}
\begin{scope}[xshift=2.2cm, yshift=-1.7cm]
\comult(0.5,-1)[1,1]; \comult(2.5,-1)[1,1]; \draw (0,-2) -- (0,-3); \flip(1,-2)[1]; \draw (3,-2) -- (3,-3); \lactionnamed(0,-3)[1,0][\scriptstyle \lambda]; \ractionnamed(3,-3)[1,0][\scriptstyle \rho]; \draw (2,-4) -- (2,-7);
\comult(1,-4)[1,1];
\map(0.5,-5)[\scriptstyle T]; \draw (1.5,-5) -- (1.5,-6);
\multsubzero(0.5,-6)[1,1];
\multsubzero(1,-7)[1,1];
\end{scope}
\begin{scope}[xshift=5.65cm, yshift=-5.8cm]
\node at (0,-0.5){=};
\end{scope}
\begin{scope}[xshift=6.3cm, yshift=0cm]
\comult(4.25,-1.5)[1.5,1.5]; \flip(2.5,-3)[1]; \comult(1.5,-1.5)[2,1.5];
\draw (0.5,-3) -- (0.5,-4); \comult(0.5,-4)[1,1]; \comult(2.5,-4)[1,1]; \draw (0,-5) -- (0,-6); \flip(1,-5)[1]; \draw (3,-5) -- (3,-6);
\draw (3.5,-4) .. controls (3.5,-4.25) and (4,-4.25) .. (4,-4.5); \draw (4,-4.5) -- (4,-5.5);
\draw (5,-3) -- (5,-5.5); \lactionnamed(0,-6)[1,0][\scriptstyle \lambda]; \lactionnamed(2,-6)[1,0][\scriptstyle \lambda]; \ractionnamed(5,-5.5)[1,0][\scriptstyle \rho]; \map(1,-7)[\scriptstyle T]; \draw (3,-7) -- (3,-8); \multsubzero(1,-8)[2,1.5];
\draw (4,-6.5) -- (4,-9.5); \multsubzero(2,-9.5)[2,1.5];
\end{scope}
\begin{scope}[xshift=11.6cm, yshift=-5.85cm]
\node at (0,-0.5){=};
\end{scope}
\begin{scope}[xshift=12.35cm, yshift=-2.4cm]
\comult(0.5,-1)[1,1]; \comult(2.5,-1)[1,1]; \draw (0,-2) -- (0,-3); \flip(1,-2)[1]; \draw (3,-2) -- (3,-3); \lactionnamed(0,-3)[1,0][\scriptstyle \lambda]; \solbraid(2.5,-3)[\scriptstyle r]; \multsubzero(2,-4)[1,1];  \map(1,-4)[\scriptstyle T];
\multsubzero(1,-5)[1.5,1.5];
\end{scope}
\begin{scope}[xshift=15.9cm, yshift=-5.8cm]
\node at (0,-0.5){=};
\end{scope}
\begin{scope}[xshift=16.3cm, yshift=-2.4cm]
\comult(0.5,-1)[1,1]; \comult(2.5,-1)[1,1]; \draw (0,-2) -- (0,-3); \flip(1,-2)[1]; \draw (3,-2) -- (3,-3); \lactionnamed(0,-3)[1,0][\scriptstyle \lambda]; \multsubzero(2,-3)[1,1]; \draw (2.5,-4) -- (2.5,-5); \map(1,-4)[\scriptstyle T];
\multsubzero(1,-5)[1.5,1.5];
\end{scope}
\end{tikzpicture}
$$
where the first equality follows using that $\lambda$ is a coalgebra morphism and $T$ is the antipode of $(A,m_{\circ},\Delta)$; the second one, using again that $\lambda$ is a coalgebra morphism; the third one, using that $\Delta_{A^2}$ is coassociative, $m_{\circ}$ is associative and Proposition~\ref{coro: morfismos de coalgebra al producto tensorial de n coalgebras}; and the last one holds since $m\circ r=m$. Item~(4) follows now mimicking the proof of Proposition~\ref{pro:rho'}. Finally we prove item~(5). By item~(3) if the first coordinates of $r$ and $r'$ coincide, then the second coordinates of $r$ and $r'$ also coincide. Consequently, by Corollary~\ref{coro: igualdad de morfismos}, we have $r'=r$.
\end{proof}

\begin{thm}\label{lem:bop->Hopf} For each braiding operator $(A,r)$ in $\mathscr{C}$, the tuple $(A,m,\eta,\Delta,\epsilon,S)$, where $m$ is as in Lemma~\ref{lem:bop->brace},  $\eta\coloneqq\eta_{\circ}$ and $S\coloneqq \lambda\circ(A\otimes T)\circ\Delta$, is a Hopf algebra. Moreover $(A,m,m_{\circ})$ is a brace.
\end{thm}

\begin{proof} By Proposition~\ref{lem:matched_pair} we know that $\lambda$ is an action. A direct computation using this and that $(A,\Delta,m_{\circ},T)$ is a Hopf algebra shows that $m\circ(\eta\otimes A)=\id$, while from Proposition~\ref{lem:matched_pair}(3) follows easily that $m\circ(A\otimes\eta)=\id$. Moreover, since $m$ is a composition of coalgebra morphisms, we have
\begin{equation}\label{ecuacion0}
\Delta\circ m=(m\otimes m)\circ \Delta_{A^2}\qquad \text{and} \qquad\epsilon\circ m=\epsilon\otimes\epsilon.
\end{equation}
It remains to prove that $m$ is associative and~$S$ is the antipode of~$(A,\Delta,m)$. Using the definition of~$m$, Proposition~\ref{lem:matched_pair}(1), the associativity of $m_{\circ}$, Proposition~\ref{coro: morfismos de coalgebra al producto tensorial de n coalgebras}, that $T$ is a coalgebra homomorphism, the coassociativity of $\Delta$, that $\lambda$ is an action, Lemma~\ref{lem:bop->brace}(4) and the first equality in~\eqref{ecuacion0}, we obtain that
$$
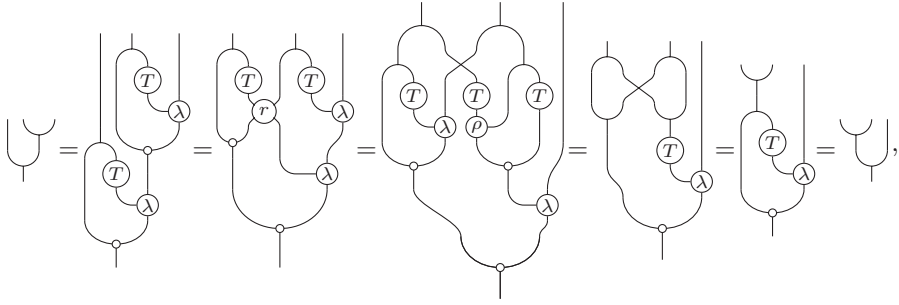
\begin{tikzpicture}[scale=0.414]
\def\unit(#1,#2){\draw (#1,#2) circle[radius=2pt] (#1,#2-0.07) -- (#1,#2-1)}
\def\mult(#1,#2)[#3,#4]{\draw (#1,#2) arc (180:360:0.5*#3 and 0.5*#4) (#1+0.5*#3, #2-0.5*#4) -- (#1+0.5*#3,#2-#4)}
\def\counit(#1,#2){\draw (#1,#2) -- (#1,#2-0.93) (#1,#2-1) circle[radius=2pt]}
\def\comult(#1,#2)[#3,#4]{\draw (#1,#2) -- (#1,#2-0.5*#4) arc (90:0:0.5*#3 and 0.5*#4) (#1,#2-0.5*#4) arc (90:180:0.5*#3 and 0.5*#4)}
\def\laction(#1,#2)[#3,#4]{\draw (#1,#2) .. controls (#1,#2-0.555*#4/2) and (#1+0.445*#4/2,#2-1*#4/2) .. (#1+1*#4/2,#2-1*#4/2) -- (#1+2*#4/2+#3*#4/2,#2-1*#4/2) (#1+2*#4/2+#3*#4/2,#2)--(#1+2*#4/2+#3*#4/2,#2-2*#4/2)}
\def\map(#1,#2)[#3]{\draw (#1,#2-0.5)  node[name=nodemap,inner sep=0pt,  minimum size=10pt, shape=circle, draw]{$#3$} (#1,#2)-- (nodemap)  (nodemap)-- (#1,#2-1)}
\def\solbraid(#1,#2)[#3]{\draw (#1,#2-0.5)  node[name=nodemap,inner sep=0pt,  minimum size=9pt, shape=circle,draw]{$#3$}
(#1-0.5,#2) .. controls (#1-0.5,#2-0.15) and (#1-0.4,#2-0.2) .. (#1-0.3,#2-0.3) (#1-0.3,#2-0.3) -- (nodemap)
(#1+0.5,#2) .. controls (#1+0.5,#2-0.15) and (#1+0.4,#2-0.2) .. (#1+0.3,#2-0.3) (#1+0.3,#2-0.3) -- (nodemap)
(#1+0.5,#2-1) .. controls (#1+0.5,#2-0.85) and (#1+0.4,#2-0.8) .. (#1+0.3,#2-0.7) (#1+0.3,#2-0.7) -- (nodemap)
(#1-0.5,#2-1) .. controls (#1-0.5,#2-0.85) and (#1-0.4,#2-0.8) .. (#1-0.3,#2-0.7) (#1-0.3,#2-0.7) -- (nodemap)
}
\def\flip(#1,#2)[#3]{\draw (
#1+1*#3,#2) .. controls (#1+1*#3,#2-0.05*#3) and (#1+0.96*#3,#2-0.15*#3).. (#1+0.9*#3,#2-0.2*#3)
(#1+0.1*#3,#2-0.8*#3)--(#1+0.9*#3,#2-0.2*#3)
(#1,#2-1*#3) .. controls (#1,#2-0.95*#3) and (#1+0.04*#3,#2-0.85*#3).. (#1+0.1*#3,#2-0.8*#3)
(#1,#2) .. controls (#1,#2-0.05*#3) and (#1+0.04*#3,#2-0.15*#3).. (#1+0.1*#3,#2-0.2*#3)
(#1+0.1*#3,#2-0.2*#3) -- (#1+0.9*#3,#2-0.8*#3)
(#1+1*#3,#2-1*#3) .. controls (#1+1*#3,#2-0.95*#3) and (#1+0.96*#3,#2-0.85*#3).. (#1+0.9*#3,#2-0.8*#3)
}
\def\raction(#1,#2)[#3,#4]{\draw (#1,#2) -- (#1,#2-2*#4/2)  (#1,#2-1*#4/2)--(#1+1*#4/2+#3*#4/2,#2-1*#4/2) .. controls (#1+1.555*#4/2+#3*#4/2,#2-1*#4/2) and (#1+2*#4/2+#3*#4/2,#2-0.555*#4/2) .. (#1+2*#4/2+#3*#4/2,#2)}
\def\doublemap(#1,#2)[#3]{\draw (#1+0.5,#2-0.5) node [name=doublemapnode,inner xsep=0pt, inner ysep=0pt, minimum height=11pt, minimum width=23pt,shape=rectangle,draw,rounded corners] {$#3$} (#1,#2) .. controls (#1,#2-0.075) .. (doublemapnode) (#1+1,#2) .. controls (#1+1,#2-0.075).. (doublemapnode) (doublemapnode) .. controls (#1,#2-0.925)..(#1,#2-1) (doublemapnode) .. controls (#1+1,#2-0.925).. (#1+1,#2-1)}
\def\doublesinglemap(#1,#2)[#3]{\draw (#1+0.5,#2-0.5) node [name=doublesinglemapnode,inner xsep=0pt, inner ysep=0pt, minimum height=11pt, minimum width=23pt,shape=rectangle,draw,rounded corners] {$#3$} (#1,#2) .. controls (#1,#2-0.075) .. (doublesinglemapnode) (#1+1,#2) .. controls (#1+1,#2-0.075).. (doublesinglemapnode) (doublesinglemapnode)-- (#1+0.5,#2-1)}
\def\ractiontr(#1,#2)[#3,#4,#5]{\draw (#1,#2) -- (#1,#2-2*#4/2)  (#1,#2-1*#4/2) node [inner sep=0pt, minimum size=3pt,shape=isosceles triangle,fill, shape border rotate=#5] {}  --(#1+1*#4/2+#3*#4/2,#2-1*#4/2) .. controls (#1+1.555*#4/2+#3*#4/2,#2-1*#4/2) and (#1+2*#4/2+#3*#4/2,#2-0.555*#4/2) .. (#1+2*#4/2+#3*#4/2,#2)  }
\def\rack(#1,#2)[#3]{\draw (#1,#2-0.5)  node[name=nodemap,inner sep=0pt,  minimum size=7.5pt, shape=circle,draw]{$#3$} (#1-1,#2) .. controls (#1-1,#2-0.5) and (#1-0.5,#2-0.5) .. (nodemap) (#1,#2)-- (nodemap)  (nodemap)-- (#1,#2-1)}
\def\rackmenoslarge(#1,#2)[#3]{\draw (#1,#2-0.5)  node[name=nodemap,inner sep=0pt,  minimum size=7.5pt, shape=circle,draw]{$#3$} (#1-1.5,#2+0.5) .. controls (#1-1.5,#2-0.5) and (#1-0.5,#2-0.5) .. (nodemap) (#1,#2)-- (nodemap)  (nodemap)-- (#1,#2-1)}
\def\racklarge(#1,#2)[#3]{\draw (#1,#2-0.5)  node[name=nodemap,inner sep=0pt,  minimum size=7.5pt, shape=circle,draw]{$#3$} (#1-2,#2+0.5) .. controls (#1-2,#2-0.5) and (#1-0.5,#2-0.5) .. (nodemap) (#1,#2)-- (nodemap)  (nodemap)-- (#1,#2-1)}
\def\rackmaslarge(#1,#2)[#3]{\draw (#1,#2-0.5)  node[name=nodemap,inner sep=0pt,  minimum size=7.5pt, shape=circle,draw]{$#3$} (#1-2.5,#2+0.5) .. controls (#1-2.5,#2-0.5) and (#1-0.5,#2-0.5) .. (nodemap) (#1,#2)-- (nodemap)  (nodemap)-- (#1,#2-1)}
\def\rackextralarge(#1,#2)[#3]{\draw (#1,#2-0.75)  node[name=nodemap,inner sep=0pt,  minimum size=7.5pt, shape=circle, draw]{$#3$} (#1-3,#2+1) .. controls (#1-3,#2-0.75) and (#1-0.5,#2-0.75) .. (nodemap) (#1,#2)-- (nodemap)  (nodemap)-- (#1,#2-1.5)}
\def\lactionnamed(#1,#2)[#3,#4][#5]{\draw (#1 + 0.5*#3 + 0.5 + 0.5*#4, #2- 0.5*#3) node[name=nodemap,inner sep=0pt,  minimum size=8pt, shape=circle,draw]{$#5$} (#1,#2)  arc (180:270:0.5*#3) (#1 + 0.5*#3,#2- 0.5*#3) --  (nodemap) (#1 + 0.5*#3 + 0.5 + 0.5*#4, #2) --  (nodemap) (nodemap) -- (#1 + 0.5*#3 + 0.5 + 0.5*#4, #2-#3)}
\def\ractionnamed(#1,#2)[#3,#4][#5]{\draw  (#1 - 0.5*#3- 0.5 - 0.5*#4, #2- 0.5*#3)  node[name=nodemap,inner sep=0pt,  minimum size=8pt, shape=circle,draw]{$#5$} (#1 - 0.5*#3, #2- 0.5*#3)  arc (270:360:0.5*#3) (#1 - 0.5*#3, #2- 0.5*#3) -- (nodemap)(#1 - 0.5*#3- 0.5 - 0.5*#4, #2)-- (nodemap) (nodemap) -- (#1 - 0.5*#3- 0.5 - 0.5*#4, #2-#3)}
\def\multsubzero(#1,#2)[#3,#4]{\draw (#1+0.5*#3, #2-0.5*#4) node [name=nodemap,inner sep=0pt, minimum size=3pt,shape=circle,fill=white, draw]{} (#1,#2) arc (180:360:0.5*#3 and 0.5*#4) (#1+0.5*#3, #2-0.5*#4) -- (#1+0.5*#3,#2-#4)}
\begin{scope}[xshift=0cm, yshift=-3.75cm]
\draw (0,0) -- (0,-1); \mult(0.5,0)[1,1]; \mult(0,-1)[1,1];
\end{scope}
\begin{scope}[xshift=1.95cm, yshift=-4.3cm]
\node at (0,-0.5){=};
\end{scope}
\begin{scope}[xshift=2.45cm, yshift=-1cm]
\draw (0.5,0) -- (0.5,-3); \draw (1,-1) -- (1,-3); \comult(1.5,0)[1,1]; \map(2,-1)[\scriptstyle T]; \draw (3,-0) -- (3,-2); \lactionnamed(2,-2)[1,0][\scriptstyle \lambda]; \multsubzero(1,-3)[2,1.5]; \comult(0.5,-3)[1,1]; \map(1,-4)[\scriptstyle T];  \lactionnamed(1,-5)[1,0][\scriptstyle \lambda]; \draw (2,-4.5) -- (2,-5); \draw (0,-4) -- (0,-6); \multsubzero(0,-6)[2,1.5];
\end{scope}
\begin{scope}[xshift=6.2cm, yshift=-4.3cm]
\node at (0,-0.5){=};
\end{scope}
\begin{scope}[xshift=6.65cm, yshift=-1cm]
\comult(0.5,0)[1,1]; \map(1,-1)[\scriptstyle T]; \comult(2.5,0)[1,1]; \map(3,-1)[\scriptstyle T]; \draw (4,-0) -- (4,-2); \lactionnamed(3,-2)[1,0][\scriptstyle \lambda]; \draw (2,-1) -- (2,-2); \solbraid(1.5,-2)[\scriptstyle r]; \draw(4,-3) .. controls (4,-3.5) and (3.5,-3.5) .. (3.5,-4); \draw (0,-1) -- (0,-3); \multsubzero(0,-3)[1,1]; \draw (2,-3) -- (2,-4); \lactionnamed(2,-4)[1,1][\scriptstyle \lambda]; \draw (0.5,-4) -- (0.5,-5); \multsubzero(0.5,-5)[3,2.5];
\end{scope}
\begin{scope}[xshift=11.4cm, yshift=-4.3cm]
\node at (0,-0.5){=};
\end{scope}
\begin{scope}[xshift=11.9cm, yshift=-0cm]
\comult(1.25,0)[1.5,1.5]; \comult(0.5,-1.5)[1,1]; \map(1,-2.5)[\scriptstyle T]; \comult(3.75,0)[1.5,1.5]; \comult(4.5,-1.5)[1,1]; \flip(2,-1.5)[1]; \map(3,-2.5)[\scriptstyle T]; \map(5,-2.5)[\scriptstyle T]; \draw (5.75,0) -- (5.75,-4.5); \draw (0,-2.5) -- (0,-4.5); \draw (2,-2.5) -- (2,-3.5); \lactionnamed(1,-3.5)[1,0][\scriptstyle \lambda]; \draw (4,-2.5) -- (4,-3.5); \ractionnamed(4,-3.5)[1,0][\scriptstyle \rho]; \draw (5,-3.5) -- (5,-4.5); \multsubzero(0,-4.5)[2,1.5]; \multsubzero(3,-4.5)[2,1.5]; \lactionnamed(4,-6)[1,0.5][\scriptstyle \lambda]; \draw(5.75,-4.5) .. controls (5.75,-5.25) and (5.25,-5.25) .. (5.25,-6); \draw(1,-6) .. controls (1,-6.5) and (2.5,-7) .. (2.5,-7.5); \mult(2.5,-7.5)[2.5,2]; \draw(5.25,-7) .. controls (5.25,-7.25) and (5,-7.25) .. (5,-7.5); \multsubzero(2.5,-7.5)[2.5,2];
\end{scope}
\begin{scope}[xshift=18.12cm, yshift=-4.3cm]
\node at (0,-0.5){=};
\end{scope}
\begin{scope}[xshift=18.55cm, yshift=-1.25cm]
\comult(0.5,0)[1,1]; \comult(2.5,0)[1,1]; \draw (0,-1) -- (0,-2); \flip(1,-1)[1]; \draw (3,-1) -- (3,-2); \mult(0,-2)[1,1]; \mult(2,-2)[1,1]; \draw (0.5,-3) -- (0.5,-4); \map(2.5,-3)[\scriptstyle T]; \draw (3.5,0) -- (3.5,-4); \lactionnamed(2.5,-4)[1,0][\scriptstyle \lambda]; \draw(0.5,-4) .. controls (0.5,-4.5) and (1,-4.5) .. (1,-5); \multsubzero(1,-5)[2.5,2];
\end{scope}
\begin{scope}[xshift=22.8cm, yshift=-4.3cm]
\node at (0,-0.5){=};
\end{scope}
\begin{scope}[xshift=23.3cm, yshift=-2cm]
\mult(0,0)[1,1]; \comult(0.5,-1)[1,1]; \draw (0,-2) -- (0,-4); \map(1,-2)[\scriptstyle T]; \draw (2,0) -- (2,-3); \lactionnamed(1,-3)[1,0][\scriptstyle \lambda];\multsubzero(0,-4)[2,1.5];
\end{scope}
\begin{scope}[xshift=26cm, yshift=-4.3cm]
\node at (0,-0.5){=};
\end{scope}
\begin{scope}[xshift=26.45cm, yshift=-3.75cm]
\draw (1.5,0) -- (1.5,-1); \mult(0,0)[1,1]; \mult(0.5,-1)[1,1];
\end{scope}
\begin{scope}[xshift=28.2cm, yshift=-4.3cm]
\node at (0,-0.5){,};
\end{scope}
\end{tikzpicture}
$$
which proves that $m$ is associative. Let us check that~$S$ is the antipode of $(A,m,\Delta)$. By the definitions of~$m$ and~$S$, the coassociativity of $\Delta$ and the facts that $\lambda$ is an action and $T$ is the antipode of $(A,m_{\circ},\Delta)$, we have
$$.
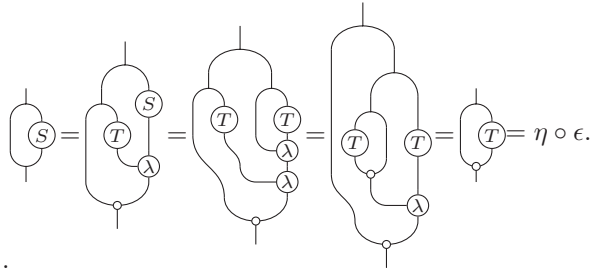
\begin{tikzpicture}[scale=0.414]
\def\mult(#1,#2)[#3,#4]{\draw (#1,#2) arc (180:360:0.5*#3 and 0.5*#4) (#1+0.5*#3, #2-0.5*#4) -- (#1+0.5*#3,#2-#4)}
\def\counit(#1,#2){\draw (#1,#2) -- (#1,#2-0.93) (#1,#2-1) circle[radius=2pt]}
\def\comult(#1,#2)[#3,#4]{\draw (#1,#2) -- (#1,#2-0.5*#4) arc (90:0:0.5*#3 and 0.5*#4) (#1,#2-0.5*#4) arc (90:180:0.5*#3 and 0.5*#4)}
\def\laction(#1,#2)[#3,#4]{\draw (#1,#2) .. controls (#1,#2-0.555*#4/2) and (#1+0.445*#4/2,#2-1*#4/2) .. (#1+1*#4/2,#2-1*#4/2) -- (#1+2*#4/2+#3*#4/2,#2-1*#4/2) (#1+2*#4/2+#3*#4/2,#2)--(#1+2*#4/2+#3*#4/2,#2-2*#4/2)}
\def\map(#1,#2)[#3]{\draw (#1,#2-0.5)  node[name=nodemap,inner sep=0pt,  minimum size=10pt, shape=circle, draw]{$#3$} (#1,#2)-- (nodemap)  (nodemap)-- (#1,#2-1)}
\def\solbraid(#1,#2)[#3]{\draw (#1,#2-0.5)  node[name=nodemap,inner sep=0pt,  minimum size=9pt, shape=circle,draw]{$#3$}
(#1-0.5,#2) .. controls (#1-0.5,#2-0.15) and (#1-0.4,#2-0.2) .. (#1-0.3,#2-0.3) (#1-0.3,#2-0.3) -- (nodemap)
(#1+0.5,#2) .. controls (#1+0.5,#2-0.15) and (#1+0.4,#2-0.2) .. (#1+0.3,#2-0.3) (#1+0.3,#2-0.3) -- (nodemap)
(#1+0.5,#2-1) .. controls (#1+0.5,#2-0.85) and (#1+0.4,#2-0.8) .. (#1+0.3,#2-0.7) (#1+0.3,#2-0.7) -- (nodemap)
(#1-0.5,#2-1) .. controls (#1-0.5,#2-0.85) and (#1-0.4,#2-0.8) .. (#1-0.3,#2-0.7) (#1-0.3,#2-0.7) -- (nodemap)
}
\def\flip(#1,#2)[#3]{\draw (
#1+1*#3,#2) .. controls (#1+1*#3,#2-0.05*#3) and (#1+0.96*#3,#2-0.15*#3).. (#1+0.9*#3,#2-0.2*#3)
(#1+0.1*#3,#2-0.8*#3)--(#1+0.9*#3,#2-0.2*#3)
(#1,#2-1*#3) .. controls (#1,#2-0.95*#3) and (#1+0.04*#3,#2-0.85*#3).. (#1+0.1*#3,#2-0.8*#3)
(#1,#2) .. controls (#1,#2-0.05*#3) and (#1+0.04*#3,#2-0.15*#3).. (#1+0.1*#3,#2-0.2*#3)
(#1+0.1*#3,#2-0.2*#3) -- (#1+0.9*#3,#2-0.8*#3)
(#1+1*#3,#2-1*#3) .. controls (#1+1*#3,#2-0.95*#3) and (#1+0.96*#3,#2-0.85*#3).. (#1+0.9*#3,#2-0.8*#3)
}
\def\raction(#1,#2)[#3,#4]{\draw (#1,#2) -- (#1,#2-2*#4/2)  (#1,#2-1*#4/2)--(#1+1*#4/2+#3*#4/2,#2-1*#4/2) .. controls (#1+1.555*#4/2+#3*#4/2,#2-1*#4/2) and (#1+2*#4/2+#3*#4/2,#2-0.555*#4/2) .. (#1+2*#4/2+#3*#4/2,#2)}
\def\doublemap(#1,#2)[#3]{\draw (#1+0.5,#2-0.5) node [name=doublemapnode,inner xsep=0pt, inner ysep=0pt, minimum height=11pt, minimum width=23pt,shape=rectangle,draw,rounded corners] {$#3$} (#1,#2) .. controls (#1,#2-0.075) .. (doublemapnode) (#1+1,#2) .. controls (#1+1,#2-0.075).. (doublemapnode) (doublemapnode) .. controls (#1,#2-0.925)..(#1,#2-1) (doublemapnode) .. controls (#1+1,#2-0.925).. (#1+1,#2-1)}
\def\doublesinglemap(#1,#2)[#3]{\draw (#1+0.5,#2-0.5) node [name=doublesinglemapnode,inner xsep=0pt, inner ysep=0pt, minimum height=11pt, minimum width=23pt,shape=rectangle,draw,rounded corners] {$#3$} (#1,#2) .. controls (#1,#2-0.075) .. (doublesinglemapnode) (#1+1,#2) .. controls (#1+1,#2-0.075).. (doublesinglemapnode) (doublesinglemapnode)-- (#1+0.5,#2-1)}
\def\ractiontr(#1,#2)[#3,#4,#5]{\draw (#1,#2) -- (#1,#2-2*#4/2)  (#1,#2-1*#4/2) node [inner sep=0pt, minimum size=3pt,shape=isosceles triangle,fill, shape border rotate=#5] {}  --(#1+1*#4/2+#3*#4/2,#2-1*#4/2) .. controls (#1+1.555*#4/2+#3*#4/2,#2-1*#4/2) and (#1+2*#4/2+#3*#4/2,#2-0.555*#4/2) .. (#1+2*#4/2+#3*#4/2,#2)  }
\def\rack(#1,#2)[#3]{\draw (#1,#2-0.5)  node[name=nodemap,inner sep=0pt,  minimum size=7.5pt, shape=circle,draw]{$#3$} (#1-1,#2) .. controls (#1-1,#2-0.5) and (#1-0.5,#2-0.5) .. (nodemap) (#1,#2)-- (nodemap)  (nodemap)-- (#1,#2-1)}
\def\rackmenoslarge(#1,#2)[#3]{\draw (#1,#2-0.5)  node[name=nodemap,inner sep=0pt,  minimum size=7.5pt, shape=circle,draw]{$#3$} (#1-1.5,#2+0.5) .. controls (#1-1.5,#2-0.5) and (#1-0.5,#2-0.5) .. (nodemap) (#1,#2)-- (nodemap)  (nodemap)-- (#1,#2-1)}
\def\racklarge(#1,#2)[#3]{\draw (#1,#2-0.5)  node[name=nodemap,inner sep=0pt,  minimum size=7.5pt, shape=circle,draw]{$#3$} (#1-2,#2+0.5) .. controls (#1-2,#2-0.5) and (#1-0.5,#2-0.5) .. (nodemap) (#1,#2)-- (nodemap)  (nodemap)-- (#1,#2-1)}
\def\rackmaslarge(#1,#2)[#3]{\draw (#1,#2-0.5)  node[name=nodemap,inner sep=0pt,  minimum size=7.5pt, shape=circle,draw]{$#3$} (#1-2.5,#2+0.5) .. controls (#1-2.5,#2-0.5) and (#1-0.5,#2-0.5) .. (nodemap) (#1,#2)-- (nodemap)  (nodemap)-- (#1,#2-1)}
\def\rackextralarge(#1,#2)[#3]{\draw (#1,#2-0.75)  node[name=nodemap,inner sep=0pt,  minimum size=7.5pt, shape=circle, draw]{$#3$} (#1-3,#2+1) .. controls (#1-3,#2-0.75) and (#1-0.5,#2-0.75) .. (nodemap) (#1,#2)-- (nodemap)  (nodemap)-- (#1,#2-1.5)}
\def\lactionnamed(#1,#2)[#3,#4][#5]{\draw (#1 + 0.5*#3 + 0.5 + 0.5*#4, #2- 0.5*#3) node[name=nodemap,inner sep=0pt,  minimum size=8pt, shape=circle,draw]{$#5$} (#1,#2)  arc (180:270:0.5*#3) (#1 + 0.5*#3,#2- 0.5*#3) --  (nodemap) (#1 + 0.5*#3 + 0.5 + 0.5*#4, #2) --  (nodemap) (nodemap) -- (#1 + 0.5*#3 + 0.5 + 0.5*#4, #2-#3)}
\def\ractionnamed(#1,#2)[#3,#4][#5]{\draw  (#1 - 0.5*#3- 0.5 - 0.5*#4, #2- 0.5*#3)  node[name=nodemap,inner sep=0pt,  minimum size=8pt, shape=circle,draw]{$#5$} (#1 - 0.5*#3, #2- 0.5*#3)  arc (270:360:0.5*#3) (#1 - 0.5*#3, #2- 0.5*#3) -- (nodemap)(#1 - 0.5*#3- 0.5 - 0.5*#4, #2)-- (nodemap) (nodemap) -- (#1 - 0.5*#3- 0.5 - 0.5*#4, #2-#3)}
\def\multsubzero(#1,#2)[#3,#4]{\draw (#1+0.5*#3, #2-0.5*#4) node [name=nodemap,inner sep=0pt, minimum size=3pt,shape=circle,fill=white, draw]{} (#1,#2) arc (180:360:0.5*#3 and 0.5*#4) (#1+0.5*#3, #2-0.5*#4) -- (#1+0.5*#3,#2-#4)}
\begin{scope}[xshift=0cm, yshift=-2.75cm]
\comult(0.5,0)[1,1]; \draw (0,-1) -- (0,-2); \map(1,-1)[\scriptstyle S];  \mult(0,-2)[1,1];
\end{scope}
\begin{scope}[xshift=1.9cm, yshift=-3.8cm]
\node at (0,-0.5){=};
\end{scope}
\begin{scope}[xshift=2.4cm, yshift=-1.25cm]
\comult(1.25,0)[1.5,1.5]; \comult(0.5,-1.5)[1,1]; \map(2,-1.5)[\scriptstyle S]; \map(1,-2.5)[\scriptstyle T]; \draw (2,-2.5) -- (2,-3.5);
\lactionnamed(1,-3.5)[1,0][\scriptstyle \lambda]; \draw (0,-2.5) -- (0,-4.5); \multsubzero(0,-4.5)[2,1.5];
\end{scope}
\begin{scope}[xshift=5.3cm, yshift=-3.8cm]
\node at (0,-0.5){=};
\end{scope}
\begin{scope}[xshift=5.8cm, yshift=-0.75cm]
\comult(1.5,0)[2,1.5]; \comult(0.5,-1.5)[1,1]; \comult(2.5,-1.5)[1,1]; \lactionnamed(2,-3.5)[1,0][\scriptstyle \lambda]; \map(3,-2.5)[\scriptstyle T]; \draw (2,-2.5) -- (2,-3.5); \map(1,-2.5)[\scriptstyle T]; \draw (1,-3.5) .. controls (1,-4) and (1.5,-4) .. (1.5,-4.5); \lactionnamed(1.5,-4.5)[1,1][\scriptstyle \lambda]; \draw (0,-2.5) -- (0,-3.5); \draw (0,-3.5) .. controls (0,-4.5) and (1,-4.5) .. (1,-5.5); \multsubzero(1,-5.5)[2,1.5];
\end{scope}
\begin{scope}[xshift=9.7cm, yshift=-3.8cm]
\node at (0,-0.5){=};
\end{scope}
\begin{scope}[xshift=10.2cm, yshift=0cm]
\comult(1,0)[2,1.5]; \comult(2,-1.5)[1.5,1.5]; \comult(1.25,-3)[1,1]; \draw (0,-1.5) -- (0,-5.5); \map(0.75,-4)[\scriptstyle T]; \draw (1.75,-4) -- (1.75,-5); \multsubzero(0.75,-5)[1,1]; \draw (2.75,-3) -- (2.75,-4); \map(2.75,-4)[\scriptstyle T]; \lactionnamed(1.25,-6)[1,1][\scriptstyle \lambda]; \draw (2.75,-5) -- (2.75,-6); \multsubzero(0.75,-7)[2,1.5]; \draw (0,-5.5) .. controls (0,-6.25) and (0.75,-6.25) .. (0.75,-7);
\end{scope}
\begin{scope}[xshift=13.8cm, yshift=-3.8cm]
\node at (0,-0.5){=};
\end{scope}
\begin{scope}[xshift=14.3cm, yshift=-2.75cm]
\comult(0.5,0)[1,1]; \draw (0,-1) -- (0,-2); \map(1,-1)[\scriptstyle T];  \multsubzero(0,-2)[1,1];
\end{scope}
\begin{scope}[xshift=17.1cm, yshift=-3.8cm]
\node at (0,-0.5){$=\eta\circ \epsilon$.};
\end{scope}
\end{tikzpicture}
$$
So, $S$ is a right inverse of $\ide_A$ respect to the convolution product constructed from $\Delta$ and $m$. In order to finish the proof it is enough to show that $S$ is right invertible respect to the same convolution product. But this is true since
$$
S\star S^2 = (\ide\star S)\circ S = \eta\circ\epsilon\circ S = \eta\circ \epsilon,
$$
where the first and last equality follow from the fact that $S$ is a coalgebra homomorphism. Finally equality~\eqref{ecuabraces} is true, because
$$
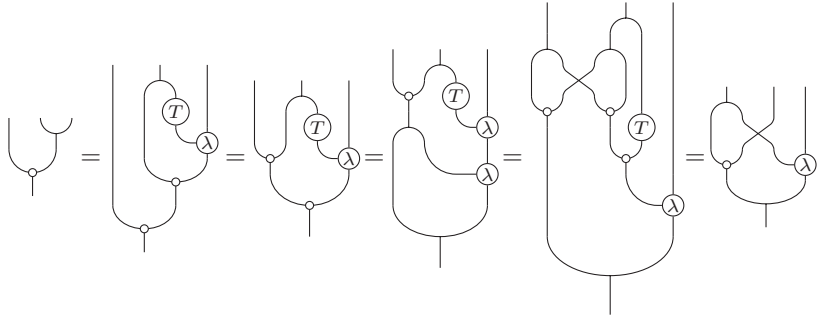
\begin{tikzpicture}[scale=0.414]
\def\mult(#1,#2)[#3,#4]{\draw (#1,#2) arc (180:360:0.5*#3 and 0.5*#4) (#1+0.5*#3, #2-0.5*#4) -- (#1+0.5*#3,#2-#4)}
\def\counit(#1,#2){\draw (#1,#2) -- (#1,#2-0.93) (#1,#2-1) circle[radius=2pt]}
\def\comult(#1,#2)[#3,#4]{\draw (#1,#2) -- (#1,#2-0.5*#4) arc (90:0:0.5*#3 and 0.5*#4) (#1,#2-0.5*#4) arc (90:180:0.5*#3 and 0.5*#4)}
\def\laction(#1,#2)[#3,#4]{\draw (#1,#2) .. controls (#1,#2-0.555*#4/2) and (#1+0.445*#4/2,#2-1*#4/2) .. (#1+1*#4/2,#2-1*#4/2) -- (#1+2*#4/2+#3*#4/2,#2-1*#4/2) (#1+2*#4/2+#3*#4/2,#2)--(#1+2*#4/2+#3*#4/2,#2-2*#4/2)}
\def\map(#1,#2)[#3]{\draw (#1,#2-0.5)  node[name=nodemap,inner sep=0pt,  minimum size=10pt, shape=circle, draw]{$#3$} (#1,#2)-- (nodemap)  (nodemap)-- (#1,#2-1)}
\def\solbraid(#1,#2)[#3]{\draw (#1,#2-0.5)  node[name=nodemap,inner sep=0pt,  minimum size=9pt, shape=circle,draw]{$#3$}
(#1-0.5,#2) .. controls (#1-0.5,#2-0.15) and (#1-0.4,#2-0.2) .. (#1-0.3,#2-0.3) (#1-0.3,#2-0.3) -- (nodemap)
(#1+0.5,#2) .. controls (#1+0.5,#2-0.15) and (#1+0.4,#2-0.2) .. (#1+0.3,#2-0.3) (#1+0.3,#2-0.3) -- (nodemap)
(#1+0.5,#2-1) .. controls (#1+0.5,#2-0.85) and (#1+0.4,#2-0.8) .. (#1+0.3,#2-0.7) (#1+0.3,#2-0.7) -- (nodemap)
(#1-0.5,#2-1) .. controls (#1-0.5,#2-0.85) and (#1-0.4,#2-0.8) .. (#1-0.3,#2-0.7) (#1-0.3,#2-0.7) -- (nodemap)
}
\def\flip(#1,#2)[#3]{\draw (
#1+1*#3,#2) .. controls (#1+1*#3,#2-0.05*#3) and (#1+0.96*#3,#2-0.15*#3).. (#1+0.9*#3,#2-0.2*#3)
(#1+0.1*#3,#2-0.8*#3)--(#1+0.9*#3,#2-0.2*#3)
(#1,#2-1*#3) .. controls (#1,#2-0.95*#3) and (#1+0.04*#3,#2-0.85*#3).. (#1+0.1*#3,#2-0.8*#3)
(#1,#2) .. controls (#1,#2-0.05*#3) and (#1+0.04*#3,#2-0.15*#3).. (#1+0.1*#3,#2-0.2*#3)
(#1+0.1*#3,#2-0.2*#3) -- (#1+0.9*#3,#2-0.8*#3)
(#1+1*#3,#2-1*#3) .. controls (#1+1*#3,#2-0.95*#3) and (#1+0.96*#3,#2-0.85*#3).. (#1+0.9*#3,#2-0.8*#3)
}
\def\raction(#1,#2)[#3,#4]{\draw (#1,#2) -- (#1,#2-2*#4/2)  (#1,#2-1*#4/2)--(#1+1*#4/2+#3*#4/2,#2-1*#4/2) .. controls (#1+1.555*#4/2+#3*#4/2,#2-1*#4/2) and (#1+2*#4/2+#3*#4/2,#2-0.555*#4/2) .. (#1+2*#4/2+#3*#4/2,#2)}
\def\doublemap(#1,#2)[#3]{\draw (#1+0.5,#2-0.5) node [name=doublemapnode,inner xsep=0pt, inner ysep=0pt, minimum height=11pt, minimum width=23pt,shape=rectangle,draw,rounded corners] {$#3$} (#1,#2) .. controls (#1,#2-0.075) .. (doublemapnode) (#1+1,#2) .. controls (#1+1,#2-0.075).. (doublemapnode) (doublemapnode) .. controls (#1,#2-0.925)..(#1,#2-1) (doublemapnode) .. controls (#1+1,#2-0.925).. (#1+1,#2-1)}
\def\doublesinglemap(#1,#2)[#3]{\draw (#1+0.5,#2-0.5) node [name=doublesinglemapnode,inner xsep=0pt, inner ysep=0pt, minimum height=11pt, minimum width=23pt,shape=rectangle,draw,rounded corners] {$#3$} (#1,#2) .. controls (#1,#2-0.075) .. (doublesinglemapnode) (#1+1,#2) .. controls (#1+1,#2-0.075).. (doublesinglemapnode) (doublesinglemapnode)-- (#1+0.5,#2-1)}
\def\ractiontr(#1,#2)[#3,#4,#5]{\draw (#1,#2) -- (#1,#2-2*#4/2)  (#1,#2-1*#4/2) node [inner sep=0pt, minimum size=3pt,shape=isosceles triangle,fill, shape border rotate=#5] {}  --(#1+1*#4/2+#3*#4/2,#2-1*#4/2) .. controls (#1+1.555*#4/2+#3*#4/2,#2-1*#4/2) and (#1+2*#4/2+#3*#4/2,#2-0.555*#4/2) .. (#1+2*#4/2+#3*#4/2,#2)  }
\def\rack(#1,#2)[#3]{\draw (#1,#2-0.5)  node[name=nodemap,inner sep=0pt,  minimum size=7.5pt, shape=circle,draw]{$#3$} (#1-1,#2) .. controls (#1-1,#2-0.5) and (#1-0.5,#2-0.5) .. (nodemap) (#1,#2)-- (nodemap)  (nodemap)-- (#1,#2-1)}
\def\rackmenoslarge(#1,#2)[#3]{\draw (#1,#2-0.5)  node[name=nodemap,inner sep=0pt,  minimum size=7.5pt, shape=circle,draw]{$#3$} (#1-1.5,#2+0.5) .. controls (#1-1.5,#2-0.5) and (#1-0.5,#2-0.5) .. (nodemap) (#1,#2)-- (nodemap)  (nodemap)-- (#1,#2-1)}
\def\racklarge(#1,#2)[#3]{\draw (#1,#2-0.5)  node[name=nodemap,inner sep=0pt,  minimum size=7.5pt, shape=circle,draw]{$#3$} (#1-2,#2+0.5) .. controls (#1-2,#2-0.5) and (#1-0.5,#2-0.5) .. (nodemap) (#1,#2)-- (nodemap)  (nodemap)-- (#1,#2-1)}
\def\rackmaslarge(#1,#2)[#3]{\draw (#1,#2-0.5)  node[name=nodemap,inner sep=0pt,  minimum size=7.5pt, shape=circle,draw]{$#3$} (#1-2.5,#2+0.5) .. controls (#1-2.5,#2-0.5) and (#1-0.5,#2-0.5) .. (nodemap) (#1,#2)-- (nodemap)  (nodemap)-- (#1,#2-1)}
\def\rackextralarge(#1,#2)[#3]{\draw (#1,#2-0.75)  node[name=nodemap,inner sep=0pt,  minimum size=7.5pt, shape=circle, draw]{$#3$} (#1-3,#2+1) .. controls (#1-3,#2-0.75) and (#1-0.5,#2-0.75) .. (nodemap) (#1,#2)-- (nodemap)  (nodemap)-- (#1,#2-1.5)}
\def\lactionnamed(#1,#2)[#3,#4][#5]{\draw (#1 + 0.5*#3 + 0.5 + 0.5*#4, #2- 0.5*#3) node[name=nodemap,inner sep=0pt,  minimum size=8pt, shape=circle,draw]{$#5$} (#1,#2)  arc (180:270:0.5*#3) (#1 + 0.5*#3,#2- 0.5*#3) --  (nodemap) (#1 + 0.5*#3 + 0.5 + 0.5*#4, #2) --  (nodemap) (nodemap) -- (#1 + 0.5*#3 + 0.5 + 0.5*#4, #2-#3)}
\def\ractionnamed(#1,#2)[#3,#4][#5]{\draw  (#1 - 0.5*#3- 0.5 - 0.5*#4, #2- 0.5*#3)  node[name=nodemap,inner sep=0pt,  minimum size=8pt, shape=circle,draw]{$#5$} (#1 - 0.5*#3, #2- 0.5*#3)  arc (270:360:0.5*#3) (#1 - 0.5*#3, #2- 0.5*#3) -- (nodemap)(#1 - 0.5*#3- 0.5 - 0.5*#4, #2)-- (nodemap) (nodemap) -- (#1 - 0.5*#3- 0.5 - 0.5*#4, #2-#3)}
\def\multsubzero(#1,#2)[#3,#4]{\draw (#1+0.5*#3, #2-0.5*#4) node [name=nodemap,inner sep=0pt, minimum size=3pt,shape=circle,fill=white, draw]{} (#1,#2) arc (180:360:0.5*#3 and 0.5*#4) (#1+0.5*#3, #2-0.5*#4) -- (#1+0.5*#3,#2-#4)}
\begin{scope}[xshift=0cm, yshift=-3.7cm]
\draw (0,0) -- (0,-1); \mult(1,0)[1,1];  \multsubzero(0,-1)[1.5,1.5];
\end{scope}
\begin{scope}[xshift=2.6cm, yshift=-4.5cm]
\node at (0,-0.5){=};
\end{scope}
\begin{scope}[xshift=3.8cm, yshift=-2cm]
\draw (-0.5,0) -- (-0.5,-4.5); \comult(1,0)[1,1]; \draw (2.5,0) -- (2.5,-2);
\draw (0.5,-1) -- (0.5,-3); \map(1.5,-1)[\scriptstyle T];
\lactionnamed(1.5,-2)[1,0][\scriptstyle \lambda];
\multsubzero(0.5,-3)[2,1.5];
\multsubzero(-0.5,-4.5)[2,1.5];
\end{scope}
\begin{scope}[xshift=7.2cm, yshift=-4.5cm]
\node at (0,-0.5){=};
\end{scope}
\begin{scope}[xshift=8.3cm, yshift=-2.5cm]
\draw (-0.5,0) -- (-0.5,-2); \comult(1,0)[1,1]; \draw (2.5,0) -- (2.5,-2);
\draw (0.5,-1) -- (0.5,-2);\map(1.5,-1)[\scriptstyle T];
\multsubzero(-0.5,-2)[1,1]; \lactionnamed(1.5,-2)[1,0][\scriptstyle \lambda];
\multsubzero(0,-3)[2.5,2];
\end{scope}
\begin{scope}[xshift=11.6cm, yshift=-4.5cm]
\node at (0,-0.5){=};
\end{scope}
\begin{scope}[xshift=12.7cm, yshift=-1.5cm]
\draw (-0.5,0) -- (-0.5,-1); \comult(1,0)[1,1]; \draw (2.5,0) -- (2.5,-2);
\multsubzero(-0.5,-1)[1,1];\map(1.5,-1)[\scriptstyle T];
\lactionnamed(1.5,-2)[1,0][\scriptstyle \lambda];
\comult(0,-2)[1,1]; \lactionnamed(0.5,-3)[2,1][\scriptstyle \lambda];
\draw (-0.5,-3) -- (-0.5,-5);
\mult(-0.5,-5)[3,2];
\end{scope}
\begin{scope}[xshift=16cm, yshift=-4.5cm]
\node at (0,-0.5){=};
\end{scope}
\begin{scope}[xshift=16.6cm, yshift=0cm]
\draw (0.5,0) -- (0.5,-1); \comult(3,0)[1,1]; \draw (4.5,0) -- (4.5,-5.5);
\comult(0.5,-1)[1,1]; \comult(2.5,-1)[1,1];
\draw (0,-2) -- (0,-3); \flip(1,-2)[1]; \draw (3,-2) -- (3,-3); \multsubzero(0,-3)[1,1]; \multsubzero(2,-3)[1,1]; \draw (0.5,-4) -- (0.5,-7.5); \draw (3.5,-1) -- (3.5,-3.5); \map(3.5,-3.5)[\scriptstyle T]; \multsubzero(2.5,-4.5)[1,1]; \lactionnamed(3,-5.5)[2,0][\scriptstyle \lambda]; \draw (2.5,-4) -- (2.5,-4.5);  \mult(0.5,-7.5)[4,2.5];
\end{scope}
\begin{scope}[xshift=21.8cm, yshift=-4.5cm]
\node at (0,-0.5){=};
\end{scope}
\begin{scope}[xshift=22.8cm, yshift=-2.7cm]
\comult(0,0)[1,1]; \draw (1.5,0) -- (1.5,-1);  \draw (2.5,0) -- (2.5,-2);
\draw (-0.5,-1) -- (-0.5,-2); \flip(0.5,-1)[1];
\multsubzero(-0.5,-2)[1,1]; \lactionnamed(1.5,-2)[1,0][\scriptstyle \lambda];
\mult(0,-3)[2.5,1.5];
\end{scope}
\end{tikzpicture}
$$
where the first equality holds by the definition of $m$; the second one, by the associativity of $m_{\circ}$; the third, by Lemma~\ref{lem:bop->brace}(2); the fourth one, since $m_{\circ}$ is a coalgebra morphism and $\lambda$ is an action; and the last one, since $\Delta$ is coassociative, $m_{\circ}$ is associative and $T$ is the antipode of $(A,m_{\circ},\Delta)$.
\end{proof}

\begin{proof}[Proof of Theorem~\ref{thm:equivalencias}] We first prove the equivalence between $\Br(A)$ and $\Bc(A)$. Let $(A,m,m_{\circ})$ be a brace in $\mathscr{C}$. By Proposition~\ref{pro:modulo}, we know that $(A,m)$ is a left $(A,m_{\circ},\Delta)$-module algebra and a left $(A,m_{\circ},\Delta)$-module coalgebra  via a same action $\lambda$. Moreover, Equality~\eqref{eq:aob} shows that $\id\colon (A,m_{\circ})\to (A,m)$ is an invertible $1$-cocycle. Furthermore, if $f$ is a morphism between braces over $A$ in $\mathscr{C}$, then $(f,f)$ is a morphism of invertible $1$-cocycles, since $f$ is a Hopf algebra morphism for both structures. Clearly the correspondence $F\colon\Br(A)\to\Bc(A)$ given by $F(A,m,m_{\circ})\coloneqq (A,m_{\circ})\xrightarrow{\id} (A,m)$ and $F(f)\coloneqq(f,f)$ is a functor. Conversely let $\pi\colon H\to A$ be an invertible $1$-cocycle. We define a new Hopf algebra structure over $A$ with multiplication
\[
m_\circ\coloneqq \pi\circ m\circ (\pi^{-1}\otimes\pi^{-1}),
\]
and the same coalgebra structure as that of $A$. The antipode is
$$
T\coloneqq \pi\circ S_H\circ\pi^{-1}.
$$
Clearly $\lambda_A\coloneqq \lambda\circ(\pi^{-1}\otimes A)$ turns $A$ into a left $(A,m_{\circ},\Delta)$-module algebra. Furthermore, using Equality~\eqref{formula equivalente a la de brace} it is easy to check that
\[
\lambda_A=m\circ (S\otimes m_\circ)\circ (\Delta\otimes A).
\]
So, by Remark~\ref{rem:brace} the tuple $(A,m,m_{\circ})$ is a brace in $\mathscr{C}$. For a morphism $(f,g)$ from $\pi\colon H\to A$ to $\xi\colon
K\to A$, the calculation
\begin{align*}
g\circ m_{\circ}&=g\circ\pi\circ m_H\circ(\pi^{-1}\otimes\pi^{-1})\\
&=\xi\circ f\circ m_H\circ (\pi^{-1}\otimes\pi^{-1})\\
&=\xi\circ m_K\circ (f\otimes f)\circ (\pi^{-1}\otimes\pi^{-1})\\
&=\xi\circ m_K\circ (\xi^{-1}\otimes \xi^{-1})\circ (g\otimes g)\\
&=m_{\circ}\circ (g\otimes g)
\end{align*}
shows that $g$ is brace morphism. Thus the correspondence $G\colon \Bc(A)\to\Br(A)$ given by $G(H\xrightarrow{\pi} A)\coloneqq (A,m,m_{\circ})$ and $G(f,g)\coloneqq g$ is a functor. Clearly $G\circ F=\id_{\Br(A)}$ and $F\circ G\simeq\id_{\Bc(A)}$.

\smallskip

Now we prove the equivalence between $\Br(A)$ and $\Bo(A)$. Let $(A,m,m_{\circ})$ be a brace in $\mathscr{C}$ and let $r\colon A\otimes A\to A\otimes A$ be given by $r\coloneqq (\lambda\otimes\rho)\circ \Delta_{A^2}$, where $\lambda$ and $\rho$ are the maps introduced in Propositions~\ref{pro:modulo} and~\ref{pro:rho}, respectively. Clearly $r$ is a coalgebra morphism. We claim that $r$ is a braiding operator of $(A,m_{\circ},\eta,\Delta,\epsilon,T)$. First note that
\begin{equation}\label{pepito}
\begin{tikzpicture}[scale=0.41, baseline=(current  bounding  box.center)]
\def\mult(#1,#2)[#3,#4]{\draw (#1,#2) arc (180:360:0.5*#3 and 0.5*#4) (#1+0.5*#3, #2-0.5*#4) -- (#1+0.5*#3,#2-#4)}
\def\counit(#1,#2){\draw (#1,#2) -- (#1,#2-0.93) (#1,#2-1) circle[radius=2pt]}
\def\comult(#1,#2)[#3,#4]{\draw (#1,#2) -- (#1,#2-0.5*#4) arc (90:0:0.5*#3 and 0.5*#4) (#1,#2-0.5*#4) arc (90:180:0.5*#3 and 0.5*#4)}
\def\laction(#1,#2)[#3,#4]{\draw (#1,#2) .. controls (#1,#2-0.555*#4/2) and (#1+0.445*#4/2,#2-1*#4/2) .. (#1+1*#4/2,#2-1*#4/2) -- (#1+2*#4/2+#3*#4/2,#2-1*#4/2) (#1+2*#4/2+#3*#4/2,#2)--(#1+2*#4/2+#3*#4/2,#2-2*#4/2)}
\def\map(#1,#2)[#3]{\draw (#1,#2-0.5)  node[name=nodemap,inner sep=0pt,  minimum size=10pt, shape=circle, draw]{$#3$} (#1,#2)-- (nodemap)  (nodemap)-- (#1,#2-1)}
\def\solbraid(#1,#2)[#3]{\draw (#1,#2-0.5)  node[name=nodemap,inner sep=0pt,  minimum size=9pt, shape=circle,draw]{$#3$}
(#1-0.5,#2) .. controls (#1-0.5,#2-0.15) and (#1-0.4,#2-0.2) .. (#1-0.3,#2-0.3) (#1-0.3,#2-0.3) -- (nodemap)
(#1+0.5,#2) .. controls (#1+0.5,#2-0.15) and (#1+0.4,#2-0.2) .. (#1+0.3,#2-0.3) (#1+0.3,#2-0.3) -- (nodemap)
(#1+0.5,#2-1) .. controls (#1+0.5,#2-0.85) and (#1+0.4,#2-0.8) .. (#1+0.3,#2-0.7) (#1+0.3,#2-0.7) -- (nodemap)
(#1-0.5,#2-1) .. controls (#1-0.5,#2-0.85) and (#1-0.4,#2-0.8) .. (#1-0.3,#2-0.7) (#1-0.3,#2-0.7) -- (nodemap)
}
\def\flip(#1,#2)[#3]{\draw (
#1+1*#3,#2) .. controls (#1+1*#3,#2-0.05*#3) and (#1+0.96*#3,#2-0.15*#3).. (#1+0.9*#3,#2-0.2*#3)
(#1+0.1*#3,#2-0.8*#3)--(#1+0.9*#3,#2-0.2*#3)
(#1,#2-1*#3) .. controls (#1,#2-0.95*#3) and (#1+0.04*#3,#2-0.85*#3).. (#1+0.1*#3,#2-0.8*#3)
(#1,#2) .. controls (#1,#2-0.05*#3) and (#1+0.04*#3,#2-0.15*#3).. (#1+0.1*#3,#2-0.2*#3)
(#1+0.1*#3,#2-0.2*#3) -- (#1+0.9*#3,#2-0.8*#3)
(#1+1*#3,#2-1*#3) .. controls (#1+1*#3,#2-0.95*#3) and (#1+0.96*#3,#2-0.85*#3).. (#1+0.9*#3,#2-0.8*#3)
}
\def\raction(#1,#2)[#3,#4]{\draw (#1,#2) -- (#1,#2-2*#4/2)  (#1,#2-1*#4/2)--(#1+1*#4/2+#3*#4/2,#2-1*#4/2) .. controls (#1+1.555*#4/2+#3*#4/2,#2-1*#4/2) and (#1+2*#4/2+#3*#4/2,#2-0.555*#4/2) .. (#1+2*#4/2+#3*#4/2,#2)}
\def\doublemap(#1,#2)[#3]{\draw (#1+0.5,#2-0.5) node [name=doublemapnode,inner xsep=0pt, inner ysep=0pt, minimum height=11pt, minimum width=23pt,shape=rectangle,draw,rounded corners] {$#3$} (#1,#2) .. controls (#1,#2-0.075) .. (doublemapnode) (#1+1,#2) .. controls (#1+1,#2-0.075).. (doublemapnode) (doublemapnode) .. controls (#1,#2-0.925)..(#1,#2-1) (doublemapnode) .. controls (#1+1,#2-0.925).. (#1+1,#2-1)}
\def\doublesinglemap(#1,#2)[#3]{\draw (#1+0.5,#2-0.5) node [name=doublesinglemapnode,inner xsep=0pt, inner ysep=0pt, minimum height=11pt, minimum width=23pt,shape=rectangle,draw,rounded corners] {$#3$} (#1,#2) .. controls (#1,#2-0.075) .. (doublesinglemapnode) (#1+1,#2) .. controls (#1+1,#2-0.075).. (doublesinglemapnode) (doublesinglemapnode)-- (#1+0.5,#2-1)}
\def\ractiontr(#1,#2)[#3,#4,#5]{\draw (#1,#2) -- (#1,#2-2*#4/2)  (#1,#2-1*#4/2) node [inner sep=0pt, minimum size=3pt,shape=isosceles triangle,fill, shape border rotate=#5] {}  --(#1+1*#4/2+#3*#4/2,#2-1*#4/2) .. controls (#1+1.555*#4/2+#3*#4/2,#2-1*#4/2) and (#1+2*#4/2+#3*#4/2,#2-0.555*#4/2) .. (#1+2*#4/2+#3*#4/2,#2)  }
\def\rack(#1,#2)[#3]{\draw (#1,#2-0.5)  node[name=nodemap,inner sep=0pt,  minimum size=7.5pt, shape=circle,draw]{$#3$} (#1-1,#2) .. controls (#1-1,#2-0.5) and (#1-0.5,#2-0.5) .. (nodemap) (#1,#2)-- (nodemap)  (nodemap)-- (#1,#2-1)}
\def\rackmenoslarge(#1,#2)[#3]{\draw (#1,#2-0.5)  node[name=nodemap,inner sep=0pt,  minimum size=7.5pt, shape=circle,draw]{$#3$} (#1-1.5,#2+0.5) .. controls (#1-1.5,#2-0.5) and (#1-0.5,#2-0.5) .. (nodemap) (#1,#2)-- (nodemap)  (nodemap)-- (#1,#2-1)}
\def\racklarge(#1,#2)[#3]{\draw (#1,#2-0.5)  node[name=nodemap,inner sep=0pt,  minimum size=7.5pt, shape=circle,draw]{$#3$} (#1-2,#2+0.5) .. controls (#1-2,#2-0.5) and (#1-0.5,#2-0.5) .. (nodemap) (#1,#2)-- (nodemap)  (nodemap)-- (#1,#2-1)}
\def\rackmaslarge(#1,#2)[#3]{\draw (#1,#2-0.5)  node[name=nodemap,inner sep=0pt,  minimum size=7.5pt, shape=circle,draw]{$#3$} (#1-2.5,#2+0.5) .. controls (#1-2.5,#2-0.5) and (#1-0.5,#2-0.5) .. (nodemap) (#1,#2)-- (nodemap)  (nodemap)-- (#1,#2-1)}
\def\rackextralarge(#1,#2)[#3]{\draw (#1,#2-0.75)  node[name=nodemap,inner sep=0pt,  minimum size=7.5pt, shape=circle, draw]{$#3$} (#1-3,#2+1) .. controls (#1-3,#2-0.75) and (#1-0.5,#2-0.75) .. (nodemap) (#1,#2)-- (nodemap)  (nodemap)-- (#1,#2-1.5)}
\def\lactionnamed(#1,#2)[#3,#4][#5]{\draw (#1 + 0.5*#3 + 0.5 + 0.5*#4, #2- 0.5*#3) node[name=nodemap,inner sep=0pt,  minimum size=8pt, shape=circle,draw]{$#5$} (#1,#2)  arc (180:270:0.5*#3) (#1 + 0.5*#3,#2- 0.5*#3) --  (nodemap) (#1 + 0.5*#3 + 0.5 + 0.5*#4, #2) --  (nodemap) (nodemap) -- (#1 + 0.5*#3 + 0.5 + 0.5*#4, #2-#3)}
\def\ractionnamed(#1,#2)[#3,#4][#5]{\draw  (#1 - 0.5*#3- 0.5 - 0.5*#4, #2- 0.5*#3)  node[name=nodemap,inner sep=0pt,  minimum size=8pt, shape=circle,draw]{$#5$} (#1 - 0.5*#3, #2- 0.5*#3)  arc (270:360:0.5*#3) (#1 - 0.5*#3, #2- 0.5*#3) -- (nodemap)(#1 - 0.5*#3- 0.5 - 0.5*#4, #2)-- (nodemap) (nodemap) -- (#1 - 0.5*#3- 0.5 - 0.5*#4, #2-#3)}
\def\multsubzero(#1,#2)[#3,#4]{\draw (#1+0.5*#3, #2-0.5*#4) node [name=nodemap,inner sep=0pt, minimum size=3pt,shape=circle,fill=white, draw]{} (#1,#2) arc (180:360:0.5*#3 and 0.5*#4) (#1+0.5*#3, #2-0.5*#4) -- (#1+0.5*#3,#2-#4)}
\begin{scope}[xshift=-1.3cm, yshift=-3.6cm]
\node at (0,-0.5){$r=$};
\end{scope}
\begin{scope}[xshift=0cm, yshift=-2.6cm]
\comult(0,0)[1,1]; \comult(2,0)[1,1];
\draw (-0.5,-1) -- (-0.5,-2); \flip(0.5,-1)[1];\draw (2.5,-1) -- (2.5,-2);
\lactionnamed(-0.5,-2)[1,0][\scriptstyle \lambda]; \ractionnamed(2.5,-2)[1,0][\scriptstyle \rho];
\end{scope}
\begin{scope}[xshift=3.1cm, yshift=-3.6cm]
\node at (0,-0.5){=};
\end{scope}
\begin{scope}[xshift=3.6cm, yshift=0cm]
\draw (0.5,0) -- (0.5,-0.5);
\comult(0.5,-0.5)[1,1]; \comult(3,0)[2,1.5]; \draw (0,-1.5) -- (0,-2.5); \flip(1,-1.5)[1]; \draw (4,-1.5) -- (4,-2.5); \lactionnamed(0,-2.5)[1,0][\scriptstyle \lambda]; \comult(2,-2.5)[1,1]; \comult(4,-2.5)[1,1]; \draw (1.5,-3.5) -- (1.5,-4.5); \flip(2.5,-3.5)[1]; \draw (4.5,-3.5) -- (4.5,-4.5); \lactionnamed(1.5,-4.5)[1,0][\scriptstyle \lambda]; \multsubzero(3.5,-4.5)[1,1]; \map(2.5,-5.5)[\scriptstyle T]; \draw (4,-5.5) -- (4,-6.5); \multsubzero(2.5,-6.5)[1.5,1.5]; \draw (1,-3.5) -- (1,-8);
\end{scope}
\begin{scope}[xshift=8.7cm, yshift=-3.6cm]
\node at (0,-0.5){=};
\end{scope}
\begin{scope}[xshift=10.3cm, yshift=-0.3cm]
\draw (3,0) -- (3,-0.5);
\comult(0.5,0)[2,1.5]; \comult(3,-0.5)[1,1];  \draw (-0.5,-1.5) -- (-0.5,-2.5); \flip(1.5,-1.5)[1]; \draw (3.5,-1.5) -- (3.5,-2.5); \comult(-0.5,-2.5)[1,1]; \comult(1.5,-2.5)[1,1]; \draw (-1,-3.5) -- (-1,-4.5); \flip(0,-3.5)[1]; \draw (2,-3.5) -- (2,-4.5); \lactionnamed(-1,-4.5)[1,0][\scriptstyle \lambda];  \lactionnamed(1,-4.5)[1,0][\scriptstyle \lambda]; \multsubzero(2.5,-2.5)[1,1]; \draw (3,-3.5) -- (3,-6.5); \map(2,-5.5)[\scriptstyle T]; \draw (0,-5.5) -- (0,-7.5); \multsubzero(2,-6.5)[1,1];
\end{scope}
\begin{scope}[xshift=14.4cm, yshift=-3.6cm]
\node at (0,-0.5){=};
\end{scope}
\begin{scope}[xshift=15.5cm, yshift=-1cm]
\comult(0,0)[1,1]; \comult(2,0)[1,1];
\draw (-0.5,-1) -- (-0.5,-2); \flip(0.5,-1)[1];\draw (2.5,-1) -- (2.5,-2);
\lactionnamed(-0.5,-2)[1,0][\scriptstyle \lambda]; \multsubzero(1.5,-2)[1,1];
\comult(0.5,-3)[1,1]; \draw (2,-3) -- (2,-5);
\draw (0,-4) -- (0,-6); \map(1,-4)[\scriptstyle T];
\multsubzero(1,-5)[1,1];
\end{scope}
\end{tikzpicture}
\end{equation}
where the first equality holds by the definition of $r$; the second one, by the definition of $\rho$; the third one, since $\Delta_{A^2}$ is coassociative; and the fourth one, since $\lambda$ is a coalgebra morphism. Combining this with the fact that $m_{\circ}$ is associative and $T$ is the antipode of $(A,m_{\circ},\Delta)$, we obtain that $m_{\circ}\circ r = m_{\circ}$. Moreover
$$
\begin{tikzpicture}[scale=0.41]
\def\counit(#1,#2){\draw (#1,#2) -- (#1,#2-0.93) (#1,#2-1) circle[radius=2pt]}
\def\comult(#1,#2)[#3,#4]{\draw (#1,#2) -- (#1,#2-0.5*#4) arc (90:0:0.5*#3 and 0.5*#4) (#1,#2-0.5*#4) arc (90:180:0.5*#3 and 0.5*#4)}
\def\laction(#1,#2)[#3,#4]{\draw (#1,#2) .. controls (#1,#2-0.555*#4/2) and (#1+0.445*#4/2,#2-1*#4/2) .. (#1+1*#4/2,#2-1*#4/2) -- (#1+2*#4/2+#3*#4/2,#2-1*#4/2) (#1+2*#4/2+#3*#4/2,#2)--(#1+2*#4/2+#3*#4/2,#2-2*#4/2)}
\def\lactiontr(#1,#2)[#3,#4,#5]{\draw (#1,#2) .. controls (#1,#2-0.555*#4/2) and (#1+0.445*#4/2,#2-1*#4/2) .. (#1+1*#4/2,#2-1*#4/2) -- (#1+2*#4/2+#3*#4/2,#2-1*#4/2)  node [inner sep=0pt, minimum size=3pt,shape=isosceles triangle,fill, shape border rotate=#5] {} (#1+2*#4/2+#3*#4/2,#2) --(#1+2*#4/2+#3*#4/2,#2-2*#4/2)}
\def\doublemap(#1,#2)[#3]{\draw (#1+0.5,#2-0.5) node [name=doublemapnode,inner xsep=0pt, inner ysep=0pt, minimum height=11pt, minimum width=23pt,shape=rectangle,draw,rounded corners] {$#3$} (#1,#2) .. controls (#1,#2-0.075) .. (doublemapnode) (#1+1,#2) .. controls (#1+1,#2-0.075).. (doublemapnode) (doublemapnode) .. controls (#1,#2-0.925)..(#1,#2-1) (doublemapnode) .. controls (#1+1,#2-0.925).. (#1+1,#2-1)}
\def\solbraid(#1,#2)[#3]{\draw (#1,#2-0.5)  node[name=nodemap,inner sep=0pt,  minimum size=9pt, shape=circle,draw]{$#3$}
(#1-0.5,#2) .. controls (#1-0.5,#2-0.15) and (#1-0.4,#2-0.2) .. (#1-0.3,#2-0.3) (#1-0.3,#2-0.3) -- (nodemap)
(#1+0.5,#2) .. controls (#1+0.5,#2-0.15) and (#1+0.4,#2-0.2) .. (#1+0.3,#2-0.3) (#1+0.3,#2-0.3) -- (nodemap)
(#1+0.5,#2-1) .. controls (#1+0.5,#2-0.85) and (#1+0.4,#2-0.8) .. (#1+0.3,#2-0.7) (#1+0.3,#2-0.7) -- (nodemap)
(#1-0.5,#2-1) .. controls (#1-0.5,#2-0.85) and (#1-0.4,#2-0.8) .. (#1-0.3,#2-0.7) (#1-0.3,#2-0.7) -- (nodemap)
}
\def\ractiontr(#1,#2)[#3,#4,#5]{\draw (#1,#2) -- (#1,#2-2*#4/2)  (#1,#2-1*#4/2) node [inner sep=0pt, minimum size=3pt,shape=isosceles triangle,fill, shape border rotate=#5] {}  --(#1+1*#4/2+#3*#4/2,#2-1*#4/2) .. controls (#1+1.555*#4/2+#3*#4/2,#2-1*#4/2) and (#1+2*#4/2+#3*#4/2,#2-0.555*#4/2) .. (#1+2*#4/2+#3*#4/2,#2)  }
\def\flip(#1,#2)[#3]{\draw (
#1+1*#3,#2) .. controls (#1+1*#3,#2-0.05*#3) and (#1+0.96*#3,#2-0.15*#3).. (#1+0.9*#3,#2-0.2*#3)
(#1+0.1*#3,#2-0.8*#3)--(#1+0.9*#3,#2-0.2*#3)
(#1,#2-1*#3) .. controls (#1,#2-0.95*#3) and (#1+0.04*#3,#2-0.85*#3).. (#1+0.1*#3,#2-0.8*#3)
(#1,#2) .. controls (#1,#2-0.05*#3) and (#1+0.04*#3,#2-0.15*#3).. (#1+0.1*#3,#2-0.2*#3)
(#1+0.1*#3,#2-0.2*#3) -- (#1+0.9*#3,#2-0.8*#3)
(#1+1*#3,#2-1*#3) .. controls (#1+1*#3,#2-0.95*#3) and (#1+0.96*#3,#2-0.85*#3).. (#1+0.9*#3,#2-0.8*#3)
}
\def\raction(#1,#2)[#3,#4]{\draw (#1,#2) -- (#1,#2-2*#4/2)  (#1,#2-1*#4/2)--(#1+1*#4/2+#3*#4/2,#2-1*#4/2) .. controls (#1+1.555*#4/2+#3*#4/2,#2-1*#4/2) and (#1+2*#4/2+#3*#4/2,#2-0.555*#4/2) .. (#1+2*#4/2+#3*#4/2,#2)}
\def\rack(#1,#2)[#3]{\draw (#1,#2-0.5)  node[name=nodemap,inner sep=0pt,  minimum size=7.5pt, shape=circle,draw]{$#3$} (#1-1,#2) .. controls (#1-1,#2-0.5) and (#1-0.5,#2-0.5) .. (nodemap) (#1,#2)-- (nodemap)  (nodemap)-- (#1,#2-1)}
\def\racklarge(#1,#2)[#3]{\draw (#1,#2-0.5)  node[name=nodemap,inner sep=0pt,  minimum size=7.5pt, shape=circle,draw]{$#3$} (#1-2,#2+0.5) .. controls (#1-2,#2-0.5) and (#1-0.5,#2-0.5) .. (nodemap) (#1,#2)-- (nodemap)  (nodemap)-- (#1,#2-1)}
\def\rackextralarge(#1,#2)[#3]{\draw (#1,#2-0.5)  node[name=nodemap,inner sep=0pt,  minimum size=7.5pt, shape=circle,draw]{$#3$} (#1-3,#2+1) .. controls (#1-3,#2-0.5) and (#1-0.5,#2-0.5) .. (nodemap) (#1,#2)-- (nodemap)  (nodemap)-- (#1,#2-1)}
\def\mult(#1,#2)[#3,#4]{\draw (#1,#2) arc (180:360:0.5*#3 and 0.5*#4) (#1+0.5*#3, #2-0.5*#4) -- (#1+0.5*#3,#2-#4)}
\def\multsubzero(#1,#2)[#3,#4]{\draw (#1+0.5*#3, #2-0.5*#4) node [name=nodemap,inner
sep=0pt, minimum size=3pt,shape=circle,fill=white, draw]{} (#1,#2) arc
(180:360:0.5*#3 and 0.5*#4) (#1+0.5*#3, #2-0.5*#4) -- (#1+0.5*#3,#2-#4)}
\def\lactionnamed(#1,#2)[#3,#4][#5]{\draw (#1 + 0.5*#3 + 0.5 + 0.5*#4, #2- 0.5*#3) node[name=nodemap,inner sep=0pt,  minimum size=8pt, shape=circle,draw]{$#5$} (#1,#2)  arc (180:270:0.5*#3) (#1 + 0.5*#3,#2- 0.5*#3) --  (nodemap) (#1 + 0.5*#3 + 0.5 + 0.5*#4, #2) --  (nodemap) (nodemap) -- (#1 + 0.5*#3 + 0.5 + 0.5*#4, #2-#3)}
\def\ractionnamed(#1,#2)[#3,#4][#5]{\draw  (#1 - 0.5*#3- 0.5 - 0.5*#4, #2- 0.5*#3)  node[name=nodemap,inner sep=0pt,  minimum size=8pt, shape=circle,draw]{$#5$} (#1 - 0.5*#3, #2- 0.5*#3)  arc (270:360:0.5*#3) (#1 - 0.5*#3, #2- 0.5*#3) -- (nodemap)(#1 - 0.5*#3- 0.5 - 0.5*#4, #2)-- (nodemap) (nodemap) -- (#1 - 0.5*#3- 0.5 - 0.5*#4, #2-#3)}
\def\map(#1,#2)[#3]{\draw (#1,#2-0.5)  node[name=nodemap,inner sep=0pt,  minimum size=10pt, shape=circle, draw]{$#3$} (#1,#2)-- (nodemap)  (nodemap)-- (#1,#2-1)}
\begin{scope}[xshift=-0.8cm, yshift=-3.75cm]
\node at (0,-0.5){$r=$};
\end{scope}
\begin{scope}[xshift=0cm, yshift=-1cm]
\comult(0.5,0)[1,1]; \comult(2.5,0)[1,1]; \draw (0,-1) -- (0,-2); \flip(1,-1)[1]; \draw (3,-1) -- (3,-2); \lactionnamed(0,-2)[1,0][\scriptstyle \lambda]; \multsubzero(2,-2)[1,1]; \draw (1,-3) -- (1,-3.5); \comult(1,-3.5)[1,1]; \draw (2.5,-3) -- (2.5,-5.5); \map(1.5,-4.5)[\scriptstyle T]; \multsubzero(1.5,-5.5)[1,1]; \draw (0.5,-4.5) -- (0.5,-6.5);
\end{scope}
\begin{scope}[xshift=3.5cm, yshift=-3.75cm]
\node at (0,-0.5){=};
\end{scope}
\begin{scope}[xshift=4.35cm, yshift=-0.25cm]
\comult(1.25,0)[1.5,1.5]; \comult(0.5,-1.5)[1,1]; \multsubzero(1,-2.5)[1,1]; \flip(2,-1.5)[1]; \comult(3.5,-0.5)[1,1]; \map(0,-2.5)[\scriptstyle S];  \multsubzero(3,-2.5)[1,1]; \draw (4,-1.5) -- (4,-2.5); \mult(0,-3.5)[1.5,1.5]; \draw (3.5,0) -- (3.5,-0.5); \draw (3.5,-3.5) ..controls (3.5,-5.25) and  (2.25,-5.25) .. (2.25,-7); \multsubzero(1.25,-7)[1,1]; \comult(0.75,-5)[1,1];  \map(1.25,-6)[\scriptstyle T]; \draw (0.25,-6) -- (0.25,-8);
\end{scope}
\begin{scope}[xshift=8.85cm, yshift=-3.75cm]
\node at (0,-0.5){=};
\end{scope}
\begin{scope}[xshift=9.65cm, yshift=0cm]
\comult(0.612,0)[1.25,1.25]; \comult(1.25,-1.25)[1,1]; \comult(3.25,-1.25)[1,1]; \flip(1.75,-2.25)[1]; \draw (3.25,0) -- (3.25,-1.25); \draw (0.75,-2.25) -- (0.75,-3.25); \draw (3.75,-2.25) -- (3.75,-2.25); \draw (0,-1.25) -- (0,-2.25); \multsubzero(0.75,-3.25)[1,1]; \map(0,-2.25)[\scriptstyle S]; \draw (0,-3.25) -- (0,-4.25);  \multsubzero(2.75,-3.25)[1,1]; \mult(0,-4.25)[1.25,1.25]; \draw (3.75,-2.25) -- (3.75,-3.25); \draw (3.25,-4.25) .. controls (3.25,-5.825)and (2.125,-5.825).. (2.125,-7.5); \comult(0.625,-5.5)[1,1]; \multsubzero(1.125,-7.5)[1,1];  \map(1.125,-6.5)[\scriptstyle T]; \draw (0.125,-6.5) -- (0.125,-8.5);
\end{scope}
\begin{scope}[xshift=13.85cm, yshift=-3.75cm]
\node at (0,-0.5){=};
\end{scope}
\begin{scope}[xshift=14.65cm, yshift=-0.75cm]
\comult(0.5,0)[1,1]; \multsubzero(1,-1)[1,1]; \draw (2,0) -- (2,-1); \draw (0,-1) -- (0,-1.75); \map(0,-1.75)[\scriptstyle S]; \comult(1.5,-2)[1,1]; \draw (0,-2.75) -- (0,-3); \mult(0,-3)[1,1]; \draw (2,-3) -- (2,-6); \comult(0.5,-4)[1,1];\map(1,-5)[\scriptstyle T];\multsubzero(1,-6)[1,1]; \draw (0,-5) -- (0,-7);
\end{scope}
\begin{scope}[xshift=17.25cm, yshift=-3.75cm]
\node at (0,-0.5){.};
\end{scope}
\end{tikzpicture}
$$
So, $r$ is composition of three coalgebra isomorphisms, and then it is a coalgebra isomorphism. Thus, by Proposition~\ref{lem:matched_pair} in order to finish the proof of the claim it suffices to check that $\lambda$ is a left action, $\rho$ is a right action and items~(1)--(4) of that lemma are fulfilled. By Propositions~\ref{pro:modulo} and~\ref{pro:rho} we know that $\lambda$ and $\rho$ are actions and items~(3) and~(4) hold. Item~(1) also holds, since
$$
\begin{tikzpicture}[scale=0.41]
\def\mult(#1,#2)[#3,#4]{\draw (#1,#2) arc (180:360:0.5*#3 and 0.5*#4) (#1+0.5*#3, #2-0.5*#4) -- (#1+0.5*#3,#2-#4)}
\def\counit(#1,#2){\draw (#1,#2) -- (#1,#2-0.93) (#1,#2-1) circle[radius=2pt]}
\def\comult(#1,#2)[#3,#4]{\draw (#1,#2) -- (#1,#2-0.5*#4) arc (90:0:0.5*#3 and 0.5*#4) (#1,#2-0.5*#4) arc (90:180:0.5*#3 and 0.5*#4)}
\def\laction(#1,#2)[#3,#4]{\draw (#1,#2) .. controls (#1,#2-0.555*#4/2) and (#1+0.445*#4/2,#2-1*#4/2) .. (#1+1*#4/2,#2-1*#4/2) -- (#1+2*#4/2+#3*#4/2,#2-1*#4/2) (#1+2*#4/2+#3*#4/2,#2)--(#1+2*#4/2+#3*#4/2,#2-2*#4/2)}
\def\map(#1,#2)[#3]{\draw (#1,#2-0.5)  node[name=nodemap,inner sep=0pt,  minimum size=10pt, shape=circle, draw]{$#3$} (#1,#2)-- (nodemap)  (nodemap)-- (#1,#2-1)}
\def\solbraid(#1,#2)[#3]{\draw (#1,#2-0.5)  node[name=nodemap,inner sep=0pt,  minimum size=9pt, shape=circle,draw]{$#3$}
(#1-0.5,#2) .. controls (#1-0.5,#2-0.15) and (#1-0.4,#2-0.2) .. (#1-0.3,#2-0.3) (#1-0.3,#2-0.3) -- (nodemap)
(#1+0.5,#2) .. controls (#1+0.5,#2-0.15) and (#1+0.4,#2-0.2) .. (#1+0.3,#2-0.3) (#1+0.3,#2-0.3) -- (nodemap)
(#1+0.5,#2-1) .. controls (#1+0.5,#2-0.85) and (#1+0.4,#2-0.8) .. (#1+0.3,#2-0.7) (#1+0.3,#2-0.7) -- (nodemap)
(#1-0.5,#2-1) .. controls (#1-0.5,#2-0.85) and (#1-0.4,#2-0.8) .. (#1-0.3,#2-0.7) (#1-0.3,#2-0.7) -- (nodemap)
}
\def\flip(#1,#2)[#3]{\draw (
#1+1*#3,#2) .. controls (#1+1*#3,#2-0.05*#3) and (#1+0.96*#3,#2-0.15*#3).. (#1+0.9*#3,#2-0.2*#3)
(#1+0.1*#3,#2-0.8*#3)--(#1+0.9*#3,#2-0.2*#3)
(#1,#2-1*#3) .. controls (#1,#2-0.95*#3) and (#1+0.04*#3,#2-0.85*#3).. (#1+0.1*#3,#2-0.8*#3)
(#1,#2) .. controls (#1,#2-0.05*#3) and (#1+0.04*#3,#2-0.15*#3).. (#1+0.1*#3,#2-0.2*#3)
(#1+0.1*#3,#2-0.2*#3) -- (#1+0.9*#3,#2-0.8*#3)
(#1+1*#3,#2-1*#3) .. controls (#1+1*#3,#2-0.95*#3) and (#1+0.96*#3,#2-0.85*#3).. (#1+0.9*#3,#2-0.8*#3)
}
\def\raction(#1,#2)[#3,#4]{\draw (#1,#2) -- (#1,#2-2*#4/2)  (#1,#2-1*#4/2)--(#1+1*#4/2+#3*#4/2,#2-1*#4/2) .. controls (#1+1.555*#4/2+#3*#4/2,#2-1*#4/2) and (#1+2*#4/2+#3*#4/2,#2-0.555*#4/2) .. (#1+2*#4/2+#3*#4/2,#2)}
\def\doublemap(#1,#2)[#3]{\draw (#1+0.5,#2-0.5) node [name=doublemapnode,inner xsep=0pt, inner ysep=0pt, minimum height=11pt, minimum width=23pt,shape=rectangle,draw,rounded corners] {$#3$} (#1,#2) .. controls (#1,#2-0.075) .. (doublemapnode) (#1+1,#2) .. controls (#1+1,#2-0.075).. (doublemapnode) (doublemapnode) .. controls (#1,#2-0.925)..(#1,#2-1) (doublemapnode) .. controls (#1+1,#2-0.925).. (#1+1,#2-1)}
\def\doublesinglemap(#1,#2)[#3]{\draw (#1+0.5,#2-0.5) node [name=doublesinglemapnode,inner xsep=0pt, inner ysep=0pt, minimum height=11pt, minimum width=23pt,shape=rectangle,draw,rounded corners] {$#3$} (#1,#2) .. controls (#1,#2-0.075) .. (doublesinglemapnode) (#1+1,#2) .. controls (#1+1,#2-0.075).. (doublesinglemapnode) (doublesinglemapnode)-- (#1+0.5,#2-1)}
\def\ractiontr(#1,#2)[#3,#4,#5]{\draw (#1,#2) -- (#1,#2-2*#4/2)  (#1,#2-1*#4/2) node [inner sep=0pt, minimum size=3pt,shape=isosceles triangle,fill, shape border rotate=#5] {}  --(#1+1*#4/2+#3*#4/2,#2-1*#4/2) .. controls (#1+1.555*#4/2+#3*#4/2,#2-1*#4/2) and (#1+2*#4/2+#3*#4/2,#2-0.555*#4/2) .. (#1+2*#4/2+#3*#4/2,#2)  }
\def\rack(#1,#2)[#3]{\draw (#1,#2-0.5)  node[name=nodemap,inner sep=0pt,  minimum size=7.5pt, shape=circle,draw]{$#3$} (#1-1,#2) .. controls (#1-1,#2-0.5) and (#1-0.5,#2-0.5) .. (nodemap) (#1,#2)-- (nodemap)  (nodemap)-- (#1,#2-1)}
\def\rackmenoslarge(#1,#2)[#3]{\draw (#1,#2-0.5)  node[name=nodemap,inner sep=0pt,  minimum size=7.5pt, shape=circle,draw]{$#3$} (#1-1.5,#2+0.5) .. controls (#1-1.5,#2-0.5) and (#1-0.5,#2-0.5) .. (nodemap) (#1,#2)-- (nodemap)  (nodemap)-- (#1,#2-1)}
\def\racklarge(#1,#2)[#3]{\draw (#1,#2-0.5)  node[name=nodemap,inner sep=0pt,  minimum size=7.5pt, shape=circle,draw]{$#3$} (#1-2,#2+0.5) .. controls (#1-2,#2-0.5) and (#1-0.5,#2-0.5) .. (nodemap) (#1,#2)-- (nodemap)  (nodemap)-- (#1,#2-1)}
\def\rackmaslarge(#1,#2)[#3]{\draw (#1,#2-0.5)  node[name=nodemap,inner sep=0pt,  minimum size=7.5pt, shape=circle,draw]{$#3$} (#1-2.5,#2+0.5) .. controls (#1-2.5,#2-0.5) and (#1-0.5,#2-0.5) .. (nodemap) (#1,#2)-- (nodemap)  (nodemap)-- (#1,#2-1)}
\def\rackextralarge(#1,#2)[#3]{\draw (#1,#2-0.75)  node[name=nodemap,inner sep=0pt,  minimum size=7.5pt, shape=circle, draw]{$#3$} (#1-3,#2+1) .. controls (#1-3,#2-0.75) and (#1-0.5,#2-0.75) .. (nodemap) (#1,#2)-- (nodemap)  (nodemap)-- (#1,#2-1.5)}
\def\lactionnamed(#1,#2)[#3,#4][#5]{\draw (#1 + 0.5*#3 + 0.5 + 0.5*#4, #2- 0.5*#3) node[name=nodemap,inner sep=0pt,  minimum size=8pt, shape=circle,draw]{$#5$} (#1,#2)  arc (180:270:0.5*#3) (#1 + 0.5*#3,#2- 0.5*#3) --  (nodemap) (#1 + 0.5*#3 + 0.5 + 0.5*#4, #2) --  (nodemap) (nodemap) -- (#1 + 0.5*#3 + 0.5 + 0.5*#4, #2-#3)}
\def\ractionnamed(#1,#2)[#3,#4][#5]{\draw  (#1 - 0.5*#3- 0.5 - 0.5*#4, #2- 0.5*#3)  node[name=nodemap,inner sep=0pt,  minimum size=8pt, shape=circle,draw]{$#5$} (#1 - 0.5*#3, #2- 0.5*#3)  arc (270:360:0.5*#3) (#1 - 0.5*#3, #2- 0.5*#3) -- (nodemap)(#1 - 0.5*#3- 0.5 - 0.5*#4, #2)-- (nodemap) (nodemap) -- (#1 - 0.5*#3- 0.5 - 0.5*#4, #2-#3)}
\def\multsubzero(#1,#2)[#3,#4]{\draw (#1+0.5*#3, #2-0.5*#4) node [name=nodemap,inner sep=0pt, minimum size=3pt,shape=circle,fill=white, draw]{} (#1,#2) arc (180:360:0.5*#3 and 0.5*#4) (#1+0.5*#3, #2-0.5*#4) -- (#1+0.5*#3,#2-#4)}
\begin{scope}[xshift=0cm, yshift=-1cm]
\solbraid(0.5,0)[\scriptstyle r]; \lactionnamed(1,-1)[1,0][\scriptstyle \lambda];  \draw (0,-1) -- (0,-2); \draw (2,0) -- (2,-1);\multsubzero(0,-2)[2,1.5];
\end{scope}
\begin{scope}[xshift=2.8cm, yshift=-2.3cm]
\node at (0,-0.5){=};
\end{scope}
\begin{scope}[xshift=3.3cm, yshift=-0.5cm]
\solbraid(1,0)[\scriptstyle r]; \lactionnamed(1.5,-1)[1,0][\scriptstyle \lambda];  \comult(0.5,-1)[1,1]; \lactionnamed(1,-2)[1,1][\scriptstyle \lambda]; \draw (0,-2) -- (0,-3); \draw (2.5,0) -- (2.5,-1); \mult(0,-3)[2.5,1.5];
\end{scope}
\begin{scope}[xshift=6.6cm, yshift=-2.3cm]
\node at (0,-0.5){=};
\end{scope}
\begin{scope}[xshift=7.1cm, yshift=0cm]
\comult(0.5,0)[1,1]; \comult(2.5,0)[1,1]; \draw (0,-1) -- (0,-2); \flip(1,-1)[1]; \draw (3,-1) -- (3,-2); \lactionnamed(0,-2)[1,0][\scriptstyle \lambda]; \multsubzero(2,-2)[1,1]; \lactionnamed(2.5,-3)[1,0][\scriptstyle \lambda]; \draw (3.5,0) -- (3.5,-3); \draw (1,-3) -- (1,-4); \mult(1,-4)[2.5,1.5];
\end{scope}
\begin{scope}[xshift=11.3cm, yshift=-2.3cm]
\node at (0,-0.5){=};
\end{scope}
\begin{scope}[xshift=11.8cm, yshift=0cm]
\comult(0.5,0)[1,1]; \comult(2.5,0)[1,1]; \draw (0,-1) -- (0,-2); \flip(1,-1)[1];  \lactionnamed(0,-2)[1,0][\scriptstyle \lambda]; \lactionnamed(3,-1)[1,0][\scriptstyle \lambda]; \draw (4,0) -- (4,-1); \lactionnamed(2,-2)[1.5,1.5][\scriptstyle \lambda]; \draw (1,-3) -- (1,-3.5); \mult(1,-3.5)[3,2];
\end{scope}
\begin{scope}[xshift=16.55cm, yshift=-2.3cm]
\node at (0,-0.5){=};
\end{scope}
\begin{scope}[xshift=17.1cm, yshift=-0.5cm]
\comult(1,0)[1,1];  \draw (0,0) -- (0,-3.5); \draw (0.5,-1) -- (0.5,-2); \draw (2.5,0) -- (2.5,-1);  \lactionnamed(1.5,-1)[1,0][\scriptstyle \lambda]; \mult(0.5,-2)[2,1.5]; \lactionnamed(0,-3.5)[1,1][\scriptstyle \lambda];
\end{scope}
\begin{scope}[xshift=20.3cm, yshift=-2.3cm]
\node at (0,-0.5){=};
\end{scope}
\begin{scope}[xshift=20.9cm, yshift=-1.8cm]
\draw (0,0) -- (0,-1); \multsubzero(1,0)[1,1];
\lactionnamed(0,-1)[1,1][\scriptstyle \lambda];
\end{scope}
\begin{scope}[xshift=23.1cm, yshift=-2.3cm]
\node at (0,-0.5){,};
\end{scope}
\end{tikzpicture}
$$
where the first and last equalities follow from~\eqref{eq:aob}; the second one, since, by Corollary~\ref{coro: igualdad de morfismos},
$$
(A\ot \lambda)\circ (\Delta\ot \lambda)\circ (r\ot A) = (A\ot \lambda)\circ (\lambda\ot m_{\circ}\ot A)\circ (\Delta_{A^2}\ot A);
$$
and the third and fourth equalities follow from Proposition~\ref{pro:modulo}. Finally, item~(2) also holds, since
$$
\begin{tikzpicture}[scale=0.41]
\def\unit(#1,#2){\draw (#1,#2) circle[radius=2pt] (#1,#2-0.07) -- (#1,#2-1)}
\def\mult(#1,#2)[#3,#4]{\draw (#1,#2) arc (180:360:0.5*#3 and 0.5*#4) (#1+0.5*#3, #2-0.5*#4) -- (#1+0.5*#3,#2-#4)}
\def\counit(#1,#2){\draw (#1,#2) -- (#1,#2-0.93) (#1,#2-1) circle[radius=2pt]}
\def\comult(#1,#2)[#3,#4]{\draw (#1,#2) -- (#1,#2-0.5*#4) arc (90:0:0.5*#3 and 0.5*#4) (#1,#2-0.5*#4) arc (90:180:0.5*#3 and 0.5*#4)}
\def\laction(#1,#2)[#3,#4]{\draw (#1,#2) .. controls (#1,#2-0.555*#4/2) and (#1+0.445*#4/2,#2-1*#4/2) .. (#1+1*#4/2,#2-1*#4/2) -- (#1+2*#4/2+#3*#4/2,#2-1*#4/2) (#1+2*#4/2+#3*#4/2,#2)--(#1+2*#4/2+#3*#4/2,#2-2*#4/2)}
\def\map(#1,#2)[#3]{\draw (#1,#2-0.5)  node[name=nodemap,inner sep=0pt,  minimum size=10pt, shape=circle, draw]{$#3$} (#1,#2)-- (nodemap)  (nodemap)-- (#1,#2-1)}
\def\solbraid(#1,#2)[#3]{\draw (#1,#2-0.5)  node[name=nodemap,inner sep=0pt,  minimum size=9pt, shape=circle,draw]{$#3$}
(#1-0.5,#2) .. controls (#1-0.5,#2-0.15) and (#1-0.4,#2-0.2) .. (#1-0.3,#2-0.3) (#1-0.3,#2-0.3) -- (nodemap)
(#1+0.5,#2) .. controls (#1+0.5,#2-0.15) and (#1+0.4,#2-0.2) .. (#1+0.3,#2-0.3) (#1+0.3,#2-0.3) -- (nodemap)
(#1+0.5,#2-1) .. controls (#1+0.5,#2-0.85) and (#1+0.4,#2-0.8) .. (#1+0.3,#2-0.7) (#1+0.3,#2-0.7) -- (nodemap)
(#1-0.5,#2-1) .. controls (#1-0.5,#2-0.85) and (#1-0.4,#2-0.8) .. (#1-0.3,#2-0.7) (#1-0.3,#2-0.7) -- (nodemap)
}
\def\flip(#1,#2)[#3]{\draw (
#1+1*#3,#2) .. controls (#1+1*#3,#2-0.05*#3) and (#1+0.96*#3,#2-0.15*#3).. (#1+0.9*#3,#2-0.2*#3)
(#1+0.1*#3,#2-0.8*#3)--(#1+0.9*#3,#2-0.2*#3)
(#1,#2-1*#3) .. controls (#1,#2-0.95*#3) and (#1+0.04*#3,#2-0.85*#3).. (#1+0.1*#3,#2-0.8*#3)
(#1,#2) .. controls (#1,#2-0.05*#3) and (#1+0.04*#3,#2-0.15*#3).. (#1+0.1*#3,#2-0.2*#3)
(#1+0.1*#3,#2-0.2*#3) -- (#1+0.9*#3,#2-0.8*#3)
(#1+1*#3,#2-1*#3) .. controls (#1+1*#3,#2-0.95*#3) and (#1+0.96*#3,#2-0.85*#3).. (#1+0.9*#3,#2-0.8*#3)
}
\def\raction(#1,#2)[#3,#4]{\draw (#1,#2) -- (#1,#2-2*#4/2)  (#1,#2-1*#4/2)--(#1+1*#4/2+#3*#4/2,#2-1*#4/2) .. controls (#1+1.555*#4/2+#3*#4/2,#2-1*#4/2) and (#1+2*#4/2+#3*#4/2,#2-0.555*#4/2) .. (#1+2*#4/2+#3*#4/2,#2)}
\def\doublemap(#1,#2)[#3]{\draw (#1+0.5,#2-0.5) node [name=doublemapnode,inner xsep=0pt, inner ysep=0pt, minimum height=11pt, minimum width=23pt,shape=rectangle,draw,rounded corners] {$#3$} (#1,#2) .. controls (#1,#2-0.075) .. (doublemapnode) (#1+1,#2) .. controls (#1+1,#2-0.075).. (doublemapnode) (doublemapnode) .. controls (#1,#2-0.925)..(#1,#2-1) (doublemapnode) .. controls (#1+1,#2-0.925).. (#1+1,#2-1)}
\def\doublesinglemap(#1,#2)[#3]{\draw (#1+0.5,#2-0.5) node [name=doublesinglemapnode,inner xsep=0pt, inner ysep=0pt, minimum height=11pt, minimum width=23pt,shape=rectangle,draw,rounded corners] {$#3$} (#1,#2) .. controls (#1,#2-0.075) .. (doublesinglemapnode) (#1+1,#2) .. controls (#1+1,#2-0.075).. (doublesinglemapnode) (doublesinglemapnode)-- (#1+0.5,#2-1)}
\def\ractiontr(#1,#2)[#3,#4,#5]{\draw (#1,#2) -- (#1,#2-2*#4/2)  (#1,#2-1*#4/2) node [inner sep=0pt, minimum size=3pt,shape=isosceles triangle,fill, shape border rotate=#5] {}  --(#1+1*#4/2+#3*#4/2,#2-1*#4/2) .. controls (#1+1.555*#4/2+#3*#4/2,#2-1*#4/2) and (#1+2*#4/2+#3*#4/2,#2-0.555*#4/2) .. (#1+2*#4/2+#3*#4/2,#2)  }
\def\rack(#1,#2)[#3]{\draw (#1,#2-0.5)  node[name=nodemap,inner sep=0pt,  minimum size=7.5pt, shape=circle,draw]{$#3$} (#1-1,#2) .. controls (#1-1,#2-0.5) and (#1-0.5,#2-0.5) .. (nodemap) (#1,#2)-- (nodemap)  (nodemap)-- (#1,#2-1)}
\def\rackmenoslarge(#1,#2)[#3]{\draw (#1,#2-0.5)  node[name=nodemap,inner sep=0pt,  minimum size=7.5pt, shape=circle,draw]{$#3$} (#1-1.5,#2+0.5) .. controls (#1-1.5,#2-0.5) and (#1-0.5,#2-0.5) .. (nodemap) (#1,#2)-- (nodemap)  (nodemap)-- (#1,#2-1)}
\def\racklarge(#1,#2)[#3]{\draw (#1,#2-0.5)  node[name=nodemap,inner sep=0pt,  minimum size=7.5pt, shape=circle,draw]{$#3$} (#1-2,#2+0.5) .. controls (#1-2,#2-0.5) and (#1-0.5,#2-0.5) .. (nodemap) (#1,#2)-- (nodemap)  (nodemap)-- (#1,#2-1)}
\def\rackmaslarge(#1,#2)[#3]{\draw (#1,#2-0.5)  node[name=nodemap,inner sep=0pt,  minimum size=7.5pt, shape=circle,draw]{$#3$} (#1-2.5,#2+0.5) .. controls (#1-2.5,#2-0.5) and (#1-0.5,#2-0.5) .. (nodemap) (#1,#2)-- (nodemap)  (nodemap)-- (#1,#2-1)}
\def\rackextralarge(#1,#2)[#3]{\draw (#1,#2-0.75)  node[name=nodemap,inner sep=0pt,  minimum size=7.5pt, shape=circle, draw]{$#3$} (#1-3,#2+1) .. controls (#1-3,#2-0.75) and (#1-0.5,#2-0.75) .. (nodemap) (#1,#2)-- (nodemap)  (nodemap)-- (#1,#2-1.5)}
\def\lactionnamed(#1,#2)[#3,#4][#5]{\draw (#1 + 0.5*#3 + 0.5 + 0.5*#4, #2- 0.5*#3) node[name=nodemap,inner sep=0pt,  minimum size=8pt, shape=circle,draw]{$#5$} (#1,#2)  arc (180:270:0.5*#3) (#1 + 0.5*#3,#2- 0.5*#3) --  (nodemap) (#1 + 0.5*#3 + 0.5 + 0.5*#4, #2) --  (nodemap) (nodemap) -- (#1 + 0.5*#3 + 0.5 + 0.5*#4, #2-#3)}
\def\ractionnamed(#1,#2)[#3,#4][#5]{\draw  (#1 - 0.5*#3- 0.5 - 0.5*#4, #2- 0.5*#3)  node[name=nodemap,inner sep=0pt,  minimum size=8pt, shape=circle,draw]{$#5$} (#1 - 0.5*#3, #2- 0.5*#3)  arc (270:360:0.5*#3) (#1 - 0.5*#3, #2- 0.5*#3) -- (nodemap)(#1 - 0.5*#3- 0.5 - 0.5*#4, #2)-- (nodemap) (nodemap) -- (#1 - 0.5*#3- 0.5 - 0.5*#4, #2-#3)}
\def\multsubzero(#1,#2)[#3,#4]{\draw (#1+0.5*#3, #2-0.5*#4) node [name=nodemap,inner sep=0pt, minimum size=3pt,shape=circle,fill=white, draw]{} (#1,#2) arc (180:360:0.5*#3 and 0.5*#4) (#1+0.5*#3, #2-0.5*#4) -- (#1+0.5*#3,#2-#4)}
\def\doublemap(#1,#2)[#3]{\draw (#1+0.5,#2-0.5) node [name=doublemapnode,inner xsep=0pt, inner ysep=0pt, minimum height=10.5pt, minimum width=23pt,shape=rectangle,draw,rounded corners] {$#3$} (#1,#2) .. controls (#1,#2-0.075) .. (doublemapnode) (#1+1,#2) .. controls (#1+1,#2-0.075).. (doublemapnode) (doublemapnode) .. controls (#1,#2-0.925)..(#1,#2-1) (doublemapnode) .. controls (#1+1,#2-0.925).. (#1+1,#2-1)}
\begin{scope}[xshift=0cm, yshift=-2.7cm]
\draw (0,0) -- (0,-1); \solbraid(1.5,0)[\scriptstyle r];
\ractionnamed(1,-1)[1,0][\scriptstyle \rho]; \draw (2,-1) -- (2,-2);
\multsubzero(0,-2)[2,1.5];
\end{scope}
\begin{scope}[xshift=2.4cm, yshift=-3.8cm]
\node at (0,-0.5){=};
\end{scope}
\begin{scope}[xshift=4.1cm, yshift=-0.3cm]
\draw (-1,0) -- (-1,-5);
\comult(0,0)[1,1]; \comult(2,0)[1,1];
\draw (-0.5,-1) -- (-0.5,-2); \flip(0.5,-1)[1];\draw (2.5,-1) -- (2.5,-2);
\lactionnamed(-0.5,-2)[1,0][\scriptstyle \lambda]; \multsubzero(1.5,-2)[1,1];
\comult(0.5,-3)[1,1]; \draw (2,-3) -- (2,-5);
\draw (0,-4) -- (0,-5); \map(1,-4)[\scriptstyle T];
\multsubzero(1,-5)[1,1]; 
\ractionnamed(0,-5)[1,0][\scriptstyle \rho];
\multsubzero(-1,-6)[2.5,2];
\end{scope}
\begin{scope}[xshift=7cm, yshift=-3.8cm]
\node at (0,-0.5){=};
\end{scope}
\begin{scope}[xshift=9cm, yshift=0cm]
\draw (-1,0) -- (-1,-2);
\comult(0,0)[1,1]; \comult(2,0)[1,1];
\draw (-0.5,-1) -- (-0.5,-2); \flip(0.5,-1)[1]; \draw (2.5,-1) -- (2.5,-2);
\lactionnamed(-0.5,-2)[1,0][\scriptstyle \lambda]; \multsubzero(1.5,-2)[1,1];
\comult(-1,-2)[1,1]; \draw (2,-3) -- (2,-7);
\draw (-1.5,-3) -- (-1.5,-4); \flip(-0.5,-3)[1];
\lactionnamed(-1.5,-4)[1,0][\scriptstyle \lambda];
\draw (0.5,-4) -- (0.5,-6); \map(-0.5,-5)[\scriptstyle T];
\multsubzero(-0.5,-6)[1,1]; 
\multsubzero(0,-7)[2,1.5];
\end{scope}
\begin{scope}[xshift=11.9cm, yshift=-3.8cm]
\node at (0,-0.5){=};
\end{scope}
\begin{scope}[xshift=12.4cm, yshift=-0.2cm]
\comult(0.5,0)[1,1]; \comult(2.5,0)[1,1]; \draw (4,0) -- (4,-2); \draw (0,-1) --
(0,-2); \flip(1,-1)[1]; \draw (3,-1) -- (3,-2);
\multsubzero(0,-2)[1,1]; \multsubzero(2,-2)[1,1]; \comult(4,-2)[1,1];
\flip(2.5,-3)[1];  \multsubzero(3.5,-4)[1,1]; \draw (4.5,-3) -- (4.5,-4);
\draw (0.5,-3) -- (0.5,-4); \lactionnamed(0.5,-4)[1.5,1.5][\scriptstyle \lambda];
\map(2.5,-5.5)[\scriptstyle T]; \draw (4,-5) -- (4,-6.5);
\multsubzero(2.5,-6.5)[1.5,1.5];
\end{scope}
\begin{scope}[xshift=17.3cm, yshift=-3.8cm]
\node at (0,-0.5){=};
\end{scope}
\begin{scope}[xshift=17.8cm, yshift=-1cm]
\multsubzero(0,0)[1,1];
\comult(0.5,-1)[1,1];
\comult(2.5,-1)[1,1];
\flip(1,-2)[1];  \multsubzero(2,-3)[1,1]; \draw (3,-2) -- (3,-3);
\draw (0,-2) -- (0,-3); \lactionnamed(0,-3)[1,0][\scriptstyle \lambda];
\map(1,-4)[\scriptstyle T]; \draw (2.5,-4) -- (2.5,-5);
\multsubzero(1,-5)[1.5,1.5];
\end{scope}
\begin{scope}[xshift=21.3cm, yshift=-3.8cm]
\node at (0,-0.5){=};
\end{scope}
\begin{scope}[xshift=21.8cm, yshift=-3.2cm]
\multsubzero(0,0)[1,1]; \draw (1.5,0) -- (1.5,-1);
\ractionnamed(1.5,-1)[1,0][\scriptstyle \rho];
\end{scope}
\begin{scope}[xshift=23.5cm, yshift=-3.8cm]
\node at (0,-0.5){,};
\end{scope}

\end{tikzpicture}
$$
where the first equality follows from~\eqref{pepito}; the secong one, from Proposition~\ref{pro:rho'}; the third, using that $c$ is natural, $m_{\circ}$ is associative and $\lambda$ is an action; the fourth one, using that $m_{\circ}$ is a coalgebra morphism; and the last one, by the definition of $\rho$. This finishes the proof of the claim. An easy computation shows that for each brace morphism $f$, the map $f$ is a braiding operator morphism. Clearly the correpondence $U\colon \Br(A)\to \Bo(A)$, given by $U(A,m,m_{\circ})\coloneqq (A,r)$ and $U(f)\coloneqq f$, is a functor.

We now construct a functor $V\colon \Bo(A)\to \Br(A)$. By Theorem~\ref{lem:bop->Hopf}, from a braiding operator $(A,r)$ we obtain a brace $(A,m,m_{\circ})$. Moreover it is clear that each braiding operator morphism is a brace morphism. Thus the correspondence $V\colon \Bo(A)\to \Br(A)$, given by $V(A,r)\coloneqq (A,m,m_{\circ})$ and $V(f)\coloneqq f$, is a functor.
	
Using equality~\eqref{eq:aob} of Remark~\ref{rem:productos} it is easy to check that,
\[
(A\otimes\epsilon)\circ r =m\circ (S\otimes m_{\circ})\circ (\Delta\otimes A),
\]		
for each braiding operator~$(A,r)$. In other words, the left actions determined by $(A,r)$ and $V(A,r)$ coincide. By Lemma~\ref{lem:bop->brace}(5), this implies that $U\circ V=\id$. On the other hand, by	Equality~\eqref{eq:ab}, we have $V\circ U(A,m,m_{\circ}) = (A,m,m_{\circ})$, for each brace $(A,m,m_{\circ})$ in $\mathscr{C}$. Consequently $V\circ U=\id$.
\end{proof}

\section{A universal construction}\label{LYZ}

We now generalize the universal construction of~\cite[Theorem 9]{MR1769723}. This construction generalizes ~\cite[Theorem 2.9]{MR1722951}. We assume that $\mathscr{C}$ is a symmetric monoidal category with countable colimits and that the tensor product commutes with colimits.

Let $(X,r)$ be a non-degenerate pair in $\mathscr{C}$ and let $Z=X\coprod SX$, where $SX=X$. Let $S\colon Z\to Z$ be the morphism induced by
$$
X \xrightarrow[]{=} SX  \xrightarrow[]{\iota_{SX}} Z\quad\text{and}\quad SX \xrightarrow[]{=} X \xrightarrow[]{\iota_X} Z,
$$
where $\iota_{SX}$ and $\iota_X$ are the canonical morphism. Over $Z$ there is a unique comultiplication such that $\iota_X$ is a coalgebra homomorphism and $S$ is a coalgebra isomorphism.

Let $r_e\colon Z\otimes Z\to Z\otimes Z$ be the map induced by
\begin{align}
\label{eq:r1}& r_1\colon X\otimes X\longrightarrow X\otimes X, && r_1\coloneqq r, \\
\label{eq:r2}& r_2\colon X\otimes SX\longrightarrow SX\otimes X, && r_2\coloneqq (S\otimes X)\circ c\circ R^{t_1}\circ (X\otimes S)\\
\label{eq:r3}& r_3\colon SX\otimes X\longrightarrow X\otimes SX, && r_3\coloneqq (X\otimes S)\circ c\circ R^{t_2}\circ (S\otimes X)\\
\label{eq:r4}& r_4\colon SX\otimes SX\longrightarrow SX\otimes SX, && r_4\coloneqq (S\otimes S)\circ c\circ r^{-1}\circ c\circ(S\otimes S).
\end{align}

\begin{pro}\label{lem:SXxSXxSX} Let $X$ be a cocommutative coalgebra in $\mathscr{C}$, let $r$ be a coalgebra automorphism of $X^2$ and let $\tilde{r}:=c\circ r^{-1}\circ c$. The following assertions hold:

\begin{enumerate}

\item If $(X,r)$ is a braided set, then so is $(X,\tilde{r})$.

\smallskip

\item If $(X,r)$ is a non-degenerate pair in $\mathscr{C}$, then so is $(X,\tilde{r})$. Moreover the maps $\widetilde{R}^{t_1}$ and $\widetilde{R}^{t_2}$, defined as the maps $R^{t_1}$ and $R^{t_2}$ of $\tilde{r}$, satisfy $\widetilde{R}^{t_1} = R^{t_2}$ and $\widetilde{R}^{t_2} = R^{t_1}$.

\end{enumerate}

\end{pro}

\begin{proof} 1)\enspace A direct computation shows that
\begin{align*}
& \tilde{r}_{12}\circ \tilde{r}_{23}\circ \tilde{r}_{12}= c_{12}\circ c_{23}\circ c_{12}\circ r_{23}^{-1}\circ r_{12}^{-1} \circ r_{23}^{-1}\circ c_{12}\circ c_{23}\circ c_{12}
\shortintertext{and}
& \tilde{r}_{23}\circ \tilde{r}_{12}\circ \tilde{r}_{23}= c_{23}\circ c_{12}\circ c_{23}\circ r_{12}^{-1}\circ r_{23}^{-1} \circ r_{12}^{-1}\circ c_{23}\circ c_{12}\circ c_{23}.
\end{align*}
From this it follows easily that $(X,\tilde{r})$ is a braided set.

\smallskip

\noindent 2) In order to prove the first assertion in item~(2) we must show that
\begin{equation}
(\epsilon\otimes X^2)\circ (\tilde{r}\otimes X)\circ (X\otimes\Delta)\quad\text{and}\quad (X^2\otimes\epsilon)\circ(X\otimes  \tilde{r})\circ(\Delta\otimes X)\label{f2}
\end{equation}
are isomorphism. Using items~(2) and~(4) of Remark~\ref{algunas formulas}, we obtain that
\begin{align}
&(X^2\otimes\epsilon)\circ(X\otimes \tilde{r})\circ(\Delta\otimes X)\circ c\circ r  = (\tau\ot X)\circ (X\ot \Delta)\label{eee1}
\shortintertext{and}
& (\epsilon\otimes X^2)\circ (\tilde{r}\otimes X)\circ (X\otimes\Delta)\circ c\circ r = (X\otimes \sigma)\circ (\Delta\ot X),\label{eee2}
\end{align}
which clearly imply that the arrows in~\eqref{f2} are isomorphisms. Finally the equalities $\widetilde{R}^{t_1} = R^{t_2}$ and $\widetilde{R}^{t_2} = R^{t_1}$ follow easily using~\eqref{eee1} and~\eqref{eee2}
\end{proof}

\begin{pro}\label{pro:non-degenerate} If $(X,r)$ is a non-degenerate pair in $\mathscr{C}$, then so is $(Z,r_e)$.
\end{pro}

\begin{proof} By Proposition~\ref{no degenerado} we must prove that
$$
(Z \ot \sigma_{r_e})\circ (\Delta_Z \ot Z)\quad\text{and}\quad (\tau_{r_e} \ot Z)\circ (Z\ot \Delta_Z)
$$
are isomorphisms. Using that $(X,r)$ and $(X,\tilde{r})$ are non-degenerate pairs and the properties of $S$ we can see that it suffices to prove that
\begin{equation}
(\epsilon\otimes X^2)\circ (c\circ R^{t_i}\otimes X)\circ (X\otimes \Delta)\quad\text{and}\quad (X^2\otimes\epsilon)\circ (X\otimes c\circ R^{t_i})\circ (\Delta\otimes X),\label{f1}
\end{equation}
where $i\in\{1,2\}$, are isomorphism. We prove that the arrows in~\eqref{f1} are isomorphisms when $i = 1$ and leave the case $i=2$, which is similar, to the reader. Since, by Definition~\ref{transposicones en variables}
$$
R^{t_1} = (\tau^{-1} \ot X) \circ (X\ot\Delta)\circ (X\ot \sigma)\circ (\Delta \ot X),
$$
we have $(\epsilon\ot X)\circ R^{t_1} = \tau^{-1}$ and $(X\ot \epsilon)\circ R^{t_1} = \sigma\circ (\tau^{-1}\ot X)\circ (X\ot \Delta)$. Hence, for $i = 1$, the arrows in~\eqref{f1} become
$$
(\tau^{-1}\ot X)\circ (X\ot\Delta)\quad\text{and}\quad (X\ot \sigma)\circ(X\ot\tau^{-1}\ot X)\circ (\Delta\ot\Delta),
$$
respectively. The first one is an isomorphism by Proposition~\ref{no degenerado}, while the second one also is by Lemma~\ref{le previo}.
\end{proof}

\begin{pro}\label{pro:YB} If $(X,r)$ is a non-degenerate braided set in $\mathscr{C}$, then so is $(Z,r_e)$.
\end{pro}

In order to prove this proposition we shall need the following two lemmas.

\begin{lem}\label{lem:XxXxSX} If $(X,r)$ is a non-degenerate braided set in $\mathscr{C}$, then
\begin{align*}
&(X\ot r) \circ (c\circ R^{t_1}\ot X)\circ  (X\ot c \circ  R^{t_1})= (c\circ R^{t_1}\ot X) \circ (X\ot c\circ R^{t_1}) \circ (r\ot X)
\shortintertext{and}
&(r\ot X) \circ (X\ot c\circ R^{t_2})\circ  (c \circ  R^{t_2}\ot X)= (X\ot c\circ R^{t_2}) \circ (c\circ R^{t_2}\ot X) \circ (X\ot r).
\end{align*}
\end{lem}

\begin{proof} By symmetry we only must prove the first equality. Let $F$ and $G$ be the endomorphisms of $X^3$ defined by
$$
F \coloneqq  (\tau\ot \tau \ot X) \circ (X\ot \sigma\ot X \ot \Delta)\circ (X\ot \Delta_{X^2})
$$
and
$$
G:= (c\ot X)\circ (X\ot c)\circ (X^2\ot\sigma) \circ (X^3\ot\sigma) \circ (\Delta_{X^2}\ot X),
$$
respectively. In order to prove the lemma it suffices to show that $F$ is an invertible map and that
\begin{equation}\label{pepe}
(c\circ R^{t_1}\ot X)\circ  (X\ot c \circ  R^{t_1}) \circ F = G
\end{equation}
and
\begin{equation}\label{pepe1}
(c\circ R^{t_1}\ot X) \circ (X\ot c\circ R^{t_1}) \circ (r\ot X) \circ F = (X\ot r)\circ G.
\end{equation}
Since $\Delta$ is coassociative and cocommutative, $\sigma$ is a coalgebra homomorphism,
\begin{equation}\label{eq1}
R^{t_1}\circ (\tau\ot X)\circ (X\ot \Delta) = (X\ot \sigma) \circ (\Delta \ot X),
\end{equation}
and $c$ is a natural isomorphism, we have
$$
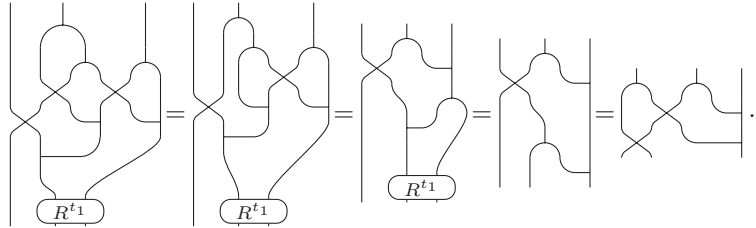
\begin{tikzpicture}[scale=0.395]
\def\counit(#1,#2){\draw (#1,#2) -- (#1,#2-0.93) (#1,#2-1) circle[radius=2pt]}
\def\comult(#1,#2)[#3,#4]{\draw (#1,#2) -- (#1,#2-0.5*#4) arc (90:0:0.5*#3 and 0.5*#4) (#1,#2-0.5*#4) arc (90:180:0.5*#3 and 0.5*#4)}
\def\laction(#1,#2)[#3,#4]{\draw (#1,#2) .. controls (#1,#2-0.555*#4/2) and (#1+0.445*#4/2,#2-1*#4/2) .. (#1+1*#4/2,#2-1*#4/2) -- (#1+2*#4/2+#3*#4/2,#2-1*#4/2) (#1+2*#4/2+#3*#4/2,#2)--(#1+2*#4/2+#3*#4/2,#2-2*#4/2)}
\def\map(#1,#2)[#3]{\draw (#1,#2-0.5)  node[name=nodemap,inner sep=0pt,  minimum size=12pt, shape=circle,draw]{$#3$} (#1,#2)-- (nodemap)  (nodemap)-- (#1,#2-1)}
\def\solbraid(#1,#2)[#3]{\draw (#1,#2-0.5)  node[name=nodemap,inner sep=0pt,  minimum size=9pt, shape=circle,draw]{$#3$}
(#1-0.5,#2) .. controls (#1-0.5,#2-0.15) and (#1-0.4,#2-0.2) .. (#1-0.3,#2-0.3) (#1-0.3,#2-0.3) -- (nodemap)
(#1+0.5,#2) .. controls (#1+0.5,#2-0.15) and (#1+0.4,#2-0.2) .. (#1+0.3,#2-0.3) (#1+0.3,#2-0.3) -- (nodemap)
(#1+0.5,#2-1) .. controls (#1+0.5,#2-0.85) and (#1+0.4,#2-0.8) .. (#1+0.3,#2-0.7) (#1+0.3,#2-0.7) -- (nodemap)
(#1-0.5,#2-1) .. controls (#1-0.5,#2-0.85) and (#1-0.4,#2-0.8) .. (#1-0.3,#2-0.7) (#1-0.3,#2-0.7) -- (nodemap)
}
\def\flip(#1,#2)[#3]{\draw (
#1+1*#3,#2) .. controls (#1+1*#3,#2-0.05*#3) and (#1+0.96*#3,#2-0.15*#3).. (#1+0.9*#3,#2-0.2*#3)
(#1+0.1*#3,#2-0.8*#3)--(#1+0.9*#3,#2-0.2*#3)
(#1,#2-1*#3) .. controls (#1,#2-0.95*#3) and (#1+0.04*#3,#2-0.85*#3).. (#1+0.1*#3,#2-0.8*#3)
(#1,#2) .. controls (#1,#2-0.05*#3) and (#1+0.04*#3,#2-0.15*#3).. (#1+0.1*#3,#2-0.2*#3)
(#1+0.1*#3,#2-0.2*#3) -- (#1+0.9*#3,#2-0.8*#3)
(#1+1*#3,#2-1*#3) .. controls (#1+1*#3,#2-0.95*#3) and (#1+0.96*#3,#2-0.85*#3).. (#1+0.9*#3,#2-0.8*#3)
}
\def\raction(#1,#2)[#3,#4]{\draw (#1,#2) -- (#1,#2-2*#4/2)  (#1,#2-1*#4/2)--(#1+1*#4/2+#3*#4/2,#2-1*#4/2) .. controls (#1+1.555*#4/2+#3*#4/2,#2-1*#4/2) and (#1+2*#4/2+#3*#4/2,#2-0.555*#4/2) .. (#1+2*#4/2+#3*#4/2,#2)}
\def\doublemap(#1,#2)[#3]{\draw (#1+0.5,#2-0.5) node [name=doublemapnode,inner xsep=0pt, inner ysep=0pt, minimum height=9pt, minimum width=25pt,shape=rectangle,draw,rounded corners] {$#3$} (#1,#2) .. controls (#1,#2-0.075) .. (doublemapnode) (#1+1,#2) .. controls (#1+1,#2-0.075).. (doublemapnode) (doublemapnode) .. controls (#1,#2-0.925)..(#1,#2-1) (doublemapnode) .. controls (#1+1,#2-0.925).. (#1+1,#2-1)}
\begin{scope}[xshift=0cm, yshift=0cm]
\draw (0,0) -- (0,-3.5); \comult(1.75,0)[1.5,1.5]; \comult(2.5,-1.5)[1,1]; \comult(4.5,-1.5)[1,1]; \draw (1,-1.5) -- (1,-2.5); \flip(1,-2.5)[1]; \flip(3,-2.5)[1]; \flip(0,-3.5)[1]; \laction(2,-3.5)[0,1]; \laction(4,-3.5)[0,1]; \raction(1,-4.5)[1,1.325]; \draw (1,-5.8) .. controls (1,-6.25) and (1.5,-6.25) .. (1.5,-6.5); \draw (5,-4.5) .. controls (5,-5) and (2.5,-6) .. (2.5,-6.5); \draw (5,-2.5) -- (5,-3.5); \doublemap(1.5,-6.5)[\scriptstyle R^{t_1}]; \draw (4.5,0) -- (4.5,-1.5); \draw (0,-4.5) -- (0,-7.5);
\end{scope}
\begin{scope}[xshift=5.5cm, yshift=-3.3cm]
\node at (0,-0.5){=};
\end{scope}
\begin{scope}[xshift=6.1cm, yshift=0cm]
\draw (0,0) -- (0,-3); \draw (1,-1) -- (1,-3); \flip(0,-3)[1]; \comult(1.5,0)[1,1]; \comult(2,-1)[1,1]; \comult(4,-1)[1,1]; \flip(2.5,-2)[1]; \draw (1.5,-2) -- (1.5,-3); \draw (4.5,-2) -- (4.5,-3); \laction(1.5,-3)[0,1]; \laction(3.5,-3)[0,1]; \raction(1,-4)[1,1]; \draw (4,0) -- (4,-1); \draw (1,-5) .. controls (1,-5.5) and (1.5,-6) .. (1.5,-6.5); \draw (4.5,-4) .. controls (4.5,-4.5) and (2.5,-6) .. (2.5,-6.5); \draw (0,-4) -- (0,-7.5); \doublemap(1.5,-6.5)[\scriptstyle R^{t_1}];
\end{scope}
\begin{scope}[xshift=11.1cm, yshift=-3.3cm]
\node at (0,-0.5){=};
\end{scope}
\begin{scope}[xshift=11.7cm, yshift=-0.7cm]
\draw (0,0) -- (0,-1); \comult(1.5,0)[1,1]; \draw (3,0) -- (3,-1); \laction(2,-1)[0,1]; \flip(0,-1)[1]; \comult(3,-2)[1,1]; \draw (1,-2) .. controls (1,-2.5) and (1.5,-2.5) .. (1.5,-3); \raction(1.5,-3)[0,1]; \doublemap(1.5,-5)[\scriptstyle R^{t_1}]; \draw (0,-2) -- (0,-6); \draw (1.5,-4) -- (1.5,-5); \draw (3.5,-3) .. controls (3.5,-3.5) and (2.5,-4.5) .. (2.5,-5);
\end{scope}
\begin{scope}[xshift=15.7cm, yshift=-3.3cm]
\node at (0,-0.5){=};
\end{scope}
\begin{scope}[xshift=16.3cm, yshift=-1.2cm]
\draw (0,0) -- (0,-1); \comult(1.5,0)[1,1]; \draw (3,0) -- (3,-1); \laction(2,-1)[0,1]; \flip(0,-1)[1]; \comult(1.5,-3)[1,1]; \draw (1,-2) .. controls (1,-2.5) and (1.5,-2.5) .. (1.5,-3); \draw (0,-2) -- (0,-5); \laction(2,-4)[0,1]; \draw (3,-2) -- (3,-4); \draw (1,-4) -- (1,-5);
\end{scope}
\begin{scope}[xshift=19.8cm, yshift=-3.3cm]
\node at (0,-0.5){=};
\end{scope}
\begin{scope}[xshift=20.35cm, yshift=-2.2cm]
\draw (0,-1) -- (0,-2); \comult(0.5,0)[1,1];  \comult(2.5,0)[1,1]; \laction(3,-1)[0,1]; \draw (4,0) -- (4,-1); \flip(1,-1)[1]; \flip(0,-2)[1]; \laction(2,-2)[2,1];
\end{scope}
\begin{scope}[xshift=24.7cm, yshift=-3.3cm]
\node at (0,-0.5){.};
\end{scope}
\end{tikzpicture}
$$
Using this, equality~\eqref{eq1} and that $c$ is a natural morphism, we obtain
$$
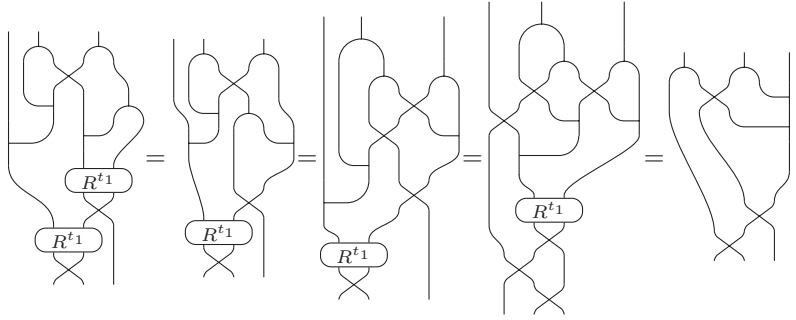
\begin{tikzpicture}[scale=0.395]
\def\counit(#1,#2){\draw (#1,#2) -- (#1,#2-0.93) (#1,#2-1) circle[radius=2pt]}
\def\comult(#1,#2)[#3,#4]{\draw (#1,#2) -- (#1,#2-0.5*#4) arc (90:0:0.5*#3 and 0.5*#4) (#1,#2-0.5*#4) arc (90:180:0.5*#3 and 0.5*#4)}
\def\laction(#1,#2)[#3,#4]{\draw (#1,#2) .. controls (#1,#2-0.555*#4/2) and (#1+0.445*#4/2,#2-1*#4/2) .. (#1+1*#4/2,#2-1*#4/2) -- (#1+2*#4/2+#3*#4/2,#2-1*#4/2) (#1+2*#4/2+#3*#4/2,#2)--(#1+2*#4/2+#3*#4/2,#2-2*#4/2)}
\def\map(#1,#2)[#3]{\draw (#1,#2-0.5)  node[name=nodemap,inner sep=0pt,  minimum size=10pt, shape=circle,draw]{$#3$} (#1,#2)-- (nodemap)  (nodemap)-- (#1,#2-1)}
\def\solbraid(#1,#2)[#3]{\draw (#1,#2-0.5)  node[name=nodemap,inner sep=0pt,  minimum size=9pt, shape=circle,draw]{$#3$}
(#1-0.5,#2) .. controls (#1-0.5,#2-0.15) and (#1-0.4,#2-0.2) .. (#1-0.3,#2-0.3) (#1-0.3,#2-0.3) -- (nodemap)
(#1+0.5,#2) .. controls (#1+0.5,#2-0.15) and (#1+0.4,#2-0.2) .. (#1+0.3,#2-0.3) (#1+0.3,#2-0.3) -- (nodemap)
(#1+0.5,#2-1) .. controls (#1+0.5,#2-0.85) and (#1+0.4,#2-0.8) .. (#1+0.3,#2-0.7) (#1+0.3,#2-0.7) -- (nodemap)
(#1-0.5,#2-1) .. controls (#1-0.5,#2-0.85) and (#1-0.4,#2-0.8) .. (#1-0.3,#2-0.7) (#1-0.3,#2-0.7) -- (nodemap)
}
\def\flip(#1,#2)[#3]{\draw (
#1+1*#3,#2) .. controls (#1+1*#3,#2-0.05*#3) and (#1+0.96*#3,#2-0.15*#3).. (#1+0.9*#3,#2-0.2*#3)
(#1+0.1*#3,#2-0.8*#3)--(#1+0.9*#3,#2-0.2*#3)
(#1,#2-1*#3) .. controls (#1,#2-0.95*#3) and (#1+0.04*#3,#2-0.85*#3).. (#1+0.1*#3,#2-0.8*#3)
(#1,#2) .. controls (#1,#2-0.05*#3) and (#1+0.04*#3,#2-0.15*#3).. (#1+0.1*#3,#2-0.2*#3)
(#1+0.1*#3,#2-0.2*#3) -- (#1+0.9*#3,#2-0.8*#3)
(#1+1*#3,#2-1*#3) .. controls (#1+1*#3,#2-0.95*#3) and (#1+0.96*#3,#2-0.85*#3).. (#1+0.9*#3,#2-0.8*#3)
}
\def\raction(#1,#2)[#3,#4]{\draw (#1,#2) -- (#1,#2-2*#4/2)  (#1,#2-1*#4/2)--(#1+1*#4/2+#3*#4/2,#2-1*#4/2) .. controls (#1+1.555*#4/2+#3*#4/2,#2-1*#4/2) and (#1+2*#4/2+#3*#4/2,#2-0.555*#4/2) .. (#1+2*#4/2+#3*#4/2,#2)}
\def\doublemap(#1,#2)[#3]{\draw (#1+0.5,#2-0.5) node [name=doublemapnode,inner xsep=0pt, inner ysep=0pt, minimum height=9pt, minimum width=25pt,shape=rectangle,draw,rounded corners] {$#3$} (#1,#2) .. controls (#1,#2-0.075) .. (doublemapnode) (#1+1,#2) .. controls (#1+1,#2-0.075).. (doublemapnode) (doublemapnode) .. controls (#1,#2-0.925)..(#1,#2-1) (doublemapnode) .. controls (#1+1,#2-0.925).. (#1+1,#2-1)}
\begin{scope}[xshift=0cm, yshift=-1cm]
\draw (0,0) -- (0,-3); \comult(1,0)[1,1]; \comult(3,0)[1,1];  \draw (0.5,-1) -- (0.5,-2); \flip(1.5,-1)[1]; \laction(0.5,-2)[0,1]; \draw (3.5,-1) .. controls (3.5,-1.5) and (4,-1.5) .. (4,-2); \comult(4,-2)[1,1]; \raction(2.5,-3)[0,1]; \draw (2.5,-2) -- (2.5,-3); \raction(0,-3)[0,1.5]; \draw (2.5,-4) -- (2.5,-4.5); \draw (4.5,-3) .. controls (4.5,-3.5) and (3.5,-4) .. (3.5,-4.5); \doublemap(2.5,-4.5)[\scriptstyle R^{t_1}]; \flip(2.5,-5.5)[1]; \draw (0,-4.5) .. controls (0,-5.5) and (1.5,-5.5) .. (1.5,-6.5); \doublemap(1.5,-6.5)[\scriptstyle R^{t_1}]; \flip(1.5,-7.5)[1]; \draw (3.5,-6.5) -- (3.5,-8.5);
\end{scope}
\begin{scope}[xshift=4.9cm, yshift=-4.8cm]
\node at (0,-0.5){=};
\end{scope}
\begin{scope}[xshift=5.5cm, yshift=-1.25cm]
\draw (0,0) -- (0,-2); \draw (0,-2) .. controls (0,-2.5) and (0.5,-2.5) .. (0.5,-3); \comult(1,0)[1,1]; \comult(3,0)[1,1];  \draw (0.5,-1) -- (0.5,-2); \flip(1.5,-1)[1]; \laction(0.5,-2)[0,1]; \comult(2.5,-2)[1,1]; \draw (3.5,-1) -- (3.5,-2); \draw (3.5,-2) .. controls (3.5,-2.5) and (4,-2.5) .. (4,-3); \laction(3,-3)[0,1]; \raction(0.5,-3)[0,1]; \draw (4,-4) .. controls (4,-4.5) and (3,-4.5) .. (3,-5); \draw (2,-3) -- (2,-5); \flip(2,-5)[1]; \draw (0.5,-4) .. controls (0.5,-4.5) and (1,-5.5) .. (1,-6); \doublemap(1,-6)[\scriptstyle R^{t_1}]; \flip(1,-7)[1]; \draw (3,-6) -- (3,-8);
\end{scope}
\begin{scope}[xshift=9.95cm, yshift=-4.8cm]
\node at (0,-0.5){=};
\end{scope}
\begin{scope}[xshift=10.5cm, yshift=-0.5cm]
\draw (0,0) -- (0,-5.5); \comult(1.25,0)[1.5,1.5]; \draw (0.5,-1.5) -- (0.5,-4.5); \comult(2,-1.5)[1,1]; \comult(4,-1.5)[1,1]; \flip(2.5,-2.5)[1]; \laction(3.5,-3.5)[0,1]; \draw (4,0) -- (4,-1.5); \draw (4.5,-2.5) -- (4.5,-3.5); \draw (1.5,-2.5) -- (1.5,-3.5); \flip(1.5,-3.5)[1]; \laction(0.5,-4.5)[0,1]; \draw (4.5,-4.5) .. controls (4.5,-5) and (3.5,-5) .. (3.5,-5.5); \draw (2.5,-4.5) -- (2.5,-5.5); \flip(2.5,-5.5)[1]; \raction(0,-5.5)[0,1.5]; \draw (0,-7) .. controls (0,-7.25) and (0.5,-7.25) .. (0.5,-7.5); \draw (2.5,-6.5) .. controls (2.5,-7) and (1.5,-7) .. (1.5,-7.5); \doublemap(0.5,-7.5)[\scriptstyle R^{t_1}];  \flip(0.5,-8.5)[1]; \draw (3.5,-6.5) -- (3.5,-9.5);
\end{scope}
\begin{scope}[xshift=15.45cm, yshift=-4.8cm]
\node at (0,-0.5){=};
\end{scope}
\begin{scope}[xshift=16cm, yshift=0cm]
\draw (0,0) -- (0,-3.5); \comult(1.75,0)[1.5,1.5]; \comult(2.5,-1.5)[1,1]; \comult(4.5,-1.5)[1,1]; \draw (1,-1.5) -- (1,-2.5); \flip(1,-2.5)[1]; \flip(3,-2.5)[1]; \flip(0,-3.5)[1]; \laction(2,-3.5)[0,1]; \laction(4,-3.5)[0,1]; \raction(1,-4.5)[1,1.325]; \draw (1,-5.8) .. controls (1,-6.25) and (1.5,-6.25) .. (1.5,-6.5); \draw (5,-4.5) .. controls (5,-5) and (2.5,-6) .. (2.5,-6.5); \draw (5,-2.5) -- (5,-3.5); \doublemap(1.5,-6.5)[\scriptstyle R^{t_1}]; \draw (4.5,0) -- (4.5,-1.5); \draw (0,-4.5) -- (0,-7.5); \flip(1.5,-7.5)[1]; \draw (0,-7.5) .. controls (0,-8) and (0.5,-8) .. (0.5,-8.5);  \flip(0.5,-8.5)[1]; \flip(1.5,-9.5)[1]; \draw (2.5,-8.5) -- (2.5,-9.5); \draw (0.5,-9.5) -- (0.5,-10.5);
\end{scope}
\begin{scope}[xshift=21.5cm, yshift=-4.8cm]
\node at (0,-0.5){=};
\end{scope}
\begin{scope}[xshift=22cm, yshift=-1.7cm]
\draw (0,-1) -- (0,-2); \comult(0.5,0)[1,1];  \comult(2.5,0)[1,1]; \laction(3,-1)[0,1]; \draw (4,0) -- (4,-1); \flip(1,-1)[1]; \flip(2.5,-5)[1]; \flip(1.5,-6)[1]; \laction(2,-2)[2,1]; \draw (4,-3) -- (4,-4); \draw (1,-2) .. controls (1,-3) and (2.5,-4) .. (2.5,-5); \draw (0,-2) .. controls (0,-3) and (1.5,-5) .. (1.5,-6); \draw (4,-4) .. controls (4,-4.5) and (3.5,-4.5) .. (3.5,-5); \draw (3.5,-6) -- (3.5,-7);
 \end{scope}
\begin{scope}[xshift=26.3cm, yshift=-4.8cm]
\node at (0,-0.5){,};
\end{scope}
\end{tikzpicture}
$$
which proves equality~\eqref{pepe} and  that $F$ is invertible, since $(c\circ R^{t_1}\ot X)\circ (X\ot c\circ R^{t_1})$ and $G$ are invertible morphisms. On the other hand, by the coassociativity of~$\Delta$ and Remark~\ref{prop de involutiva y de braided} we have
$$
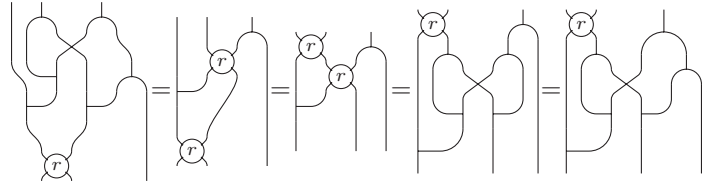
\begin{tikzpicture}[scale=0.395]
\def\counit(#1,#2){\draw (#1,#2) -- (#1,#2-0.93) (#1,#2-1) circle[radius=2pt]}
\def\comult(#1,#2)[#3,#4]{\draw (#1,#2) -- (#1,#2-0.5*#4) arc (90:0:0.5*#3 and 0.5*#4) (#1,#2-0.5*#4) arc (90:180:0.5*#3 and 0.5*#4)}
\def\laction(#1,#2)[#3,#4]{\draw (#1,#2) .. controls (#1,#2-0.555*#4/2) and (#1+0.445*#4/2,#2-1*#4/2) .. (#1+1*#4/2,#2-1*#4/2) -- (#1+2*#4/2+#3*#4/2,#2-1*#4/2) (#1+2*#4/2+#3*#4/2,#2)--(#1+2*#4/2+#3*#4/2,#2-2*#4/2)}
\def\map(#1,#2)[#3]{\draw (#1,#2-0.5)  node[name=nodemap,inner sep=0pt,  minimum size=12pt, shape=circle,draw]{$#3$} (#1,#2)-- (nodemap)  (nodemap)-- (#1,#2-1)}
\def\solbraid(#1,#2)[#3]{\draw (#1,#2-0.5)  node[name=nodemap,inner sep=0pt,  minimum size=9pt, shape=circle,draw]{$#3$}
(#1-0.5,#2) .. controls (#1-0.5,#2-0.15) and (#1-0.4,#2-0.2) .. (#1-0.3,#2-0.3) (#1-0.3,#2-0.3) -- (nodemap)
(#1+0.5,#2) .. controls (#1+0.5,#2-0.15) and (#1+0.4,#2-0.2) .. (#1+0.3,#2-0.3) (#1+0.3,#2-0.3) -- (nodemap)
(#1+0.5,#2-1) .. controls (#1+0.5,#2-0.85) and (#1+0.4,#2-0.8) .. (#1+0.3,#2-0.7) (#1+0.3,#2-0.7) -- (nodemap)
(#1-0.5,#2-1) .. controls (#1-0.5,#2-0.85) and (#1-0.4,#2-0.8) .. (#1-0.3,#2-0.7) (#1-0.3,#2-0.7) -- (nodemap)
}
\def\flip(#1,#2)[#3]{\draw (
#1+1*#3,#2) .. controls (#1+1*#3,#2-0.05*#3) and (#1+0.96*#3,#2-0.15*#3).. (#1+0.9*#3,#2-0.2*#3)
(#1+0.1*#3,#2-0.8*#3)--(#1+0.9*#3,#2-0.2*#3)
(#1,#2-1*#3) .. controls (#1,#2-0.95*#3) and (#1+0.04*#3,#2-0.85*#3).. (#1+0.1*#3,#2-0.8*#3)
(#1,#2) .. controls (#1,#2-0.05*#3) and (#1+0.04*#3,#2-0.15*#3).. (#1+0.1*#3,#2-0.2*#3)
(#1+0.1*#3,#2-0.2*#3) -- (#1+0.9*#3,#2-0.8*#3)
(#1+1*#3,#2-1*#3) .. controls (#1+1*#3,#2-0.95*#3) and (#1+0.96*#3,#2-0.85*#3).. (#1+0.9*#3,#2-0.8*#3)
}
\def\raction(#1,#2)[#3,#4]{\draw (#1,#2) -- (#1,#2-2*#4/2)  (#1,#2-1*#4/2)--(#1+1*#4/2+#3*#4/2,#2-1*#4/2) .. controls (#1+1.555*#4/2+#3*#4/2,#2-1*#4/2) and (#1+2*#4/2+#3*#4/2,#2-0.555*#4/2) .. (#1+2*#4/2+#3*#4/2,#2)}
\def\doublemap(#1,#2)[#3]{\draw (#1+0.5,#2-0.5) node [name=doublemapnode,inner xsep=0pt, inner ysep=0pt, minimum height=11pt, minimum width=25pt,shape=rectangle,draw,rounded corners] {$#3$} (#1,#2) .. controls (#1,#2-0.075) .. (doublemapnode) (#1+1,#2) .. controls (#1+1,#2-0.075).. (doublemapnode) (doublemapnode) .. controls (#1,#2-0.925)..(#1,#2-1) (doublemapnode) .. controls (#1+1,#2-0.925).. (#1+1,#2-1)}
\begin{scope}[xshift=0cm, yshift=0cm]
\draw (0,0) -- (0,-2); \draw (0,-2) .. controls (0,-2.5) and (0.5,-2.5) .. (0.5,-3); \comult(1,0)[1,1]; \comult(3,0)[1,1];  \draw (0.5,-1) -- (0.5,-2); \flip(1.5,-1)[1]; \laction(0.5,-2)[0,1]; \draw (3.5,-1) .. controls (3.5,-1.5) and (4,-1.5) .. (4,-2); \comult(4,-2)[1,1]; \raction(2.5,-3)[0,1]; \draw (2.5,-2) -- (2.5,-3); \raction(0.5,-3)[0,1]; \draw (2.5,-4) .. controls (2.5,-4.5) and (2,-4.5) .. (2,-5); \draw (0.5,-4) .. controls (0.5,-4.5) and (1,-4.5) .. (1,-5); \solbraid(1.5,-5)[\scriptstyle r]; \draw (4.5,-3) -- (4.5,-6);
\end{scope}
\begin{scope}[xshift=4.98cm, yshift=-2.5cm]
\node at (0,-0.5){=};
\end{scope}
\begin{scope}[xshift=5.51cm, yshift=-0.5cm]
\draw (0,0) -- (0,-2); \draw (1,0) -- (1,-1); \comult(2.5,0)[1,1]; \solbraid(1.5,-1)[\scriptstyle r]; \raction(0,-2)[0,1]; \draw (2,-2) .. controls (2,-2.5) and (1,-3.5) .. (1,-4); \solbraid(0.5,-4)[\scriptstyle r]; \draw (0,-3) -- (0,-4); \draw (3,-1) -- (3,-5);
\end{scope}
\begin{scope}[xshift=8.96cm, yshift=-2.5cm]
\node at (0,-0.5){=};
\end{scope}
\begin{scope}[xshift=9.47cm, yshift=-1cm]
\comult(2.5,0)[1,1]; \solbraid(0.5,0)[\scriptstyle r]; \solbraid(1.5,-1)[\scriptstyle r];\draw (3,-1) -- (3,-4); \draw (2,-2) -- (2,-4);
\raction(0,-2)[0,1];\draw (0,-1) -- (0,-2);\draw (0,-3) -- (0,-4);
\end{scope}
\begin{scope}[xshift=12.97cm, yshift=-2.5cm]
\node at (0,-0.5){=};
\end{scope}
\begin{scope}[xshift=13.53cm, yshift=-0.25cm]
\comult(3.5,0)[1,1]; \solbraid(0.5,0)[\scriptstyle r]; \comult(1,-1)[1,1]; \comult(3,-1)[1,1]; \flip(1.5,-2)[1]; \draw (0.5,-2) -- (0.5,-3); \draw (3.5,-2) -- (3.5,-3); \laction(0.5,-3)[0,1]; \raction(2.5,-3)[0,1]; \draw (0,-1) -- (0,-4); \raction(0,-4)[0,1.5]; \draw (2.5,-4) -- (2.5,-5.5); \draw (4,-1) -- (4,-5.5);
\end{scope}
\begin{scope}[xshift=17.97cm, yshift=-2.5cm]
\node at (0,-0.5){=};
\end{scope}
\begin{scope}[xshift=18.47cm, yshift=-0.25cm]
\comult(3.25,0)[1.5,1.5]; \solbraid(0.5,0)[\scriptstyle r]; \comult(1,-1)[1,1]; \comult(4,-1.5)[1,1]; \flip(1.5,-2)[1]; \draw (0.5,-2) -- (0.5,-3); \draw (3.5,-2.5) -- (3.5,-3); \laction(0.5,-3)[0,1]; \raction(2.5,-3)[0,1]; \draw (0,-1) -- (0,-4); \raction(0,-4)[0,1.5]; \draw (2.5,-4) -- (2.5,-5.5); \draw (4.5,-2.5) -- (4.5,-5.5); \draw (2.5,-1.5) -- (2.5,-2);
\end{scope}
\begin{scope}[xshift=23.22cm, yshift=-2.5cm]
\node at (0,-0.5){.};
\end{scope}
\end{tikzpicture}
$$
Using this, equality~\eqref{pepe}, that $r$ is a coalgebra morphism, that $c$ is natural and Re\-mark~\ref{prop de involutiva y de braided}, we obtain
$$
\begin{tikzpicture}[scale=0.395]
\def\counit(#1,#2){\draw (#1,#2) -- (#1,#2-0.93) (#1,#2-1) circle[radius=2pt]}
\def\comult(#1,#2)[#3,#4]{\draw (#1,#2) -- (#1,#2-0.5*#4) arc (90:0:0.5*#3 and 0.5*#4) (#1,#2-0.5*#4) arc (90:180:0.5*#3 and 0.5*#4)}
\def\laction(#1,#2)[#3,#4]{\draw (#1,#2) .. controls (#1,#2-0.555*#4/2) and (#1+0.445*#4/2,#2-1*#4/2) .. (#1+1*#4/2,#2-1*#4/2) -- (#1+2*#4/2+#3*#4/2,#2-1*#4/2) (#1+2*#4/2+#3*#4/2,#2)--(#1+2*#4/2+#3*#4/2,#2-2*#4/2)}
\def\map(#1,#2)[#3]{\draw (#1,#2-0.5)  node[name=nodemap,inner sep=0pt,  minimum size=10pt, shape=circle,draw]{$#3$} (#1,#2)-- (nodemap)  (nodemap)-- (#1,#2-1)}
\def\solbraid(#1,#2)[#3]{\draw (#1,#2-0.5)  node[name=nodemap,inner sep=0pt,  minimum size=9pt, shape=circle,draw]{$#3$}
(#1-0.5,#2) .. controls (#1-0.5,#2-0.15) and (#1-0.4,#2-0.2) .. (#1-0.3,#2-0.3) (#1-0.3,#2-0.3) -- (nodemap)
(#1+0.5,#2) .. controls (#1+0.5,#2-0.15) and (#1+0.4,#2-0.2) .. (#1+0.3,#2-0.3) (#1+0.3,#2-0.3) -- (nodemap)
(#1+0.5,#2-1) .. controls (#1+0.5,#2-0.85) and (#1+0.4,#2-0.8) .. (#1+0.3,#2-0.7) (#1+0.3,#2-0.7) -- (nodemap)
(#1-0.5,#2-1) .. controls (#1-0.5,#2-0.85) and (#1-0.4,#2-0.8) .. (#1-0.3,#2-0.7) (#1-0.3,#2-0.7) -- (nodemap)
}
\def\flip(#1,#2)[#3]{\draw (
#1+1*#3,#2) .. controls (#1+1*#3,#2-0.05*#3) and (#1+0.96*#3,#2-0.15*#3).. (#1+0.9*#3,#2-0.2*#3)
(#1+0.1*#3,#2-0.8*#3)--(#1+0.9*#3,#2-0.2*#3)
(#1,#2-1*#3) .. controls (#1,#2-0.95*#3) and (#1+0.04*#3,#2-0.85*#3).. (#1+0.1*#3,#2-0.8*#3)
(#1,#2) .. controls (#1,#2-0.05*#3) and (#1+0.04*#3,#2-0.15*#3).. (#1+0.1*#3,#2-0.2*#3)
(#1+0.1*#3,#2-0.2*#3) -- (#1+0.9*#3,#2-0.8*#3)
(#1+1*#3,#2-1*#3) .. controls (#1+1*#3,#2-0.95*#3) and (#1+0.96*#3,#2-0.85*#3).. (#1+0.9*#3,#2-0.8*#3)
}
\def\raction(#1,#2)[#3,#4]{\draw (#1,#2) -- (#1,#2-2*#4/2)  (#1,#2-1*#4/2)--(#1+1*#4/2+#3*#4/2,#2-1*#4/2) .. controls (#1+1.555*#4/2+#3*#4/2,#2-1*#4/2) and (#1+2*#4/2+#3*#4/2,#2-0.555*#4/2) .. (#1+2*#4/2+#3*#4/2,#2)}
\def\doublemap(#1,#2)[#3]{\draw (#1+0.5,#2-0.5) node [name=doublemapnode,inner xsep=0pt, inner ysep=0pt, minimum height=9pt, minimum width=25pt,shape=rectangle,draw,rounded corners] {$#3$} (#1,#2) .. controls (#1,#2-0.075) .. (doublemapnode) (#1+1,#2) .. controls (#1+1,#2-0.075).. (doublemapnode) (doublemapnode) .. controls (#1,#2-0.925)..(#1,#2-1) (doublemapnode) .. controls (#1+1,#2-0.925).. (#1+1,#2-1)}
\begin{scope}[xshift=0cm, yshift=0cm]
\draw (0,0) -- (0,-2); \draw (0,-2) .. controls (0,-2.5) and (0.5,-2.5) .. (0.5,-3); \comult(1,0)[1,1]; \comult(3,0)[1,1];  \draw (0.5,-1) -- (0.5,-2); \flip(1.5,-1)[1]; \laction(0.5,-2)[0,1]; \draw (3.5,-1) .. controls (3.5,-1.5) and (4,-1.5) .. (4,-2); \comult(4,-2)[1,1]; \raction(2.5,-3)[0,1]; \draw (2.5,-2) -- (2.5,-3); \raction(0.5,-3)[0,1]; \draw (2.5,-4) .. controls (2.5,-4.5) and (2,-4.5) .. (2,-5); \draw (0.5,-4) .. controls (0.5,-4.5) and (1,-4.5) .. (1,-5); \solbraid(1.5,-5)[\scriptstyle r];\draw (4.5,-3) -- (4.5,-4); \draw (4.5,-4) .. controls (4.5,-5) and (3,-5) .. (3,-6); \doublemap(2,-6)[\scriptstyle R^{t_1}]; \flip(2,-7)[1]; \doublemap(1,-8)[\scriptstyle R^{t_1}]; \flip(1,-9)[1]; \draw (1,-6) -- (1,-8); \draw (3,-8) -- (3,-10);
\end{scope}
\begin{scope}[xshift=5cm, yshift=-4.5cm]
\node at (0,-0.5){=};
\end{scope}
\begin{scope}[xshift=5.6cm, yshift=-1.2cm]
\solbraid(0.5,1)[\scriptstyle r]; \draw (0,0) -- (0,-3); \comult(1,0)[1,1]; \comult(3,0)[1,1];  \draw (0.5,-1) -- (0.5,-2); \flip(1.5,-1)[1]; \laction(0.5,-2)[0,1]; \draw (3.5,-1) .. controls (3.5,-1.5) and (4,-1.5) .. (4,-2); \comult(4,-2)[1,1]; \raction(2.5,-3)[0,1]; \draw (2.5,-2) -- (2.5,-3); \raction(0,-3)[0,1.5]; \draw (2.5,-4) -- (2.5,-4.5); \draw (4.5,-3) .. controls (4.5,-3.5) and (3.5,-4) .. (3.5,-4.5); \doublemap(2.5,-4.5)[\scriptstyle R^{t_1}]; \flip(2.5,-5.5)[1]; \draw (0,-4.5) .. controls (0,-5.5) and (1.5,-5.5) .. (1.5,-6.5); \doublemap(1.5,-6.5)[\scriptstyle R^{t_1}]; \flip(1.5,-7.5)[1]; \draw (3.5,-6.5) -- (3.5,-8.5);
\end{scope}
\begin{scope}[xshift=10.6cm, yshift=-4.5cm]
\node at (0,-0.5){=};
\end{scope}
\begin{scope}[xshift=11.2cm, yshift=-2.5cm]
\solbraid(1.5,2)[\scriptstyle r]; \draw (1,1) .. controls (1,0.5) and (0.5,0.5) .. (0.5,0);
\draw (2,1) .. controls (2,0.5) and (2.5,0.5) .. (2.5,0); \draw (4,2) -- (4,0);
\draw (0,-1) -- (0,-2); \comult(0.5,0)[1,1];  \comult(2.5,0)[1,1]; \laction(3,-1)[0,1]; \draw (4,0) -- (4,-1); \flip(1,-1)[1]; \flip(2.5,-5)[1]; \flip(1.5,-6)[1]; \laction(2,-2)[2,1]; \draw (4,-3) -- (4,-4); \draw (1,-2) .. controls (1,-3) and (2.5,-4) .. (2.5,-5); \draw (0,-2) .. controls (0,-3) and (1.5,-5) .. (1.5,-6); \draw (4,-4) .. controls (4,-4.5) and (3.5,-4.5) .. (3.5,-5); \draw (3.5,-6) -- (3.5,-7);
\end{scope}
\begin{scope}[xshift=15.7cm, yshift=-4.5cm]
\node at (0,-0.5){=};
\end{scope}
\begin{scope}[xshift=16.3cm, yshift=-1cm]
\comult(0.5,0)[1,1]; \comult(2.5,0)[1,1]; \draw (0,-1) -- (0,-2); \draw (3,-1) -- (3,-2); \flip(1,-1)[1]; \solbraid(0.5,-2)[\scriptstyle r]; \solbraid(2.5,-2)[\scriptstyle r]; \draw (4,-0) -- (4,-3); \laction(3,-3)[0,1]; \laction(2,-4)[2,1]; \draw (2,-3) -- (2,-4); \draw (1,-3) -- (1,-4); \draw (0,-3) -- (0,-4); \draw (1,-4) .. controls (1,-5) and (2.5,-5) .. (2.5,-6); \draw (4,-5) .. controls (4,-5.5) and (3.5,-5.5) .. (3.5,-6); \flip(2.5,-6)[1]; \flip(1.5,-7)[1]; \draw (3.5,-7) -- (3.5,-8); \draw (0,-4) .. controls (0,-5) and (1.5,-6) .. (1.5,-7);
\end{scope}
\begin{scope}[xshift=20.8cm, yshift=-4.5cm]
\node at (0,-0.5){=};
\end{scope}
\begin{scope}[xshift=21.4cm, yshift=-1.1cm]
\draw (0,-1) -- (0,-2); \comult(0.5,0)[1,1];  \comult(2.5,0)[1,1]; \laction(3,-1)[0,1]; \draw (4,0) -- (4,-1); \flip(1,-1)[1]; \flip(2.5,-5)[1]; \flip(1.5,-6)[1]; \laction(2,-2)[2,1]; \draw (4,-3) -- (4,-4); \draw (1,-2) .. controls (1,-3) and (2.5,-4) .. (2.5,-5); \draw (0,-2) .. controls (0,-3) and (1.5,-5) .. (1.5,-6); \draw (4,-4) .. controls (4,-4.5) and (3.5,-4.5) .. (3.5,-5); \draw (3.5,-6) -- (3.5,-7); \solbraid(3,-7)[\scriptstyle r]; \draw (1.5,-7) -- (1.5,-8);
\end{scope}
\begin{scope}[xshift=25.6cm, yshift=-4.5cm]
\node at (0,-0.5){,};
\end{scope}
\end{tikzpicture}
$$
which proves equality~\eqref{pepe1}.
\end{proof}

\begin{lem}\label{lem:XxSXxX} If $(X,r)$ is a non-degenerate braided set in $\mathscr{C}$, then
\[
(c\circ R^{t_2}\ot X) \circ (X\ot r)\circ  (c \circ  R^{t_1}\ot X) = (X \ot c\circ R^{t_1}) \circ (r\ot X) \circ (X\ot c\circ R^{t_2}).
\]
\end{lem}

\begin{proof} Let $G$ be the automorphism of $X^3$ defined by
$$
G \coloneqq  (\tau \ot X^2) \circ (X\ot \Delta \ot X)\circ (X^2\ot \sigma)\circ (X\ot \Delta \ot X).
$$
In order to prove the lemma it suffices to show that
$$
(c\circ R^{t_2}\ot X) \circ (X\ot r)\circ  (c \circ  R^{t_1}\ot X) \circ G = (X \ot c\circ R^{t_1}) \circ (r\ot X) \circ (X\ot c\circ R^{t_2}) \circ G.
$$
By equality~\eqref{eq1} and the fact that $c$ is natural, $\Delta$ is coassociative and cocommutative, and $\sigma$ is a coalgebra morphism, we have
$$
\begin{tikzpicture}[scale=0.395]
\def\counit(#1,#2){\draw (#1,#2) -- (#1,#2-0.93) (#1,#2-1) circle[radius=2pt]}
\def\comult(#1,#2)[#3,#4]{\draw (#1,#2) -- (#1,#2-0.5*#4) arc (90:0:0.5*#3 and 0.5*#4) (#1,#2-0.5*#4) arc (90:180:0.5*#3 and 0.5*#4)}
\def\laction(#1,#2)[#3,#4]{\draw (#1,#2) .. controls (#1,#2-0.555*#4/2) and (#1+0.445*#4/2,#2-1*#4/2) .. (#1+1*#4/2,#2-1*#4/2) -- (#1+2*#4/2+#3*#4/2,#2-1*#4/2) (#1+2*#4/2+#3*#4/2,#2)--(#1+2*#4/2+#3*#4/2,#2-2*#4/2)}
\def\map(#1,#2)[#3]{\draw (#1,#2-0.5)  node[name=nodemap,inner sep=0pt,  minimum size=10pt, shape=circle,draw]{$#3$} (#1,#2)-- (nodemap)  (nodemap)-- (#1,#2-1)}
\def\solbraid(#1,#2)[#3]{\draw (#1,#2-0.5)  node[name=nodemap,inner sep=0pt,  minimum size=9pt, shape=circle,draw]{$#3$}
(#1-0.5,#2) .. controls (#1-0.5,#2-0.15) and (#1-0.4,#2-0.2) .. (#1-0.3,#2-0.3) (#1-0.3,#2-0.3) -- (nodemap)
(#1+0.5,#2) .. controls (#1+0.5,#2-0.15) and (#1+0.4,#2-0.2) .. (#1+0.3,#2-0.3) (#1+0.3,#2-0.3) -- (nodemap)
(#1+0.5,#2-1) .. controls (#1+0.5,#2-0.85) and (#1+0.4,#2-0.8) .. (#1+0.3,#2-0.7) (#1+0.3,#2-0.7) -- (nodemap)
(#1-0.5,#2-1) .. controls (#1-0.5,#2-0.85) and (#1-0.4,#2-0.8) .. (#1-0.3,#2-0.7) (#1-0.3,#2-0.7) -- (nodemap)
}
\def\flip(#1,#2)[#3]{\draw (
#1+1*#3,#2) .. controls (#1+1*#3,#2-0.05*#3) and (#1+0.96*#3,#2-0.15*#3).. (#1+0.9*#3,#2-0.2*#3)
(#1+0.1*#3,#2-0.8*#3)--(#1+0.9*#3,#2-0.2*#3)
(#1,#2-1*#3) .. controls (#1,#2-0.95*#3) and (#1+0.04*#3,#2-0.85*#3).. (#1+0.1*#3,#2-0.8*#3)
(#1,#2) .. controls (#1,#2-0.05*#3) and (#1+0.04*#3,#2-0.15*#3).. (#1+0.1*#3,#2-0.2*#3)
(#1+0.1*#3,#2-0.2*#3) -- (#1+0.9*#3,#2-0.8*#3)
(#1+1*#3,#2-1*#3) .. controls (#1+1*#3,#2-0.95*#3) and (#1+0.96*#3,#2-0.85*#3).. (#1+0.9*#3,#2-0.8*#3)
}
\def\raction(#1,#2)[#3,#4]{\draw (#1,#2) -- (#1,#2-2*#4/2)  (#1,#2-1*#4/2)--(#1+1*#4/2+#3*#4/2,#2-1*#4/2) .. controls (#1+1.555*#4/2+#3*#4/2,#2-1*#4/2) and (#1+2*#4/2+#3*#4/2,#2-0.555*#4/2) .. (#1+2*#4/2+#3*#4/2,#2)}
\def\doublemap(#1,#2)[#3]{\draw (#1+0.5,#2-0.5) node [name=doublemapnode,inner xsep=0pt, inner ysep=0pt, minimum height=9pt, minimum width=25pt,shape=rectangle,draw,rounded corners] {$#3$} (#1,#2) .. controls (#1,#2-0.075) .. (doublemapnode) (#1+1,#2) .. controls (#1+1,#2-0.075).. (doublemapnode) (doublemapnode) .. controls (#1,#2-0.925)..(#1,#2-1) (doublemapnode) .. controls (#1+1,#2-0.925).. (#1+1,#2-1)}
\begin{scope}[xshift=0cm, yshift=0cm]
\draw (0,0) -- (0,-2); \comult(2,0)[1,1]; \comult(1.5,-1)[1,1]; \raction(0,-2)[0,1]; \draw (2,-2) .. controls (2,-2.5) and (1.5,-3.5) .. (1.5,-4); \draw (0,-3) .. controls (0,-3.5) and (0.5,-3.5) .. (0.5,-4); \doublemap(0.5,-4)[\scriptstyle R^{t_1}]; \flip(0.5,-5)[1]; \draw (3.5,-0) -- (3.5,-1); \laction(2.5,-1)[0,1]; \draw (3.5,-2) ..  controls (3.5,-3) and (2.5,-5.5) .. (2.5,-6); \draw (0.5,-6) -- (0.5,-7); \solbraid(2,-6)[\scriptstyle r];
\end{scope}
\begin{scope}[xshift=4cm, yshift=-3.05cm]
\node at (0,-0.5){=};
\end{scope}
\begin{scope}[xshift=4.5cm, yshift=0cm]
\comult(2.5,0)[1,1]; \comult(0.5,0)[1,1]; \laction(3,-1)[0,1]; \laction(1,-1)[0,1]; \draw (0,-1) .. controls (0,-1.5) and (0.5,-2.5) .. (0.5,-3); \draw (2,-2) .. controls (2,-2.5) and (1.5,-2.5) .. (1.5,-3); \flip(0.5,-3)[1]; \comult(1.5,-4)[1,1]; \flip(2,-5)[1]; \comult(3.5,-4)[1,1]; \draw (0.5,-4) -- (0.5,-7); \draw (1,-5) -- (1,-6);\draw (4,-5) -- (4,-6); \laction(1,-6)[0,1]; \raction(3,-6)[0,1]; \draw (4,-2) .. controls (4,-2.5) and (3.5,-3.5) .. (3.5,-4); \draw (4,0) -- (4,-1);
\end{scope}
\begin{scope}[xshift=9cm, yshift=-3.05cm]
\node at (0,-0.5){=};
\end{scope}
\begin{scope}[xshift=9.5cm, yshift=-0.25cm]
\comult(0.5,0)[1,1]; \comult(2.75,0)[1.5,1.5]; \flip(0,-1)[1]; \flip(1,-2)[1]; \draw (2,-1.5) -- (2,-2); \comult(3.5,-1.5)[1,1]; \draw (3,-2.5) -- (3,-3.5); \flip(4,-2.5)[1]; \comult(5.5,-1.5)[1,1]; \draw (6,-2.5) -- (6,-3.5); \draw (5.5,0) -- (5.5,-1.5); \laction(3,-3.5)[0,1]; \laction(5,-3.5)[0,1]; \comult(2,-3)[1,1]; \draw (2.5,-4) .. controls (2.5,-4.25) and (3,-4.25) .. (3,-4.5); \draw (0,-2) -- (0,-3); \laction(0,-3)[0,1]; \flip(3,-4.5)[1]; \draw (1.5,-4) .. controls (1.5,-4.5) and (2,-5) .. (2,-5.5); \laction(2,-5.5)[0,1]; \draw (1,-4) -- (1,-6.5); \draw (6,-4.5) .. controls (6,-5) and (5,-5) .. (5,-5.5); \raction(4,-5.5)[0,1];
\end{scope}
\begin{scope}[xshift=16cm, yshift=-3.05cm]
\node at (0,-0.5){=};
\end{scope}
\begin{scope}[xshift=16.5cm, yshift=0cm]
\comult(1.5,0)[2,2]; \comult(0.5,-3)[1,1]; \flip(2.5,-2)[1]; \comult(2.5,-3)[1,1]; \draw (0.5,-2) -- (0.5,-3); \comult(4,0)[1,1]; \draw (3.5,-1) -- (3.5,-2); \flip(4.5,-1)[1]; \comult(6,0)[1,1]; \draw (6.5,-1) -- (6.5,-2); \draw (0,-4) -- (0,-5); \laction(5.5,-2)[0,1]; \laction(0,-5)[0,1]; \flip(1,-4)[1]; \draw (4.5,-2) -- (4.5,-3); \flip(4,-4)[1]; \draw (3.5,-3) .. controls (3.5,-3.5) and (4,-3.5) .. (4,-4); \draw (4.5,-3) .. controls (4.5,-3.5) and (5,-3.5) .. (5,-4); \draw (3,-4) -- (3,-5); \draw (2,-5) -- (2,-5.5); \laction(3,-5)[0,1]; \laction(2,-5.5)[0.665,1.5]; \draw (6.5,-3) .. controls (6.5,-3.5) and (6,-4.5) .. (6,-5); \raction(5,-5)[0,1]; \draw (5,-6) -- (5,-7); \draw (1,-6) -- (1,-7);
\end{scope}
\begin{scope}[xshift=23.3cm, yshift=-3.05cm]
\node at (0,-0.5){,};
\end{scope}
\end{tikzpicture}
$$
while by Remarks~\ref{prop de involutiva y de braided} and~\ref{algunas formulas}, the definition of $R^{t_2}$ and the fact that $\sigma$ is a coalgebra morphism, we have
$$
\begin{tikzpicture}[scale=0.395]
\def\counit(#1,#2){\draw (#1,#2) -- (#1,#2-0.93) (#1,#2-1) circle[radius=2pt]}
\def\comult(#1,#2)[#3,#4]{\draw (#1,#2) -- (#1,#2-0.5*#4) arc (90:0:0.5*#3 and 0.5*#4) (#1,#2-0.5*#4) arc (90:180:0.5*#3 and 0.5*#4)}
\def\laction(#1,#2)[#3,#4]{\draw (#1,#2) .. controls (#1,#2-0.555*#4/2) and (#1+0.445*#4/2,#2-1*#4/2) .. (#1+1*#4/2,#2-1*#4/2) -- (#1+2*#4/2+#3*#4/2,#2-1*#4/2) (#1+2*#4/2+#3*#4/2,#2)--(#1+2*#4/2+#3*#4/2,#2-2*#4/2)}
\def\map(#1,#2)[#3]{\draw (#1,#2-0.5)  node[name=nodemap,inner sep=0pt,  minimum size=10pt, shape=circle,draw]{$#3$} (#1,#2)-- (nodemap)  (nodemap)-- (#1,#2-1)}
\def\solbraid(#1,#2)[#3]{\draw (#1,#2-0.5)  node[name=nodemap,inner sep=0pt,  minimum size=9pt, shape=circle,draw]{$#3$}
(#1-0.5,#2) .. controls (#1-0.5,#2-0.15) and (#1-0.4,#2-0.2) .. (#1-0.3,#2-0.3) (#1-0.3,#2-0.3) -- (nodemap)
(#1+0.5,#2) .. controls (#1+0.5,#2-0.15) and (#1+0.4,#2-0.2) .. (#1+0.3,#2-0.3) (#1+0.3,#2-0.3) -- (nodemap)
(#1+0.5,#2-1) .. controls (#1+0.5,#2-0.85) and (#1+0.4,#2-0.8) .. (#1+0.3,#2-0.7) (#1+0.3,#2-0.7) -- (nodemap)
(#1-0.5,#2-1) .. controls (#1-0.5,#2-0.85) and (#1-0.4,#2-0.8) .. (#1-0.3,#2-0.7) (#1-0.3,#2-0.7) -- (nodemap)
}
\def\flip(#1,#2)[#3]{\draw (
#1+1*#3,#2) .. controls (#1+1*#3,#2-0.05*#3) and (#1+0.96*#3,#2-0.15*#3).. (#1+0.9*#3,#2-0.2*#3)
(#1+0.1*#3,#2-0.8*#3)--(#1+0.9*#3,#2-0.2*#3)
(#1,#2-1*#3) .. controls (#1,#2-0.95*#3) and (#1+0.04*#3,#2-0.85*#3).. (#1+0.1*#3,#2-0.8*#3)
(#1,#2) .. controls (#1,#2-0.05*#3) and (#1+0.04*#3,#2-0.15*#3).. (#1+0.1*#3,#2-0.2*#3)
(#1+0.1*#3,#2-0.2*#3) -- (#1+0.9*#3,#2-0.8*#3)
(#1+1*#3,#2-1*#3) .. controls (#1+1*#3,#2-0.95*#3) and (#1+0.96*#3,#2-0.85*#3).. (#1+0.9*#3,#2-0.8*#3)
}
\def\raction(#1,#2)[#3,#4]{\draw (#1,#2) -- (#1,#2-2*#4/2)  (#1,#2-1*#4/2)--(#1+1*#4/2+#3*#4/2,#2-1*#4/2) .. controls (#1+1.555*#4/2+#3*#4/2,#2-1*#4/2) and (#1+2*#4/2+#3*#4/2,#2-0.555*#4/2) .. (#1+2*#4/2+#3*#4/2,#2)}
\def\doublemap(#1,#2)[#3]{\draw (#1+0.5,#2-0.5) node [name=doublemapnode,inner xsep=0pt, inner ysep=0pt, minimum height=9pt, minimum width=25pt,shape=rectangle,draw,rounded corners] {$#3$} (#1,#2) .. controls (#1,#2-0.075) .. (doublemapnode) (#1+1,#2) .. controls (#1+1,#2-0.075).. (doublemapnode) (doublemapnode) .. controls (#1,#2-0.925)..(#1,#2-1) (doublemapnode) .. controls (#1+1,#2-0.925).. (#1+1,#2-1)}
\begin{scope}[xshift=0cm, yshift=-0.75cm]
\comult(0.5,0)[1,1]; \comult(2.5,0)[1,1]; \draw (0,-1) -- (0,-2); \laction(0,-2)[0,1];  \flip(1,-1)[1]; \draw (3,-1) -- (3,-2); \draw (4,0) -- (4,-2); \laction(3,-2)[0,1]; \laction(2,-2.5)[0.665,1.5]; \draw (2,-2) -- (2,-2.5); \draw (1,-3) ..  controls (1,-3.5) and (2,-4.5) .. (2,-5); \draw (4,-4) ..  controls (4,-4.5) and (3,-4.5) .. (3,-5);  \doublemap(2,-5)[\scriptstyle R^{t_2}]; \flip(2,-6)[1];
\end{scope}
\begin{scope}[xshift=4.5cm, yshift=-3.8cm]
\node at (0,-0.5){=};
\end{scope}
\begin{scope}[xshift=5cm, yshift=-0.25cm]
\comult(0.5,0)[1,1]; \comult(2.5,0)[1,1]; \draw (0,-1) -- (0,-2); \flip(1,-1)[1]; \draw (3,-1) -- (3,-2); \draw (4,0) -- (4,-3); \laction(3,-3)[0,1]; \laction(2,-3.5)[0.665,1.5]; \draw (2,-3) -- (2,-3.5); \solbraid(2.5,-2)[\scriptstyle r]; \draw (1,-3) ..  controls (1,-3.5) and (2,-5.5) .. (2,-6); \draw (4,-5) ..  controls (4,-5.5) and (3,-5.5) .. (3,-6);  \doublemap(2,-6)[\scriptstyle R^{t_2}]; \flip(2,-7)[1];  \laction(0,-2)[0,1];
\end{scope}
\begin{scope}[xshift=9.5cm, yshift=-3.8cm]
\node at (0,-0.5){=};
\end{scope}
\begin{scope}[xshift=10cm, yshift=-0.95cm]
\solbraid(1,0)[\scriptstyle r];  \comult(0.5,-1)[1,1];  \laction(1.5,-1)[0,1]; \draw (2.5,0) -- (2.5,-1); \laction(1,-2)[0,1.5];  \draw (0,-2) ..  controls (0,-2.5) and (1,-4) .. (1,-4.5); \draw (2.5,-3.5) ..  controls (2.5,-4) and (2,-4) .. (2,-4.5); \doublemap(1,-4.5)[\scriptstyle R^{t_2}]; \flip(1,-5.5)[1];
\end{scope}
\begin{scope}[xshift=13cm, yshift=-3.8cm]
\node at (0,-0.5){=};
\end{scope}
\begin{scope}[xshift=13.5cm, yshift=-1.25cm]
\solbraid(0.5,0)[\scriptstyle r];   \draw (0,-1) -- (0,-3); \laction(1,-1)[0,1]; \comult(2,-2)[1,1]; \raction(0,-3)[0,1.5]; \draw (2,0) -- (2,-1); \draw (2.5,-3) ..  controls (2.5,-3.5) and (1.5,-4.5) .. (1.5,-5); \draw (0,-4.5) ..  controls (0,-4.75) and (0.5,-4.75) .. (0.5,-5); \flip(0.5,-5)[1];
\end{scope}
\begin{scope}[xshift=16.5cm, yshift=-3.8cm]
\node at (0,-0.5){=};
\end{scope}
\begin{scope}[xshift=17cm, yshift=-0.75cm]
\solbraid(0.5,0)[\scriptstyle r]; \draw (0,-1) -- (0,-4); \comult(1,-1)[1,1]; \comult(3,-1)[1,1]; \draw (3,0) -- (3,-1); \flip(1.5,-2)[1]; \draw (0.5,-2) -- (0.5,-3); \draw (3.5,-2) -- (3.5,-3); \laction(2.5,-3)[0,1]; \laction(0.5,-3)[0,1]; \raction(0,-4)[0,1.5]; \flip(0.5,-6)[1]; \draw (3.5,-4) ..  controls (3.5,-5) and (1.5,-5) .. (1.5,-6); \draw (0,-5.5) ..  controls (0,-5.75) and (0.5,-5.75) .. (0.5,-6);
\end{scope}
\begin{scope}[xshift=21cm, yshift=-3.8cm]
\node at (0,-0.5){=};
\end{scope}
\begin{scope}[xshift=21.5cm, yshift=0cm]
\comult(0.5,0)[1,1]; \comult(2.5,0)[1,1]; \draw (3.5,0) -- (3.5,-2); \draw (0,-1) -- (0,-2); \flip(1,-1)[1]; \draw (3,-1) -- (3,-2); \raction(2,-2)[0,1]; \solbraid(0.5,-2)[\scriptstyle r]; \comult(3.5,-2)[1,1]; \flip(2,-3)[1]; \draw (1,-3) -- (1,-4); \laction(1,-4)[0,1]; \draw (0,-3) -- (0,-5); \raction(0,-5)[0.665,1.5]; \draw (4,-3) -- (4,-4); \laction(3,-4)[0,1]; \draw (4,-5) ..  controls (4,-6) and (1.5,-6) .. (1.5,-7);
\draw (0,-6.5) ..  controls (0,-6.75) and (0.5,-6.75) .. (0.5,-7); \flip(0.5,-7)[1];
\end{scope}
\begin{scope}[xshift=25.8cm, yshift=-3.8cm]
\node at (0,-0.5){.};
\end{scope}
\end{tikzpicture}
$$
Consequently,
$$
\begin{tikzpicture}[scale=0.395]
\def\counit(#1,#2){\draw (#1,#2) -- (#1,#2-0.93) (#1,#2-1) circle[radius=2pt]}
\def\comult(#1,#2)[#3,#4]{\draw (#1,#2) -- (#1,#2-0.5*#4) arc (90:0:0.5*#3 and 0.5*#4) (#1,#2-0.5*#4) arc (90:180:0.5*#3 and 0.5*#4)}
\def\laction(#1,#2)[#3,#4]{\draw (#1,#2) .. controls (#1,#2-0.555*#4/2) and (#1+0.445*#4/2,#2-1*#4/2) .. (#1+1*#4/2,#2-1*#4/2) -- (#1+2*#4/2+#3*#4/2,#2-1*#4/2) (#1+2*#4/2+#3*#4/2,#2)--(#1+2*#4/2+#3*#4/2,#2-2*#4/2)}
\def\map(#1,#2)[#3]{\draw (#1,#2-0.5)  node[name=nodemap,inner sep=0pt,  minimum size=10pt, shape=circle,draw]{$#3$} (#1,#2)-- (nodemap)  (nodemap)-- (#1,#2-1)}
\def\solbraid(#1,#2)[#3]{\draw (#1,#2-0.5)  node[name=nodemap,inner sep=0pt,  minimum size=9pt, shape=circle,draw]{$#3$}
(#1-0.5,#2) .. controls (#1-0.5,#2-0.15) and (#1-0.4,#2-0.2) .. (#1-0.3,#2-0.3) (#1-0.3,#2-0.3) -- (nodemap)
(#1+0.5,#2) .. controls (#1+0.5,#2-0.15) and (#1+0.4,#2-0.2) .. (#1+0.3,#2-0.3) (#1+0.3,#2-0.3) -- (nodemap)
(#1+0.5,#2-1) .. controls (#1+0.5,#2-0.85) and (#1+0.4,#2-0.8) .. (#1+0.3,#2-0.7) (#1+0.3,#2-0.7) -- (nodemap)
(#1-0.5,#2-1) .. controls (#1-0.5,#2-0.85) and (#1-0.4,#2-0.8) .. (#1-0.3,#2-0.7) (#1-0.3,#2-0.7) -- (nodemap)
}
\def\flip(#1,#2)[#3]{\draw (
#1+1*#3,#2) .. controls (#1+1*#3,#2-0.05*#3) and (#1+0.96*#3,#2-0.15*#3).. (#1+0.9*#3,#2-0.2*#3)
(#1+0.1*#3,#2-0.8*#3)--(#1+0.9*#3,#2-0.2*#3)
(#1,#2-1*#3) .. controls (#1,#2-0.95*#3) and (#1+0.04*#3,#2-0.85*#3).. (#1+0.1*#3,#2-0.8*#3)
(#1,#2) .. controls (#1,#2-0.05*#3) and (#1+0.04*#3,#2-0.15*#3).. (#1+0.1*#3,#2-0.2*#3)
(#1+0.1*#3,#2-0.2*#3) -- (#1+0.9*#3,#2-0.8*#3)
(#1+1*#3,#2-1*#3) .. controls (#1+1*#3,#2-0.95*#3) and (#1+0.96*#3,#2-0.85*#3).. (#1+0.9*#3,#2-0.8*#3)
}
\def\raction(#1,#2)[#3,#4]{\draw (#1,#2) -- (#1,#2-2*#4/2)  (#1,#2-1*#4/2)--(#1+1*#4/2+#3*#4/2,#2-1*#4/2) .. controls (#1+1.555*#4/2+#3*#4/2,#2-1*#4/2) and (#1+2*#4/2+#3*#4/2,#2-0.555*#4/2) .. (#1+2*#4/2+#3*#4/2,#2)}
\def\doublemap(#1,#2)[#3]{\draw (#1+0.5,#2-0.5) node [name=doublemapnode,inner xsep=0pt, inner ysep=0pt, minimum height=9pt, minimum width=25pt,shape=rectangle,draw,rounded corners] {$#3$} (#1,#2) .. controls (#1,#2-0.075) .. (doublemapnode) (#1+1,#2) .. controls (#1+1,#2-0.075).. (doublemapnode) (doublemapnode) .. controls (#1,#2-0.925)..(#1,#2-1) (doublemapnode) .. controls (#1+1,#2-0.925).. (#1+1,#2-1)}
\begin{scope}[xshift=0cm, yshift=-1.25cm]
\draw (0,0) -- (0,-2); \comult(2,0)[1,1]; \comult(1.5,-1)[1,1]; \raction(0,-2)[0,1]; \draw (2,-2) .. controls (2,-2.5) and (1.5,-3.5) .. (1.5,-4); \draw (0,-3) .. controls (0,-3.5) and (0.5,-3.5) .. (0.5,-4); \doublemap(0.5,-4)[\scriptstyle R^{t_1}]; \flip(0.5,-5)[1]; \draw (3.5,-0) -- (3.5,-1); \laction(2.5,-1)[0,1]; \draw (3.5,-2) ..  controls (3.5,-3) and (2.5,-5.5) .. (2.5,-6); \draw (0.5,-6) -- (0.5,-7); \solbraid(2,-6)[\scriptstyle r]; \doublemap(0.5,-7)[\scriptstyle R^{t_2}];  \flip(0.5,-8)[1]; \draw (2.5,-7) -- (2.5,-9);
\end{scope}
\begin{scope}[xshift=4cm, yshift=-5.3cm]
\node at (0,-0.5){=};
\end{scope}
\begin{scope}[xshift=4.5cm, yshift=-0.25cm]
\comult(1.5,0)[2,2]; \comult(4,-1)[1,1]; \flip(2.5,-2)[1]; \flip(4.5,-2)[1]; \comult(6,-1)[1,1]; \draw (0.5,-2) -- (0.5,-3);  \draw (6.5,-2) -- (6.5,-3); \draw (6,0) -- (6,-1); \draw (4,0) -- (4,-1); \flip(3.5,-3)[1]; \laction(5.5,-3)[0,1]; \raction(4.5,-4)[0,2]; \comult(0.5,-3)[1,1]; \comult(2.5,-3)[1,1]; \draw (3.5,-4) -- (3.5,-5); \draw (0,-4) -- (0,-5); \flip(1,-4)[1]; \draw (3,-4) -- (3,-5); \raction(2,-5)[0,1]; \solbraid(0.5,-5)[\scriptstyle r]; \comult(3.5,-5)[1,1]; \flip(2,-6)[1]; \draw (1,-6) -- (1,-7); \laction(1,-7)[0,1]; \draw (0,-6) -- (0,-8); \raction(0,-8)[0.665,1.5]; \draw (4,-6) -- (4,-7); \laction(3,-7)[0,1]; \draw (4,-8) ..  controls (4,-9) and (1.5,-9) .. (1.5,-10); \draw (0,-9.5) ..  controls (0,-9.75) and (0.5,-9.75) .. (0.5,-10); \flip(0.5,-10)[1]; \draw (4.5,-6) -- (4.5,-11);
\end{scope}
\begin{scope}[xshift=11.5cm, yshift=-5.3cm]
\node at (0,-0.5){=};
\end{scope}
\begin{scope}[xshift=12cm, yshift=0cm]
\comult(1.5,0)[2,2]; \comult(4,-1)[1,1]; \flip(2.5,-2)[1]; \flip(4.5,-2)[1]; \comult(6,-1)[1,1]; \draw (0.5,-2) -- (0.5,-3);  \draw (6.5,-2) -- (6.5,-3); \draw (6,0) -- (6,-1); \draw (4,0) -- (4,-1); \flip(3.5,-3)[1]; \laction(5.5,-3)[0,1]; \raction(4.5,-4)[0,2]; \comult(0.5,-3)[1,1]; \comult(2.5,-3)[1,1]; \draw (3.5,-4) -- (3.5,-5); \draw (0,-4) -- (0,-5); \flip(1,-4)[1]; \draw (3,-4) -- (3,-5); \raction(2,-5)[0,1]; \solbraid(1.5,-7)[\scriptstyle r]; \comult(3.5,-5)[1,1]; \flip(2,-6)[1]; \draw (1,-5) -- (1,-7); \draw (2,-8) -- (2,-9); \draw (0,-5) -- (0,-8); \raction(0,-8)[0,1];  \laction(0,-9)[0.665,1.5]; \draw (4,-6) -- (4,-7); \laction(3,-7)[0,1]; \draw (4,-8) ..  controls (4,-9) and (3,-9.5) .. (3,-10.5); \flip(2,-10.5)[1]; \draw (4.5,-6) -- (4.5,-11.5);
\end{scope}
\begin{scope}[xshift=19cm, yshift=-5.3cm]
\node at (0,-0.5){=};
\end{scope}
\begin{scope}[xshift=19.5cm, yshift=0cm]
\comult(1.5,0)[2,2]; \comult(4,-1)[1,1]; \flip(2.5,-2)[1]; \flip(4.5,-2)[1]; \comult(6,-1)[1,1]; \draw (0.5,-2) -- (0.5,-3);
\draw (6.5,-2) -- (6.5,-3); \draw (6,0) -- (6,-1); \draw (4,0) -- (4,-1); \flip(3.5,-3)[1]; \laction(5.5,-3)[0,1]; \raction(4.5,-4)[0,2];  \comult(0.5,-3)[1,1]; \comult(2.5,-3)[1,1]; \raction(1,-4)[0,1]; \draw (0,-4) -- (0,-8); \draw (3,-4) ..  controls (3,-4.5) and (2,-4.5) .. (2,-5); \draw (3.5,-4) -- (3.5,-5); \flip(1,-5)[1]; \comult(3.5,-5)[1,1]; \flip(2,-6)[1]; \draw (1,-6) -- (1,-7); \draw (4,-6) -- (4,-7); \draw (4.5,-6) -- (4.5,-11.5); \solbraid(1.5,-7)[\scriptstyle r]; \laction(3,-7)[0,1]; \raction(0,-8)[0,1]; \laction(0,-9)[0.665,1.5]; \draw (2,-8) -- (2,-9);
\draw (4,-8) ..  controls (4,-9) and (3,-9.5) .. (3,-10.5); \flip(2,-10.5)[1];
\end{scope}
\begin{scope}[xshift=26.4cm, yshift=-5.3cm]
\node at (0,-0.5){,};
\end{scope}
\end{tikzpicture}
$$
where the second equality follows from Remark~\ref{prop de involutiva y de braided}, and the third one, from the fact that $c$ is a natural  isomorphism and $\Delta$ is cocommutative. But, by the naturalness of $c$ and the fact that $\Delta$ is cocommutative,
$$
\begin{tikzpicture}[scale=0.395]
\def\counit(#1,#2){\draw (#1,#2) -- (#1,#2-0.93) (#1,#2-1) circle[radius=2pt]}
\def\comult(#1,#2)[#3,#4]{\draw (#1,#2) -- (#1,#2-0.5*#4) arc (90:0:0.5*#3 and 0.5*#4) (#1,#2-0.5*#4) arc (90:180:0.5*#3 and 0.5*#4)}
\def\laction(#1,#2)[#3,#4]{\draw (#1,#2) .. controls (#1,#2-0.555*#4/2) and (#1+0.445*#4/2,#2-1*#4/2) .. (#1+1*#4/2,#2-1*#4/2) -- (#1+2*#4/2+#3*#4/2,#2-1*#4/2) (#1+2*#4/2+#3*#4/2,#2)--(#1+2*#4/2+#3*#4/2,#2-2*#4/2)}
\def\map(#1,#2)[#3]{\draw (#1,#2-0.5)  node[name=nodemap,inner sep=0pt,  minimum size=10pt, shape=circle,draw]{$#3$} (#1,#2)-- (nodemap)  (nodemap)-- (#1,#2-1)}
\def\solbraid(#1,#2)[#3]{\draw (#1,#2-0.5)  node[name=nodemap,inner sep=0pt,  minimum size=9pt, shape=circle,draw]{$#3$}
(#1-0.5,#2) .. controls (#1-0.5,#2-0.15) and (#1-0.4,#2-0.2) .. (#1-0.3,#2-0.3) (#1-0.3,#2-0.3) -- (nodemap)
(#1+0.5,#2) .. controls (#1+0.5,#2-0.15) and (#1+0.4,#2-0.2) .. (#1+0.3,#2-0.3) (#1+0.3,#2-0.3) -- (nodemap)
(#1+0.5,#2-1) .. controls (#1+0.5,#2-0.85) and (#1+0.4,#2-0.8) .. (#1+0.3,#2-0.7) (#1+0.3,#2-0.7) -- (nodemap)
(#1-0.5,#2-1) .. controls (#1-0.5,#2-0.85) and (#1-0.4,#2-0.8) .. (#1-0.3,#2-0.7) (#1-0.3,#2-0.7) -- (nodemap)
}
\def\flip(#1,#2)[#3]{\draw (
#1+1*#3,#2) .. controls (#1+1*#3,#2-0.05*#3) and (#1+0.96*#3,#2-0.15*#3).. (#1+0.9*#3,#2-0.2*#3)
(#1+0.1*#3,#2-0.8*#3)--(#1+0.9*#3,#2-0.2*#3)
(#1,#2-1*#3) .. controls (#1,#2-0.95*#3) and (#1+0.04*#3,#2-0.85*#3).. (#1+0.1*#3,#2-0.8*#3)
(#1,#2) .. controls (#1,#2-0.05*#3) and (#1+0.04*#3,#2-0.15*#3).. (#1+0.1*#3,#2-0.2*#3)
(#1+0.1*#3,#2-0.2*#3) -- (#1+0.9*#3,#2-0.8*#3)
(#1+1*#3,#2-1*#3) .. controls (#1+1*#3,#2-0.95*#3) and (#1+0.96*#3,#2-0.85*#3).. (#1+0.9*#3,#2-0.8*#3)
}
\def\raction(#1,#2)[#3,#4]{\draw (#1,#2) -- (#1,#2-2*#4/2)  (#1,#2-1*#4/2)--(#1+1*#4/2+#3*#4/2,#2-1*#4/2) .. controls (#1+1.555*#4/2+#3*#4/2,#2-1*#4/2) and (#1+2*#4/2+#3*#4/2,#2-0.555*#4/2) .. (#1+2*#4/2+#3*#4/2,#2)}
\def\doublemap(#1,#2)[#3]{\draw (#1+0.5,#2-0.5) node [name=doublemapnode,inner xsep=0pt, inner ysep=0pt, minimum height=11pt, minimum width=25pt,shape=rectangle,draw,rounded corners] {$#3$} (#1,#2) .. controls (#1,#2-0.075) .. (doublemapnode) (#1+1,#2) .. controls (#1+1,#2-0.075).. (doublemapnode) (doublemapnode) .. controls (#1,#2-0.925)..(#1,#2-1) (doublemapnode) .. controls (#1+1,#2-0.925).. (#1+1,#2-1)}
\begin{scope}[xshift=0cm, yshift=-0.25cm]
\comult(0.5,0)[1,1]; \comult(2.5,0)[1,1]; \raction(1,-1)[0,1]; \draw (0,-1) -- (0,-5); \draw (3,-1) ..  controls (3,-1.5) and (2,-1.5) .. (2,-2); \draw (3.5,0) -- (3.5,-2); \flip(1,-2)[1]; \comult(3.5,-2)[1,1]; \flip(2,-3)[1]; \draw (1,-3) -- (1,-4); \draw (4,-3) -- (4,-4);
\solbraid(1.5,-4)[\scriptstyle r]; \laction(3,-4)[0,1]; \raction(0,-5)[0,1]; \laction(0,-6)[0.665,1.5]; \draw (2,-5) -- (2,-6);
\draw (4,-5) ..  controls (4,-6) and (3,-6.5) .. (3,-7.5); \flip(2,-7.5)[1];
\end{scope}
\begin{scope}[xshift=4.5cm, yshift=-4.05cm]
\node at (0,-0.5){=};
\end{scope}
\begin{scope}[xshift=5cm, yshift=0cm]
\comult(0.5,0)[1,1]; \flip(0,-1)[1]; \draw (2.5,0) -- (2.5,-1); \comult(2.5,-1)[1,1]; \raction(1,-2)[0,1]; \draw (4.5,0) -- (4.5,-1); \comult(4.5,-1)[1,1]; \solbraid(3.5,-2)[\scriptstyle r]; \draw (3,-3) ..  controls (3,-3.5) and (2,-3.5) .. (2,-4); \draw (4,-3) ..  controls (4,-4) and (3,-4) .. (3,-5); \draw (1,-3) -- (1,-4); \flip(1,-4)[1]; \flip(2,-5)[1];  \draw (0,-2) -- (0,-5); \raction(0,-5)[0,1]; \laction(0,-6)[0.665,1.5]; \draw (5,-2) ..  controls (5,-3) and (4,-5) .. (4,-6); \laction(3,-6)[0,1]; \draw (4,-7) ..  controls (4,-7.5) and (3,-7.5) .. (3,-8); \draw (2,-7.5) -- (2,-8); \flip(2,-8)[1];
\end{scope}
\begin{scope}[xshift=10.5cm, yshift=-4.05cm]
\node at (0,-0.5){=};
\end{scope}
\begin{scope}[xshift=11cm, yshift=0cm]
\comult(0.5,0)[1,1]; \comult(2.5,0)[1,1]; \draw (0,-1) -- (0,-2); \flip(1,-1)[1]; \draw (3,-1) -- (3,-2); \comult(4.5,0)[1,1];  \flip(4,-1)[1];
\solbraid(3.5,-2)[\scriptstyle r]; \raction(0,-2)[0,1]; \draw (2,-2) -- (2,-3); \raction(2,-3)[0,1]; \draw (4,-3) -- (4,-4);  \laction(2,-4)[0.665,1.5]; \draw (5,-2) -- (5,-5.5); \flip(4,-5.5)[1]; \draw (0,-3) -- (0,-6.5);  \laction(0,-6.5)[1.2,2.5];
\draw (0,-2) -- (0,-5); \draw (5,-6.5) -- (5,-9);
\end{scope}
\begin{scope}[xshift=16.5cm, yshift=-4.05cm]
\node at (0,-0.5){=};
\end{scope}
\begin{scope}[xshift=17cm, yshift=-1.7cm]
\comult(0.5,0)[1,1]; \comult(2.5,0)[1,1]; \draw (0,-1) -- (0,-2); \flip(1,-1)[1]; \draw (5,-1) -- (5,-2); \comult(4.5,0)[1,1];  \flip(3,-1)[1]; \raction(0,-2)[0,1]; \flip(2,-2)[1]; \solbraid(4.5,-2)[\scriptstyle r]; \laction(0,-3)[0.665,1.5]; \raction(3,-3)[0,1]; \draw (5,-3) -- (5,-4); \laction(3,-4)[0.665,1.5];  \draw (2,-4.5) -- (2,-5.5);
\end{scope}
\begin{scope}[xshift=22.3cm, yshift=-4.05cm]
\node at (0,-0.5){,};
\end{scope}
\end{tikzpicture}
$$
and therefore
$$
\begin{tikzpicture}[scale=0.395]
\def\counit(#1,#2){\draw (#1,#2) -- (#1,#2-0.93) (#1,#2-1) circle[radius=2pt]}
\def\comult(#1,#2)[#3,#4]{\draw (#1,#2) -- (#1,#2-0.5*#4) arc (90:0:0.5*#3 and 0.5*#4) (#1,#2-0.5*#4) arc (90:180:0.5*#3 and 0.5*#4)}
\def\laction(#1,#2)[#3,#4]{\draw (#1,#2) .. controls (#1,#2-0.555*#4/2) and (#1+0.445*#4/2,#2-1*#4/2) .. (#1+1*#4/2,#2-1*#4/2) -- (#1+2*#4/2+#3*#4/2,#2-1*#4/2) (#1+2*#4/2+#3*#4/2,#2)--(#1+2*#4/2+#3*#4/2,#2-2*#4/2)}
\def\map(#1,#2)[#3]{\draw (#1,#2-0.5)  node[name=nodemap,inner sep=0pt,  minimum size=10pt, shape=circle,draw]{$#3$} (#1,#2)-- (nodemap)  (nodemap)-- (#1,#2-1)}
\def\solbraid(#1,#2)[#3]{\draw (#1,#2-0.5)  node[name=nodemap,inner sep=0pt,  minimum size=9pt, shape=circle,draw]{$#3$}
(#1-0.5,#2) .. controls (#1-0.5,#2-0.15) and (#1-0.4,#2-0.2) .. (#1-0.3,#2-0.3) (#1-0.3,#2-0.3) -- (nodemap)
(#1+0.5,#2) .. controls (#1+0.5,#2-0.15) and (#1+0.4,#2-0.2) .. (#1+0.3,#2-0.3) (#1+0.3,#2-0.3) -- (nodemap)
(#1+0.5,#2-1) .. controls (#1+0.5,#2-0.85) and (#1+0.4,#2-0.8) .. (#1+0.3,#2-0.7) (#1+0.3,#2-0.7) -- (nodemap)
(#1-0.5,#2-1) .. controls (#1-0.5,#2-0.85) and (#1-0.4,#2-0.8) .. (#1-0.3,#2-0.7) (#1-0.3,#2-0.7) -- (nodemap)
}
\def\flip(#1,#2)[#3]{\draw (
#1+1*#3,#2) .. controls (#1+1*#3,#2-0.05*#3) and (#1+0.96*#3,#2-0.15*#3).. (#1+0.9*#3,#2-0.2*#3)
(#1+0.1*#3,#2-0.8*#3)--(#1+0.9*#3,#2-0.2*#3)
(#1,#2-1*#3) .. controls (#1,#2-0.95*#3) and (#1+0.04*#3,#2-0.85*#3).. (#1+0.1*#3,#2-0.8*#3)
(#1,#2) .. controls (#1,#2-0.05*#3) and (#1+0.04*#3,#2-0.15*#3).. (#1+0.1*#3,#2-0.2*#3)
(#1+0.1*#3,#2-0.2*#3) -- (#1+0.9*#3,#2-0.8*#3)
(#1+1*#3,#2-1*#3) .. controls (#1+1*#3,#2-0.95*#3) and (#1+0.96*#3,#2-0.85*#3).. (#1+0.9*#3,#2-0.8*#3)
}
\def\raction(#1,#2)[#3,#4]{\draw (#1,#2) -- (#1,#2-2*#4/2)  (#1,#2-1*#4/2)--(#1+1*#4/2+#3*#4/2,#2-1*#4/2) .. controls (#1+1.555*#4/2+#3*#4/2,#2-1*#4/2) and (#1+2*#4/2+#3*#4/2,#2-0.555*#4/2) .. (#1+2*#4/2+#3*#4/2,#2)}
\def\doublemap(#1,#2)[#3]{\draw (#1+0.5,#2-0.5) node [name=doublemapnode,inner xsep=0pt, inner ysep=0pt, minimum height=9pt, minimum width=25pt,shape=rectangle,draw,rounded corners] {$#3$} (#1,#2) .. controls (#1,#2-0.075) .. (doublemapnode) (#1+1,#2) .. controls (#1+1,#2-0.075).. (doublemapnode) (doublemapnode) .. controls (#1,#2-0.925)..(#1,#2-1) (doublemapnode) .. controls (#1+1,#2-0.925).. (#1+1,#2-1)}
\begin{scope}[xshift=0cm, yshift=-1.2cm]
\draw (0,0) -- (0,-2); \comult(2,0)[1,1]; \comult(1.5,-1)[1,1]; \raction(0,-2)[0,1]; \draw (2,-2) .. controls (2,-2.5) and (1.5,-3.5) .. (1.5,-4); \draw (0,-3) .. controls (0,-3.5) and (0.5,-3.5) .. (0.5,-4); \doublemap(0.5,-4)[\scriptstyle R^{t_1}]; \flip(0.5,-5)[1]; \draw (3.5,-0) -- (3.5,-1); \laction(2.5,-1)[0,1]; \draw (3.5,-2) ..  controls (3.5,-3) and (2.5,-5.5) .. (2.5,-6); \draw (0.5,-6) -- (0.5,-7); \solbraid(2,-6)[\scriptstyle r]; \doublemap(0.5,-7)[\scriptstyle R^{t_2}];  \flip(0.5,-8)[1]; \draw (2.5,-7) -- (2.5,-9);
\end{scope}
\begin{scope}[xshift=4cm, yshift=-5.25cm]
\node at (0,-0.5){=};
\end{scope}
\begin{scope}[xshift=4.6cm, yshift=0.6cm]
\comult(1.5,-1)[2,2]; \draw (5,-1) -- (5,-2); \comult(5,-2)[1,1];\draw (5.5,-3) -- (5.5,-4); \comult(7,-3)[1,1]; \draw (7.5,-4) -- (7.5,-5); \flip(2.5,-3)[2];\flip(5.5,-4)[1]; \flip(4.5,-5)[1]; \draw (0.5,-3) -- (0.5,-6); \draw (2.5,-5) -- (2.5,-6); \comult(0.5,-6)[1,1]; \comult(2.5,-6)[1,1]; \draw (0,-7) -- (0,-8); \flip(1,-7)[1]; \draw (5,-7) -- (5,-8); \comult(4.5,-6)[1,1];  \flip(3,-7)[1]; \raction(0,-8)[0,1]; \flip(2,-8)[1]; \solbraid(4.5,-8)[\scriptstyle r]; \laction(0,-9)[0.665,1.5]; \raction(3,-9)[0,1]; \draw (5,-9) -- (5,-10); \laction(3,-10)[0.665,1.5];  \draw (2,-10.5) -- (2,-11.5); \laction(6.5,-5)[0,1]; \raction(5.5,-6)[0,2]; \draw (7,-1) -- (7,-3);\draw (5.5,-8) -- (5.5,-11.5);
\end{scope}
\begin{scope}[xshift=12.7cm, yshift=-5.25cm]
\node at (0,-0.5){=};
\end{scope}
\begin{scope}[xshift=13.3cm, yshift=0.6cm]
\comult(0.5,-3)[1,1];\draw (0,-4) -- (0,-5);\comult(2.5,-2)[1,1];\draw (2,-3) -- (2,-4) ;\comult(6,-1)[2,2]; \flip(1,-4)[1]; \flip(3,-3)[2]; \flip(2,-5)[1]; \draw (0.5,-1) -- (0.5,-3); \draw (2.5,-1) -- (2.5,-2); \raction(0,-5)[0,1]; \laction(0,-6)[0,2]; \draw (5,-5) -- (5,-6); \draw (7,-3) -- (7,-6); \comult(3,-6)[1,1]; \comult(5,-6)[1,1]; \comult(7,-6)[1,1];  \flip(3.5,-7)[1]; \flip(5.5,-7)[1]; \flip(4.5,-8)[1];\draw (2.5,-7) -- (2.5,-8); \solbraid(3,-8)[\scriptstyle r]; \laction(3.5,-9)[0,1]; \raction(2.5,-10)[0.665,1.5]; \draw (2.5,-9) -- (2.5,-10); \draw (2,-8) -- (2,-11.5); \laction(6.5,-8)[0,1]; \draw (7.5,-7) -- (7.5,-8); \raction(5.5,-9)[0.665,1.5]; \draw (5.5,-10.5) -- (5.5,-11.5);
\end{scope}
\begin{scope}[xshift=21.3cm, yshift=-5.25cm]
\node at (0,-0.5){=};
\end{scope}
\begin{scope}[xshift=21.9cm, yshift=-1.25cm]
\draw (0,0) -- (0,-2); \comult(2,0)[1,1]; \comult(1.5,-1)[1,1]; \raction(0,-2)[0,1]; \draw (3.5,0) -- (3.5,-1); \draw (2,-2) -- (2,-4);
\draw (3.5,-2) .. controls (3.5,-2.5) and (3,-3.5) .. (3,-4); \doublemap(2,-4)[\scriptstyle R^{t_2}]; \flip(2,-5)[1];
\draw (0,-3) .. controls (0,-4) and (1,-5) .. (1,-6); \solbraid(1.5,-6)[\scriptstyle r]; \draw (3,-6) -- (3,-7);
\laction(2.5,-1)[0,1]; \doublemap(2,-7)[\scriptstyle R^{t_1}]; \flip(2,-8)[1];  \draw (1,-7) -- (1,-9);
\end{scope}

\begin{scope}[xshift=26.2cm, yshift=-5.25cm]
\node at (0,-0.5){,};
\end{scope}
\end{tikzpicture}
$$
where the second equality holds by the coassociativity of $\Delta_{X^3}$ and Remark~\ref{prop de involutiva y de braided}, and the last one equality follows from the first one by simmetry.
\end{proof}

\begin{proof}[Proof of Proposition~\ref{pro:YB}] Using Remark~\ref{Rt2=Rt1a la-1} one proves that $r_e$ is bijective. Moreover by Proposition~\ref{pro:non-degenerate} the pair $(Z,r_e)$ is non-degenerate. Clearly $r_e$ satisfies the braid equation on $X\otimes X\otimes X$, and from Proposition~\ref{lem:SXxSXxSX}(1) it follows that $r_e$ satisfies the braid equation on $SX\otimes SX\otimes SX$. Moreover, using Lemma~\ref{lem:XxXxSX} we obtain that $r_e$ satisfies the braid equation on $X\otimes X\otimes SX$ and $SX\ot X\ot X$, while using Lemma~\ref{lem:XxSXxX} we obtain that $r_e$ also satisfies the braid equation on $X\otimes SX\otimes X$. Finally the cases $SX\ot SX\ot X$ and $X\ot X\ot SX$ follows applying Lemma~\ref{lem:XxXxSX}, with $r$ replaced by $\tilde{r}$, and taking into account Proposition~\ref{lem:SXxSXxSX}(2), while the case $SX\ot X\ot SX$ is similar, but we must apply Lemma~\ref{lem:XxSXxX} instead of Lemma~\ref{lem:XxXxSX}.
\end{proof}

\begin{pro}\label{pro:grupo_de_estructura} Let $(X,r)$ be a non-degenerate braided set in $\mathscr{C}$. There exists a cocommutative Hopf algebra $H$ and a morphism of coalgebras $\gamma\colon X\to H$ such that $m_H \circ (\gamma\ot \gamma) = m_H\circ (\gamma\ot \gamma)\circ r$, which satisfies the following universal property: If $K$ is a cocommutative Hopf algebra and $f\colon X\to K$ is a morphism of coalgebras such that $m_K \circ (f\ot f) = m_K\circ (f\ot f)\circ r$, then $f$ factorizes univocally through $\gamma$.
\end{pro}

\begin{proof} Let $\ov{\jmath}\colon X\to \ov{L}$ be the free Hopf algebra over the coalgebra $X$ (see Proposition~\ref{const de antipoda} and Remark~\ref{Z especial}) and let $\ov{\pi}\colon \ov{L}\to H$ be the coequalizer of
\[
\ov{m}\circ (\ov{m}\ot \ov{L})\circ(\ov{L}\otimes \ov{m}\circ (\ov{\jmath}\ot \ov{\jmath})\circ r\otimes\ov{L})\quad\text{and} \quad \ov{m}\circ (\ov{m}\ot \ov{L})\circ(\ov{L}\otimes \ov{m}\circ (\ov{\jmath}\ot \ov{\jmath})\otimes \ov{L}),
\]
where $\ov{m}$ denotes the multiplication map of $\ov{L}$. By Remark~\ref{coegalizador'} and Corollary~\ref{cociente de Hopf algebra}, the map $\gamma\coloneqq \ov{\pi}\circ \ov{\jmath}$ satisfies the required conditions.
\end{proof}

\begin{defn} Let $(X,r)$ be a non-degenerate braided set in $\mathscr{C}$. The pair $(H,\gamma)$ constructed in Proposition~\ref{pro:grupo_de_estructura} is called the \emph{structure Hopf algebra} of $(X,r)$.
\end{defn}

The following theorem is one of the main results of this paper.

\begin{thm}\label{thm:main} Let $(X,r)$ be a non-degenerate braided set in $\mathscr{C}$ and let $(H,\gamma)$ be the structure Hopf algebra  of $(X,r)$. There is a unique braiding operator $(H,r_H)$ such that
\[
r_H\circ (\gamma\otimes\gamma)=(\gamma\otimes\gamma)\circ r.
\]
Furthermore, if $(K,r_k)$ is a braiding operator and $f\colon (X,r)\to (K,r_K)$ is a morphism of braided sets in $\mathscr{C}$, then there is a unique braided operator morphism $\phi$ from $(H,r_H)$ to $(K,r_K)$ such that $f=\phi\circ\gamma$.
\end{thm}

In order to prove this theorem we shall need the following lemma.

\begin{lem}\label{lem:chinos} Let $(X,r)$ be a non-degenerate braided set in $\mathscr{C}$. The following equalities hold:
\begin{align*}
&(X\otimes r_2)\circ (r\otimes S)\circ (X\otimes\Delta) = (X\otimes S\otimes X)\circ (\Delta\circ\sigma \otimes X)\circ (X\otimes c)\circ (\Delta\otimes X),\\
&(r_3\otimes X)\circ (S\otimes r)\circ (\Delta\otimes X) = (X\otimes S\otimes X)\circ (X\otimes \Delta\circ\tau)\circ (c\otimes X)\circ (X\otimes \Delta),\\
&(X\otimes r)\circ (r_2\otimes X)\circ (X\otimes S\otimes X) \circ (X\otimes\Delta)=(S\otimes X^2)\circ(\Delta\otimes X)\circ r\circ J_2^{-1},\\
&(r\otimes X)\circ (X\otimes r_3)\circ (X\otimes S\otimes X) \circ (\Delta\otimes X)=(X^2\otimes S)\circ(X \otimes\Delta)\circ r\circ K_2^{-1},
\end{align*}
where $J_2\coloneqq (\tau \otimes X)\circ (X \otimes\Delta)$ and $K_2\coloneqq (X\otimes\sigma)\circ (\Delta\otimes X)$.
\end{lem}

\begin{proof} We will prove the first and third equalities, since the second and fourth one are valid by symmetry. For the first one note that by the very definitions of $r_2$ and $R^{t_1}$, and the fact that $\Delta$ is coassociative and cocommutative, $c$ is a natural isomorphism and $\sigma$ is a coalgebra homomorphism, we have:
$$
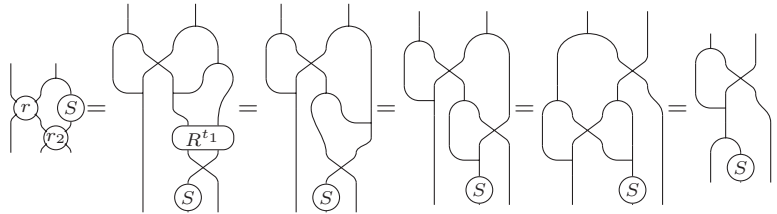
\begin{tikzpicture}[scale=0.395]
\def\mult(#1,#2)[#3,#4]{\draw (#1,#2) arc (180:360:0.5*#3 and 0.5*#4) (#1+0.5*#3,-0.5*#4) -- (#1+0.5*#3,-#4)}
\def\counit(#1,#2){\draw (#1,#2) -- (#1,#2-0.93) (#1,#2-1) circle[radius=2pt]}
\def\comult(#1,#2)[#3,#4]{\draw (#1,#2) -- (#1,#2-0.5*#4) arc (90:0:0.5*#3 and 0.5*#4) (#1,#2-0.5*#4) arc (90:180:0.5*#3 and 0.5*#4)}
\def\laction(#1,#2)[#3,#4]{\draw (#1,#2) .. controls (#1,#2-0.555*#4/2) and (#1+0.445*#4/2,#2-1*#4/2) .. (#1+1*#4/2,#2-1*#4/2) -- (#1+2*#4/2+#3*#4/2,#2-1*#4/2) (#1+2*#4/2+#3*#4/2,#2)--(#1+2*#4/2+#3*#4/2,#2-2*#4/2)}
\def\map(#1,#2)[#3]{\draw (#1,#2-0.5)  node[name=nodemap,inner sep=0pt,  minimum size=10pt, shape=circle, draw]{$#3$} (#1,#2)-- (nodemap)  (nodemap)-- (#1,#2-1)}
\def\solbraid(#1,#2)[#3]{\draw (#1,#2-0.5)  node[name=nodemap,inner sep=0pt,  minimum size=9pt, shape=circle,draw]{$#3$}
(#1-0.5,#2) .. controls (#1-0.5,#2-0.15) and (#1-0.4,#2-0.2) .. (#1-0.3,#2-0.3) (#1-0.3,#2-0.3) -- (nodemap)
(#1+0.5,#2) .. controls (#1+0.5,#2-0.15) and (#1+0.4,#2-0.2) .. (#1+0.3,#2-0.3) (#1+0.3,#2-0.3) -- (nodemap)
(#1+0.5,#2-1) .. controls (#1+0.5,#2-0.85) and (#1+0.4,#2-0.8) .. (#1+0.3,#2-0.7) (#1+0.3,#2-0.7) -- (nodemap)
(#1-0.5,#2-1) .. controls (#1-0.5,#2-0.85) and (#1-0.4,#2-0.8) .. (#1-0.3,#2-0.7) (#1-0.3,#2-0.7) -- (nodemap)
}
\def\flip(#1,#2)[#3]{\draw (
#1+1*#3,#2) .. controls (#1+1*#3,#2-0.05*#3) and (#1+0.96*#3,#2-0.15*#3).. (#1+0.9*#3,#2-0.2*#3)
(#1+0.1*#3,#2-0.8*#3)--(#1+0.9*#3,#2-0.2*#3)
(#1,#2-1*#3) .. controls (#1,#2-0.95*#3) and (#1+0.04*#3,#2-0.85*#3).. (#1+0.1*#3,#2-0.8*#3)
(#1,#2) .. controls (#1,#2-0.05*#3) and (#1+0.04*#3,#2-0.15*#3).. (#1+0.1*#3,#2-0.2*#3)
(#1+0.1*#3,#2-0.2*#3) -- (#1+0.9*#3,#2-0.8*#3)
(#1+1*#3,#2-1*#3) .. controls (#1+1*#3,#2-0.95*#3) and (#1+0.96*#3,#2-0.85*#3).. (#1+0.9*#3,#2-0.8*#3)
}
\def\raction(#1,#2)[#3,#4]{\draw (#1,#2) -- (#1,#2-2*#4/2)  (#1,#2-1*#4/2)--(#1+1*#4/2+#3*#4/2,#2-1*#4/2) .. controls (#1+1.555*#4/2+#3*#4/2,#2-1*#4/2) and (#1+2*#4/2+#3*#4/2,#2-0.555*#4/2) .. (#1+2*#4/2+#3*#4/2,#2)}
\def\doublemap(#1,#2)[#3]{\draw (#1+0.5,#2-0.5) node [name=doublemapnode,inner xsep=0pt, inner ysep=0pt, minimum height=9pt, minimum width=23pt,shape=rectangle,draw,rounded corners] {$#3$} (#1,#2) .. controls (#1,#2-0.075) .. (doublemapnode) (#1+1,#2) .. controls (#1+1,#2-0.075).. (doublemapnode) (doublemapnode) .. controls (#1,#2-0.925)..(#1,#2-1) (doublemapnode) .. controls (#1+1,#2-0.925).. (#1+1,#2-1)}
\def\doublesinglemap(#1,#2)[#3]{\draw (#1+0.5,#2-0.5) node [name=doublesinglemapnode,inner xsep=0pt, inner ysep=0pt, minimum height=11pt, minimum width=23pt,shape=rectangle,draw,rounded corners] {$#3$} (#1,#2) .. controls (#1,#2-0.075) .. (doublesinglemapnode) (#1+1,#2) .. controls (#1+1,#2-0.075).. (doublesinglemapnode) (doublesinglemapnode)-- (#1+0.5,#2-1)}
\def\ractiontr(#1,#2)[#3,#4,#5]{\draw (#1,#2) -- (#1,#2-2*#4/2)  (#1,#2-1*#4/2) node [inner sep=0pt, minimum size=3pt,shape=isosceles triangle,fill, shape border rotate=#5] {}  --(#1+1*#4/2+#3*#4/2,#2-1*#4/2) .. controls (#1+1.555*#4/2+#3*#4/2,#2-1*#4/2) and (#1+2*#4/2+#3*#4/2,#2-0.555*#4/2) .. (#1+2*#4/2+#3*#4/2,#2)  }
\def\rack(#1,#2)[#3]{\draw (#1,#2-0.5)  node[name=nodemap,inner sep=0pt,  minimum size=7.5pt, shape=circle,draw]{$#3$} (#1-1,#2) .. controls (#1-1,#2-0.5) and (#1-0.5,#2-0.5) .. (nodemap) (#1,#2)-- (nodemap)  (nodemap)-- (#1,#2-1)}
\def\rackmenoslarge(#1,#2)[#3]{\draw (#1,#2-0.5)  node[name=nodemap,inner sep=0pt,  minimum size=7.5pt, shape=circle,draw]{$#3$} (#1-1.5,#2+0.5) .. controls (#1-1.5,#2-0.5) and (#1-0.5,#2-0.5) .. (nodemap) (#1,#2)-- (nodemap)  (nodemap)-- (#1,#2-1)}
\def\racklarge(#1,#2)[#3]{\draw (#1,#2-0.5)  node[name=nodemap,inner sep=0pt,  minimum size=7.5pt, shape=circle,draw]{$#3$} (#1-2,#2+0.5) .. controls (#1-2,#2-0.5) and (#1-0.5,#2-0.5) .. (nodemap) (#1,#2)-- (nodemap)  (nodemap)-- (#1,#2-1)}
\def\rackmaslarge(#1,#2)[#3]{\draw (#1,#2-0.5)  node[name=nodemap,inner sep=0pt,  minimum size=7.5pt, shape=circle,draw]{$#3$} (#1-2.5,#2+0.5) .. controls (#1-2.5,#2-0.5) and (#1-0.5,#2-0.5) .. (nodemap) (#1,#2)-- (nodemap)  (nodemap)-- (#1,#2-1)}
\def\rackextralarge(#1,#2)[#3]{\draw (#1,#2-0.75)  node[name=nodemap,inner sep=0pt,  minimum size=7.5pt, shape=circle, draw]{$#3$} (#1-3,#2+1) .. controls (#1-3,#2-0.75) and (#1-0.5,#2-0.75) .. (nodemap) (#1,#2)-- (nodemap)  (nodemap)-- (#1,#2-1.5)}
\begin{scope}[xshift=0cm, yshift=-2cm]
\draw (0,0) -- (0,-1); \comult(1.5,0)[1,1]; \solbraid(0.5,-1)[\scriptstyle r]; \map(2,-1)[\scriptstyle S]; \solbraid(1.5,-2)[\scriptstyle r_2]; \draw (0,-2) -- (0,-3);
\end{scope}
\begin{scope}[xshift=2.85cm, yshift=-3.05cm]
\node at (0,-0.5){=};
\end{scope}
\begin{scope}[xshift=3.4cm, yshift=0cm]
\draw (0.5,0) -- (0.5,-0.5); \comult(0.5,-0.5)[1,1]; \comult(2.75,0)[1.5,1.5]; \flip(1,-1.5)[1]; \draw (0,-1.5) -- (0,-2.5); \comult(3.5,-1.5)[1,1]; \raction(2,-2.5)[0,1]; \laction(0,-2.5)[0,1]; \draw (2,-3.5).. controls (2,-3.75) and (2.5,-3.75).. (2.5,-4); \draw (4,-2.5).. controls (4,-3) and (3.5,-3).. (3.5,-4); \doublemap(2.5,-4)[\scriptstyle R^{t_1}]; \flip(2.5,-5)[1]; \map(2.5,-6)[\scriptstyle S]; \draw (3.5,-6) -- (3.5,-7); \draw (1,-3.5) -- (1,-7);
\end{scope}
\begin{scope}[xshift=7.9cm, yshift=-3.05cm]
\node at (0,-0.5){=};
\end{scope}
\begin{scope}[xshift=8.5cm, yshift=0cm]
\draw (0.5,0) -- (0.5,-0.5); \comult(0.5,-0.5)[1,1]; \comult(2.75,0)[1.5,1.5]; \flip(1,-1.5)[1]; \draw (0,-1.5) -- (0,-2.5); \comult(2,-2.5)[1,1]; \laction(0,-2.5)[0,1]; \laction(2.5,-3.5)[0,1]; \draw (3.5,-1.5) -- (3.5,-3.5); \draw (1.5,-3.5) .. controls (1.5,-4) and (2,-4.5) .. (2,-5); \draw (3.5,-4.5) .. controls (3.5,-4.75) and (3,-4.75) .. (3,-5); \flip(2,-5)[1]; \map(2,-6)[\scriptstyle S]; \draw (1,-3.5) -- (1,-7); \draw (3,-6) -- (3,-7);
\end{scope}
\begin{scope}[xshift=12.5cm, yshift=-3.05cm]
\node at (0,-0.5){=};
\end{scope}
\begin{scope}[xshift=13.1cm, yshift=-0.25cm]
\draw (0.5,0) -- (0.5,-0.5); \comult(0.5,-0.5)[1,1]; \comult(2.75,0)[1.5,1.5]; \flip(1,-1.5)[1]; \draw (0,-1.5) -- (0,-2.5); \comult(2,-2.5)[1,1]; \laction(0,-2.5)[0,1]; \draw (3.5,-1.5) -- (3.5,-3.5); \flip(2.5,-3.5)[1]; \draw (1.5,-3.5) -- (1.5,-4.5); \draw (1,-3.5) -- (1,-6.5); \laction(1.5,-4.5)[0,1]; \map(2.5,-5.5)[\scriptstyle S]; \draw (3.5,-4.5) -- (3.5,-6.5);
\end{scope}
\begin{scope}[xshift=17.1cm, yshift=-3.05cm]
\node at (0,-0.5){=};
\end{scope}
\begin{scope}[xshift=17.7cm, yshift=-0.25cm]
\comult(1.5,0)[2,1.5]; \comult(0.5,-2.5)[1,1]; \comult(2.5,-2.5)[1,1]; \flip(2.5,-1.5)[1]; \flip(1,-3.5)[1];  \draw (0.5,-1.5) -- (0.5,-2.5); \draw (3.5,0) -- (3.5,-1.5);  \draw (0,-3.5) -- (0,-4.5); \draw (3,-3.5) -- (3,-4.5); \laction(0,-4.5)[0,1]; \laction(2,-4.5)[0,1]; \draw (4,-3.5) -- (4,-6.5); \draw (3.5,-2.5) .. controls (3.5,-3) and (4,-3) .. (4,-3.5); \map(3,-5.5)[\scriptstyle S]; \draw (1,-5.5) -- (1,-6.5);
\end{scope}
\begin{scope}[xshift=22.2cm, yshift=-3.05cm]
\node at (0,-0.5){=};
\end{scope}
\begin{scope}[xshift=22.8cm, yshift=-1cm]
\comult(0.5,0)[1,1]; \draw (2,0) -- (2,-1); \draw (0,-1) -- (0,-2); \laction(0,-2)[0,1]; \flip(1,-1)[1]; \comult(1,-3)[1,1]; \draw (0.5,-4) -- (0.5,-5); \map(1.5,-4)[\scriptstyle S];  \draw (2.5,-3) -- (2.5,-5); \draw (2,-2) .. controls (2,-2.5) and (2.5,-2.5) .. (2.5,-3);
\end{scope}
\begin{scope}[xshift=25.6cm, yshift=-3.05cm]
\node at (0,-0.5){,};
\end{scope}
\end{tikzpicture}
$$
as desired. Since $J_2$ is an isomorphism, in order to prove the third one it suffices to note that
$$
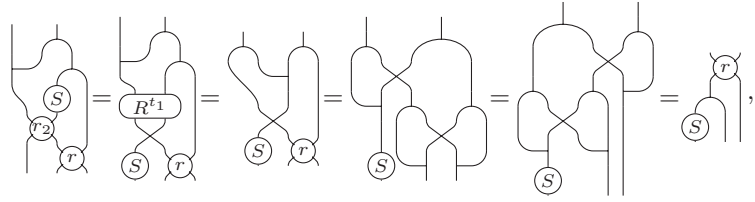
\begin{tikzpicture}[scale=0.395]
\def\mult(#1,#2)[#3,#4]{\draw (#1,#2) arc (180:360:0.5*#3 and 0.5*#4) (#1+0.5*#3,-0.5*#4) -- (#1+0.5*#3,-#4)}
\def\counit(#1,#2){\draw (#1,#2) -- (#1,#2-0.93) (#1,#2-1) circle[radius=2pt]}
\def\comult(#1,#2)[#3,#4]{\draw (#1,#2) -- (#1,#2-0.5*#4) arc (90:0:0.5*#3 and 0.5*#4) (#1,#2-0.5*#4) arc (90:180:0.5*#3 and 0.5*#4)}
\def\laction(#1,#2)[#3,#4]{\draw (#1,#2) .. controls (#1,#2-0.555*#4/2) and (#1+0.445*#4/2,#2-1*#4/2) .. (#1+1*#4/2,#2-1*#4/2) -- (#1+2*#4/2+#3*#4/2,#2-1*#4/2) (#1+2*#4/2+#3*#4/2,#2)--(#1+2*#4/2+#3*#4/2,#2-2*#4/2)}
\def\map(#1,#2)[#3]{\draw (#1,#2-0.5)  node[name=nodemap,inner sep=0pt,  minimum size=10pt, shape=circle, draw]{$#3$} (#1,#2)-- (nodemap)  (nodemap)-- (#1,#2-1)}
\def\solbraid(#1,#2)[#3]{\draw (#1,#2-0.5)  node[name=nodemap,inner sep=0pt,  minimum size=9pt, shape=circle,draw]{$#3$}
(#1-0.5,#2) .. controls (#1-0.5,#2-0.15) and (#1-0.4,#2-0.2) .. (#1-0.3,#2-0.3) (#1-0.3,#2-0.3) -- (nodemap)
(#1+0.5,#2) .. controls (#1+0.5,#2-0.15) and (#1+0.4,#2-0.2) .. (#1+0.3,#2-0.3) (#1+0.3,#2-0.3) -- (nodemap)
(#1+0.5,#2-1) .. controls (#1+0.5,#2-0.85) and (#1+0.4,#2-0.8) .. (#1+0.3,#2-0.7) (#1+0.3,#2-0.7) -- (nodemap)
(#1-0.5,#2-1) .. controls (#1-0.5,#2-0.85) and (#1-0.4,#2-0.8) .. (#1-0.3,#2-0.7) (#1-0.3,#2-0.7) -- (nodemap)
}
\def\flip(#1,#2)[#3]{\draw (
#1+1*#3,#2) .. controls (#1+1*#3,#2-0.05*#3) and (#1+0.96*#3,#2-0.15*#3).. (#1+0.9*#3,#2-0.2*#3)
(#1+0.1*#3,#2-0.8*#3)--(#1+0.9*#3,#2-0.2*#3)
(#1,#2-1*#3) .. controls (#1,#2-0.95*#3) and (#1+0.04*#3,#2-0.85*#3).. (#1+0.1*#3,#2-0.8*#3)
(#1,#2) .. controls (#1,#2-0.05*#3) and (#1+0.04*#3,#2-0.15*#3).. (#1+0.1*#3,#2-0.2*#3)
(#1+0.1*#3,#2-0.2*#3) -- (#1+0.9*#3,#2-0.8*#3)
(#1+1*#3,#2-1*#3) .. controls (#1+1*#3,#2-0.95*#3) and (#1+0.96*#3,#2-0.85*#3).. (#1+0.9*#3,#2-0.8*#3)
}
\def\raction(#1,#2)[#3,#4]{\draw (#1,#2) -- (#1,#2-2*#4/2)  (#1,#2-1*#4/2)--(#1+1*#4/2+#3*#4/2,#2-1*#4/2) .. controls (#1+1.555*#4/2+#3*#4/2,#2-1*#4/2) and (#1+2*#4/2+#3*#4/2,#2-0.555*#4/2) .. (#1+2*#4/2+#3*#4/2,#2)}
\def\doublemap(#1,#2)[#3]{\draw (#1+0.5,#2-0.5) node [name=doublemapnode,inner xsep=0pt, inner ysep=0pt, minimum height=9pt, minimum width=23pt,shape=rectangle,draw,rounded corners] {$#3$} (#1,#2) .. controls (#1,#2-0.075) .. (doublemapnode) (#1+1,#2) .. controls (#1+1,#2-0.075).. (doublemapnode) (doublemapnode) .. controls (#1,#2-0.925)..(#1,#2-1) (doublemapnode) .. controls (#1+1,#2-0.925).. (#1+1,#2-1)}
\def\doublesinglemap(#1,#2)[#3]{\draw (#1+0.5,#2-0.5) node [name=doublesinglemapnode,inner xsep=0pt, inner ysep=0pt, minimum height=11pt, minimum width=23pt,shape=rectangle,draw,rounded corners] {$#3$} (#1,#2) .. controls (#1,#2-0.075) .. (doublesinglemapnode) (#1+1,#2) .. controls (#1+1,#2-0.075).. (doublesinglemapnode) (doublesinglemapnode)-- (#1+0.5,#2-1)}
\def\ractiontr(#1,#2)[#3,#4,#5]{\draw (#1,#2) -- (#1,#2-2*#4/2)  (#1,#2-1*#4/2) node [inner sep=0pt, minimum size=3pt,shape=isosceles triangle,fill, shape border rotate=#5] {}  --(#1+1*#4/2+#3*#4/2,#2-1*#4/2) .. controls (#1+1.555*#4/2+#3*#4/2,#2-1*#4/2) and (#1+2*#4/2+#3*#4/2,#2-0.555*#4/2) .. (#1+2*#4/2+#3*#4/2,#2)  }
\def\rack(#1,#2)[#3]{\draw (#1,#2-0.5)  node[name=nodemap,inner sep=0pt,  minimum size=7.5pt, shape=circle,draw]{$#3$} (#1-1,#2) .. controls (#1-1,#2-0.5) and (#1-0.5,#2-0.5) .. (nodemap) (#1,#2)-- (nodemap)  (nodemap)-- (#1,#2-1)}
\def\rackmenoslarge(#1,#2)[#3]{\draw (#1,#2-0.5)  node[name=nodemap,inner sep=0pt,  minimum size=7.5pt, shape=circle,draw]{$#3$} (#1-1.5,#2+0.5) .. controls (#1-1.5,#2-0.5) and (#1-0.5,#2-0.5) .. (nodemap) (#1,#2)-- (nodemap)  (nodemap)-- (#1,#2-1)}
\def\racklarge(#1,#2)[#3]{\draw (#1,#2-0.5)  node[name=nodemap,inner sep=0pt,  minimum size=7.5pt, shape=circle,draw]{$#3$} (#1-2,#2+0.5) .. controls (#1-2,#2-0.5) and (#1-0.5,#2-0.5) .. (nodemap) (#1,#2)-- (nodemap)  (nodemap)-- (#1,#2-1)}
\def\rackmaslarge(#1,#2)[#3]{\draw (#1,#2-0.5)  node[name=nodemap,inner sep=0pt,  minimum size=7.5pt, shape=circle,draw]{$#3$} (#1-2.5,#2+0.5) .. controls (#1-2.5,#2-0.5) and (#1-0.5,#2-0.5) .. (nodemap) (#1,#2)-- (nodemap)  (nodemap)-- (#1,#2-1)}
\def\rackextralarge(#1,#2)[#3]{\draw (#1,#2-0.75)  node[name=nodemap,inner sep=0pt,  minimum size=7.5pt, shape=circle, draw]{$#3$} (#1-3,#2+1) .. controls (#1-3,#2-0.75) and (#1-0.5,#2-0.75) .. (nodemap) (#1,#2)-- (nodemap)  (nodemap)-- (#1,#2-1.5)}
\begin{scope}[xshift=0cm, yshift=-0.75cm]
\draw (0,0) -- (0,-1); \comult(1.5,0)[1,1]; \raction(0,-1)[0,1]; \comult(2,-1)[1,1]; \map(1.5,-2)[\scriptstyle S]; \draw (2.5,-2) -- (2.5,-4); \draw (0,-2) .. controls (0,-2.5) and (0.5,-2.5) .. (0.5,-3); \solbraid(1,-3)[\scriptstyle r_2]; \solbraid(2,-4)[\scriptstyle r]; \draw (0.5,-4) -- (0.5,-5);
\end{scope}
\begin{scope}[xshift=3cm, yshift=-2.8cm]
\node at (0,-0.5){=};
\end{scope}
\begin{scope}[xshift=3.6cm, yshift=-0.5cm]
\draw (0,0) -- (0,-1); \comult(1.5,0)[1,1]; \raction(0,-1)[0,1]; \comult(2,-1)[1,1]; \draw (0,-2) .. controls (0,-2.25) and (0.5,-2.25) .. (0.5,-2.5); \doublemap(0.5,-2.5)[\scriptstyle R^{t_1}]; \flip(0.5,-3.5)[1]; \draw (1.5,-2) -- (1.5,-2.5); \draw (2.5,-2) -- (2.5,-4.5); \solbraid(2,-4.5)[\scriptstyle r]; \map(0.5,-4.5)[\scriptstyle S];
\end{scope}
\begin{scope}[xshift=6.6cm, yshift=-2.8cm]
\node at (0,-0.5){=};
\end{scope}
\begin{scope}[xshift=7.2cm, yshift=-1cm]
\comult(0.5,0)[1,1]; \comult(2.5,0)[1,1]; \laction(1,-1)[0,1]; \draw (3,-1) -- (3,-3.5); \draw (0,-1) .. controls (0,-1.5) and (1,-2.) .. (1,-2.5); \draw (2,-2) -- (2,-2.5); \flip(1,-2.5)[1]; \solbraid(2.5,-3.5)[\scriptstyle r]; \map(1,-3.5)[\scriptstyle S];
\end{scope}
\begin{scope}[xshift=10.7cm, yshift=-2.8cm]
\node at (0,-0.5){=};
\end{scope}
\begin{scope}[xshift=11.3cm, yshift=-0.5cm]
\draw (0.5,0) -- (0.5,-0.5); \comult(0.5,-0.5)[1,1]; \comult(2,-2.5)[1,1]; \flip(1,-1.5)[1]; \comult(3,0)[2,1.5]; \draw (4,-1.5) -- (4,-2.5); \comult(4,-2.5)[1,1]; \draw (0,-1.5) -- (0,-2.5); \laction(0,-2.5)[0,1]; \flip(2.5,-3.5)[1]; \draw (1.5,-3.5) -- (1.5,-4.5); \draw (4.5,-3.5) -- (4.5,-4.5); \laction(1.5,-4.5)[0,1]; \raction(3.5,-4.5)[0,1]; \draw (1,-3.5) -- (1,-4.5); \map(1,-4.5)[\scriptstyle S];
\end{scope}
\begin{scope}[xshift=16.25cm, yshift=-2.8cm]
\node at (0,-0.5){=};
\end{scope}
\begin{scope}[xshift=16.85cm,  yshift=0cm]
\comult(1.5,0)[2,1.5]; \flip(2.5,-1.5)[1]; \comult(4,-0.5)[1,1]; \draw (4,0) -- (4,-0.5); \draw (0.5,-1.5) -- (0.5,-2.5); \comult(0.5,-2.5)[1,1]; \comult(2.5,-2.5)[1,1]; \draw (0,-3.5) -- (0,-4.5); \flip(1,-3.5)[1]; \draw (3,-3.5) -- (3,-4.5); \laction(0,-4.5)[0,1]; \laction(2,-4.5)[0,1];  \draw (4.5,-1.5) -- (4.5,-2.5); \raction(3.5,-2.5)[0,1]; \draw (3.5,-3.5) -- (3.5,-6.5); \map(1,-5.5)[\scriptstyle S]; \draw (3,-5.5) -- (3,-6.5);
\end{scope}
\begin{scope}[xshift=21.9cm, yshift=-2.8cm]
\node at (0,-0.5){=};
\end{scope}
\begin{scope}[xshift=22.75cm,  yshift=-1.7cm]
\solbraid(1,0)[\scriptstyle r]; \comult(0.5,-1)[1,1]; \map(0,-2)[\scriptstyle S]; \draw (1,-2) -- (1,-3); \draw (1.5,-1) -- (1.5,-3);
\end{scope}
\begin{scope}[xshift=24.6cm, yshift=-2.8cm]
\node at (0,-0.5){,};
\end{scope}
\end{tikzpicture}
$$
by the definitions of the arrows $r_2$ and  $R^{t_1}$,  and the facts that $c$ is a natural isomorphism, $\Delta$ is cocommutative, $\Delta$ and $\Delta_{X^2}$ are coassociative, and $\sigma$ is a coalgebra homomorphism.
\end{proof}

\begin{proof}[Proof of Theorem~\ref{thm:main}] Let $L\coloneqq T(Z)$ be the free bialgebra on $Z$ (where $Z$ is at the beginning of this section) and let $\ov{L}$ be the free Hopf algebra over the coalgebra $X$ (see Proposition~\ref{const de antipoda} and Remark~\ref{Z especial}). We define $r_L\colon L\otimes L\to L\otimes L$ on $Z^0\otimes Z^n$ and on $Z^m\otimes Z^0$ by $c$, and on $Z^m\otimes Z^n$ with $m,n\geq1$, by
$$
r_L\coloneqq r_e(m,m+n)\circ r_e(m-1,m+n-1)\circ\cdots\circ r_e(2,n+2)\circ r_e(1,n+1),
$$
where $r_e(i,k) \coloneqq r_{e_{i,i+1}}\circ r_{e_{i+1,i+2}}\circ\cdots\circ r_{e_{k-1,k}}$. It is easy to see that $r_L$ is a braiding satisfying \eqref{eq:bo2}--\eqref{eq:bo5}. We assert that $r_L$ induces a map $\ov{r}\colon\ov{L}\otimes\ov{L}\to\ov{L}\otimes \ov{L}$. Let $\pi \colon L \to \ov{L}$ be the canonical morphism and let $f_1$, $f_2$ and $f_3$ be as at the discussion below Corollary~\ref{cociente de Hopf algebra}. By Remark~\ref{rem:producto_tensorial_colimite} in order to check this it suffices to verify that $(\pi\ot \pi)\circ r_L$ coequalize the morphisms $\hat{f}_1\ot L$, $\hat{f}_2\ot L$ and $\hat{f}_3\ot L$, where $\hat{f}_i\coloneqq m\circ (m\ot L)\circ (L\ot f_i\ot L)$, and coequalize the morphisms $L\ot\hat{f}_1 $, $L\ot \hat{f}_2$ and $L\ot\hat{f}_3$. We leave this task to the reader (use Lemma~\ref{lem:chinos}). Clearly $\ov{r}$ is a braiding satisfying \eqref{eq:bo2}--\eqref{eq:bo5}. Let $\ov{\pi}\colon \ov{L}\to H$ be the canonical epimorphism and let $\ov{m}$ be the multiplication map of $\ov{L}$. Using the very definition of $\ov{r}$ and $H$ it is easy to see that $\ov{\pi}\circ \ov{m}\circ \ov{r} = \ov{\pi}\circ \ov{m}$. Hence
\begin{align*}
(\ov{\pi}\otimes\ov{\pi})\circ \ov{r}\circ (\ov{m}\otimes\ov{L}) &= (\ov{\pi}\otimes\ov{\pi})\circ (\ov{L}\otimes \ov{m})\circ (\ov{r}\otimes\ov{L})\circ (\ov{L}\otimes\ov{r})\\
&=(\ov{\pi}\otimes\ov{\pi})\circ (\ov{L}\otimes \ov{m})\circ (\ov{L}\otimes\ov{r})\circ (\ov{r}\otimes\ov{L})\circ (\ov{L}\otimes\ov{r})\\
&=(\ov{\pi}\otimes\ov{\pi})\circ (\ov{L}\otimes \ov{m})\circ (\ov{r}\otimes\ov{L})\circ (\ov{L}\otimes\ov{r})\circ (\ov{r}\otimes\ov{L})\\
&=(\ov{\pi}\otimes\ov{\pi})\circ \ov{r}\circ (\ov{m}\otimes\ov{L})\circ(\ov{r}\otimes\ov{L}),
\shortintertext{and similarly}
(\ov{\pi}\otimes\ov{\pi})\circ\ov{r}\circ (\ov{L}\otimes \ov{m}) &= (\ov{\pi}\otimes\ov{\pi})\circ\ov{r}\circ (\ov{L}\otimes \ov{m})\circ (\ov{L}\otimes\ov{r}).
\end{align*}
Consequently $\ov{r}$ induces a morphism $r_H\colon H\otimes H\to H\otimes H$, which is evidently a braiding that satisfies \eqref{eq:bo2}--\eqref{eq:bo5}. We claim that $r_H$ is a coalgebra isomorphism. It is clear that $r_H$ is invertible. Let $p\colon L\to H$ be the canonical epimorphism $p\coloneqq \ov{\pi}\circ\pi$. In order to prove the claim we must show that
\begin{equation}\label{r_H es morfismo de coalgebras}
\Delta_{H^2}\circ r_H\circ (p\ot p)= (r_H\ot r_H)\circ \Delta_{H^2}\circ (p\ot p)
\end{equation}
on $Z^m\ot Z^n$ for all $m,n\in \mathds{N}_0$. Clearly this is true for $m=0$ and $n\in \mathds{N}_0$, for $m\in \mathds{N}_0$ and $n=0$, and for $m=n=1$. Assume that it is true for $n=1$ and $m\in \mathds{N}$. Then
$$
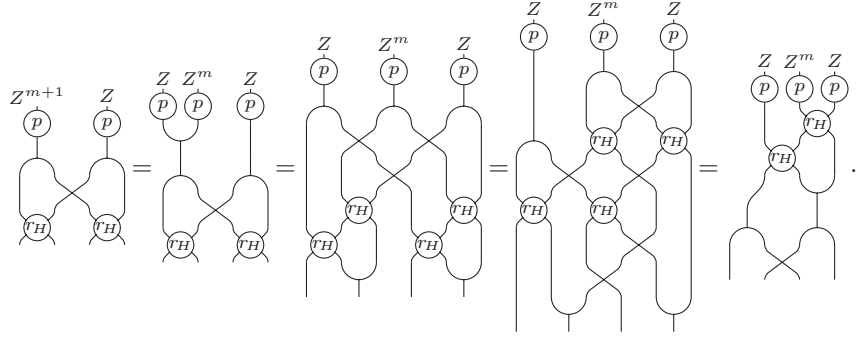
\begin{tikzpicture}[scale=0.46]
\def\mult(#1,#2)[#3,#4]{\draw (#1,#2) arc (180:360:0.5*#3 and 0.5*#4) (#1+0.5*#3, #2-0.5*#4) -- (#1+0.5*#3,#2-#4)}
\def\counit(#1,#2){\draw (#1,#2) -- (#1,#2-0.93) (#1,#2-1) circle[radius=2pt]}
\def\comult(#1,#2)[#3,#4]{\draw (#1,#2) -- (#1,#2-0.5*#4) arc (90:0:0.5*#3 and 0.5*#4) (#1,#2-0.5*#4) arc (90:180:0.5*#3 and 0.5*#4)}
\def\laction(#1,#2)[#3,#4]{\draw (#1,#2) .. controls (#1,#2-0.555*#4/2) and (#1+0.445*#4/2,#2-1*#4/2) .. (#1+1*#4/2,#2-1*#4/2) -- (#1+2*#4/2+#3*#4/2,#2-1*#4/2) (#1+2*#4/2+#3*#4/2,#2)--(#1+2*#4/2+#3*#4/2,#2-2*#4/2)}
\def\map(#1,#2)[#3]{\draw (#1,#2-0.5)  node[name=nodemap,inner sep=0pt,  minimum size=10pt, shape=circle, draw]{$#3$} (#1,#2)-- (nodemap)  (nodemap)-- (#1,#2-1)}
\def\solbraid(#1,#2)[#3]{\draw (#1,#2-0.5)  node[name=nodemap,inner sep=0pt,  minimum size=9pt, shape=circle,draw]{$#3$}
(#1-0.5,#2) .. controls (#1-0.5,#2-0.15) and (#1-0.4,#2-0.2) .. (#1-0.3,#2-0.3) (#1-0.3,#2-0.3) -- (nodemap)
(#1+0.5,#2) .. controls (#1+0.5,#2-0.15) and (#1+0.4,#2-0.2) .. (#1+0.3,#2-0.3) (#1+0.3,#2-0.3) -- (nodemap)
(#1+0.5,#2-1) .. controls (#1+0.5,#2-0.85) and (#1+0.4,#2-0.8) .. (#1+0.3,#2-0.7) (#1+0.3,#2-0.7) -- (nodemap)
(#1-0.5,#2-1) .. controls (#1-0.5,#2-0.85) and (#1-0.4,#2-0.8) .. (#1-0.3,#2-0.7) (#1-0.3,#2-0.7) -- (nodemap)
		}
\def\flip(#1,#2)[#3]{\draw (
#1+1*#3,#2) .. controls (#1+1*#3,#2-0.05*#3) and (#1+0.96*#3,#2-0.15*#3).. (#1+0.9*#3,#2-0.2*#3)
(#1+0.1*#3,#2-0.8*#3)--(#1+0.9*#3,#2-0.2*#3)
(#1,#2-1*#3) .. controls (#1,#2-0.95*#3) and (#1+0.04*#3,#2-0.85*#3).. (#1+0.1*#3,#2-0.8*#3)
(#1,#2) .. controls (#1,#2-0.05*#3) and (#1+0.04*#3,#2-0.15*#3).. (#1+0.1*#3,#2-0.2*#3)
(#1+0.1*#3,#2-0.2*#3) -- (#1+0.9*#3,#2-0.8*#3)
(#1+1*#3,#2-1*#3) .. controls (#1+1*#3,#2-0.95*#3) and (#1+0.96*#3,#2-0.85*#3).. (#1+0.9*#3,#2-0.8*#3)
		}
\def\raction(#1,#2)[#3,#4]{\draw (#1,#2) -- (#1,#2-2*#4/2)  (#1,#2-1*#4/2)--(#1+1*#4/2+#3*#4/2,#2-1*#4/2) .. controls (#1+1.555*#4/2+#3*#4/2,#2-1*#4/2) and (#1+2*#4/2+#3*#4/2,#2-0.555*#4/2) .. (#1+2*#4/2+#3*#4/2,#2)}
\def\doublemap(#1,#2)[#3]{\draw (#1+0.5,#2-0.5) node [name=doublemapnode,inner xsep=0pt, inner ysep=0pt, minimum height=11pt, minimum width=23pt,shape=rectangle,draw,rounded corners] {$#3$} (#1,#2) .. controls (#1,#2-0.075) .. (doublemapnode) (#1+1,#2) .. controls (#1+1,#2-0.075).. (doublemapnode) (doublemapnode) .. controls (#1,#2-0.925)..(#1,#2-1) (doublemapnode) .. controls (#1+1,#2-0.925).. (#1+1,#2-1)}
\def\doublesinglemap(#1,#2)[#3]{\draw (#1+0.5,#2-0.5) node [name=doublesinglemapnode,inner xsep=0pt, inner ysep=0pt, minimum height=11pt, minimum width=23pt,shape=rectangle,draw,rounded corners] {$#3$} (#1,#2) .. controls (#1,#2-0.075) .. (doublesinglemapnode) (#1+1,#2) .. controls (#1+1,#2-0.075).. (doublesinglemapnode) (doublesinglemapnode)-- (#1+0.5,#2-1)}
\def\ractiontr(#1,#2)[#3,#4,#5]{\draw (#1,#2) -- (#1,#2-2*#4/2)  (#1,#2-1*#4/2) node [inner sep=0pt, minimum size=3pt,shape=isosceles triangle,fill, shape border rotate=#5] {}  --(#1+1*#4/2+#3*#4/2,#2-1*#4/2) .. controls (#1+1.555*#4/2+#3*#4/2,#2-1*#4/2) and (#1+2*#4/2+#3*#4/2,#2-0.555*#4/2) .. (#1+2*#4/2+#3*#4/2,#2)  }
\def\rack(#1,#2)[#3]{\draw (#1,#2-0.5)  node[name=nodemap,inner sep=0pt,  minimum size=7.5pt, shape=circle,draw]{$#3$} (#1-1,#2) .. controls (#1-1,#2-0.5) and (#1-0.5,#2-0.5) .. (nodemap) (#1,#2)-- (nodemap)  (nodemap)-- (#1,#2-1)}
\def\rackmenoslarge(#1,#2)[#3]{\draw (#1,#2-0.5)  node[name=nodemap,inner sep=0pt,  minimum size=7.5pt, shape=circle,draw]{$#3$} (#1-1.5,#2+0.5) .. controls (#1-1.5,#2-0.5) and (#1-0.5,#2-0.5) .. (nodemap) (#1,#2)-- (nodemap)  (nodemap)-- (#1,#2-1)}
\def\racklarge(#1,#2)[#3]{\draw (#1,#2-0.5)  node[name=nodemap,inner sep=0pt,  minimum size=7.5pt, shape=circle,draw]{$#3$} (#1-2,#2+0.5) .. controls (#1-2,#2-0.5) and (#1-0.5,#2-0.5) .. (nodemap) (#1,#2)-- (nodemap)  (nodemap)-- (#1,#2-1)}
\def\rackmaslarge(#1,#2)[#3]{\draw (#1,#2-0.5)  node[name=nodemap,inner sep=0pt,  minimum size=7.5pt, shape=circle,draw]{$#3$} (#1-2.5,#2+0.5) .. controls (#1-2.5,#2-0.5) and (#1-0.5,#2-0.5) .. (nodemap) (#1,#2)-- (nodemap)  (nodemap)-- (#1,#2-1)}
\def\rackextralarge(#1,#2)[#3]{\draw (#1,#2-0.75)  node[name=nodemap,inner sep=0pt,  minimum size=7.5pt, shape=circle, draw]{$#3$} (#1-3,#2+1) .. controls (#1-3,#2-0.75) and (#1-0.5,#2-0.75) .. (nodemap) (#1,#2)-- (nodemap)  (nodemap)-- (#1,#2-1.5)}
\begin{scope}[xshift=0cm, yshift=-2.5cm]
\node at (0.5,0.3) {$\scriptstyle Z^{m+1}$}; \node at (2.5,0.3) {$\scriptstyle Z$}; \map(0.5,0)[\scriptstyle p]; \map(2.5,0)[\scriptstyle p]; \comult(0.5,-1)[1,1]; \comult(2.5,-1)[1,1]; \flip(1,-2)[1]; \draw (0,-2) -- (0,-3); \draw (3,-2) -- (3,-3); \solbraid(0.5,-3)[\scriptstyle r_{\hspace{-1pt}H}]; \solbraid(2.5,-3)[\scriptstyle r_{\hspace{-1pt} H}];
\end{scope}
\begin{scope}[xshift=3.5cm, yshift=-3.85cm]
\node at (0,-0.5){=};
\end{scope}
\begin{scope}[xshift=4.1cm, yshift=-2cm]
\node at (0,0.3) {$\scriptstyle Z$}; \node at (1,0.3) {$\scriptstyle Z^m$}; \node at (2.5,0.3) {$\scriptstyle Z$}; \mult(0,-1)[1,1]; \draw (2.5,-1) -- (2.5,-2); \map(0,0)[\scriptstyle p]; \map(1,0)[\scriptstyle p]; \map(2.5,0)[\scriptstyle p]; \comult(0.5,-2)[1,1]; \comult(2.5,-2)[1,1]; \flip(1,-3)[1]; \draw (0,-3) -- (0,-4); \draw (3,-3) -- (3,-4); \solbraid(0.5,-4)[\scriptstyle r_{\hspace{-1pt}H}]; \solbraid(2.5,-4)[\scriptstyle r_{\hspace{-1pt} H}];
\end{scope}
\begin{scope}[xshift=7.6cm, yshift=-3.85cm]
\node at (0,-0.5){=};
\end{scope}
\begin{scope}[xshift=8.2cm, yshift=-1cm]
\node at (0.5,0.3) {$\scriptstyle Z$}; \node at (2.5,0.3) {$\scriptstyle Z^m$}; \node at (4.5,0.3) {$\scriptstyle Z$}; \map(0.5,0)[\scriptstyle p]; \map(2.5,0)[\scriptstyle p]; \map(4.5,0)[\scriptstyle p]; \comult(0.5,-1)[1,1]; \comult(2.5,-1)[1,1]; \comult(4.5,-1)[1,1]; \flip(1,-2)[1]; \flip(3,-2)[1]; \flip(2,-3)[1]; \draw (1,-3) -- (1,-4);  \draw (0,-2) -- (0,-5); \solbraid(1.5,-4)[\scriptstyle r_{\hspace{-1pt}H}]; \solbraid(0.5,-5)[\scriptstyle r_{\hspace{-1pt}H}]; \mult(1,-6)[1,1]; \draw (2,-5) -- (2,-6); \draw (0,-6) -- (0,-7); \draw (4,-3) -- (4,-4); \draw (5,-2) -- (5,-4); \draw (3,-4) -- (3,-5); \solbraid(4.5,-4)[\scriptstyle r_{\hspace{-1pt}H}]; \solbraid(3.5,-5)[\scriptstyle r_{\hspace{-1pt}H}]; \draw (5,-5) -- (5,-6); \mult(4,-6)[1,1]; \draw (3,-6) -- (3,-7);
\end{scope}
\begin{scope}[xshift=13.7cm, yshift=-3.85cm]
\node at (0,-0.5){=};
\end{scope}
\begin{scope}[xshift=14.2cm, yshift=0cm]
\node at (0.5,0.3) {$\scriptstyle Z$}; \node at (2.5,0.3) {$\scriptstyle Z^m$}; \node at (4.5,0.3) {$\scriptstyle Z$}; \map(0.5,0)[\scriptstyle p]; \map(2.5,0)[\scriptstyle p]; \map(4.5,0)[\scriptstyle p];  \comult(2.5,-1)[1,1]; \comult(4.5,-1)[1,1]; \flip(3,-2)[1]; \draw (2,-2) -- (2,-3); \draw (5,-2) -- (5,-3); \solbraid(2.5,-3)[\scriptstyle r_{\hspace{-1pt}H}]; \solbraid(4.5,-3)[\scriptstyle r_{\hspace{-1pt}H}]; \draw (0.5,-1) -- (0.5,-3);  \comult(0.5,-3)[1,1]; \flip(1,-4)[1]; \draw (0,-4) -- (0,-5);  \solbraid(0.5,-5)[\scriptstyle r_{\hspace{-1pt}H}]; \solbraid(2.5,-5)[\scriptstyle r_{\hspace{-1pt}H}]; \flip(3,-4)[1]; \draw (4,-5) -- (4,-6); \flip(3,-6)[1]; \draw (2,-6) -- (2,-7); \flip(2,-7)[1]; \draw (1,-6) -- (1,-8); \mult(1,-8)[1,1]; \draw (0,-6) -- (0,-9); \draw (3,-8) -- (3,-9); \draw (5,-4) -- (5,-8); \draw (4,-7) -- (4,-8); \mult(4,-8)[1,1];
\end{scope}
\begin{scope}[xshift=19.7cm, yshift=-3.85cm]
\node at (0,-0.5){=};
\end{scope}
\begin{scope}[xshift=20.3cm, yshift=-1.5cm]
\node at (1,0.3) {$\scriptstyle Z$}; \node at (2,0.3) {$\scriptstyle Z^m$}; \node at (3,0.3) {$\scriptstyle Z$}; \map(1,0)[\scriptstyle p]; \map(2,0)[\scriptstyle p]; \map(3,0)[\scriptstyle p]; \solbraid(2.5,-1)[\scriptstyle r_{\hspace{-1pt}H}]; \draw (1,-1) -- (1,-2); \solbraid(1.5,-2)[\scriptstyle r_{\hspace{-1pt}H}]; \draw (3,-2) -- (3,-3);  \mult(2,-3)[1,1]; \draw (1,-3) .. controls (1,-3.5) and (0.5,-3.5) .. (0.5,-4); \comult(0.5,-4)[1,1]; \comult(2.5,-4)[1,1]; \draw (0,-5) -- (0,-6); \flip(1,-5)[1]; \draw (3,-5) -- (3,-6);
\end{scope}
\begin{scope}[xshift=23.85cm, yshift=-3.85cm]
\node at (0,-0.5){.};
\end{scope}
\end{tikzpicture}
$$
So equality~\eqref{r_H es morfismo de coalgebras} holds for all $m\in \mathds{N}$ and $n=1$, and a similar inductive argument proves now that~\eqref{r_H es morfismo de coalgebras} hold for all  $m,n\in \mathds{N}$, which finishes the proof of the claim.

It is now evident that $(H,r_H)$ is a braiding operator.

By the very definition of $r_H$ and $\gamma$ it is clear that $(\gamma\otimes\gamma)\circ r=r_H\circ(\gamma\otimes\gamma)$.

Let us prove the universal property. Since $\ov{L}$ is the universal cocommutative Hopf algebra of $X$, the map $f$ induces a Hopf algebra map $\widehat{f}\colon\ov{L}\to K$. The equalities
\[
\widehat{f}\circ m\circ \ov{r}=m_K\circ (\widehat{f}\otimes \widehat{f})\circ\ov{r}=m_K\circ r_K\circ (\widehat{f}\otimes\widehat{f})=m_K\circ (\widehat{f}\otimes\widehat{f}) =\widehat{f}\circ m
\]
show that $\widehat{f}$ induces a Hopf algebra morphism $\phi\colon H\to K$. Clearly $f=\phi\circ\gamma$. It remains to check that
\begin{equation}\label{prop univ}
(\phi \ot \phi) \circ r_H= r_K\circ(\phi\ot\phi).
\end{equation}
But using Corollary~\ref{braiding operators y antipodas 2} and the definition of $r_L$ it is easy to check that
$$
(\phi\circ p \ot \phi\circ p) \circ r_L= r_K\circ(\phi\circ p\ot\phi\circ p),
$$
from which Equality~\eqref{prop univ} follows immediately.
\end{proof}

\section{Solutions on coalgebras of primitive elements}
\label{examples}
Let $k$ be a field and let $V$ be a $k$-vector space. Let $X$ denote the coalgebra with underlying $k$-vector space $k\oplus V$, set of primitive elements $V$ and group-like element $1\in k$. In this section we study solutions of the braid equation on~$X$.

\begin{pro}\label{automorfismos no degenerados de coalgebras de X^2} A $k$-linear map $r\colon X^2\to X^2$ is a non-degenerate coalgebra automorphism if and only if $r(1\ot 1) = 1\ot 1$ and there exist $k$-linear maps
\begin{equation*}
g \colon V\to V, \quad h \colon V\to V,\quad \sigma_V \colon V^2\to V \quad\text{and}\quad \tau_V \colon V^2\to V
\end{equation*}
such that

\begin{enumerate}

\smallskip

\item $r(1\ot v) = g(v) \ot 1$ and $r(v\ot 1)=1\ot h(v)$ for all $v\in V$,

\smallskip

\item $r(v\ot w) = 1 \ot \tau_V(v\ot w) + g(w) \ot h(v) + \sigma_V(v\ot w) \ot  1$ for all $v,w\in V$,

\smallskip

\item the maps $g$ and $h$ are bijective.

\end{enumerate}
Moreover, under these conditions we have
\begin{equation*}
\sigma(1\ot 1)= \tau(1\ot 1)= \sigma^{-1}(1\ot 1)=\tau^{-1}(1\ot 1)=1,
\end{equation*}
and
\begin{xalignat*}{3}
&\sigma(1\ot v)= g(v), &&\sigma(v\ot 1)= 0, &&\sigma(v\ot w)=\sigma_V(v\ot w),\\
&\tau(v\ot 1)= h(v),&&\tau(1\ot v)= 0 , &&\tau(v\ot w)=\tau_V(v\ot w),\\
&\sigma^{-1}(1\ot v)= g^{-1}(v),&& \sigma^{-1}(v\ot 1)= 0,&&\sigma^{-1}(v\ot w)=-g^{-1}(\sigma_V(v\ot g^{-1}(w))),\\
&\tau^{-1}(v\ot 1)= h^{-1}(v),&& \tau^{-1}(1\ot v)= 0,&&\tau^{-1}(v\ot w)= -h^{-1}(\tau_V(h^{-1}(v)\ot w)).
\end{xalignat*}
for all $v,w\in V$.
\end{pro}

\begin{proof} Assume first that $r$ is a non-degenerate coalgebra automorphism. Then we have $r(1\ot 1)=1\ot 1$, because $1\ot 1$ is the unique group-like element in $X^2$. Consequently,
\begin{equation}\label{eq2}
\sigma(1\ot 1)= \tau(1\ot 1)= \sigma^{-1}(1\ot 1)=\tau^{-1}(1\ot 1)=1,
\end{equation}
as desired. Moreover, the equality $(\epsilon \ot \epsilon) \circ r = \epsilon \ot \epsilon$ implies that
$$
W\coloneqq (k\ot V) \oplus (V\ot V)\oplus (V\ot k)
$$
is an invariant vector subspace of $X^2$, and so
\begin{equation}\label{eq4}
\sigma(W)\subseteq V\quad\text{and}\quad \tau(W) \subseteq V.
\end{equation}
Let $v\in V$ arbitrary and let $\sigma:=(X\ot \epsilon)\circ r$ and $\tau:=(\epsilon\ot X)\circ r$. By~\eqref{eq2},
\begin{align*}
r(1\ot v) &=(\sigma\ot \tau)\circ \Delta_{X^2}(1\ot v) \\
&= (\sigma\ot \tau)(1\ot 1\ot 1\ot v + 1\ot v\ot 1\ot 1)\\
&=1\ot \tau(1\ot v) + \sigma(1\ot v) \ot 1
\end{align*}
and, similarly,
$$
r(v\ot 1) = 1\ot \tau(v\ot 1) + \sigma(v\ot 1) \ot 1.
$$
So, by~\eqref{eq4} there exist $f,g,h,i\in \End_k(V)$ such that
\begin{equation}\label{eq3}
r(1\ot v)=1\ot f(v) + g(v)\ot 1\quad\text{and}\quad r(v\ot 1)=1\ot h(v) + i(v)\ot 1.
\end{equation}
Thus,
\begin{equation}\label{eq5}
\sigma(1\ot v)=g(v),\quad \sigma(v\ot 1)=i(v),\quad \tau(1\ot v)=f(v)\quad\text{and}\quad \tau(v\ot 1)=h(v).
\end{equation}
We next prove that $g$ and $h$ are bi\-jec\-tive maps and
$$
\sigma^{-1}(1\ot v)=g^{-1}(v)\quad\text{and}\quad \tau^{-1}(v\ot 1)=h^{-1}(v).
$$
In fact by symmetry it suffices to prove this for $g$, which follows easily using that
$$
v = (\epsilon\ot X)(1\ot v) = \sigma\circ (X\ot \sigma^{-1}) \circ (\Delta\ot X)(1\ot v) = \sigma(1\ot \sigma^{-1}(1\ot v))
$$
and
$$
v = (\epsilon\ot X)(1\ot v) = \sigma^{-1}\circ (X\ot \sigma) \circ (\Delta\ot X)(1\ot v) = \sigma^{-1}(1\ot g(v)).
$$
We claim now that $i=0$. In fact, using that
$$
\Delta_{X^2}(v\ot w) = 1\ot 1\ot v\ot w + 1\ot w\ot v \ot 1 + v\ot 1\ot 1\ot w + v\ot w\ot 1\ot 1,
$$
we obtain
$$
(\sigma\ot \sigma)\circ \Delta_{X^2}(v\ot w) = 1\ot \sigma(v\ot w) + g(w)\ot i(v) + i(v)\ot g(w) + \sigma(v\ot w)\ot 1.
$$
Consequently, since
$$
\Delta\circ \sigma(v\ot w) = 1\ot \sigma(v\ot w) + \sigma(v\ot w) \ot 1,
$$
$\sigma$ is a coalgebra morphism and $g$ is bijective, we have
$$
w\ot i(v) + i(v)\ot w = 0 \quad\text{for all $v,w\in V$,}
$$
which implies $i=0$ as desired. A similar computation using that $h$ is bijective and $\tau$ is a coalgebra automorphism proves that $f=0$. Hence, by~\eqref{eq3} statement~(1) is satisfied. Let $\sigma_V\colon V^2\to V$ and~$\tau_V\colon V^2\to V$ be the maps induced by $\sigma$ and $\tau$, respectively. Using that $r = (\sigma\ot\tau)\circ \Delta_{X^2}$ and equalities~\eqref{eq2} and~\eqref{eq5} we obtain statement~(2). Moreover
$$
\sigma^{-1}(v\ot 1) = 0\quad\text{and}\quad\tau^{-1}(1\ot v)=0 \quad\text{for all $v\in V$.}
$$
In fact
$$
0 = (\epsilon\ot X)(v\ot 1) = \sigma^{-1}\circ (X\ot \sigma) \circ (\Delta\ot X)(v\ot 1) = \sigma^{-1}(v\ot 1),
$$
and the equality for $\tau$ is similar. In order to finish this part of the proof only remains to see that
$$
\sigma^{-1}(v\ot w)=-g^{-1}(\sigma_V(v\ot g^{-1}(w)))\quad\text{and}\quad \tau^{-1}(v\ot w)= -h^{-1}(\tau_V(h^{-1}(v)\ot w)).
$$
But the first equality follows from the fact that
\begin{align*}
v\ot w & = (X\ot \sigma^{-1})\circ (\Delta \ot X)\circ (X\ot \sigma)\circ (\Delta \ot X)(v\ot w)\\
&=(X\ot \sigma^{-1})\circ (\Delta \ot X)\circ (X\ot \sigma)(1\ot v\ot w + v\ot 1\ot w)\\
&= (X\ot \sigma^{-1})\circ (\Delta \ot X)(1\ot \sigma_V(v\ot w) + v\ot g(w))\\
&= (X\ot \sigma^{-1})(1\ot 1\ot \sigma_V(v\ot w) + 1\ot v\ot g(w) +  v\ot 1\ot g(w))\\
&= 1\ot g^{-1}(\sigma_V(v\ot w)) + \sigma^{-1}(v\ot g(w)) + v\ot w
\end{align*}
and the second one is similar.

\smallskip

Conversely, a direct computation shows that if $r(1\ot 1) = 1\ot 1$ and there are maps
\begin{equation*}
g \colon V\to V, \quad h \colon V\to V,\quad \sigma_V \colon V^2\to V \quad\text{and}\quad \tau_V \colon V^2\to V,
\end{equation*}
such that statements (1)--(3) are satisfied, then $r$ is a non-degenerate coalgebra automorphism.
\end{proof}

\begin{thm}\label{soluciones no degeneradas de X^2} Let $r\colon X^2 \to X^2$ be a non-degenerate coalgebra automorphism of~$X^2$ and let $g$, $h$, $\sigma_V$ and $\tau_V$ be as in Proposition~\ref{automorfismos no degenerados de coalgebras de X^2}. Then $r$ is a solution of the braid equation if and only if the following conditions are fulfilled:

\begin{enumerate}

\smallskip

\item $h\circ g=g\circ h$,

\smallskip

\item $\sigma_V\circ (g\ot g)=g\circ \sigma_V$,

\smallskip

\item $\tau_V\circ (g\ot g)=g\circ \tau_V$,

\smallskip

\item $\sigma_V\circ (h\ot h)=h\circ \sigma_V$,

\smallskip

\item  $\tau_V\circ (h\ot h)=h\circ \tau_V$,

\smallskip

\item $\sigma_V\circ (V\ot g)=g\circ \sigma_V\circ (h\ot V)$,

\smallskip

\item $\tau_V\circ (h\ot V)=h\circ \tau_V\circ (V\ot g)$,

\smallskip

\item  For all $u,v,w\in V$,
\begin{align*}
\qquad\,\quad \tau_V(\tau_V(u\ot v)\ot w)&=\tau_V(h(u)\ot\tau_V(v\ot w))\\
&+ h(\tau_V(u \ot \sigma_V(v\ot w))) + \tau_V(\tau_V(u\ot g(w))\ot h(v)),\\
\qquad\,\quad\sigma_V(u\ot \sigma_V(v\ot w)) &= \sigma_V(\sigma_V(u\ot v)\ot g(w))\\
&+ g(\sigma_V(\tau_V(u\ot v)\ot w)) + \sigma_V(g(v) \ot \sigma_V(h(u)\ot w))
\shortintertext{and}
\qquad\,\,\,\quad \tau_V(\sigma_V(u\ot v)\!\ot g(w)) & + \tau_V(g(v) \ot \sigma_V(h(u)\ot w))\\
& = \sigma_V(h(u)\ot \tau_V(v\ot w))\! + \sigma_V(\tau_V(u\ot g(w))\ot h(v)).
\end{align*}
\end{enumerate}
\end{thm}

\begin{proof} It is easy to check that $r$ satisfies the braid equation on $k^3$, $k^2\ot V$ and $V\ot k^2$; that $r$ satisfies the braid equation on $k\ot V\ot k$ if and only if $g\circ h=h\circ g$; and that $r$ satisfies the braid equation on $V\ot k\ot V$ if and only if conditions~(6) and~(7) are fulfilled. Moreover a direct computation proves that if $g\circ h=h\circ g$, then $r$ satisfies the braid equation on $k\ot V^2$ if and only if conditions~(2) and~(3) are fulfilled, and that $r$ satisfies the braid equation on $V^2\ot k$ if and only if conditions~(4) and~(5) are fulfilled. Finally, a number of obvious direct calculations proves that if conditions~(1)--(7) are fulfilled, then $r$ satisfies the braid equation on $V^3$ if and only if condition~(8) holds.
\end{proof}

\begin{rem}\label{previo} Under Conditions~(2) and~(6), items~(5) and~(7) are equivalent to
$$
\sigma_V\circ (g^{-1}\ot V) = \sigma_V\circ (h\ot V)\quad\text{and}\quad \tau_V\circ (V\ot h^{-1}) = \tau_V\circ (V\ot g),
$$
respectively.
\end{rem}

Let $\varsigma\colon V^2\to V$ be a linear map. Following a common terminology in linear al\-ge\-bra, we call the vector subspaces
\begin{align*}
&\rad_L(\varsigma)\coloneqq\{v\in V: \varsigma(v\ot w)=0 \text{ for all $w\in V$} \}
\shortintertext{and}
&\rad_R(\varsigma)\coloneqq \{w\in V: \varsigma(v\ot w)=0 \text{ for all $v\in V$}\},
\end{align*}
of $V^2$, the {\em left} and the {\em right radicals} of $\varsigma$, respectively.

\begin{rem} Let $r\colon X^2 \to X^2$ be a non-degenerate solution of the braid equation. From items~(2)--(5) of Theorem~\ref{soluciones no degeneradas de X^2} it follows that
$$
f(\rad_L(\varsigma))=\rad_L(\varsigma)\quad\text{and}\quad f(\rad_R(\varsigma))=\rad_R(\varsigma),
$$
for all $f\in \{g,h\}$ and $\varsigma\in \{\sigma_V,\tau_V\}$. Moreover, by Remark~\ref{previo}
$$
\ima(g-h^{-1})\subseteq \rad_R(\tau_V)\cap \rad_L(\sigma_V).
$$
Consequently, if $\rad_R(\tau_V)\cap \rad_L(\sigma_V)=0$, then $g=h^{-1}$. Finally, by the first and second equalities in Theorem~\ref{soluciones no degeneradas de X^2}(8), if $v\in \rad_L(\sigma_V)\cap\rad_R(\tau_V)$, then $\sigma_V(-\ot v)$ and $\tau_V(v\ot -)$ take their values in $\rad_L(\sigma_V)\cap \rad_R(\tau)$. So, $r$ induces a solution $\bar{r}\colon \ov{X}^2  \to \ov{X}^2$, where $\ov{X}\coloneqq k\oplus \frac{V}{\rad_R(\tau_V)\cap \rad_L(\sigma_V)}$.
\end{rem}

\begin{exa} If $h=g=\ide_V$ and $\sigma_V$ is an associative multiplication map, then the map $r\colon X^2 \to X^2$ defined by
\begin{itemize}

\smallskip

\item[-] $r(1\ot 1) = 1\ot 1$,

\smallskip

\item[-] $r(1\ot v) = v \ot 1$ and $r(v\ot 1)=1\ot v$ for all $v\in V$,

\smallskip

\item[-] $r(v\ot w) = - 1 \ot  \sigma_V(w\ot v) + w \ot v + \sigma_V(v\ot w) \ot  1$ for all $v,w\in V$,

\smallskip

\end{itemize}
is a non-degenerate solution of the braid equation.
Compare with~\cite{MR1726279}. 
\end{exa}

\begin{example}\label{ejemplos de asociados a racks} Let $X$, $r$, $g$, $h$, $\sigma_V$ and $\tau_V$ be as in Proposition~\ref{automorfismos no degenerados de coalgebras de X^2}. From that proposition and Theorem~\ref{soluciones no degeneradas de X^2} it follows that $(X,r)$ is a braided set associated with a rack (or, which is equivalent, that $(X,r)$ is a non-degenerate braided set with $h=\ide_V$ and $\tau_V=0$) if and only if $g$ is a $k$-linear automorphism such that
\begin{align*}
&\sigma_V\circ (g \ot g)= \sigma_V\circ (V \ot g)= g\circ \sigma_V
\shortintertext{and, for all $u,v,w\in V$,}
&\sigma_V(u\ot \sigma_V(v\ot w))= \sigma_V(\sigma_V(u\ot v)\ot g(w)) + \sigma_V(g(v) \ot \sigma_V(u\ot w)).
\end{align*}
Note that, when $g=\ide_V$, then the first two conditions are fulfilled and the last one says that the map $\sigma_V(u\ot -)$ is a derivation. In other words, if $r(1\ot v) = v\ot 1$ for all $v\in V$, then $(X,r)$ is a braided set associated with a rack if and only if $(V,\sigma_V)$ is a Leibniz algebra.
Compare with~\cite{NP}. 
\end{example}

\bibliographystyle{abbrv}
\bibliography{refs}

\end{document}